\documentclass[oneside]{amsart}
\usepackage{amsmath,amssymb,color,amsthm,graphicx,cite}
\usepackage{color,bm}
\graphicspath{{../../img_pdf/}{../img_pdf/}{./img_pdf/}}
\usepackage[T1]{fontenc}
\usepackage{enumerate} 
\usepackage[backref]{hyperref}
\RequirePackage{luatex85}
\usepackage[all]{xy}
\newtheorem{theorem}{Theorem}[section]
\newtheorem{proposition}[theorem]{Proposition}
\newtheorem{lemma}[theorem]{Lemma}
\newtheorem{corollary}[theorem]{Corollary}
\newtheorem{claim}{Claim}

\newtheorem{problem}{Problem}
\theoremstyle{definition}
\newtheorem{definition}{Definition}
\newtheorem{main}{Theorem}

\def\Z{\mathbb{Z} }

\def\R{\mathbb{R} }
\def\A{\mathbb{A} }

\def\bp{\partial^{+} }
\def\BD{\mathop{\mathrm{BD}}}
\def\nbd{neighborhood }

\def\Sv{\mathop{\mathrm{Sing}}(v)}

\def\Pv{\mathop{\mathrm{Per}}(v)}
\def\Cv{\mathop{\mathrm{Cl}}(v)}

\author{Tomoo Yokoyama}
\date{\today}
\address{Department of Mathematics, Faculty of Science, Saitama University, Shimo-Okubo 255, Sakura-ku, Saitama-shi, 338-8570 Japan\\}
\email{tyokoyama@rimath.saitama-u.ac.jp}
\thanks{The author was partially supported by JSPS Grant Number 20K03583 and 24K06733}
\subjclass[2010]{}

\makeatletter
\@namedef{subjclassname@2020}{%
  \textup{2020} Mathematics Subject Classification}
\makeatother

\title[Decompositions of surface flows]{Decompositions of surface flows}
\date{\today}


\begin{document}

\maketitle

\begin{abstract}
Flows on surfaces are one of the most fundamental and classical objects in dynamical systems, and are studied from various areas (e.g. integrable systems, differential equations, fluid mechanics). Though hyperbolic flows and recurrent flows on surfaces are classified and characterized using various topological invariants, no complete finite invariants captured both hyperbolicity and recurrence. Moreover, no topological frameworks described even generic time evaluations of gradient flows or incompressible flows (e.g. flows around a circular cylinder placed in uniform flow, solutions of Euler equations and incompressible Navier-Stokes equations). In this paper, to construct a foundation for describing fluid phenomena and capturing hyperbolicity and recurrence, under regularity for the singular point set and tameness of genus and ends, we construct complete finite invariants of flows on (possibly non-compact) surfaces by reconstructing surfaces by gluing five kinds of invariant open subsets, which are trivial flow boxes, transverse/periodic annuli, periodic M{\"o}bius bands, and locally dense Q-sets. Such invariants imply a topological framework that can convert various time evaluations of fluids into walks in graphs without losing topological information, and provide a new tool for analyzing fluid phenomena and differential equations through combinatorics and topological data analysis. Furthermore, such invariants partially revive ``Markus-Neumann theorem''.
\end{abstract}

\section{Introduction}
Flows on surfaces are one of the most fundamental and classical objects in dynamical systems and are wildly used as models of various fluid phenomena. 
However, from a dynamical system's perspective, there were no topological frameworks for describing hyperbolic and recurrent behaviors of unsteady flows on surfaces.
In addition, from a physical point of view, no topological frameworks exist that describe time-dependent flow around a circular cylinder placed in uniform flow, which are Hamiltonian flows on a punctured plane. This is due to the absence of complete finite invariants of flows on surfaces that can capture such fluid phenomena. 
Therefore, such complete finite invariants are necessary to construct a foundation for capturing hyperbolicity and recurrence and topologically describing fluid phenomena.

%

\subsection{Backgrounds}
Flows on surfaces are studied from various aspects (e.g. hyperbolic dynamical systems, recurrent dynamics, topological dynamics, integrable systems, differential equations, fluid mechanics).

\subsubsection{Background from a dynamical system's point of view}
The stability of flows, necessary to recognize time evaluations of fluids as sequences of their topologies, is widely studied. 
In fact, Andronov and Pontryagin introduced structural stability and showed the structural stability and open denseness of Morse-Smale flows on spheres \cite{andronov1937coarse}. 
Peixoto generalized the result as follows \cite{peixoto1962structural}: A vector field $v \in \chi^r(S)$ is structurally stable if and only if $v \in \Sigma^r(S)$, where $\chi^r(S)$ ($r \in \Z_{\geq 1}$) is the set of $C^r$-vector fields (with the $C^r$- topology) on an orientable closed surface $S$ and $\Sigma^r(S)$ is the subspace of $\chi^r(S)$ formed by the Morse-Smale $C^r$-vector fields on $S$.
Moreover, $\Sigma^r(S)$ is open and dense in $\chi^r(S)$.
%
It is also known that the Peixoto graph is a complete finite invariant for Morse-Smale flows without limit cycles (i.e. Morse flows), three-colour graph is a complete invariant for Morse-Smale flows \cite{oshemkov1998classification}, and an equipped graph \cite{kruglov2018topological} is a complete finite invariant for $\Omega$-stable flows on compact surfaces, which are ``Morse-Smale flow without non-existence of heteroclinic separatrices''.  

\subsubsection{Background from a recurrent dynamics' point of view}
Though ``almost'' flows on surfaces are hyperbolic from a dynamical system point of view, non-hyperbolic flows (e.g. area-preserving flows) on surfaces naturally appear from differential equations and physical points of view. 
In particular, the measurable properties of area-preserving flows are studied from various aspects  \cite{chaika2021singularity,conze2011cocycles,forni1997solutions,frkaczek2012ergodic,forni2002deviation,kanigowski2016ratner,kulaga2012self,ravotti2017quantitative,ulcigrai2011absence}. 
The study of area-preserving flows for their connection with solid-state physics and pseudo-periodic topology was initiated by Novikov \cite{novikov1982hamiltonian}. 
The orbits of such flows also arise in pseudo-periodic topology, as hyperplane sections of periodic manifolds (cf. \cite{arnol1991topological,zorich1999leaves}).
Moreover, non-hyperbolic flows with wandering domains appear on compact surfaces which are the restrictions of flows on $3$-manifolds to the boundary and are the resulting flows generated by the projections of vector fields on $3$-manifolds to two-dimensional subspaces  (e.g. \cite[Figure~18]{sakajo2020discrete}). 

\subsubsection{Background from an integrable system's point of view and physical point of view}
From a fluid mechanics point of view, the topological properties of area-preserving flows have been paid attention to \cite{bakker1991bifurcations,bro1999streamline,hartnack1999streamline}, and the topological classification of Hamiltonian flows with finitely many singular points has been investigated on a plane $\R^2$ \cite{aref1998stagnation}, on a sphere $\mathbb{S}^2$ \cite{kidambi2000streamline}, and a torus \cite{moffatt2001topology}. 
These classifications are generalized to such flows on closed surfaces from an integrable system point of view \cite{bolsinov2004integrable} and compact surfaces from a dynamical system point of view \cite{sakajo2014unique,sakajo2015transitions,sakajo2018tree,sakajo2020discrete,yokoyama2021complete,yokoyama2013word,yokoyama2021combinatorial}. 
In addition, for any positive integer $r \in \Z_{\geq 1}$, 
a Hamiltonian vector field $v \in \mathcal{H}^r(S)$ is structurally stable in $\mathcal{H}^r(S)$ if and only if $v \in \mathcal{H}^r_*(S)$,  and that $\mathcal{H}^r_*(S)$ is open dense in $\mathcal{H}^r(S)$ \cite{ma2005geometric}, where $\mathcal{H}^r(S)$ is the set of Hamiltonian $C^r$ vector fields on an orientable closed connected surface $S$ and $\mathcal{H}^r_*(S)$ is the set of Hamiltonian $C^r$ vector fields whose singular points are non-degenerate and whose saddle connection are self-connected. 
The computable representations of complete finite invariants for structurally stable Hamiltonian flows and their generic intermediate flows are constructed \cite{bolsinov2004integrable,sakajo2014unique,sakajo2015transitions,sakajo2018tree,yokoyama2021complete,yokoyama2013word}. 
The representation is implemented and applied to analyze various phenomena (e.g. vortices in heart \cite{sakajo2023topological}, meteorology \cite{uda2021identification}, oceanography \cite{sakajo2022identification}, and fluids in powder classifying devices \cite{sakajo2020discrete}). 

\subsubsection{Background from a computational point of view}

The Conley Theory allows dynamical systems to be decomposed into finite blocks, each of which is a chain recurrent or gradient-like. 
The resulting directed graph is called a Morse graph. 
Though Morse graphs of various flows become finite graphs, the finite representability of Morse graphs is crucial to using graph algorithms because only finite data can be computed. 
In fact, the Morse graphs are implemented as a computer software CHomP (Computational Homology Project software) \cite{AKKMOP2009} and there are several relative works to analyzing dynamical systems using algorithm \cite{MSW2015,HMMN2014,MMW2016}. 
The software widely analyzes global dynamics and bifurcations in several areas. 
The finite realizability of the orbit class spaces is studied \cite{BHS2011,BHSV2011}. 
Therefore, the question of how to reduce dynamical systems into finite information is essential. 
On the other hand, since the Morse graphs of any recurrent dynamics on connected manifolds are singleton, they cannot capture any information on recurrent dynamics. 

\subsubsection{Background from a topological dynamics' point of view}

There are (non-finite) complete invariants for non-wandering flows (i.e. flows without wandering domains). 
Recall that Hamiltonian (or, more generally, area-preserving) flows are examples of non-wandering flows. 
Non-wandering flows with finitely many singular points on compact surfaces are classified up to a graph-equivalence by using a topological invariant, called a Conley-Lyapunov-Peixoto graph, which is equipped with rotation and weight functions \cite{nikolaev2001non}. 


In addition, the following assertion was stated in the papers~\cite{markus1954global,neumann1975classification,neumann1976global}: 
An orbit complex (also called separatrix configuration) of a flow is also a complete invariant for the set of flows with ``finitely many separatrices'' in the sense of Markus. 
Though these papers are referred to more than one hundred papers as mentioned in \cite{buendia2018markus}, Buend{\'\i}a and L\'opez have pointed out that the assertion does not hold. 

In fact, the orbit complex is not a complete invariant for flows having  ``finitely many separatrices''  in the sense of Markus, and it is not a complete invariant even for flows with finitely many singular points but without limit separatrices in the sense of Markus \cite{buendia2018markus}. 
One of their counterexamples is a toral flow which consists of one singular point and non-recurrent orbits. 
In \cite{buendia2018markus}, the authors have modified the invariant into a complete invariant, called a separatrix configuration, for flows whose essential singular points are at most finite. 
By definition of separatrix configuration in the sense of Buend{\'\i}a and L\'opez (see  \cite[Definition 2.2]{buendia2018markus} for details), note that the separatrix configuration contains the complement of the interior of the union of non-recurrent orbits, and so that it contains infinitely many data in general. 

Therefore, we pose the following problem to justify the results by relying as much as possible on the ``Markus-Neumann theorem''.
 
\begin{problem} \label{prob:MN}
Find a large class for which the ``Markus-Neumann theorem'' holds. Or justify the ``Markus-Neumann theorem'' with as few additional conditions as possible.
\end{problem}

\subsubsection{Background from a differential equation's point of view}

Though the non-degeneracy of singular points are ``generic'' conditions from a dynamical system point of view, the no-slip boundary condition as in Figure~\ref{blowup} appearing in fluid phenomena on punctured surfaces is a fundamental condition from a differential equation point of view. 
Therefore, studies of flows with finitely many singular points on compact surfaces must be extended to analyze several fluid phenomena (e.g. flows on non-compact surfaces with no-slip boundaries). 

\subsection{Statements of main results}
In this paper, to achieve a foundation of the analysis of flows, we construct decompositions of flows on surfaces (see Theorem~\ref{lem0c}, Theorem~\ref{homogeneity}, Theorem~\ref{th:bd1}), and construct a complete finite invariant of such flows under regularity of singular points and tameness of genus and ends. 
As an application, we show that the orbit complex of any flow of finite type is a complete finite invariant 
(see Theorem~\ref{main:03}). 
To construct complete finite invariants, we need some regularity of singular points because degenerate singular points can have infinitely many parabolic sectors (see Figure~\ref{sing} and Definition~\ref{def:elli_sector}
). 
Notice that \cite[Theorem~3]{cobo2010flows} implies that every singular point of any non-wandering flows with finitely many singular points on compact surfaces is either a center or a multi-saddle, and that \cite[Theorem~A]{kibkalo2021topological} says that every isolated singular point of a gradient flow on a surface is finitely sectored (i.e. there is a \nbd consisting of the singular points and finitely many non-elliptic sectors) (see Figure~\ref{sectors_all_01} and definition of (elliptic) sectors in \S~\ref{sec:sector} for details).

Therefore, we will construct decompositions and complete finite invariants for flows possibly with finitely ``sectored-like'' singular points and with/without no-slip boundaries on compact/non-compact surfaces. 
Because the descriptions of such flows are relatively complicated, we state the main results under regularity conditions in this section and generalize them in the following sections.  

\subsubsection{Detailed descriptions of main results}

To state more precisely, we recall some concepts. 
An orbit $O$ is {\bf periodic} if it is homeomorphic to a circle, {\bf locally dense} if its closure has a nonempty interior, and {\bf exceptional} if it is neither locally dense nor proper (i.e. embedded as a submanifold). 
Denote by $\bm{\mathop{\mathrm{Per}}(v)}$  (resp. $\bm{\mathrm{P}(v)}$, $\bm{\mathrm{LD}(v)}$) the union of periodic (resp. non-recurrent, locally dense) orbits for the flow $v$.
A flow on a surface is {\bf quasi-regular} if any singular points are either multi-saddles, sinks, $\partial$-sinks, sources, $\partial$-sources, or centers (see Figure~\ref{Fig:quasiregular} and definition in \S~\ref{sec:def_sing} and \S~\ref{sec_multi-saddle}).
By definition, every quasi-regular flow on a compact surface has finitely many singular points.
\begin{figure}
\begin{center}
\includegraphics[scale=0.325]{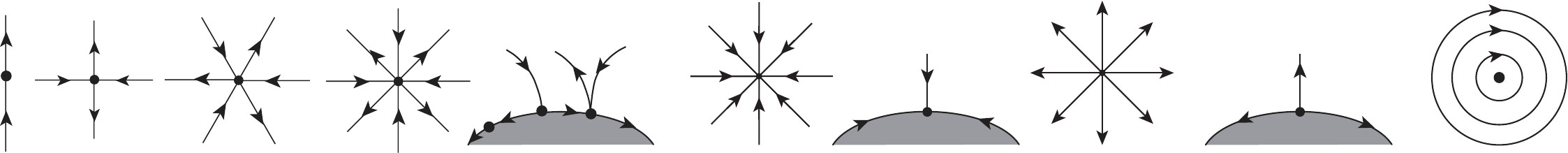}
\end{center}
\caption{Multi-saddles, a sink,  a $\partial$-sink, a source, a $\partial$-source, and a center.}
\label{Fig:quasiregular}
\end{figure} 
For a quasi-regular flow $v$ on a compact surface $S$, denote by $\bm{\mathop{\mathrm{Bd}}(v)}$ by the union of singular points, periodic orbits on $\partial S$, the closure of the union of limit cycles, orbits from or to multi-saddles, orbits between a $\partial$-source and a $\partial$-sink on the boundary $\partial S$, and exceptional orbits. 
We show that quasi-regular flows can be decomposed into finitely many invariant flow boxes, periodic annuli, transverse annuli (see Figure~\ref{fb} and Definitions~\ref{def:tr_flow_box}--\ref{def:per_annulus}), and essential subsets of locally dense orbits. 
\begin{figure}
\begin{center}
\includegraphics[scale=0.35]{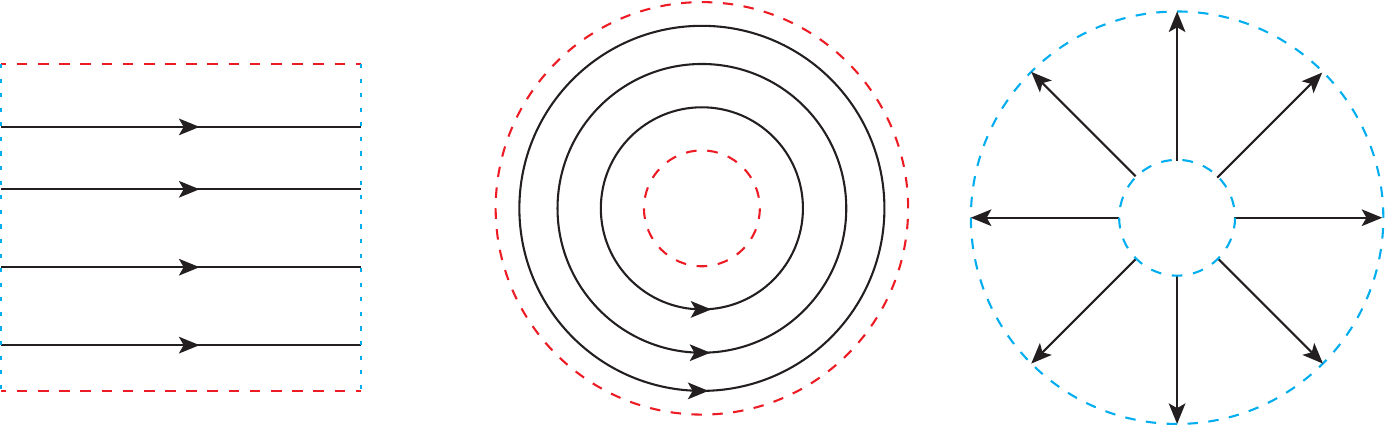}
\end{center}
\caption{Three fundamental domains: left invariant flow box, middle periodic annulus, right transverse annulus.}
\label{fb}
\end{figure}

\begin{main}\label{lem0c}
Each connected component of $S - \mathop{\mathrm{Bd}}(v)$ for a quasi-regular flow $v$ on a compact surface $S$ is one of the following invariant open subsets exclusively:
\\
$(1)$ A trivial flow box in $\mathrm{P}(v)$, whose orbit space is an open interval,
\\
$(2)$ An annulus in $\mathrm{P}(v)$, whose orbit space is a circle,
\\
$(3)$ A torus in $\mathop{\mathrm{Per}}(v)$, whose orbit space is a circle,
\\
$(4)$ A Klein bottle in $\mathop{\mathrm{Per}}(v)$, whose orbit space is an interval,
\\
$(5)$ An annulus in $\mathop{\mathrm{Per}}(v)$, whose orbit space is an open interval,
\\
$(6)$ A M{\"o}bius band in $\mathop{\mathrm{Per}}(v)$, whose orbit space is an interval,
or
\\
$(7)$ An essential subset in $\mathrm{LD}(v)$, whose orbit class space is a singleton.

Moreover, the following statements hold: 
\\
{\rm(a)}  For any connected component $U$ in $\mathrm{P}(v) \setminus \mathop{\mathrm{Bd}}(v)$ and for any points $x,y \in U$, we have $\alpha(x) = \alpha(y)$ and $\omega(x) = \omega(y)$. 
\\
{\rm(b)} The orbit closure of a point in a connected component in $\mathrm{LD}(v)$ is the orbit closure of any point in the connected component. 
\\
{\rm(c)} The boundary $\partial O(x)$ of a point $x \in \mathrm{LD}(v)$ consists of finitely many multi-saddle separatrices and singular points. 
\\
{\rm(d)} Any boundary component of a connected component of $S - \mathop{\mathrm{Bd}}(v)$ which does not intersect $\mathrm{E}(v)$ is a finite union of closed orbits, separatrices from or to multi-saddles, and separatrices on $\partial S$ between $\partial$-sources and  $\partial$-sinks. 
\end{main}

%
Using the decomposition in the previous theorem, a complete finite invariant is constructed for quasi-regular flows with finitely many limit cycles on punctured spheres (or, more generally, flows of finite type on compact surfaces) (see Theorem~\ref{injection} for details). 
As mentioned above, from physical and differential equation points of view, the no-slip boundary condition and non-compactness appearing in fluid phenomena on punctured surfaces are fundamental. 
Therefore, we extend Theorem~\ref{lem0c} for flows with no-slip boundaries on non-compact surfaces (or more general flows) (see Theorem~\ref{cor:bd1_1} and Theorem~\ref{th:bd1} for details). 
Theorem~\ref{cor:bd1_1} is applied to topologically characterize Hamiltonian flows on unbounded surfaces \cite{yokoyama2021topological}. 
%

As mentioned, the ``Markus-Neumann theorem'' is not true even for a toral flow that consists of one singular point and non-recurrent orbits.  
To answer Problem~\ref{prob:MN} affirmatively, we show that the ``Markus-Neumann theorem'' is true in the space of flows of finite type, 
which contains Morse-Smale flows and generic Hamiltonian flows on compact surfaces. 
Here, a quasi-regular flow is {\bf of finite type} if every recurrent orbit is closed and there are at most finitely many limit cycles.
In other words, using the decompositions, we show the following completeness. 

\begin{main}\label{main:03}
The orbit complex (see \S~\ref{sec:ord_cpx} for the definition) is a complete invariant for a flow of finite type on an orientable compact surface. 
\end{main}


Applying Theorem~\ref{lem0c}, we can reduce flows of finite type on compact surfaces into labeled graphs without loss of topological information as follows.  

\begin{main}\label{main:04}
There is a complete finite invariant, which is a labeled graph, for the space of flows of finite type on compact surfaces.  
\end{main}

Indeed, such a complete finite invariant is constructed in Theorem~\ref{injection}. 
The simplified complete finite invariant for the space of flows generated by structurally stable Hamiltonian vector fields is implemented and applied to analyze various phenomena (e.g. vortices in heart \cite{sakajo2023topological}, meteorology \cite{uda2021identification}, oceanography \cite{sakajo2022identification}, and fluids in powder classifying devices \cite{sakajo2020discrete}). 
Furthermore, the combinatorial structure of the ``moduli space'' of Hamiltonian (resp. gradient) vector fields is analyzed using variations of such complete finite invariants. 
For instance, under the conditions of the non-existence of creations and annihilations of singular points, the space of Hamiltonian (resp. gradient) vector fields on compact surfaces has non-contractible connected components by using combinatorics and simple homotopy theory \cite{yokoyama2021combinatorial} (resp. \cite{yokoyama2022combinatorial}).

The present paper consists of thirteen sections.
In the next section, as preliminaries, we introduce some fundamental concepts.
In \S~3 and 4, the properties of surface flows are described to show Theorem~\ref{lem0c}. 
In \S~5, Theorem~\ref{lem0c} is demonstrated. 
The remaining sections deal with three independent applications.
In \S~6--8, a finite enumerable complete invariant for surface flows of finite type is constructed explicitly, which implies that the ``Markus-Neumann theorem'' is true in a large class under mild conditions, which contains Morse-Smale flows and generic Hamiltonian flows on compact surfaces. 
In particular, Theorem~\ref{main:03} is proved.  
In \S~9, we apply our decompositions to Hamiltonian flows and Morse-Smale flows, which are the foundation for the theory of analyzing flows such as incompressible fluids and potential flows to construct computable representations of complete invariants.
In \S~10--13, the results in \S~5 are extended to flows with non-slip conditions and flows on non-compact surfaces, to construct the foundation to describe and analyze a wide class of differential equations and fluid phenomena.
In fact, in \S~10--12, we extend results in \S~5 to those for flows with no-slip boundary condition or more general flows.
In \S~13,
we extend results in \S~12 to those for flows on non-compact surfaces.
In Appendix, we demonstrate some fundamental statements. 
In particular, we show a characterization of a center, which complements studies of germs by Bendixson \cite{bendixson1901courbes}, Dumortier \cite{dumortier1977singularities}, and Roussarie \cite{roussarie2021top}.

\section{Preliminaries}

\subsection{Definitions from topology}
Let $A$ be a subset of a topological space $X$. 
Denoted by $\overline{A}$ the closure of $A$, by $\mathrm{int}A$ the interior of $A$, and by $\partial A := \overline{A} - \mathrm{int}A$ the boundary of $A$, where $B - C$ is used instead of $B \setminus C$ when $C \subseteq B$. 
We define the border $\bm{\partial^-  A}$ by $A - \mathrm{int}A$ of $A$.
Denote by $\bm{\bp A} := \overline{A} - A$ the coborder (or frontier) of $A$.
Then $\partial A = \partial^-  A \sqcup \bp A$, where $\sqcup$ denotes a disjoint union.
A {\bf boundary component} of a subset $A$ is a connected component of the boundary $\partial A$ of $A$. 

\subsubsection{$T_0$-separation axiom}
A point $x$ of a topological space $X$ is $T_0$ (or Kolmogorov) if for any point $y \neq x \in X$ there is an open subset $U$ of $X$ such that $|\{x, y \} \cap U| =1$, where $|A|$ is the cardinality of a subset $A$.
A topological space is $T_0$ if each point is $T_0$.

\subsubsection{$T_0$-tification}
By a decomposition, we mean a family $\mathcal{F}$ of pairwise disjoint nonempty subsets of a set $X$ such that $X = \bigsqcup \mathcal{F}$, where $\bigsqcup$ denotes a disjoint union. 
For a topological space $X$, define the {\bf class} $\hat{x} := \{ y \in X \mid \overline{\{ x \}} = \overline{\{ y \}} \}$ for any point $x \in X$ and a decomposition $\hat{X} := \{ \hat{z} \mid z \in X \}$ of classes. 
Then the decomposition $\hat{X}$ is a $T_0$ space as a quotient space, which is called the $T_0$-tification (or Kolmogorov quotient) of $X$.

\subsubsection{Dimensions}
By {\bf dimension}, we mean Lebesgue covering dimension on a topological space.  
According to Urysohn's theorem, the Lebesgue covering dimension, the large inductive dimension, and the small inductive dimension correspond to separable metrizable spaces. 
By definition, a topological space is zero-dimensional if and only if it has a base consisting of closed and open subsets.

\begin{definition}
A compact metrizable space $X$ whose small inductive dimension is $n > 0$ is an $n$-dimensional {\bf Cantor manifold} \cite{urysohn1925memoire} if the complement $X - L$ for any closed subset $L$ of $X$ whose small inductive dimension is less than $n - 1$ is connected.
\end{definition}
Hurewicz-Menger-Tumarkin Theorem in dimension theory says that a connected topological manifold is a Cantor manifold. 
In addition, it is known that any open connected subset of a homogeneous locally compact metrizable space is a Cantor manifold \cite[Theorem~2.1]{krupski1991recent}. 
Recall that a topological space $X$ is homogeneous if for each two points $x, y \in X$ there is a homeomorphism $h \colon  X \to X$ with $h(x)=y$, and that $X$ is locally compact if any point of $X$ has a compact neighborhood.
Then any open connected subset of a compact manifold is a Cantor manifold \cite{hurewicz1928dimension,tumarkin1928structure,urysohn1925memoire} (cf.  \cite[p.93. Example VI 11]{hurewicz2015dimension}).

\subsubsection{Curves and loops}
A curve is a continuous mapping $C: I \to X$ where $I$ is a non-degenerate connected subset of a circle $\mathbb{S}^1$. 
Here, a non-degenerate subset contains at least two points. 
A curve is simple if it is injective.
We also denote by $C$ the image of a curve $C$.
Denote by $\partial C := C(\partial I)$ the boundary of a curve $C$ if $C$ can be extended into a continuous map whose domain is $I \cup \partial I$, where $\partial I$ is the boundary of $I \subset \mathbb{S}^1$. 
Put $\mathop{\mathrm{int}} C := C \setminus \partial C$ if $\partial C$ is defined. 
A simple curve is a simple closed curve if its domain is $\mathbb{S}^1$ (i.e. $I = \mathbb{S}^1$).
A simple closed curve is also called a {\bf loop}. 
An {\bf arc} is a simple curve whose domain is an interval. 


\subsubsection{Essential subsets}

Let $S$ be a compact surface and $S^*$ the resulting surface from $S$ by collapsing all boundary components into singletons. 
In other words, the resulting surface $S^*$ is a quotient space $S/\sim^*$, where $x \sim^* y$ if there is a boundary component of $S$ which contains $x$ and $y$. 
For a subset $U$ of $S$, put the resulting subset $U^* := U/\sim^*$. 
We call that a subset $A$ in $S$ is {\bf inessential} if there is a disjoint union of open disks in $S^*$ which is a \nbd of the resulting subset $A^*$ in $S^*$. 
A subset in $S$ is {\bf essential} if it is not inessential. 
Notice that every loop on the boundary of a compact surface is inessential, and that a loop $\gamma$ in a compact surface $S$ with $\gamma \cap \partial S = \emptyset$ is essential if and only if $\gamma$ is not null homotopic in $S^*$. 


\subsection{Defitions from dynamical systems}
A {\bf flow} is a continuous $\R$-action on a topological space.
Let $v : \R \times X \to X$ be a flow on a topological space $X$.
For $t \in \R$, define $v_t : X \to X$ by $v_t := v(t, \cdot )$.
For a point $x$ of $X$, we denote by $O(x)$ the orbit of $x$, $O^+(x)$ the positive orbit (i.e. $O^+(x) := \{ v_t(x) \mid t > 0 \} = v(\R_{>0},x)$), and $O^-(x)$ the negative orbit (i.e. $O^-(x) := \{ v_t(x) \mid t < 0 \} = v(\R_{<0},x)$).
Recall that a point $x$ of $X$ is singular if $x = v_t(x)$ for any $t \in \R$, and is periodic if there is a positive number $T > 0$ such that $x = v_T(x)$ and $x \neq v_t(x)$ for any $t \in (0, T)$.
A singular point is isolated if there is its \nbd which contains no singular points except it.  

\begin{definition}
A point is {\bf closed} if it is either singular or periodic. 
\end{definition}
An orbit is singular (resp. periodic, closed) if it contains a singular (resp. periodic, closed) point. 
Denote by $\mathop{\mathrm{Sing}}(v)$ the set of singular points and by $\mathop{\mathrm{Per}}(v)$ (resp. $\mathop{\mathrm{Cl}}(v)$) the union of periodic (resp. closed) orbits.

\subsubsection{Wandering domains and non-wandering points}
A point is wandering if there are its neighborhood $U$ and a positive number $N$ such that $v_t(U) \cap U = \emptyset$ for any $t > N$. 
Such $U$ is called a {\bf wandering domain}. 
A point is {\bf non-wandering} if it is not wandering (i.e. for any its neighborhood $U$ and for any positive number $N$, there is a number $t \in \mathbb{R}$ with $|t| > N$ such that $v_t(U) \cap U \neq \emptyset$).
Denote by $\Omega (v)$ the set of non-wandering points, called the {\bf non-wandering set}. 
A flow is {\bf non-wandering} if the non-wandering set is the whole space. 

\subsubsection{Orbit arcs}
An {\bf orbit arc} is an arc contained in an orbit. 
A maximal orbit arc is an orbit arc that is maximal with respect to the inclusion order. 

\subsubsection{Limit sets of points and orbits}
The $\omega$-limit set of a point $x$ is $\omega(x) := \bigcap_{n\in \mathbb{R}}\overline{\{v_t(x) \mid t > n\}} = \bigcap_{n\in \mathbb{R}} \overline{v(\R_{>n},x)}$. 
The $\alpha$-limit set of a point $x$ is $\alpha(x) := \bigcap_{n\in \mathbb{R}}\overline{\{v_t(x) \mid t < n\}} = \bigcap_{n\in \mathbb{R}} \overline{v(\R_{<n},x)}$. 
For an orbit $O$, define $\omega(O) := \omega(x)$ and $\alpha(O) := \alpha(x)$ for some point $x \in O$.
Note that an $\omega$-limit (resp. $\alpha$-limit) set of an orbit is independent of the choice of point in the orbit. 

\subsubsection{Recurrence and Poisson stability}
A point $x$ of $X$ is {\bf Poisson stable} (or strongly recurrent) if $x \in \omega(x) \cap \alpha(x)$,  a point $x$ is positively recurrent (or positively Poisson stable) if $x \in \omega(x)$, and a point $x$ is negatively recurrent (or negatively Poisson stable) if $x \in \alpha(x)$. 
\begin{definition}
A point $x$ is {\bf recurrent} (cf. \cite{marzougui2002area}) (or positively or negatively Poisson stable) if $x \in \omega(x) \cup \alpha(x)$.
\end{definition}
Denote by $\mathrm{R}(v)$ (resp. $\mathrm{P}(v)$) the set of non-closed recurrent points (resp. non-recurrent) points of the flow $v$.
Notice that some authors refer to the property $x \in \omega(x) \cap \alpha(x)$ as recurrence (cf. \cite{nikolaev2013foliations}).

\subsubsection{Invariance, saturation, orbit class, and orbit class space}
A subset $A$ is positive invariant if $\bigcup_{x \in A} O^+(x) \subseteq A$. 
\begin{definition}
A subset is {\bf invariant} (or {\bf saturated}) if it is a union of orbits. 
\end{definition}
The saturation of a subset is the union of orbits intersecting it.
Denote by $\mathrm{Sat}_v(A)$ or $v(A)$ the saturation of a subset $A$.
Then $\mathrm{Sat}_v(A) = v(A) = \bigcup_{a \in A} O(a)$.

\subsubsection{Orbit classes and quotient spaces of flows}

\begin{definition}
The {\bf {\rm(}orbit{\rm)} class} $\hat{O}$ of an orbit $O$ is the union of orbits each of whose orbit closure corresponds to $\overline{O}$ {\rm(i.e.} $\hat{O} = \{ y \in X \mid \overline{O(y)} = \overline{O} \}${\rm)}.
\end{definition}
For a flow $v$ on a topological space $X$, the {\bf orbit space} $T/v$ (resp. {\bf orbit class space} $T/\hat{v}$) of an invariant subset $T$ of $X$ is a quotient space $T/\sim$ defined by $x \sim y$ if $O(x) = O(y)$ (resp. $\overline{O(x)} = \overline{O(y)}$). 
Notice that an orbit space $T/v$ is the set $\{ O(x) \mid x \in T \}$ with the quotient topology. 
Moreover, the orbit class space $T/\hat{v}$ is the set $\{ \hat{O}(x) \mid x \in T \}$ with the quotient topology. 
Note that the orbit class space $X/{\hat{v}}$ is  the $T_0$-tification of the orbit space $X/v$, and that 
$T/v$ (resp. $T/\hat{v}$) is a subset of the orbit space $X/v$ (resp. orbit class space $X/\hat{v}$).

\subsubsection{Separatrices}
A {\bf separatrix} is a non-singular orbit whose $\alpha$-limit or $\omega$-limit set is a singular point.
A separatrix is {\bf connecting} if each of its $\omega$-limit set and $\alpha$-limit sets is a singular point.

\subsection{Concepts from flows on surfaces}
By a {\bf surface},  we mean a paracompact two-dimensional manifold, that does not need to be orientable.
From now on, we suppose that flows are on surfaces unless otherwise stated.
Let $v$ be a flow on surface $S$. 

\subsubsection{Topological properties of orbits}
From codimension one foliation theory, we define the following concepts. 
\begin{definition}
An orbit is {\bf proper} if it is an embedded submanifold, {\bf locally dense} if its closure has a nonempty interior, and {\bf exceptional} if it is neither proper nor locally dense.  
\end{definition}

Denote by $\mathrm{LD}(v)$ (resp. $\mathrm{E}(v)$) the union of locally dense (resp. exceptional) orbits for the flow $v$.
We have the following observation. 

\begin{lemma}[Lemma~2.1\cite{yokoyama2023topological}]\label{lem:decomp}
The following statements hold for a flow $v$ on a paracompact manifold $S$: 
\\
{\rm(1)} The set $\mathrm{P}(v)$ of non-recurrent points is the union of non-closed proper orbits. 
\\
{\rm(2)} The subset $\mathop{\mathrm{Cl}}(v) \sqcup \mathrm{LD}(v) \sqcup \mathrm{E}(v)$ is the set of recurrent points. 
\\
{\rm(3)} The subset $\mathrm{R}(v) = \mathrm{LD}(v) \sqcup \mathrm{E}(v)$ is the union of non-proper orbits. 
\\
{\rm(4)} $S = \mathop{\mathrm{Cl}}(v) \sqcup \mathrm{P}(v) \sqcup \mathrm{R}(v)$. 
\end{lemma}

Note that an orbit on a paracompact manifold (e.g. a surface) is proper if and only if it has a neighborhood in which the orbit is closed \cite{yokoyama2019properness}.  
Note that the closure $\mathcal{M}$ of every exceptional orbit is transversely a Cantor set (i.e.
there is a small neighborhood $U$ of a recurrent point of $\mathcal{M}$ such that $\mathcal{M} \cap U$ is a product of an open interval and a Cantor set) because a Cantor set is characterized as a compact metrizable perfect totally disconnected space \cite{yokoyama2021poincare}. 

\subsubsection{Properties for points}

A point is proper (resp. locally dense, exceptional) if so is its orbit. 
By \cite[Proposition~2.2]{yokoyama2016topological} and \cite[Proposition~2.3]{yokoyama2016topological}, if $x \in \mathrm{LD}(v)$ (resp. $x \in \mathrm{E}(v)$) then $\hat{O}(x) = \overline{O(x)} \cap \mathrm{LD}(v )= \overline{O(x)} \setminus (\mathop{\mathrm{Sing}}(v) \sqcup \mathrm{P}(v))$ (resp. $\hat{O}(x) = \overline{O(x)} \cap \mathrm{E}(v) = \overline{O(x)} \setminus (\mathop{\mathrm{Sing}}(v) \sqcup \mathrm{P}(v))$). 

\subsubsection{Q-sets}
A {\bf quasi-minimal set} (or a {\bf Q-set} for brevity) is defined to be the orbit closure of a non-closed recurrent point.
It is known that a total number of Q-sets for $v$ cannot exceed $g$ if $S$ is an orientable surface of genus $g$ \cite{mayer1943trajectories}, and $\frac{p-1}{2}$ if $S$ is a non-orientable surface of genus $p$ \cite{markley1970number} (cf. \cite[Remark 2]{aranson1996maier}).
Therefore, the closure $\overline{\mathrm{R}(v)}$ is a finite union of Q-sets.
%
%
A (quasi-)minimal set is exceptional (resp. locally dense) if it is a closure of an orbit intersecting $\mathrm{E}(v)$ (resp. $\mathrm{LD}(v)$).

\subsubsection{Topological equivalence}
A flow $v$ on a surface $M$ is {\bf topologically equivalent} to a flow $w$ on a surface $N$ if there is a homeomorphism $h \colon M \to N$ such that the images of any orbits of $v$ are orbits of $w$ with preservation of the direction in time. 
Then, the homeomorphism $h \colon M \to N$ is called the {\bf topologically equivalent homeomorphism}. 

\subsubsection{Locally topological equivalence}

The restriction of a flow $v$ to a subset $U$ of $S$ is {\rm(}{\bf locally}{\rm)} {\bf topologically equivalent} to one of a flow $w$ to a subset $V$ on a topological space $T$ if there is a homeomorphism $h \colon U \to V$ such that the images of any orbit arcs of $v$ in $U$ are orbit arcs of $w$, the inverse images of any orbit arcs of $w$ in $V$ are orbit arcs of $v$, and $h$ and $h^{-1}$ preserve the directions of orbit arc of $v$ and $w$. 

\subsubsection{Canonical local structures}

We define a trivial flow box as follows, as the left in Figure~\ref{fb}.

\begin{definition}\label{def:tr_flow_box}
The restriction of a flow to a disk is a {\bf trivial flow box} if it is locally topologically equivalent to the restriction of the flow generated by a vector field $(0,1)$ to a square $I \times J$ where $I$ and $J$ are non-degenerate intervals. 
\end{definition}
%
We also define a transverse annulus as follows, as the middle in Figure~\ref{fb}. 

\begin{definition}\label{def:trans_annulus}
The restriction of a flow to an annulus is a {\bf transverse annulus} if it is locally topologically equivalent to the restriction of the flow generated by a vector field $(-x,-y)$ to an annulus $\{ (x,y) \in \R^2 \mid x^2 + y^2 \in I \}$ where $I \subset (0,1)$ is a non-degenerate interval. 
\end{definition}
The annulus is also called a {\bf transverse annulus}.
Similarly, we define a periodic annulus as follows, as the right in Figure~\ref{fb}. 

\begin{definition}\label{def:per_annulus}
The restriction of a flow to an annulus is a {\bf periodic annulus} if it is locally topologically equivalent to the restriction of the flow generated by a vector field $(-y,x)$ to an annulus $\{ (x,y) \in \R^2 \mid x^2 + y^2 \in I \}$ where $I \subset (0,1)$ is a non-degenerate interval. 
\end{definition}

The annulus is also called a {\bf periodic annulus}.

\subsubsection{Types of topologically non-degenerate singular points}\label{sec:def_sing}

We recall the following types of singular points as in Figure~\ref{Fig:quasiregular}.
\begin{definition}
A singular point outside of the boundary is a {\bf sink} (resp. {\bf source}) if there is its open \nbd to which the restriction of the flow is locally topologically equivalent to a flow on the Euclidean space $\R^2$ generated by a vector field $X(p) = -p$ (resp. $X(p) = p$). 
\end{definition}

\begin{definition}
A singular point on the boundary is a {\bf $\partial$-sink} (resp. {\bf $\partial$-source}) if there is its open \nbd to which the restriction of the flow is locally topologically equivalent to a flow on the half plane $\R_{\geq 0} \times \R$ generated by a vector field $X(p) = -p$ (resp. $X(p) = p$). 
\end{definition}

\begin{definition}
A singular point is a {\bf center} if there is its open \nbd to which the restriction of the flow is locally topologically equivalent to a flow on the Euclidean space $\R^2$ generated by a vector field $X(x,y) = (-y,x)$. 
\end{definition}

We have the following observation. 

\begin{proposition}\label{prop:char_center}
Every singular point $x$ of a flow on a surface is a center if and only if there is an open disk $V$ with $x \in V$ such that the set difference $V - \{ x \}$ is contained in the set of periodic points. 
\end{proposition}

The proof of the previous proposition is shown in Appendix~\ref{appendix}. 
An invariant disk is a {\bf center disk} if it is a union of one center and periodic orbits and is a neighborhood of the center.
%
%

\begin{definition}
A singular point outside of the boundary is a {\bf saddle} if there is its open \nbd to which the restriction of the flow is locally topologically equivalent to a flow on the Euclidean space $\R^2$ generated by a vector field $X(x,y) = (-x,y)$. 
\end{definition}

\begin{definition}
A singular point on the boundary is a {\bf $\partial$-saddle} if there is its open \nbd to which the restriction of the flow is locally topologically equivalent to a flow on the half plane $\R_{\geq 0} \times \R$ generated by a vector field $X(x,y) = (-x,y)$ or $X(x,y) = (x,-y)$. 
\end{definition}

\subsubsection{Multi-saddles}\label{sec_multi-saddle}

A $\partial$-$k$-saddle (resp. $k$-saddle) is an isolated singular point on (resp. outside of) $\partial S$ with exactly $(2k + 2)$-separatrices, counted with multiplicity as in  Figure~\ref{multi-saddles}.
\begin{figure}
\begin{center}
\includegraphics[scale=0.6]{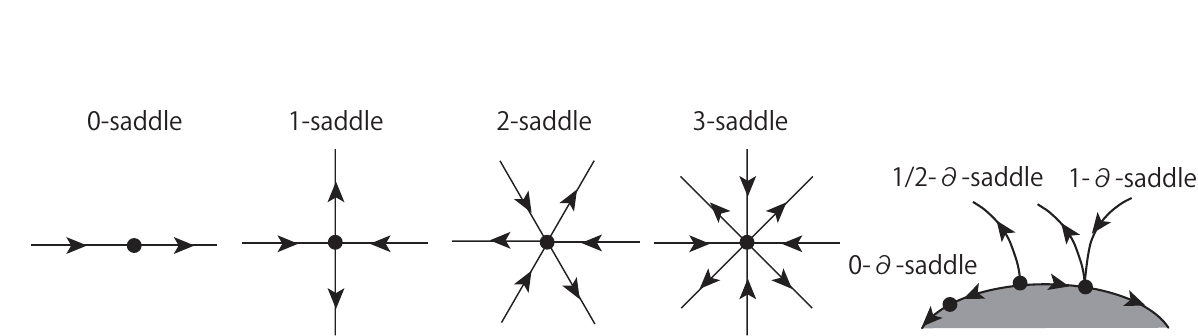}
\end{center}
\caption{Examples of multi-saddles}
\label{multi-saddles}
\end{figure} 
In other words, A $\partial$-$k/2$-saddle (resp. $k$-saddle) is an isolated singular point on (resp. outside of) $\partial S$ with exactly $k+1$ (resp. $(2k + 2)$) hyperbolic sectors (see Definition~\ref{def:hyp_sector} for the definition of hyperbolic sector). 
Here, a hyperbolic sector is a local structure as in Figure~\ref{hyp-sectors}. 
\begin{figure}
\begin{center}
\includegraphics[scale=0.3]{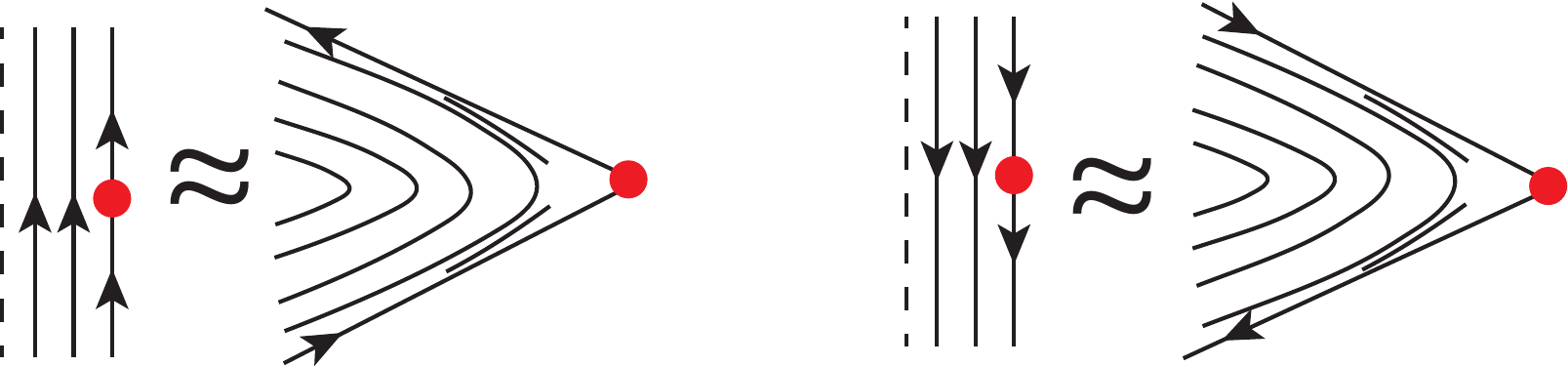}
\end{center}
\caption{A hyperbolic sector with anti-clockwise flow direction and a hyperbolic sector with clockwise flow direction}
\label{hyp-sectors}
\end{figure}

\begin{definition}
A {\bf multi-saddle} is a $k$-saddle or a $\partial$-$(k/2)$-saddle for some $k \in \mathbb{Z}_{\geq 0}$.
\end{definition}

By definition, notice that a $1$-saddle is a saddle and that a $\partial$-$(1/2)$-saddle is a $\partial$-saddle.
%

\subsubsection{{\rm(}Semi-{\rm)}multi-saddle separatrices}

A separatrix is {\bf semi-multi-saddle} if it is from or to a multi-saddle.
A connecting separatrix is a saddle (resp. {\bf multi-saddle}) separatrix if its $\alpha$-limit and $\omega$-limit set are saddles and $\partial$-saddles (resp. multi-saddles).

\subsubsection{{\rm(}Multi-{\rm)}saddle connection diagrams}

To construct decompositions of surfaces, we recall the following concepts. 

\begin{definition}
The {\bf multi-saddle connection diagram} $D(v)$ is the union of multi-saddles and multi-saddle separatrices.
\end{definition}

\begin{definition}
A {\bf multi-saddle connection} is a connected component of the multi-saddle connection diagram.
\end{definition}

Note that a multi-saddle connection is also called a poly-cycle.
If each multi-saddle in the multi-saddle connection diagram $D(v)$ is either a saddle or a $\partial$-saddle, then $D(v)$ is also called the {\bf saddle connection diagram}.
A {\bf saddle connection} is a connected component of the saddle connection diagram.


\subsubsection{Collars}
Recall that an {\bf annular} subset is homeomorphic to an annulus. 
\begin{definition}
An open annular subset $\mathbb{A}$ of a surface is a {\bf collar} of an invariant subset $\gamma$ if there is an open connected neighborhood $U$ of $\gamma$ such that $\mathbb{A}$ is a connected component of the complement $U - \gamma$. 
\end{definition}
Note that an open annular subset $\mathbb{A}$ of a surface is a collar of a singular point $x$ if and only if the union $\mathbb{A} \sqcup \{ x \}$ is an open disk and a neighborhood of $x$. 

\subsubsection{Circuits}
By a {\bf cycle} or a periodic circuit, we mean a periodic orbit.

\begin{definition}
A {\bf circuit} is either a singular point, a cycle, or an image of an oriented circle by a continuous orientation-preserving mapping which is a directed graph but not a singleton and which is the union of separatrices and finitely many singular points. 
\end{definition}

A circuit if {\bf trivial} if it is a singular point. 
Note that a non-trivial circuit is not a local injection in general (see Figure~\ref{NAC02+}). 
Here, a mapping is a local injection if, for any point in a domain, there is its neighborhood to which the restriction is injection. 
\begin{figure}
\begin{center}
\includegraphics[scale=0.25]{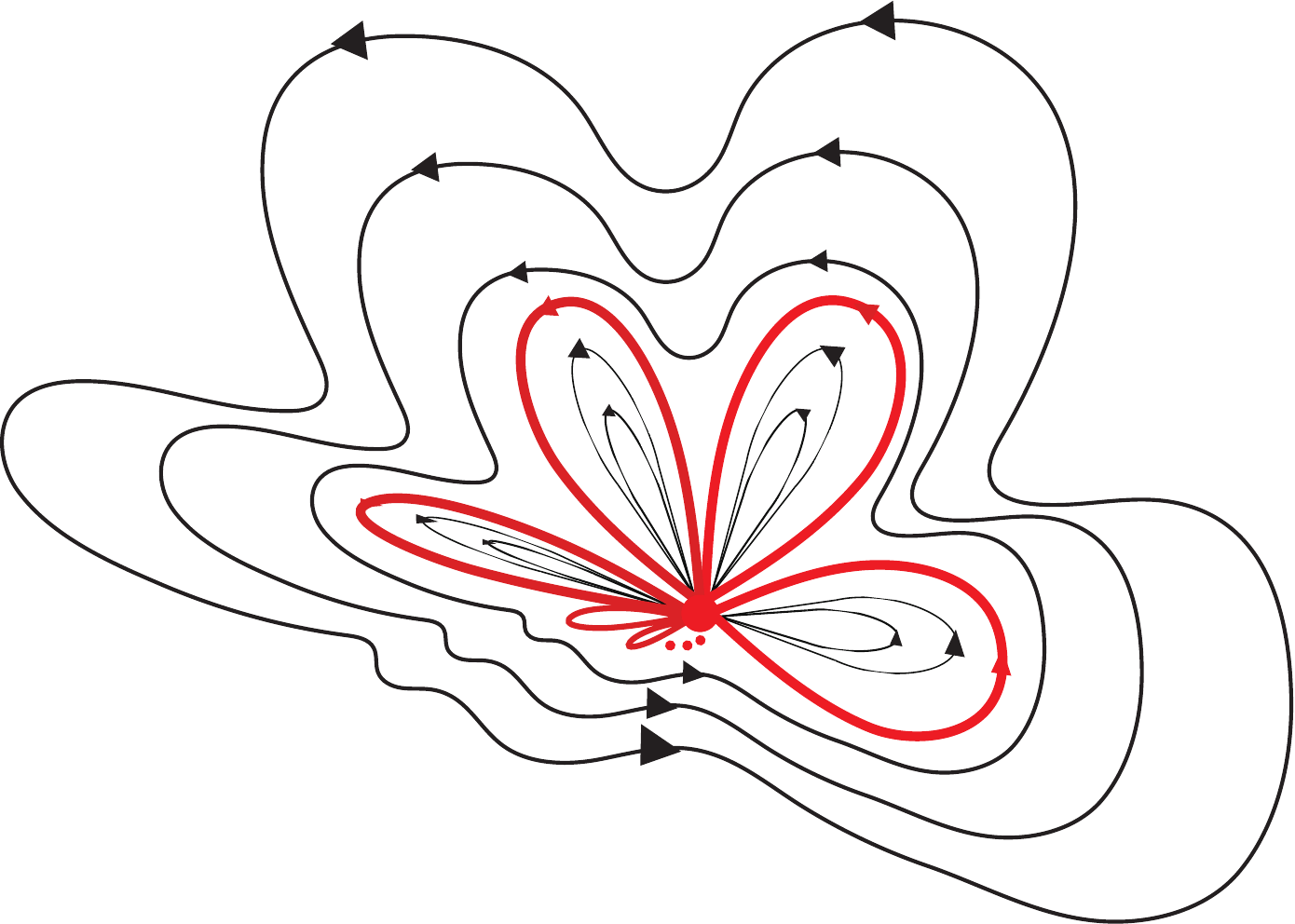}
\end{center}
\caption{A degenerate singular point with infinitely many connecting separatrices}
\label{NAC02+}
\end{figure}
Moreover, there are circuits with infinitely many edges, and any non-trivial non-periodic circuit contains non-recurrent orbits as in Figure~\ref{NAC02+}. 
\begin{definition}
A non-trivial circuit is a {\bf multi-saddle circuit} if it is contained in a multi-saddle connection. 
\end{definition}
%
\begin{definition}
A non-trivial circuit $\gamma$ is a {\bf limit circuit} if there is a point $x \notin \gamma$ with $\omega(x) = \gamma$ or $\alpha(x) = \gamma$. 
\end{definition}
A {\bf limit cycle} is a limit circuit in $\mathop{\mathrm{Per}}(v)$.

\subsubsection{Semi-attracting circuits and semi-repelling circuits}

A trivial circuit $x$ is {\bf semi-attracting} (resp. {\bf semi-repelling}) if there is its collar which is contained in the stable (resp. unstable) manifold of $x$.  
In other words, a semi-attracting trivial circuit is either a $\partial$-source or a source, and a semi-repelling trivial circuit is either a $\partial$-sink or a sink.

\begin{definition}
A circuit $\gamma$ is a {\bf semi-attracting} (resp. {\bf semi-repelling}) circuit with respect to a small collar $\A$ if $\omega(x)  = \gamma$ (resp. $\alpha(x) = \gamma$) and $O^+(x) \subset \A$ (resp. $O^-(x) \subset \A$) for any point $x \in \A$. 
\end{definition}

Then $\A$ is called a semi-attracting (resp. a semi-repelling) {\bf collar basin} of $\gamma$. 
By the generalization of the Poincar{\'e}-Bendixson theorem (cf. \cite[Corollary~5.8]{yokoyama2021poincare}), any limit circuit is either semi-attracting or semi-repelling.

\subsubsection{One-sided circuits and two-sided circuits}

One-sidedness is defined as follows. 

\begin{definition}
A non-trivial circuit $\gamma$ is {\bf one-sided} if, for any small neighborhood $U$ of $\gamma$,  there are a point $x \in \mathrm{P}(v) \cap \gamma$ and a collar $V \subset U$ of $\gamma$ such that the union $V \sqcup \gamma$ is a neighborhood of $x$. 
\end{definition}
A non-trivial circuit $\gamma$ is {\bf two-sided} if it is not one-sided (i.e. there is a small neighborhood $U$ of $\gamma$ such that the union $V \sqcup \gamma$ for any collar $V \subset U$ of $\gamma$ is not a neighborhood of any point in $\mathrm{P}(v) \cap \gamma$). 

For a circuit $\mu$ which is a simple closed curve, notice that the circuit $\mu$ is one-sided if and only if it is either a boundary component of a surface or has a small neighborhood which is a M{\"o}bius band, and that the circuit $\mu$ is two-sided if and only if it has an open small annular neighborhood $\mathbb{A}$ such that the complement $\mathbb{A} - \mu$ consists of two open annuli.

\subsubsection{Ss-multi-saddle connection diagrams}

To construct decompositions, we introduce the following invariant subsets. 

\begin{definition}
An invariant subset is an {\bf ss-component} if it is either a sink, a $\partial$-sink, a source, a $\partial$-source, a limit circuit, or an exceptional Q-set.
\end{definition}

In other words, an invariant subset is an ss-component if and only if it is either a semi-attracting/semi-repelling circuit or an exceptional Q-set.

\begin{definition}
A semi-multi-saddle separatrix is an {\bf ss-separatrix} if it connects a multi-saddle and an ss-component.
\end{definition}

In other words, a semi-multi-saddle separatrix is an ss-separatrix if its $\alpha$-limit or $\omega$-limit set is an ss-component. 

\begin{definition}
The {\bf ss-multi-saddle connection diagram} $\bm{D_{\mathrm{ss}}(v)}$ is the union of multi-saddles, multi-saddle separatrices, ss-separatrices, and ss-components.
\end{definition}

An {\bf ss-multi-saddle connection} is a connected component of the ss-multi-saddle connection diagram.
Note that $D_{\mathrm{ss}}(v)$ is the union of the multi-saddle connection diagram, ss-separatrices, and ss-components.

\subsubsection{Concepts of ``border'' invariant subsets in interiors}

We define notations of invariant subsets as follows: 
\\
 $\bm{\mathop{\mathrm{P}_{\mathrm{ms}}}(v)} \subset D_{\mathrm{ss}}(v)$ : The union of multi-saddle separatrices in $\mathrm{int}\mathrm{P}(v)$. 
\\
$\bm{\mathop{\mathrm{P}_{\mathrm{ss}}}(v)} \subset D_{\mathrm{ss}}(v)$ : The union of ss-separatrices in $\mathrm{int}\mathrm{P}(v)$. 
\\
$\bm{\mathop{\partial_{\mathrm{P}(v)}}} \subset \partial S$ : The union of orbits on $\partial S$ which connect $\partial$-sinks and $\partial$-sources. 
\\
$\bm{\mathop{\mathrm{P}_{\mathrm{sep}}}(v)} := \mathop{\mathrm{P}_{\mathrm{ms}}}(v) \sqcup \mathop{\mathrm{P}_{\mathrm{ss}}}(v) \sqcup \mathop{\partial_{\mathrm{P}(v)}} \subset \mathop{\mathrm{int}} \mathrm{P}(v)$
\\
$\bm{\partial_{\mathop{\mathrm{Per}}(v)}} := \partial S \cap \mathop{\mathrm{int}} \mathop{\mathrm{Per}}(v)$
\\
$\bm{\mathop{\mathrm{Per}_{1}}(v)}$ : The union of one-sided periodic orbits in $\mathrm{int}(\mathop{\mathrm{Per}}(v) - \partial_{\mathop{\mathrm{Per}}(v)})$ 
\\
$\bm{\mathop{\mathrm{Bd}}_{\mathrm{int}}(v)} := \mathop{\mathrm{P}_{\mathrm{sep}}}(v) \sqcup \partial_{\mathop{\mathrm{Per}}(v)}$
\\
$\bm{\mathop{\mathrm{BD}}_{\mathrm{int}}(v)} := \mathop{\mathrm{P}_{\mathrm{sep}}}(v) \sqcup \partial_{\mathop{\mathrm{Per}}(v)} \sqcup \mathop{\mathrm{Per}_{1}}(v) = \mathop{\mathrm{Bd}}_{\mathrm{int}}(v) \sqcup \mathop{\mathrm{Per}_{1}}(v)$
\\
$\bm{\mathop{\mathrm{Bd}}_{\partial}(v)} := \partial \mathop{\mathrm{Sing}}(v) \cup \partial \mathop{\mathrm{Per}}(v) \cup \partial \mathrm{P}(v) \cup \partial \mathrm{LD}(v) \cup \partial \mathrm{E}(v)$
\\
$
\bm{\mathop{\mathrm{Bd}}(v)} := \mathop{\mathrm{Bd}_{\mathrm{int}}}(v) \sqcup \mathop{\mathrm{Bd}_{\partial}}(v)
$
\\
$\bm{\mathop{\mathrm{BD}}(v)} := \mathop{\mathrm{Bd}}(v) \sqcup \mathop{\mathrm{Per}_{1}}(v)$

We call $\mathop{\mathrm{BD}}(v)$ (resp. $\mathop{\mathrm{Bd}}(v)$) the {\bf border} (resp. {\bf weak border}) {\bf point set} for the flow $v$. 
By definitions, we have the following equalities: 
\[
\mathop{\mathrm{Bd}}(v) = \partial \mathop{\mathrm{Sing}}(v) \cup \partial \mathop{\mathrm{Per}}(v) \cup \partial_{\mathop{\mathrm{Per}}(v)} \cup \partial \mathrm{P}(v) \cup  \mathop{\mathrm{P}_{\mathrm{sep}}}(v) \cup \partial \mathrm{LD}(v) \cup \partial \mathrm{E}(v)
\]
\[
\mathop{\mathrm{BD}}(v) = \partial \mathop{\mathrm{Sing}}(v) \cup \partial \mathop{\mathrm{Per}}(v) \cup \partial_{\mathop{\mathrm{Per}}(v)} \cup \mathop{\mathrm{Per}_{1}}(v) \cup \partial \mathrm{P}(v) \cup  \mathop{\mathrm{P}_{\mathrm{sep}}}(v) \cup \partial \mathrm{LD}(v) \cup \partial \mathrm{E}(v)
\]

Note that the definition of $\mathop{\mathrm{Bd}}(v)$ before Theorem~\ref{lem0c} coincides with this definition when the flow $v$ is quasi-regular (see Proposition~\ref{cor:bd01}). 
We will prove that $\mathop{\mathrm{Bd}}(v) = \mathop{\mathrm{Sing}}(v) \sqcup \partial^{-} \mathop{\mathrm{Per}}(v) \sqcup \partial^{-} \mathrm{P}(v) \sqcup \mathrm{E}(v) \sqcup \mathop{\mathrm{Bd}_{\mathrm{int}}}(v)$ if $v$ is a flow with finitely many singular points on a compact surface (see Lemma~\ref{lem010}).
Note that the union $\partial_{\mathop{\mathrm{Per}}(v)} \sqcup \mathop{\mathrm{Per}_{1}}(v)$ is the union of one-sided periodic orbits in $\mathrm{int}\mathop{\mathrm{Per}}(v)$. 



\subsubsection{Border points}

Note that 
if $S$ is compact and $\Sv$ is finite then $\mathop{\partial_{\mathrm{P}(v)}}$ is a finite union of non-recurrent orbits on $\partial S \cap \mathop{\mathrm{int}} \mathrm{P}(v)$. 
%
We have the following observation. 
\begin{lemma}\label{lem:semi-multi-saddles}
The following statements hold for a quasi-regular flow $v$ on a surface $S$: 
\\
{\rm(1)} The union of non-singular orbits in limit circuits is contained in the union $\mathop{\mathrm{P}_{\mathrm{ms}}}(v) \sqcup \partial^{-} \mathrm{P}(v)$. 
\\
{\rm(2)} The union $\mathop{\mathrm{P}_{\mathrm{ms}}}(v) \sqcup \mathop{\mathrm{P}_{\mathrm{ss}}}(v)$ is the union of proper semi-multi-saddle separatrices in $\mathop{\mathrm{int}} \mathrm{P}(v)$. 
\\
{\rm(3)} The union $\mathop{\mathrm{P}_{\mathrm{sep}}}(v)$ 
 is the union of proper semi-multi-saddle separatrices in $\mathop{\mathrm{int}} \mathrm{P}(v)$ and non-recurrent orbits on the boundary $\partial S$ in $\mathop{\mathrm{int}} \mathrm{P}(v)$. 
\end{lemma}

\begin{proof}
The quasi-regularity implies that any singular point is either a multi-saddle, a sink, a $\partial$-sink, a source, a $\partial$-source, or a center. 
Fix an orbit $O \subseteq \mathrm{P}(v)$. 
By \cite[Corollary~2.9]{yokoyama2016topological}, the union $\mathrm{LD}(v)$ is open and so $\overline{O} \cap \mathrm{LD}(v) = \emptyset$. 

Suppose that  $O$ is contained in a limit circuit. 
Then the orbit $O$ connects multi-saddles and so is a multi-saddle separatrix. 
If $O \subset \mathop{\mathrm{int}} \mathrm{P}(v)$, then assertion (1) holds. 
If $O \not\subset \mathop{\mathrm{int}} \mathrm{P}(v)$, then $O \subset \partial^{-} \mathrm{P}(v)$, which implies assertion (1). 

Suppose that $O$ is contained in $\mathop{\mathrm{P}_{\mathrm{ms}}}(v) \sqcup \mathop{\mathrm{P}_{\mathrm{ss}}}(v)$. 
Since any multi-saddle separatrix is a semi-multi-saddle separatrix, we may assume that $O \subset \mathop{\mathrm{P}_{\mathrm{ss}}}(v)$ (i.e. $O$ is an ss-separatrix in $\mathrm{int}\mathrm{P}(v)$). 
Conversely, suppose that $O$ is a proper semi-multi-saddle separatrix in $\mathop{\mathrm{int}} \mathrm{P}(v)$. 
By time reversion if necessary, we may assume that $\alpha(O)$ is a multi-saddle. 
If $O$ is a multi-saddle separatrix, then $O \subset \mathop{\mathrm{P}_{\mathrm{ms}}}(v)$. 
Thus we may assume that $O$ is not a multi-saddle separatrix. 
Then $\omega(O)$ is not a multi-saddle. 
By the generalization of the Poincar{\'e}-Bendixson theorem (cf. \cite[Corollary~5.8]{yokoyama2021poincare}) and from the quasi-regularity of $v$, the $\omega$-limit set $\omega(O)$ is either a $\partial$-sink, a sink, a limit circuit, or an exceptional Q-set.
This means that the $\omega$-limit set $\omega(O)$ is an ss-component and so the orbit $O$ is an ss-separatrix in $\mathrm{int}\mathrm{P}(v)$ (i.e. $O \subset \mathop{\mathrm{P}_{\mathrm{ss}}}(v)$), which implies assertion (2).

By definitions, the union $\mathop{\partial_{\mathrm{P}(v)}} \subset \partial S$ consists of non-recurrent orbits on the boundary $\partial S$. 
By assertion (2), the union $\mathop{\mathrm{P}_{\mathrm{sep}}}(v)$ is contained in the union of proper semi-multi-saddle separatrices in $\mathop{\mathrm{int}} \mathrm{P}(v)$. 
Conversely, suppose that $O$ is a non-recurrent orbit on the boundary $\partial S$ in $\mathop{\mathrm{int}} \mathrm{P}(v)$.
Then $\alpha(O)$ and $\omega(O)$ are singletons, which are singular points. 
If $\alpha(O)$ and $\omega(O)$ are multi-saddles, then $O$ is a proper multi-saddle separatrices in $\mathop{\mathrm{int}} \mathrm{P}(v)$ and so $O \subseteq \mathop{\mathrm{P}_{\mathrm{ms}}}(v)$. 
Thus we may assume that $\omega(O)$ is not a multi-saddle. 
The quasi-regularity of $v$ imiplies that $\omega(O)$ is a $\partial$-sink.  
If $\alpha(O)$ is a multi-saddle, then $O$ is an ss-separatrix in $\mathrm{int}\mathrm{P}(v)$ and so $O \subset \mathop{\mathrm{P}_{\mathrm{ss}}}(v)$. 
Thus we may assume that $\alpha(O)$ is not a multi-saddle. 
The quasi-regularity of $v$ imiplies that $\alpha(O)$ is a $\partial$-source. 
This means that $O \subseteq \mathop{\partial_{\mathrm{P}(v)}}$. 
Therefore, assertion (3) holds because of assertion (2). 
\end{proof}

We will show that the union $\partial^{-} \mathrm{P}(v) \sqcup \mathop{\mathrm{P}_{\mathrm{ms}}}(v) \sqcup \mathop{\mathrm{P}_{\mathrm{ss}}}(v)$ is the union of proper semi-multi-saddle separatrices if $v$ is quasi-regular (see Corollary~\ref{cor:semi-multi-separatrices}). 

\subsubsection{Transversality}
Notice that we can define transversality using tangential spaces of surfaces because each flow on a compact surface is topologically equivalent to a $C^1$-flow by Gutierrez's smoothing theorem~\cite{gutierrez1986smoothing}.
However, to explicitly modify transverse arcs, we define transversality immediately as follows.  

\begin{definition}
A curve $C$ is {\bf transverse} to $v$ at a point $p \in C \cap \partial S$ if there are a small neighborhood $U$ of $p$ and a homeomorphism $h:U \to [-1,1] \times [0,1]$ with $h(p) = 0$ such that $h^{-1}([-1,1] \times \{t \})$ for any $t \in [0, 1]$ is an orbit arc and $h^{-1}(\{0\} \times [0,1]) = C \cap U$. 
\end{definition}

\begin{definition}
A curve $C$ is {\bf transverse} to $v$ at a point $p \in C \setminus \partial S$ if there are a small neighborhood $U$ of $p$ and a homeomorphism $h:U \to [-1,1]^2$ with $h(p) = 0$ such that $h^{-1}([-1,1] \times \{t \})$ for any $t \in [-1, 1]$ is an orbit arc and $h^{-1}(\{0\} \times [-1,1]) = C \cap U$. 
\end{definition}

A simple curve $C$ is transverse to $v$ if so is it at any point in $C$.  
A simple curve $C$ is transverse to $v$ is called a {\bf transverse arc }.
A simple closed curve is a {\bf closed transversal} if it transverses to $v$.

\section{Properties of surface flows}

This section states the following lemmas in preparation for demonstrating Theorem~\ref{lem0c}. 
Let $v$ be a flow on a compact connected surface $S$. 

\subsection{Properties of borders and frontiers}

Recall notations of the border $\partial^- A$ and the frontier $\bp A$ for a subset $A \subseteq S$ as follows:
\[ 
\begin{split}
\partial^- A &= A - \mathop{\mathrm{int}} A
\\
\bp A &= \overline{A} - A
\end{split}
\]
Moreover, recall notations of sets of points of various types: 
\\
$\Sv$ : the singular point set
\\
$\Pv$ : the periodic point set
\\
$\Cv = \Sv \sqcup \Pv$ : the closed point set
\\
$\mathrm{P}(v)$ : the union of non-closed proper orbits, which is also the set of non-recurrent points
\\
$\mathrm{LD}(v)$ : the union of locally orbits
\\
$\mathrm{E}(v)$ : the union of exceptional orbits
\\
$\mathrm{R}(v) = \mathrm{LD}(v) \sqcup \mathrm{E}(v)$ : the set of non-closed recurrent points, which is also the union of non-proper orbits

 \subsubsection{Inclusion relations for borders and frontiers}

We have the following inclusion relations. 
%

\begin{lemma}\label{lem3-04}
The following statements hold for a flow $v$ on a compact surface $S$:
\\
$(1)$ 
$\bp \mathrm{LD}(v) \subseteq  \partial^{-} \mathop{\mathrm{Sing}}(v) \sqcup \partial^{-} \mathrm{P}(v)$.
\\
$(2)$ 
$\bp \mathrm{E}(v) \subseteq  \partial^{-} \mathop{\mathrm{Sing}}(v) \sqcup \partial^{-} \mathrm{P}(v)$.
\\
$(3)$ 
$\bp \mathop{\mathrm{Per}}(v) \subseteq  \partial^{-} \mathop{\mathrm{Sing}}(v) \sqcup \partial^{-} \mathrm{P}(v)$.
\\
$(4)$ 
$\bp \mathop{\mathrm{Sing}}(v) = \emptyset$.  
\\
$(5)$ 
$\bp \mathrm{P}(v) \subseteq  \partial^{-} \mathop{\mathrm{Sing}}(v) \sqcup  \partial^{-} \mathop{\mathrm{Per}}(v) \sqcup \partial^{-} \mathrm{LD}(v) \sqcup \partial^{-} \mathrm{E}(v)$.

In particular, we have
$\bp \mathop{\mathrm{Per}}(v) \cup \bp \mathrm{LD}(v) \cup \bp \mathrm{E}(v) \subseteq \partial^{-} \mathop{\mathrm{Sing}}(v) \sqcup \partial^{-} \mathrm{P}(v)$. 
\end{lemma}

\begin{proof}
As mentioned, it is known that a total number of Q-sets for $v$ is finite  \cite{markley1970number,mayer1943trajectories}.
By  \cite[Proposition~2.2]{yokoyama2016topological}, assertions $(1)$ and $(2)$ hold.
 \cite[Lemma 2.3]{yokoyama2016topological} implies that $\overline{\mathop{\mathrm{Per}}(v)} \cap \mathrm{R}(v) = \emptyset$ and so $\bp \mathop{\mathrm{Per}}(v) \subseteq  \partial^{-} \mathop{\mathrm{Sing}}(v) \sqcup \partial^{-} \mathrm{P}(v)$ (i.e assertion $(3)$ holds).
The closedness of the singular point set $\mathop{\mathrm{Sing}}(v)$ implies $\bp \mathop{\mathrm{Sing}}(v) = \emptyset$.  
By definition of $\bp$, assertion $(5)$ holds.
\end{proof}


%

\subsection{Non-trivial recurrence}

We recall the following Ma{\v \i}er's result \cite{mayer1943trajectories} (cf. \cite[Theorem 2.4.4 p.32]{nikolaev1999flows}, \cite[Theorem 4.2]{aranson1996maier}). 

\begin{lemma}[Ma{\v \i}er]\label{pos_rec}
Let $v$ be a flow on a compact surface $S$. 
A point $x \in \omega(z)$ for some point $z \in S$ with $\omega(x) \setminus  \mathop{\mathrm{Cl}}(v) \neq \emptyset$ is non-closed positively recurrent. 
In particular, we have $x \in \omega(x) \cap \mathrm{R}(v)$. 
\end{lemma}

Although the results are widely referred, the original Ma{\v \i}er's Russian paper is not translated. 
Because no proof is written in English, even in textbooks, as far as the author knows, we state a proof of this result in Appendix~\ref{M_result}.

\subsection{Inherited properties for double covers}
We state the inherited properties for double covers. 

\begin{lemma}\label{lem:ori_cover}
Let $v$ be a flow on a compact surface $S$ and $p \colon  \widetilde{S} \to S$ a double cover. 
Then the lift $\widetilde{v}$ of $v$ to $\widetilde{S}$ is well-defined and the following statement holds for any point $x \in S$: 
\\
$(1)$ Either $\hat{O}_{\widetilde{v}}(x_+) = \hat{O}_{\widetilde{v}}(x_-) = p^{-1}(\hat{O}(x) )$ or $\hat{O}_{\widetilde{v}}(x_+) \sqcup \hat{O}_{\widetilde{v}}(x_-) = p^{-1}(\hat{O}(x) )$
\\
$(2)$ $p^{-1}(\mathop{\mathrm{Sing}}(v)) = \mathop{\mathrm{Sing}}(\widetilde{v})$
\\ 
$(3)$ $p^{-1}(\mathop{\mathrm{Per}}(v)) = \mathop{\mathrm{Per}}(\widetilde{v})$
\\ 
$(4)$ $p^{-1}(\mathrm{P}(v)) = \mathrm{P}(\widetilde{v})$
\\ 
$(5)$  $p^{-1}(\mathrm{LD}(v)) = \mathrm{LD}(\widetilde{v})$
\\ 
$(6)$  $p^{-1}(\mathrm{E}(v)) = \mathrm{E}(\widetilde{v})$
\end{lemma}

\begin{proof}  
By Gutierrez's smoothing theorem~\cite{gutierrez1986smoothing}, the flow is generated by a vector field up to topological equivalence. 
Notice that any flow generated by a vector field $X$ on a connected compact manifold induces the induced flow on the finite cover because $X$ can be lifted to the cover.  

Let $\widetilde{v}$ be the lift of $v$ on a double cover $\widetilde{S}$ of $S$. 
For a point $x \in S$, put $\{ x_+, x_- \} := p^{-1}(x)$. 
For a subset $A \subseteq S$, since covering maps are locally homeomorphisms, we have $p^{-1}(A) = \{ x_+, x_- \mid x \in A \}$ and so $\overline{p^{-1}(A)} = p^{-1}(\overline{A}) = \{ x_+, x_- \mid x \in \overline{A} \}$.  

\begin{claim}\label{claim:007}
Assertions $(2)$ and $(3)$ hold, and assertion $(1)$ for closed points holds.
\end{claim}

\begin{proof}[Proof of Claim~\ref{claim:007}]
Suppose that $x \in \mathop{\mathrm{Cl}}(v)$. 
Then the lift $p^{-1}(O(x))$ consists of one or two closed orbits of $\widetilde{v}$ and so of one or two orbit classes. 
Since lifts of singular (resp. periodic, non-closed) points are singular (resp. periodic, non-closed), assertions $(2)$ and $(3)$ hold. 
Moreover, assertion $(1)$ for closed points holds. 
\end{proof}

%
%

\begin{claim}\label{claim:008}
Assertion $(4)$ holds, and assertion $(1)$ for non-recurrent points holds.
\end{claim}

\begin{proof}[Proof of Claim~\ref{claim:008}]
Let $x \in S$ be a non-recurrent point.  
Since the non-recurrence is invariant under taking finite coverings, we have that $x \in \mathrm{P}(v)$ if and only if $p^{-1}(x) \subseteq \mathrm{P}(\widetilde{v})$. 
This means that assertion $(4)$ holds. 
Moreover, if $x \in \mathrm{P}(v)$, then either $O_{\widetilde{v}}(x_+) = O_{\widetilde{v}}(x_-) = \hat{O}_{\widetilde{v}}(x_+) = \hat{O}_{\widetilde{v}}(x_-) = p^{-1}(\hat{O}(x) )= p^{-1}(O(x))$ or $O_{\widetilde{v}}(x_+) \sqcup O_{\widetilde{v}}(x_-) = \hat{O}_{\widetilde{v}}(x_+) \sqcup \hat{O}_{\widetilde{v}}(x_-) = p^{-1}(\hat{O}(x)) = p^{-1}(O(x))$, because of the properness of $x$. 
This implies that assertion $(1)$ for non-recurrent points holds. 
\end{proof}

Assertions $(2)$--$(4)$ imply that $p^{-1}(\mathrm{R}(v)) = \mathrm{R}(\widetilde{v})$. 
%

\begin{claim}\label{claim:009}
Assertion $(1)$ holds.
\end{claim}

\begin{proof}[Proof of Claim~\ref{claim:009}]
Fix a point $x \in S$. 
By Claim~\ref{claim:007} and Claim~\ref{claim:008}, we may assume that $x \in S - (\mathop{\mathrm{Cl}}(v) \sqcup \mathrm{P}(v)) = \mathrm{R}(v)$. 
Then $p^{-1}(x) \subseteq \mathrm{R}(\widetilde{v})$. 
\cite[Proposition~2.2]{yokoyama2016topological} implies that $\hat{O}(x) = \overline{O(x)} \setminus (\mathop{\mathrm{Sing}}(v) \sqcup \mathrm{P}(v)) = \overline{O(x)} \cap \mathrm{R}(\widetilde{v})$, $\hat{O}_{\widetilde{v}}(x_+) = \overline{O_{\widetilde{v}}(x_+)} \setminus (\mathop{\mathrm{Sing}}(\widetilde{v}) \sqcup \mathrm{P}(\widetilde{v}))  = \overline{O_{\widetilde{v}}(x_+)} \cap \mathrm{R}(\widetilde{v})$, and $\hat{O}_{\widetilde{v}}(x_-) = \overline{O_{\widetilde{v}}(x_-)} \setminus (\mathop{\mathrm{Sing}}(\widetilde{v}) \sqcup \mathrm{P}(\widetilde{v})) = \overline{O_{\widetilde{v}}(x_-)} \cap \mathrm{R}(\widetilde{v})$. 
Since the cover $p$ is a local homeomorphism, we have that $\hat{O}_{\widetilde{v}}(x_+) \cup \hat{O}_{\widetilde{v}}(x_-) \subseteq p^{-1}(\hat{O}(x)) = p^{-1}(\overline{O(x)} \cap \mathrm{R}(v))  \subseteq p^{-1}(\overline{O(x)}) \cap p^{-1}((\mathrm{R}(v))) = \overline{p^{-1}(O(x))} \cap \mathrm{R}(\widetilde{v}) = \overline{O_{\widetilde{v}}(x_-) \cup O_{\widetilde{v}}(x_+)} \cap \mathrm{R}(\widetilde{v}) = (\overline{O_{\widetilde{v}}(x_-)} \cup \overline{O_{\widetilde{v}}(x_+)}) \cap \mathrm{R}(\widetilde{v}) = \hat{O}_{\widetilde{v}}(x_+) \cup \hat{O}_{\widetilde{v}}(x_-)$. 
Therefore, we obtain  $\hat{O}_{\widetilde{v}}(x_+) \cup \hat{O}_{\widetilde{v}}(x_-) = p^{-1}(\hat{O}(x))$. 
This implies assertion $(1)$. 
\end{proof}

\begin{claim}\label{claim:010}
Assertions $(5)$ and $(6)$ hold for any point $x \in \mathrm{R}(v)$ with $\hat{O}_{\widetilde{v}}(x_+) = \hat{O}_{\widetilde{v}}(x_-)$. 
\end{claim}

\begin{proof}[Proof of Claim~\ref{claim:010}]
Let $x \in S$ be a point $x \in \mathrm{R}(v)$ with $\hat{O}_{\widetilde{v}}(x_+) = \hat{O}_{\widetilde{v}}(x_-)$.
Then $p^{-1}(\hat{O}(x)) = \hat{O}_{\widetilde{v}}(x_+) = \hat{O}_{\widetilde{v}}(x_-) \subseteq \mathrm{R}(\widetilde{v})$ and $p^{-1}(\overline{O(x)}) = \overline{p^{-1}(O(x))} = \overline{O_{\widetilde{v}}(x_+)} =\overline{O_{\widetilde{v}}(x_-)}$. 
The local homeomorphic property of the covering map $p$ implies that $\mathop{\mathrm{int}} \, \overline{O(x)} \neq \emptyset$ if and only if $\mathop{\mathrm{int}} \, \overline{O_{\widetilde{v}}(x_+)} = \mathop{\mathrm{int}} \, \overline{O_{\widetilde{v}}(x_-)} \neq \emptyset$. 
Hence assertion $(5)$ holds. 
Since $\mathrm{E}(v) = S - (\mathop{\mathrm{Cl}}(v) \sqcup \mathrm{P}(v) \sqcup \mathrm{LD}(v))$, by $p^{-1}(\mathrm{R}(v)) = \mathrm{R}(\widetilde{v})$, assertions $(2)$--$(5)$ implies assertion $(6)$. 
\end{proof}

\begin{claim}\label{claim:010+}
Assertions $(5)$ and $(6)$ hold for any point $x \in \mathrm{R}(v)$. 
\end{claim}

\begin{proof}[Proof of Claim~\ref{claim:010+}]
Let $x \in S$ be a point $x \in \mathrm{R}(v)$.
By Claim~\ref{claim:010}, we may assume that $\hat{O}_{\widetilde{v}}(x_+) \neq \hat{O}_{\widetilde{v}}(x_-)$. 
Then $p^{-1}(\hat{O}(x)) = \hat{O}_{\widetilde{v}}(x_+) \sqcup \hat{O}_{\widetilde{v}}(x_-) \subseteq \mathrm{R}(\widetilde{v})$, $x_+ \notin \overline{O_{\widetilde{v}}(x_-)}$, and $x_- \notin \overline{O_{\widetilde{v}}(x_+)}$. 
Since the complement $S - \overline{O_{\widetilde{v}}(x_-)}$ is an invariant open \nbd of $x_+$, the local homeomorphic property of the covering map $p$ and the homogeneity of orbits imply that $\mathop{\mathrm{int}} \, \overline{O(x)} \neq \emptyset$ if and only if $\mathop{\mathrm{int}} \, \overline{O_{\widetilde{v}}(x_+)} \neq \emptyset$. 
Similarly, we have that $\mathop{\mathrm{int}} \, \overline{O(x)} \neq \emptyset$ if and only if $\mathop{\mathrm{int}} \, \overline{O_{\widetilde{v}}(x_-)} \neq \emptyset$. 
This implies assertion $(5)$ holds. 
Since $\mathrm{E}(v) = S - (\mathop{\mathrm{Cl}}(v) \sqcup \mathrm{P}(v) \sqcup \mathrm{LD}(v))$, assertions $(2)$--$(5)$ implies assertion $(6)$. 
\end{proof}
This completes the proof. 
\end{proof}

\section{On the border point sets}

This section also states the following four statements (Proposition~\ref{prop:cs}, Lemma~\ref{lem010}, Proposition~\ref{lem011}, and Proposition~\ref{cor:bd01}) in preparation for demonstrating Theorem~\ref{lem0c}.

\subsection{Border of non-recurrent closed orbits}

To describe the border of non-recurrent closed orbits, we have the following statement. 

\begin{lemma}\label{lem:annuli}
Let $S$ be a compact surface, $I \subset S$ a closed transverse arc, and $(O_n)_{n \in \Z_{\geq 0}}$ a sequence of pairwise disjoint loops intersecting $I$ such that, for any $n \neq m \in \Z_{\geq 0}$, the pair $O_n$ and $O_m$ are homotopic to each other. 
By renumbering $(O_n)_{n \in \Z_{\geq 0}}$, we may assume that the connected component of $S - \bigsqcup_{n \in \Z_{\geq 0}} O_{n}$ whose boundary is $O_{n} \sqcup O_{n+1}$ for any $n \in \Z_{\geq 0}$ is an open annulus. 
\end{lemma}

\begin{proof}
By the compactness of $I$, there is a subsequence of the sequence of points $x_n \in O_n \cap I$ which convergences to a point in $I$ monotonically.  
By taking a subsequence of $(O_n)_{n \in \Z_{\geq 0}}$, we may assume that the sequence $(x_n)_{n \in \Z_{\geq 0}}$ of $x_n \in O_n$ which convergences to a point in $I$ monotonically. 
For any $n \in \Z_{\geq 0}$, fix a connected component $A_n$ of $S - \bigsqcup_{k \in \Z_{\geq 0}}O_k$  whose boundary contains $O_n \sqcup O_{n+1}$.  

Suppose that there are infinitely many loops $O_n$ which are not null homotopic. 
By taking a subsequence of $(O_n)_{n \in \Z_{\geq 0}}$, we may assume that any loops $O_n$ are not null homotopic. 
From the finiteness of genus and boundary components of $S$, the complement $S - \bigsqcup_{n \in \Z_{\geq 0}}O_n$ consists of open annuli except for finitely many domains. 
By taking a subsequence of $(O_n)_{n \in \Z_{\geq 0}}$, we may assume that the connected component $A_n$ is an open annulus whose boundary is $O_n \sqcup O_{n+1}$ for any $n \in \Z_{\geq 0}$. 

Suppose that there are at most finitely many loops $O_n$ which are not null homotopic. 
By taking a subsequence of $(O_n)_{n \in \Z_{\geq 0}}$, we may assume that the loops $O_n$ are null homotopic. 
From the finiteness of genus and boundary components of $S$, the complement $S - \bigsqcup_{n \in \Z_{\geq 0}}O_n$ consists of open punctured annuli $A_{i_n}$ except finitely many domains. 
On the other hand, each loop $O_n$ is null homotopic and so bounds an open disk. 
Since any loops $O_n$ intersect the closed transverse arc $I$, the open punctured annuli $A_{i_n}$ except at most one domain are annuli. 
By taking a subsequence of $(O_n)_{n \in \Z_{\geq 0}}$, we may assume that the connected component $A_n$ is an open annulus whose boundary is $O_n \sqcup O_{n+1}$ for any $n \in \Z_{\geq 0}$. 
\end{proof}

A non-singular orbit of a point $x$ is a connecting quasi-separatrix \cite{yokoyama2021poincare} if $\omega(x) \cup \alpha(x) \subseteq \Sv$. 
We show the following statement. 
\begin{lemma}\label{lem:cs-}
Let $v$ be a flow on a compact surface and $x \in \partial^{-} \mathrm{P}(v)$ a point. 
Then the orbit $O(x)$ is a connecting quasi-separatrix. 
\end{lemma}

\begin{proof}
Since $\Sv$ is closed, we have $\mathrm{P}(v) \cap \partial \Sv = \emptyset$. 
From $\partial^{-} \mathrm{P}(v) = \mathrm{P}(v) \cap \partial \mathrm{P}(v) = \mathrm{P}(v) \cap \partial (S  - \mathrm{P}(v)) = \mathrm{P}(v) \cap \partial (\Sv \sqcup \Pv \sqcup \mathrm{R}(v)) \subseteq  \partial \Pv \sqcup \partial \mathrm{R}(v)$, we obtain $\partial^{-} \mathrm{P}(v) = (\mathrm{P}(v) \cap \partial \Pv) \cup (\mathrm{P}(v) \cap \partial \mathrm{R}(v))$. 

The Ma{\v \i}er and Markley works \cite{markley1969poincare,markley1970number} imply the closure $\overline{\mathrm{R}(v)}$ is a finite union of Q-sets. 
For any point $x \in \mathrm{P}(v) \cap \partial \mathrm{R}(v)$, by \cite[Theorem~A(b)]{yokoyama2021poincare} and the dual statement, the point $x$ is contained in $\omega$-limit sets or $\alpha$-limit set of points and so the orbit $O(x)$ is a connecting quasi-separatrix. 

Fix any point $x \in \mathrm{P}(v) \cap \partial \Pv$. 
By Gutierrez's smoothing theorem, we may assume that the flow $v$ is generated by an integrable vector field. 
The flow box theorem (cf. \cite[Theorem~4.2.6, p.95]{markley2023flows}, \cite[Theorem 1.1, p.45]{aranson1996introduction}) implies there is an open trivial flow box $B$ which contains the non-singular point $x$. 
From $x \in \mathrm{P}(v)$, taking $B$ small, we may assume that $O(x) \cap B$ consists of an open orbit arc. 
Let $(w\vert_{S - (\overline{B} \setminus O(x))})'$ (resp. $(v\vert_{S - (\overline{B} \setminus O(x))})'$) be the vector field of the flow $w\vert_{S - (\overline{B} \setminus O(x))}$ (resp. $v\vert_{S - (\overline{B} \setminus O(x))}$) by differentiating. 

\begin{claim}\label{claim:pret_B}
Replacing a flow box in $B \setminus O(x)$, we can obtain the resulting flow $w$ with $\Sv = \mathop{\mathrm{Sing}(w)}$ and $(w\vert_{S - (\overline{B} \setminus O(x))})' = (v\vert_{S - (\overline{B} \setminus O(x))})'$ such that the point $x$ is contained in the $\omega$-limit set of a point with respect to $w$. 
\end{claim}

\begin{proof}[Proof of Claim~\ref{claim:pret_B}]
By $x \in \partial \Pv$, there are a closed transverse arc $I$ one of whose boundary components is $x$ and a sequence $(a_n)_{n \in \Z_{\geq 0}}$ of periodic points $a_n \in B \cap I$ such that $(a_n)_{n \in \Z_{\geq 0}}$ converges to $x$ monotonically in the closed interval $I$. 
By taking a subsequence of $(a_n)_{n \in \Z_{\geq 0}}$, from the finiteness of genus and boundary components of the compact surface $S$, we may assume that for any $n \neq m \in \Z_{\geq 0}$, the pair $O(a_n)$ and $O(a_m)$ are homotopic to each other such that $O(a_n) \neq O(a_m)$. 
From Lemma~\ref{lem:annuli}, we may assume that the connected component $A_n$ whose boundary is $O(a_n) \sqcup O(a_{n+1})$ for any $n \in \Z_{\geq 0}$ is an open annulus. 
By construction of $A_n$, we have $A_n \cap (O(x) \sqcup \bigsqcup_{n \in \Z_{\geq 0}}O(a_n)) = \emptyset$.
Then the disjoint union $\sqcup_{n \in \Z_{\geq 0}} A_n \sqcup O(a_{n+1})$ is an invariant open annulus whose boundary contains $O(x)$. 
By replacing the flow box $B$ as in the right of Figure~\ref{Fig:pert_B}, the resulting flow $w$ satisfies that $(w\vert_{S - (\overline{B} \setminus O(x))})' = (v\vert_{S - (\overline{B} \setminus O(x))})'$
 and that the orbit $O_w(a_1)$ contains the subset $\{a_n \mid n \in \Z_{\geq 0} \}$. 
\begin{figure}
\begin{center}
\includegraphics[scale=0.95]{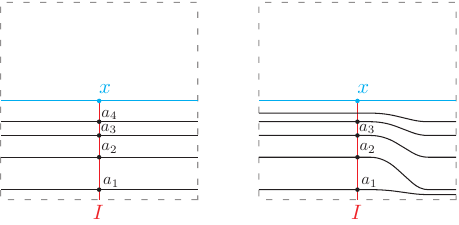}
\end{center}
\caption{Left: the flow box $B$; right: perturabated flow box from the flow box $B$}
\label{Fig:pert_B}
\end{figure} 
From $\lim_{n \to \infty}a_n = x$, the $\omega$-limit set $\omega_{w}(a_1)$ contains $x$. 
\end{proof}

\cite[Theorem~A(b)]{yokoyama2021poincare} implies that $O_w(x)$ is a connecting quasi-separatrix with respect to $w$. 
Since $\overline{O_w(x)} \subset S - (\overline{B} \setminus O(x))$, we have $O(x) = O_w(x)$, $\omega(x) = \omega_w(x) \subseteq \mathop{\mathrm{Sing}(w)} = \Sv$, and $\alpha(x) = \alpha_w(x) \subseteq \mathop{\mathrm{Sing}(w)} = \Sv$. 
This means that $O(x)$ is a connecting quasi-separatrix with respect to $v$. 
\end{proof}

\subsection{Border of non-recurrent closed orbits for flows with finitely many singular points}

In this section, from now on, let $v$ be a flow with finitely many singular points on a compact connected surface $S$ unless otherwise stated.

\begin{proposition}\label{prop:cs}
The following statements hold for any flow $v$ with finitely many singular points on a compact surface: 
\\
{\rm(1)} Each orbit in $\partial^{-} \mathrm{P}(v)$ is a connecting separatrix.
\\
{\rm(2)} Each limit circuit of $v$ is a boundary component of its semi-attracting or semi-repelling collar basin. 
\\
{\rm(3)} For any point $x$ whose $\omega$-limit set or $\alpha$-limit set is a limit circuit, we have $x \in \mathop{\mathrm{int}} \mathrm{P}(v)$. 
\\
{\rm(4)} If $v$ is quasi-regular, then $\partial^{-} \mathrm{P}(v)$ is a finite union of multi-saddle  separatrices.
\end{proposition}

\begin{proof}
Let $v$ be a flow with finitely many singular points on a compact surface. 
Lemma~\ref{lem:cs-} implies assertion (1).  
From \cite[Lemma~8 and Proposition~51]{yokoyama2021poincare}, assertions (2) and (3) hold. 

\begin{claim}\label{claim:051}
Assertion (4) holds. 
\end{claim}

\begin{proof}[Proof of Claim~\ref{claim:051}]
Suppose that $v$ is quasi-regular. 
Fix a point $x \in \partial^{-} \mathrm{P}(v)$. 
By assertion (1), the finiteness of singular points implies that the $\omega$-limit set $\omega(x)$ is a singular point. 
Since each point whose $\omega$-limit set is either a $\partial$-sink or a sink is contained in $\mathop{\mathrm{int}} \mathrm{P}(v)$, the $\omega$-limit set $\omega(x)$ is neither $\partial$-sink nor a sink. 
This means that the $\omega$-limit set $\omega(x)$ is a multi-saddle. 
By symmetry, so is $\alpha$-limit set $\alpha(x)$. 
\end{proof}

This completes the proof. 
\end{proof}

The finiteness of $\mathop{\mathrm{Sing}}(v)$ in Proposition~\ref{prop:cs}(1) is necessary. 
In fact, there is a flow with a point $y \in \partial^{-} \mathrm{P}(v)$ whose $\omega$-limit set is not a point (see Figure~\ref{NAC}).
\begin{figure}
\begin{center}
\includegraphics[scale=0.25]{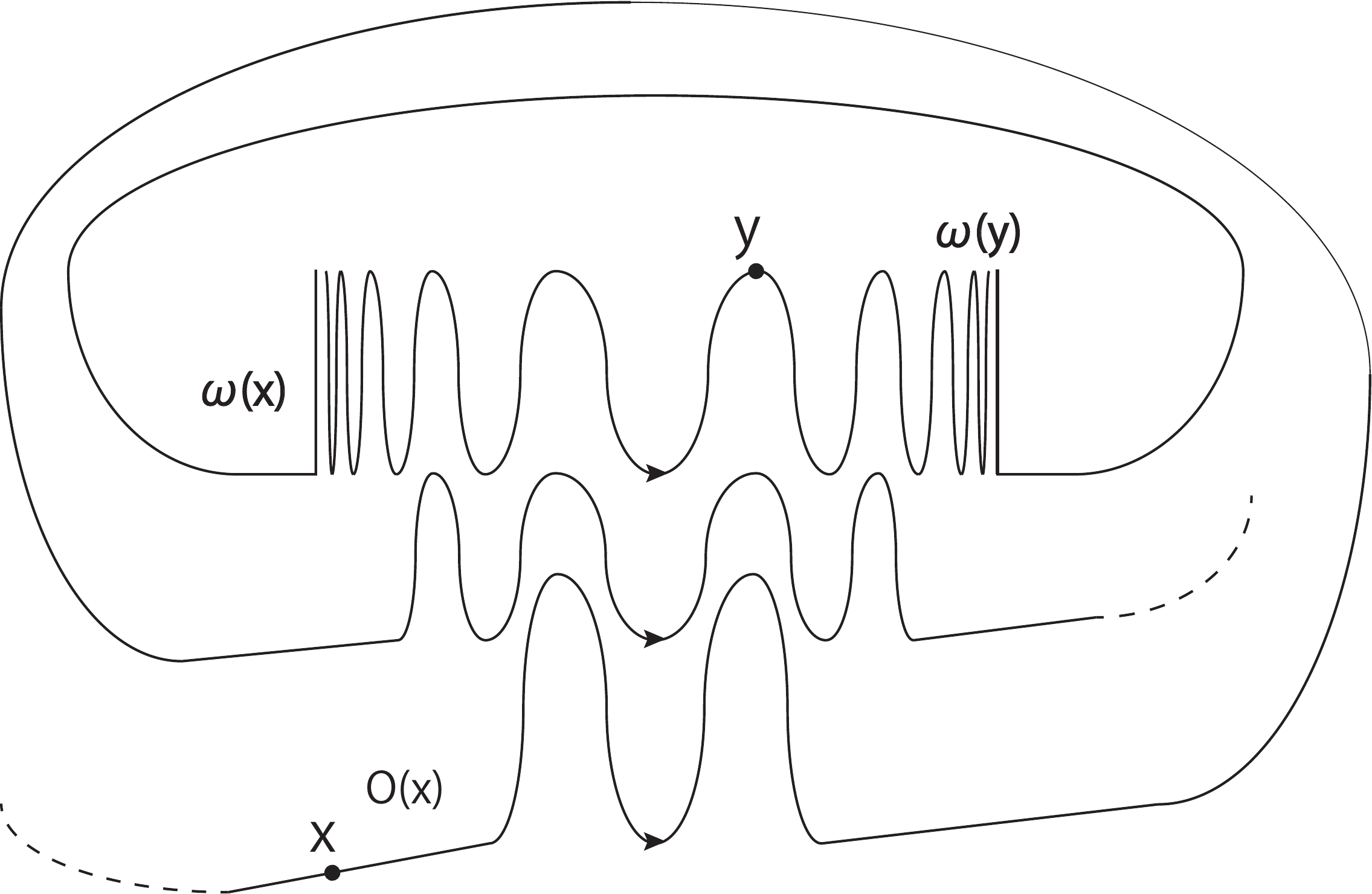}
\end{center}
\caption{An $\omega$-limit set which is not arcwise-connected.}
\label{NAC}
\end{figure}
The quasi-regularity of $\mathop{\mathrm{Sing}}(v)$ in Proposition~\ref{prop:cs}(1) is necessary.
In fact, there is a flow with infinitely many connecting separatrices (see Figure~\ref{NAC02+}).

\subsubsection{On proper semi-multi-saddle separatrices}

We define $\mathop{\mathrm{P}_{\mathrm{semi}}}(v)$ for a quasi-regular flow on a surface as follows. 

\begin{definition}\label{semi_P_qr}
We define $\mathop{\mathrm{P}_{\mathrm{semi}}}(v)$ as the union of proper semi-multi-saddle separatrices. 
\end{definition}

Note that we will define $\mathop{\mathrm{P}_{\mathrm{semi}}}(v)$ in more general setting (see Definition~\ref{semi_P} for details). 
The previous proposition and Lemma~\ref{lem:semi-multi-saddles} implies the following statement. 

\begin{corollary}\label{cor:semi-multi-separatrices}
If $v$ is quasi-regular, then the union $\partial^{-} \mathrm{P}(v) \sqcup \mathop{\mathrm{P}_{\mathrm{ms}}}(v) \sqcup \mathop{\mathrm{P}_{\mathrm{ss}}}(v)$ is the union $\mathop{\mathrm{P}_{\mathrm{semi}}}(v)$ of proper semi-multi-saddle separatrices. 
\end{corollary}

\subsection{Descriptions of border point sets and relative concepts}

Recall that $\mathop{\mathrm{BD}}(v) = \mathop{\mathrm{Bd}}(v) \sqcup \mathop{\mathrm{Per}_{1}}(v)$ and $\mathop{\mathrm{Bd}}(v) = \partial \mathop{\mathrm{Sing}}(v) \cup \partial \mathop{\mathrm{Per}}(v) \cup \partial_{\mathop{\mathrm{Per}}(v)} \cup \partial \mathrm{P}(v) \cup \partial \mathrm{LD}(v) \cup  \mathop{\mathrm{P}_{\mathrm{sep}}}(v) \cup \partial \mathrm{E}(v)$. 
In this subsection, we demonstrate the following description of the (weak) border point set below (see the proof of the following lemma in \ref{prf:lem44}).

\begin{lemma}\label{lem010}
The following statements hold for a flow $v$ with finitely many singular points on a compact surface $S$: 
\\
{\rm(1)}  $\mathop{\mathrm{Bd}}(v)$ and $\mathop{\mathrm{BD}}(v)$ are closed. 
\\
{\rm(2)} $\mathop{\mathrm{Bd}}(v) = \mathop{\mathrm{Sing}}(v) \sqcup \partial^{-} \mathop{\mathrm{Per}}(v) \sqcup \partial_{\mathop{\mathrm{Per}}(v)} \sqcup \partial^{-} \mathrm{P}(v) \sqcup \mathop{\mathrm{P}_{\mathrm{sep}}}(v) \sqcup \mathrm{E}(v) = \mathop{\mathrm{Sing}}(v) \sqcup \partial^{-} \mathop{\mathrm{Per}}(v) \sqcup \partial^{-} \mathrm{P}(v) \sqcup \mathrm{E}(v) \sqcup \mathop{\mathrm{Bd}}_{\mathrm{int}}(v)$. 
%
\\
{\rm(3)} The union $\mathop{\mathrm{P}_{\mathrm{sep}}}(v)$
consists of finitely many connecting separatrices.
\\
{\rm(4)} The union $\partial_{\mathop{\mathrm{Per}}(v)} \sqcup \mathop{\mathrm{Per}_{1}}(v)$ is the finite union of one-sided periodic orbits in $\mathop{\mathrm{int}} \mathop{\mathrm{Per}}(v)$. 
\\
{\rm(5)} If there are at most finitely many limit cycles, then the union $\partial^{-} \mathop{\mathrm{Per}}(v)$ is the finite union of limit cycles. 
\\
{\rm(6)} If $v$ is quasi-regular, then the union $\partial^{-} \mathrm{P}(v)\sqcup \mathop{\mathrm{P}_{\mathrm{sep}}}(v) = \mathop{\mathrm{P}_{\mathrm{semi}}}(v) \sqcup \mathop{\partial_{\mathrm{P}(v)}}$ is the finite union of proper semi-multi-saddle separatrices and separatrices on $\partial S$ between $\partial$-sources and $\partial$-sinks. 
In particular, the quasi-regularity implies $\partial^{-} \mathrm{P}(v) \subseteq \mathop{\mathrm{P}_{\mathrm{semi}}}(v)$. 
\end{lemma}

To demonstrate the previous lemma, we show the following properties.

\begin{lemma}\label{lem001}
For a flow $v$ with finitely many singular points on a compact surface, we have $\mathrm{LD}(v) \cap \overline{\mathrm{P}(v)} = \emptyset$.
\end{lemma}

\begin{proof}
By Lemma~\ref{lem:ori_cover}, we may assume that $S$ is connected and orientable. 
Assume that there is a point $x \in \mathrm{LD}(v) \cap \overline{\mathrm{P}(v)}$.
\cite[Theorem VI]{cherry1937topological} implies that the orbit class $\hat{O}(x)$ contains infinitely many Poisson stable orbits and so there is a non-closed recurrent point $z$ with $x \in \omega(z) \cap \alpha(z)$. 
Since the orbit closure $\overline{O(x)}$ is a neighborhood of $x$, there are an open trivial flow box $U \subset \overline{O(x)}$ centered at $x$ and an open transverse arc $I \subset U$ intersecting $x$ with $v((-\varepsilon,\varepsilon),I) = U$ for some $\varepsilon>0$. 
Moreover, there is a sequence $(x_n)_{n \in \Z_{>0}}$ in $I \cap \overline{O(x)} \cap \partial^- \mathrm{P}(v) \subseteq \overline{\mathrm{LD}(v)}$ converging to $x$. 
Then $x_n \in \omega(z) \cap \alpha(z)$ for any $n \in \Z_{>0}$. 
Since neither $\omega(x_n)$ nor $\alpha(x_n)$ are limit cycles, Lemma~\ref{pos_rec} and the dual statement for $\alpha$-limit sets imply that $\overline{O(x_n)} - O(x_n)$ consists of at most two singular points.
By renumbering, we may assume that $(x_n)_{n \in \Z_{>0}}$ converges to $x$ and is decreasing with respect to the transverse direction $I$ from one side in the trivial flow box $U$ such that $O(x_n) \cap O(x_m) = \emptyset$ for any $n \neq m \in \Z_{>0}$.  

\begin{claim}\label{claim:014}
There is a subsequence $(x_{k_n})_{n \in \Z_{>0}}$ of $(x_n)_{n \in \Z_{>0}}$ such that $\overline{O(x_n)}$ is a simple closed curve. 
\end{claim}

\begin{proof}[Proof of Claim~\ref{claim:014}]
Suppose that there is no subsequence $(x_{k_n})_{n \in \Z_{>0}}$ of $(x_n)_{n \in \Z_{>0}}$ such that $\overline{O(x_n)}$ is a simple closed curve. 
By renumbering, each closure $\overline{O(x_n)}$ is not a simple closed curve. 
Then $\omega(x_n) \neq \alpha(x_n)$ for any $n \in \Z_{>0}$. 
Since $|\mathop{\mathrm{Sing}}(v)| < \infty$, there are singular points $\alpha \neq \omega \in \overline{\mathrm{LD}(v)}$ and a subsequence $(x_{k_n})_{n \in \Z_{>0}}$ of $(x_n)_{n \in \Z_{>0}}$ such that $\omega(x_{k_n}) = \omega$ and $\alpha(x_{k_n}) = \alpha$. 
Since $S$ is compact, there are an open disk $D \subseteq \overline{O(x)}$ and three points $x_{k_{n}}, x_{k_{n}+1}, x_{k_{n}+2}$ in the subsequence such that $\partial D = \overline{O(x_{k_{n}})} \cup \overline{O(x_{k_{n}+2})}$
and $x_{k_{n}+1} \in \mathrm{int}D$. 
Then $\overline{\mathrm{LD}(v)} \cap \mathrm{int}D = \emptyset$ and so $x_{k_{n}+1} \notin \overline{\mathrm{LD}(v)}$, which contradicts the choice of $x_{k_{n}+1}$.
\end{proof}

From Claim~\ref{claim:014}, by taking a subsequence $(x_{k_n})_{n \in \Z_{>0}}$ of $(x_{k_n})_{n \in \Z_{>0}}$, we may assume that there is a singular point $\gamma$ with $\omega(x_n) = \alpha(x_n) = \gamma$ for any $n \in \Z_{>0}$. 
Since the isotopic classes of pairwise disjoint loops on a compact surface are finite, by taking a subsequence $(x_{k_n})_{n \in \Z_{>0}}$ of $(x_{k_n})_{n \in \Z_{>0}}$, we may assume that any pair of $\overline{O(x_n)}$ and $\overline{O(x_m)}$ are homotopic relative to $\gamma$. 

\begin{claim}\label{claim:015}
For any $n \in \Z_{>0}$, the loop $\overline{O(x_n)}$ is null homotopic and so intersects no closed transversals. 
\end{claim}

\begin{proof}[Proof of Claim~\ref{claim:015}]
Assume $\overline{O(x_n)}$ is not null homotopic.
Then $\overline{O(x_n)}$ does not bound any open disks.  
By renumbering $(x_{k_n})_{n \in \Z_{>0}}$, we may assume that there is a connected component of the complement of $\overline{O(x_n)} \cup \overline{O(x_{n+2})}$ containing $O(x_{n+1})$ is an open disk, which contradicts $x_{n+1} \in \overline{O(x)} \subseteq \overline{\mathrm{LD}(v)}$. 
\end{proof}

From Claim~\ref{claim:015}, by taking a subsequence $(x_{k_n})_{n \in \Z_{>0}}$ of $(x_n)_{n \in \Z_{>0}}$, we may assume that there are open disks $D_n$ whose boundary is a loop $\partial D_n = \{ \gamma \} \sqcup O_n$ for any $n \in \Z_{>0}$.

\begin{claim}\label{claim:016}
We may assume that the intersection $O(x_n) \cap U$ is one orbit arc for any $n \in \Z_{>0}$.
\end{claim}

\begin{proof}[Proof of Claim~\ref{claim:016}]
Assume that there is a subsequence  $(x_{k_n})_{n \in \Z_{>0}}$ of $(x_n)_{n \in \Z_{>0}}$ such that $O(x_{k_n}) \cap U$ consists of at least two maximal orbit arcs. 
If the restriction of the return map on $I$ near $I \cap O(x_{k_n})$ is non-orientable, then there are an orbit arc in $O(x_{k_n})$ and a sub-arc of $I$ whose union is a one-sided loop. 
By the finiteness of the genus of $S$, there are at most finitely many pairwise disjoint one-sided loops.
Therefore, since the sequence $(x_{k_n})_{n \in \Z_{>0}}$ on $I$ converges to $x$, by taking a subsequence of $(x_{k_n})_{n \in \Z_{>0}}$, we may assume that the restriction of the return map on $I$ near $I \cap O(x_{k_n})$ is orientable. 
Then the waterfall construction (cf. \cite[Lemma~3.3.7 p.86]{Candel2000foliation})  to the loop consisting an orbit arc in $O(x_{k_n})$ and a transverse arc in $U$ implies the existence of a closed transversal intersecting $O(x_{k_n})$, which contradicts the non-existence of closed transversals from Claim~\ref{claim:015}.  
\end{proof}

Since the complement $S - \overline{O(x_n)}$ consists of two connected components and one of them is an open disk and since the sequence $(x_n)_{n \in \Z_{>0}}$ in $I$ converges to $x$, Claim~\ref{claim:016} implies that the loop $O(x_{n+1}) \sqcup \{ \gamma \} = \overline{O(x_{n+1})}$ for any $n \in \Z_{>0}$ contains a disk which contains either $O(x_{n})$ or $O(x_{n+2})$, which contradicts $x_{n}, x_{n+2} \in \overline{O(x)} \subseteq \overline{\mathrm{LD}(v)}$. 
Thus $\mathrm{LD}(v) \cap \overline{\mathrm{P}(v)} = \emptyset$.
\end{proof}


\begin{lemma}\label{lem002}
For a flow $v$ with finitely many singular points on a compact surface, the union $\mathrm{LD}(v)$ is open.
\end{lemma}

\begin{proof}
Since the singular point set $\mathop{\mathrm{Sing}}(v)$ is closed, Lemma~\ref{lem3-04} and Lemma~\ref{lem001} imply that $\overline{\mathop{\mathrm{Cl}}(v) \sqcup \mathrm{P}(v) \sqcup \mathrm{E}(v)} \cap \mathrm{LD}(v) = \emptyset$ and so that the union $\mathrm{LD}(v) = S - \overline{\mathop{\mathrm{Cl}}(v) \sqcup \mathrm{P}(v) \sqcup \mathrm{E}(v)}$ is open.
\end{proof}

The finiteness of singular points in  Lemma~\ref{lem001} and Lemma~\ref{lem002} is necessary, because there is a flow $v$ on closed surfaces with countably many singular points such that $\mathrm{LD}(v)$ is not open (cf. \cite[Example~2.10]{yokoyama2016topological}). 
We obtain the following description of a neighborhood of $\mathrm{E}(v)$.

\begin{lemma}\label{lem3-06}
For a flow $v$ with finitely many singular points on a compact surface, we obtain $\partial^{-} \mathrm{E}(v) = \mathrm{E}(v) \subseteq  \mathrm{int}(\mathrm{E}(v) \sqcup \mathrm{int}\mathrm{P}(v))  = \mathrm{E}(v) \sqcup \mathrm{int}\mathrm{P}(v)$.
\end{lemma}

\begin{proof}
By the Ma{\v \i}er theorem \cite{mayer1943trajectories} (cf. \cite[Remark~2]{aranson1996maier}), the closure $\overline{\mathrm{E}(v)}$ is the finite union of closures of exceptional orbits and so $\partial^{-} \mathrm{E}(v) = \mathrm{E}(v)$.
The closedness of $\mathop{\mathrm{Sing}}(v)$ and Lemma~\ref{lem3-04} imply $\mathrm{E}(v) \cap \overline{\mathop{\mathrm{Cl}}(v) \sqcup \mathrm{LD}(v)} = \emptyset$. 
Then we have that $\mathrm{E}(v) \subseteq S -  \overline{\mathop{\mathrm{Cl}}(v) \sqcup \mathrm{LD}(v)}
= \mathrm{int}(\mathrm{P}(v) \sqcup \mathrm{E}(v))$.
%
%

\begin{claim}\label{claim:017}
$\overline{\partial^{-} \mathrm{P}(v)} \cap \mathrm{E}(v) = \emptyset$. 
\end{claim}

\begin{proof}[Proof of Claim~\ref{claim:017}]
Assume that there is a point $x \in \overline{\partial^{-} \mathrm{P}(v)} \cap \mathrm{E}(v)$. 
Then there are a transverse closed arc $\gamma$, a sequence $\{ x_n \}_{n \in \Z_{>0}}$ on $\gamma \cap \partial^{-} \mathrm{P}(v)$ convergening to $x$ monoronically from a one side, and a sequence $\{ O_n \}_{n \in \Z_{>0}}$ of orbits of $x_n$ in $\partial^{-} \mathrm{P}(v)$ and  with $x \in \partial \gamma \setminus \bigsqcup_n O_n$ such that $x \in \overline{\bigcup_n O_n \cap \gamma}$ and 
$O_n \neq O_m$ for any $n \neq m \in \Z_{>0}$. 

Assume that there are infinitely many orbits $O_{k_n}$ which are homoclinic. 
Taking a subsequence, we may assume that there is a singular point $y$ with $y = \omega(O_n) = \alpha(O_n)$ for any $n \in \Z_{>0}$. 
Since $S$ is compact, taking a subsequence, we may assume that $\{ y \} \sqcup O_{n}$ and $\{ y \} \sqcup O_{n+1}$ are homotopic relative to $\{ y \}$. 
As above, taking a subsequence, for any $n \in \Z_{>0}$, we may assume that a union $\{ y \} \sqcup O_{2n} \sqcup O_{2n+2}$ bounds an open disk $B_n$, and that the open disk $B_n$  contains $O_{2n+1}$ such that $B_n - O_{2n+1}$ consists of two open disks.  
For any $n \in \Z_{>0}$, by $O_{2n+1} \subset \partial^{-} \mathrm{P}(v)$, let $C_n$ be the one of two open disks $B_n - O_{2n+1}$ intersecting $S - \mathrm{P}(v)$. 
Since a closed disk whose boundary is saturated does not intersect $\mathrm{R}(v)$, the disk $C_n$ intersects $S - (\mathrm{P}(v) \sqcup \mathrm{R}(v)) = \mathop{\mathrm{Cl}}(v)$ for any $n \in \Z_{>0}$. 
Since a null homotopic periodic orbit bounds a singular point, the disk $C_n$ contains singular points for any $n \in \Z_{>0}$. 
Hence there are infinitely many singular points, which contradicts the finite hypothesis of singular points. 

Thus, there are at most finitely many orbits $O_{k_n}$ which are homoclinic. 
Taking a subsequence, we may assume that there are singular points $\alpha \neq \omega$ with $\omega = \omega(O_n)$ and $\alpha = \alpha(O_n)$ for any $n$. 
Since $S$ is compact, taking a subsequence, we may assume that the closed arcs $\{ \alpha, \omega \} \sqcup O_{n}$ and $\{ \alpha, \omega \} \sqcup O_{n+1}$ are homotopic relative to $\{ \alpha, \omega \}$. 
As above, taking a subsequence, for any $n \in \Z_{>0}$, we may assume that a union $\{ \alpha, \omega \} \sqcup O_{2n} \sqcup O_{2n+2}$ bounds an open disk $B_n$, and that the open disk $B_n$ contains $O_{2n+1}$ such that $B_n - O_{2n+1}$ consists of two open disks.  
By $O_{2n+1} \subset \partial^{-} \mathrm{P}(v)$, let $C_n$ be the one of two open disks $B_n - O_{2n+1}$ intersecting $S - \mathrm{P}(v)$. 
Since $\partial C_n \subset \{ \alpha, \omega \} \sqcup O_{2n} \sqcup O_{2n+1} \sqcup O_{2n+2}$, 
the disk $C_n$ intersects $S - (\mathrm{P}(v) \sqcup \mathrm{R}(v)) = \mathop{\mathrm{Cl}}(v)$. 
Since a null homotopic periodic orbit bounds a singular point, the disk $C_n$ contains singular points. 
Hence, there are infinitely many singular points, which contradicts the finite hypothesis of singular points. 
\end{proof}

By Claim~\ref{claim:017}, the set difference $\mathrm{int}(\mathrm{E}(v) \sqcup \mathrm{P}(v))  \setminus \overline{\partial^{-} \mathrm{P}(v)} \subseteq (\mathrm{E}(v) \sqcup \mathrm{P}(v)) \setminus \overline{\partial^{-} \mathrm{P}(v)} = \mathrm{E}(v) \sqcup \mathrm{int}\mathrm{P}(v)$ is an open neighborhood of $\mathrm{E}(v)$. 
Since $\mathrm{int}\mathrm{P}(v) \subset \mathrm{int}(\mathrm{E}(v) \sqcup \mathrm{P}(v))$, we have $E \sqcup \mathrm{int}\mathrm{P}(v) \subseteq \mathrm{int}(\mathrm{E}(v) \sqcup \mathrm{int}\mathrm{P}(v))$ and so $\mathrm{int}(\mathrm{E}(v) \sqcup \mathrm{int}\mathrm{P}(v)) = \mathrm{E}(v) \sqcup \mathrm{int}\mathrm{P}(v)$. 
\end{proof}

Recall the frontier $\bp A = \overline{A} - A$ of a subset $A \subseteq S$.
The following frontier's property holds.

\begin{lemma}\label{lem009}
For a flow $v$ with finitely many singular points on a compact surface, we have $\mathrm{E}(v) \subseteq \bp \mathrm{P}(v) \subseteq \mathop{\mathrm{Sing}}(v) \sqcup \partial^{-} \mathop{\mathrm{Per}}(v)  \sqcup \mathrm{E}(v)$. 
\end{lemma}

\begin{proof}
By \cite[Lemma 2.3]{yokoyama2016topological}, we obtain $\mathrm{E}(v) \subseteq \overline{\mathrm{int}\mathrm{P}(v)} \setminus \mathrm{P}(v) \subseteq \bp \mathrm{P}(v)$.
Lemma~\ref{lem002} implies that the union $\mathrm{LD}(v) \sqcup \mathrm{int}\mathop{\mathrm{Per}}(v)$ are open and so that $\overline{\mathrm{P}(v)} - (\mathrm{P}(v) \sqcup \mathrm{E}(v)) = \overline{\mathrm{P}(v)} \setminus (\mathrm{int}\mathop{\mathrm{Per}}(v) \sqcup \mathrm{P}(v) \sqcup \mathrm{R}(v)) \subseteq S - (\mathrm{int}\mathop{\mathrm{Per}}(v) \sqcup \mathrm{P}(v) \sqcup \mathrm{R}(v)) = \mathop{\mathrm{Sing}}(v) \sqcup \partial^{-} \mathop{\mathrm{Per}}(v)$.
Therefore $\bp \mathrm{P}(v) \subseteq \mathop{\mathrm{Sing}}(v) \sqcup \partial^{-} \mathop{\mathrm{Per}}(v)  \sqcup \mathrm{E}(v)$.
\end{proof}

\subsubsection{Properties of the set $\mathrm{P}(v)$ of non-recurrent points}

We obtain the following description of a neighborhood of $\mathop{\mathrm{Per}}(v)$.

\begin{lemma}\label{lem007a}
Let $v$ be a flow with finitely many singular points on a compact surface. 
Then the disjoint union $\mathop{\mathrm{Per}}(v) \sqcup \mathrm{int}\mathrm{P}(v)$ is an open neighborhood of $\mathop{\mathrm{Per}}(v)$.
In particular, we have $\mathop{\mathrm{Per}}(v) \cap \overline{\partial^{-} \mathrm{P}(v)} = \emptyset$.
\end{lemma}

\begin{proof}
By Lemma~\ref{lem:ori_cover}, taking the orientation double covering and the double of the surface $S$ if necessary, we may assume that $S$ is closed and orientable.
Fix a point $x \in \mathop{\mathrm{Per}}(v)$.
Take a small collar $U \subseteq S - \mathop{\mathrm{Sing}}(v)$ of $O(x)$.
The flow box theorem 
 implies that there are an open annulus $\A \subseteq U$ which is also a collar of $O(x)$ and a transverse closed arc $T_{+} \subseteq \overline{\A}$ such that $x$ is contained in $\partial T_{+}$.
Then there are a small subinterval $I_{+} \subseteq T_{+}$ containing $x$ and the first return map $f_{v+}$ on $T_{+}$ induced by $v$ whose domain contains $I_{+}$. 
Then $f_{v+}(x) = x$.
Identify $T_{+}$ with $[0,1]$ (resp. $x$ with $0$).

\begin{claim}\label{claim:011}
We may assume that $\mathrm{Sat}_v(T_+) \subseteq \mathop{\mathrm{Per}}(v) \sqcup \mathrm{int}\mathrm{P}(v)$. 
\end{claim}

\begin{proof}[Proof of Claim~\ref{claim:011}]
Suppose that the restriction $f_{v+}\vert_{S} \colon I_{+} \to T_{+}$ is either attracting or repelling near zero.
Then $O$ is a limit cycle. 
By taking $T_+$ small, we may assume that $\mathrm{Sat}_v(T_+) \subseteq \mathop{\mathrm{Per}}(v) \sqcup \mathrm{int}\mathrm{P}(v)$.

Suppose that the restriction $f_{v+}\vert_{S} \colon I_{+} \to T_{+}$ is neither attracting nor repelling near zero.
Then zero is an accumulation point of the fixed point set $\mathop{\mathrm{Fix}}(f_{v+})$.
Shortening $T_+$, we may assume that $f_{v+}(1) = 1$ and so $S_+ = T_+$.
Then the saturation of $T_+$ is a closed annulus, and each orbit intersecting $\mathrm{Fix}(f_{v+})$ is contained in
$\mathop{\mathrm{Per}}(v)$ and each orbit intersecting $T_{+} - \mathrm{Fix}(f_v)$ is contained in $\mathop{\mathrm{int}} \mathrm{P}(v)$.
Therefore $\mathrm{Sat}_v(T_+) \subseteq \mathop{\mathrm{Per}}(v) \sqcup \mathrm{int}\mathrm{P}(v)$.
\end{proof}

By symmetry, there is a transverse closed arc $T_{-} \subseteq \overline{\A}$ with $x \in \partial T_{-}$ such that $T_{-} \cap U = \emptyset$ and $\mathrm{Sat}_v(T_-) \subseteq \mathop{\mathrm{Per}}(v) \sqcup \mathrm{int}\mathrm{P}(v)$.
Then the union $T := T_- \cup T_+$ is a transverse closed arc with $x \in \mathop{\mathrm{int}} T$ such that $\mathrm{Sat}_v(T) \subseteq \mathop{\mathrm{Per}}(v) \sqcup \mathrm{int}\mathrm{P}(v)$.
Since $\mathrm{Sat}_v(T)$ is a \nbd of $x$, the disjoint union $\mathop{\mathrm{Per}}(v) \sqcup \mathop{\mathrm{int}} \mathrm{P}(v)$ is a neighborhood of $x$ and so of $\mathop{\mathrm{Per}}(v)$.

In addition, since $(\mathop{\mathrm{Per}}(v) \sqcup \mathop{\mathrm{int}} \mathrm{P}(v)) \cap \partial^{-} \mathrm{P}(v) = \emptyset$, we obtain $\mathop{\mathrm{int}}(\mathop{\mathrm{Per}}(v) \sqcup \mathop{\mathrm{int}} \mathrm{P}(v)) \cap \partial^{-} \mathrm{P}(v) = \emptyset$ and so  $\emptyset = \mathop{\mathrm{int}}(\mathop{\mathrm{Per}}(v) \sqcup \mathop{\mathrm{int}} \mathrm{P}(v)) \cap \overline{\partial^{-} \mathrm{P}(v)} \supseteq \mathop{\mathrm{Per}}(v) \cap \overline{\partial^{-} \mathrm{P}(v)}$ because of $\mathop{\mathrm{Per}}(v) \subseteq  \mathop{\mathrm{int}} (\mathop{\mathrm{Per}}(v) \sqcup \mathop{\mathrm{int}} \mathrm{P}(v))$. 
Therefore $\mathop{\mathrm{Per}}(v) \cap \overline{\partial^{-} \mathrm{P}(v)} = \emptyset$.
\end{proof}

We describe the border of $\mathop{\mathrm{Per}}(v)$. 

\begin{lemma}\label{lem007}
Let $v$ be a flow with finitely many singular points on a compact surface $S$.
Then the border $\partial^{-} \mathop{\mathrm{Per}}(v) \subseteq \Pv \cap \overline{\mathrm{int}\mathrm{P}(v)}$ is the union of periodic orbits contained in the closure of the union of limit cycles.
\end{lemma}

\begin{proof}
By Lemma~\ref{lem:ori_cover}, taking the orientation double covering if necessary, we may assume that $S$ is orientable.
From Lemma~\ref{lem007a}, the union $\mathop{\mathrm{Per}}(v) \sqcup \mathrm{int}\mathrm{P}(v)$ is an open neighborhood of $\mathop{\mathrm{Per}}(v)$ and so $\partial^{-} \mathop{\mathrm{Per}}(v) \subset \Pv \cap \overline{\mathrm{int}\mathrm{P}(v)}$.

Fix a periodic orbit $O \subseteq \partial^{-} \mathop{\mathrm{Per}}(v)$.
If $O$ is a limit cycle, then $O$ is contained in the union of limit cycles.
Thus, we may assume that $O$ is not a limit cycle.
By the compactness of $O$, there is a transverse closed arc $T \subseteq \overline{\A}$ whose boundary contains a point of $O$ and whose interior does not intersect $O$.
Identify $T$ with $[0,1]$ (resp. $O \cap T$ with $0$).
Let $f_v$ be the first return map on $T$ induced by $v$.
Then $f_v(0) = 0$. 
By the compactness of $O$, the flow box theorem implies that the domain of $f_v$ contains a non-degenerate interval $I$ in $[0,1]$ with $0 \in I$.
The fact  $O \subset \overline{\mathrm{int}\mathrm{P}(v)}$ implies that the first return map $f_v$ is not identical. 
Since $O$ is not a limit cycle, the return map $f_v\vert_{I}: I \to T$ is neither attracting nor repelling near zero.
This implies that $0 \in I$ is an accumulation point of the fixed point set $\mathop{\mathrm{Fix}}(f_v)$.
Since $O \subseteq \overline{\mathrm{int}\mathrm{P}(v)}$, the transversal subinterval $I$ has a convergence sequence to zero of fixed points which are either attracting or repelling from at least one side with respect to $f_v$.
This means that $O$ is contained in the closure of the union of limit cycles intersecting $I$.
Therefore, the border $\partial^{-} \mathop{\mathrm{Per}}(v)$ is contained in the closure of the union of limit cycles, and so is contained in the union of periodic orbits of the closure of the union of limit cycles.

Conversely, any periodic orbits contained in the closure of the union of limit cycles are contained in $\partial^{-} \mathop{\mathrm{Per}}(v)$. 
This means that the border $\partial^{-} \mathop{\mathrm{Per}}(v)$ is the union of periodic orbits contained in the closure of the union of limit cycles.
\end{proof}

\subsubsection{Properties of the boundary points, borders, and frontiers}

We summarize the properties of the borders and frontiers as follows.

\begin{proposition}\label{prop4-12}
Let $v$ be a flow with finitely many singular points on a compact surface $S$.
The following statements hold:
\\
$(1)$ 
$\bp \mathrm{LD}(v) = \partial \mathrm{LD}(v)$ $(\mathrm{i.e.} \,\, \mathrm{LD}(v)$ is open $)$.
\\
$(2)$ 
$\bp \mathop{\mathrm{Sing}}(v) = \emptyset$.
\\
$(3)$ 
$\bp \mathop{\mathrm{Per}}(v) \sqcup \bp \mathrm{LD}(v) \sqcup \bp \mathrm{E}(v)
\subseteq \mathop{\mathrm{Sing}}(v) \sqcup \partial^{-} \mathrm{P}(v)$.
\\
$(4)$ 
$\bp \mathrm{P}(v) \subseteq \mathop{\mathrm{Sing}}(v) \sqcup \partial^{-} \mathop{\mathrm{Per}}(v)  \sqcup \mathrm{E}(v)$.
\end{proposition}

\begin{proof}
Assertion $(1)$ is immediate from Lemma~\ref{lem002}.
The finiteness of $\mathop{\mathrm{Sing}}(v)$ implies assertion $(2)$.
Assertion $(3)$ is obtained by Lemma~\ref{lem3-04}.
Assertion $(4)$ is assured by Lemma~\ref{lem009}. 
\end{proof}


\begin{proposition}\label{prop4-12.5}
Let $v$ be a flow with finitely many singular points on a compact surface $S$.
The following statements hold:
\\
$(1)$  
$\partial^{-}  \mathop{\mathrm{Sing}}(v) =  \mathop{\mathrm{Sing}}(v)$.
\\
$(2)$ 
$\partial^{-} \mathop{\mathrm{Per}}(v) \subseteq \overline{\mathrm{int}\mathrm{P}(v)}$ is the union of periodic orbits contained in the closure of the union of limit cycles.
\\
$(3)$ 
$\partial^{-} \mathrm{P}(v) = (\bp \mathop{\mathrm{Per}}(v) \cup \bp \mathrm{LD}(v) \cup \bp \mathrm{E}(v)) \setminus \mathop{\mathrm{Sing}}(v)$ consists of connecting separatrices.
\\
$(4)$ 
$\partial^{-} \mathrm{LD}(v) = \emptyset$.
\\
$(5)$ 
$\partial^{-} \mathrm{E}(v) = \mathrm{E}(v) = \bp \mathrm{P}(v) \cap \mathrm{int}(\mathrm{E}(v) \sqcup \mathrm{int}\mathrm{P}(v))$.

Moverover, a disjoint union $\mathop{\mathrm{Sing}}(v) \sqcup \partial^{-} \mathop{\mathrm{Per}}(v) \sqcup \partial^{-} \mathrm{P}(v) \sqcup \mathrm{E}(v)$ is closed. 
\end{proposition}

\begin{proof}
The finiteness of $\mathop{\mathrm{Sing}}(v)$ implies assertion $(1)$.
Assertion $(2)$ is followed from Lemma~\ref{lem007}. 
Proposition~\ref{prop:cs} and Proposition~\ref{prop4-12} imply assertion $(3)$. 
Assertion $(4)$ is obtained by Lemma~\ref{lem002}. 
From Lemma~\ref{lem3-06}, assertion $(5)$ holds.
\end{proof}

\begin{proposition}\label{thm018}
Let $v$ be a flow with finitely many singular points on a compact surface $S$.
The following statements hold:
\\
$(1)$
$\partial \mathop{\mathrm{Per}}(v) \subseteq \mathop{\mathrm{Sing}}(v) \sqcup \partial^{-} \mathop{\mathrm{Per}}(v) \sqcup \partial^{-} \mathrm{P}(v)$.
\\
$(2)$
$\partial \mathrm{LD}(v) = \partial^+ \mathrm{LD}(v) \subseteq \mathop{\mathrm{Sing}}(v) \sqcup \partial^{-}  \mathrm{P}(v)$.
\\
$(3)$ 
$\partial \mathrm{E}(v) = \overline{\mathrm{E}(v)} \subseteq \mathop{\mathrm{Sing}}(v) \sqcup \partial^{-}  \mathrm{P}(v) \sqcup \mathrm{E}(v)$.
\\
$(4)$ 
$\partial \mathrm{P}(v) \subseteq \mathop{\mathrm{Sing}}(v) \sqcup \partial^{-} \mathop{\mathrm{Per}}(v) \sqcup \partial^{-} \mathrm{P}(v) \sqcup \mathrm{E}(v)$.
\end{proposition}

\begin{proof}
Proposition~\ref{prop4-12} implies assertions (1) - (4).
\end{proof}

The finiteness of singular points is necessary in the previous proposition {\rm(3)}.
In fact, a construction in Figure~\ref{NAC} implies that there is a flow with an orbit in $\partial^{-} \mathrm{P}(v) \cap (\bp \mathrm{LD}(v) \cup \bp \mathrm{E}(v))$ which is not a separatrix.

\subsubsection{Proof of description of the {\rm(}weak{\rm)} border point set}\label{prf:lem44}

We demonstrate the following description of the (weak) border point set.

\begin{proof}[Proof of Lemma~\ref{lem010}]
By the compactness of $S$, assertion {\rm(4)} holds. 
From Proposition~\ref{prop4-12.5}, assertion {\rm(5)} holds. 

\begin{claim}\label{claim:018}
Assertion {\rm(3)} holds. 
\end{claim}
\begin{proof}[Proof of Claim~\ref{claim:018}]
By the finiteness of multi-saddles, the union of proper semi-multi-saddle separatrices consists of finitely many orbits and so does the union $\mathop{\mathrm{P}_{\mathrm{ms}}}(v) \sqcup \mathop{\mathrm{P}_{\mathrm{ss}}}(v)$ of semi-multi-saddle separatrices in $\mathop{\mathrm{int}} \mathrm{P}(v)$. 
The compactness of $S$ implies that there are at most finitely connected components of $\partial S$ and so that $\mathop{\partial_{\mathrm{P}(v)}}$ consists of finitely many orbits. 
Therefore, the union $\mathop{\mathrm{P}_{\mathrm{sep}}}(v) = \mathop{\mathrm{P}_{\mathrm{ms}}}(v) \sqcup \mathop{\mathrm{P}_{\mathrm{ss}}}(v) \sqcup \mathop{\partial_{\mathrm{P}(v)}}$ consists of finitely many connecting separatrices.
\end{proof}

\begin{claim}\label{claim:019}
Assertion {\rm(1)} holds. 
\end{claim}
\begin{proof}[Proof of Claim~\ref{claim:019}]
By definition, the finite union $\mathop{\mathrm{Bd}}_{\partial}(v) = \partial \mathop{\mathrm{Sing}}(v) \cup \partial \mathop{\mathrm{Per}}(v) \cup \partial \mathrm{P}(v) \cup \partial \mathrm{LD}(v) \cup \partial \mathrm{E}(v)$ is closed. 
By assertion {\rm(4)}, the union $\partial_{\mathop{\mathrm{Per}}(v)}$ is closed. 
By assertion {\rm(3)}, the closure of the finite union $\mathop{\mathrm{P}_{\mathrm{sep}}}(v)$ of orbits is the finite union of the closures of orbits in $\mathop{\mathrm{P}_{\mathrm{sep}}}(v)$. 
By the fininteness of $\Sv$, we have $\overline{\mathop{\mathrm{P}_{\mathrm{sep}}}(v)} \subseteq \partial^- \Sv \sqcup \mathop{\mathrm{P}_{\mathrm{sep}}}(v)$. 
Threrefore, we obtain $\overline{\mathop{\mathrm{Bd}}_{\mathrm{int}}(v)} = \overline{\mathop{\mathrm{P}_{\mathrm{sep}}}(v) \sqcup \partial_{\mathop{\mathrm{Per}}(v)}} = \overline{\mathop{\mathrm{P}_{\mathrm{sep}}}(v)} \sqcup \partial_{\mathop{\mathrm{Per}}(v)} \subseteq \mathop{\mathrm{Bd}}_{\mathrm{int}}(v) \sqcup \partial^- \Sv \subseteq \mathop{\mathrm{Bd}}_{\mathrm{int}}(v) \sqcup \mathop{\mathrm{Bd}}_{\partial}(v) = \mathop{\mathrm{Bd}}(v)$. 
This means that $\mathop{\mathrm{Bd}}(v)$ is closed. 
The compactness of $S$ implies that there are at most finitely connected components of $\partial S$ and so that $\mathop{\mathrm{Per}_{1}}(v)$ consists of finitely many orbits. 
Therefore, the union $\mathop{\mathrm{BD}}(v) = \mathop{\mathrm{Bd}}(v) \sqcup \mathop{\mathrm{Per}_{1}}(v)$ is closed. 
\end{proof}

\begin{claim}\label{claim:020}
Assertion {\rm(2)} holds. 
\end{claim}
\begin{proof}[Proof of Claim~\ref{claim:020}]
The finiteness of $\mathop{\mathrm{Sing}}(v)$ implies $\partial  \mathop{\mathrm{Sing}}(v) =  \mathop{\mathrm{Sing}}(v)$.
Proposition~\ref{thm018} implies $\partial \mathop{\mathrm{Per}}(v) \subseteq \mathop{\mathrm{Sing}}(v) \sqcup \partial^{-} \mathop{\mathrm{Per}}(v) \sqcup \partial^{-} \mathrm{P}(v)$, $\partial \mathrm{P}(v) \subseteq \mathop{\mathrm{Sing}}(v) \sqcup \partial^{-} \mathop{\mathrm{Per}}(v) \sqcup \partial^{-} \mathrm{P}(v) \sqcup \mathrm{E}(v)$, $\partial \mathrm{LD}(v) \subseteq \mathop{\mathrm{Sing}}(v) \sqcup \partial^{-} \mathrm{P}(v)$, and $\partial \mathrm{E}(v) = \overline{\mathrm{E}(v)} \subseteq  \mathop{\mathrm{Sing}}(v) \cup \partial^{-} \mathrm{P}(v) \sqcup \mathrm{E}(v)$.
Therefore, we have $\mathop{\mathrm{Bd}}_{\partial}(v) = \mathop{\mathrm{Sing}}(v) \sqcup \partial^{-} \mathop{\mathrm{Per}}(v) \sqcup \partial^{-} \mathrm{P}(v) \sqcup \mathrm{E}(v)$ and so $\mathop{\mathrm{Bd}}(v) = \mathop{\mathrm{Sing}}(v) \sqcup \partial^{-} \mathop{\mathrm{Per}}(v) \sqcup \partial^{-} \mathrm{P}(v) \sqcup \mathrm{E}(v) \sqcup \mathop{\mathrm{Bd}}_{\mathrm{int}}(v)$.
%
\end{proof}

Suppose that $v$ is quasi-regular. 
From Corollary~\ref{cor:semi-multi-separatrices}, the union $\partial^{-} \mathrm{P}(v)\sqcup \mathop{\mathrm{P}_{\mathrm{ms}}}(v) \sqcup \mathop{\mathrm{P}_{\mathrm{ss}}}(v)$ is the finite union $\mathop{\mathrm{P}_{\mathrm{semi}}}(v)$ of proper semi-multi-saddle separatrices. 
By $\mathop{\mathrm{P}_{\mathrm{sep}}}(v) = \mathop{\mathrm{P}_{\mathrm{ms}}}(v) \sqcup \mathop{\mathrm{P}_{\mathrm{ss}}}(v) \sqcup \mathop{\partial_{\mathrm{P}(v)}}$, we have $\partial^{-} \mathrm{P}(v) \sqcup \mathop{\mathrm{P}_{\mathrm{sep}}}(v) = \partial^{-} \mathrm{P}(v) \sqcup \mathop{\mathrm{P}_{\mathrm{ms}}}(v) \sqcup \mathop{\mathrm{P}_{\mathrm{ss}}}(v) \sqcup \mathop{\partial_{\mathrm{P}(v)}} = \mathop{\mathrm{P}_{\mathrm{semi}}}(v) \sqcup \mathop{\partial_{\mathrm{P}(v)}}$. 
Therefore the union $\partial^{-} \mathrm{P}(v)\sqcup \mathop{\mathrm{P}_{\mathrm{sep}}}(v)$ is the finite union of proper semi-multi-saddle separatrices and separatrices on $\partial S$ between $\partial$-sources and $\partial$-sinks. 
\end{proof}

\subsection{Complement of the weak border point set}
We have the following properties. 

\begin{proposition}\label{lem011}
The following statements hold for a flow $v$ with finitely many on a compact surface $S$: 
\\
$(1)$ The complement $S - \mathop{\mathrm{Bd}}(v)$ is open. 
\\
$(2)$ $S - \mathop{\mathrm{Bd}}(v) = (\mathrm{int}\mathop{\mathrm{Per}}(v) - \partial_{\mathop{\mathrm{Per}}(v)}) \sqcup (\mathrm{int}\mathrm{P}(v) - \mathop{\mathrm{P}_{\mathrm{sep}}}(v)) \sqcup \mathrm{LD}(v)$. 
\\
$(3)$ The invariant subsets $\mathrm{int}\mathop{\mathrm{Per}}(v) - \partial_{\mathop{\mathrm{Per}}(v)}$, $\mathrm{int}\mathrm{P}(v) - \mathop{\mathrm{P}_{\mathrm{sep}}}(v)$, and $\mathrm{LD}(v)$ are open.
\\
$(4)$ The union $\mathrm{LD}(v)/\hat{v}$ is a finite space such that each connected component of $\mathrm{LD}(v)$ is an orbit class.
\\
$(5)$ If $v$ is quasi-regular, then the set differences $\mathrm{P}(v) - \mathop{\mathrm{P}_{\mathrm{semi}}}(v)$ and $\mathrm{P}(v) - (\mathop{\mathrm{P}_{\mathrm{semi}}}(v) \sqcup \mathop{\partial_{\mathrm{P}(v)}}) = \mathrm{int}\mathrm{P}(v) - \mathop{\mathrm{P}_{\mathrm{sep}}}(v)$ are open. 
\end{proposition}

\begin{proof}
Lemma~\ref{lem010} implies assertion {\rm(1)}. 
By $\mathop{\mathrm{Bd}}(v) = \mathop{\mathrm{Sing}}(v) \sqcup \partial^{-} \mathop{\mathrm{Per}}(v) \sqcup \partial_{\mathop{\mathrm{Per}}(v)} \sqcup \partial^{-} \mathrm{P}(v) \sqcup  \mathop{\mathrm{P}_{\mathrm{sep}}}(v) \sqcup \mathrm{E}(v)$, we have $S - \mathop{\mathrm{Bd}}(v) = (\mathrm{int}\mathop{\mathrm{Per}}(v) - \partial_{\mathop{\mathrm{Per}}(v)}) \sqcup (\mathrm{int}\mathrm{P}(v) - \mathop{\mathrm{P}_{\mathrm{sep}}}(v)) \sqcup \mathrm{LD}(v)$.
This means that assertion {\rm(2)} holds.
%
By the finiteness of the boundary components, the set difference $\mathrm{int}\mathop{\mathrm{Per}}(v) - \partial_{\mathop{\mathrm{Per}}(v)}$ is open. 
Since $\mathop{\mathrm{Bd}}(v)$ is closed, 
the set difference $\mathrm{int}\mathrm{P}(v) \setminus \mathop{\mathrm{Bd}}(v) = \mathrm{int}\mathrm{P}(v) \setminus (\mathop{\mathrm{Sing}}(v) \sqcup \partial^{-} \mathop{\mathrm{Per}}(v) \sqcup \partial_{\mathop{\mathrm{Per}}(v)} \sqcup \partial^{-} \mathrm{P}(v) \sqcup  \mathop{\mathrm{P}_{\mathrm{sep}}}(v) \sqcup \mathrm{E}(v)) = \mathrm{int}\mathrm{P}(v) - \mathop{\mathrm{P}_{\mathrm{sep}}}(v)$ is open.
Lemma~\ref{lem002} implies the openness of the union $\mathrm{LD}(v)$. 
Thus assertion $(3)$ is true.  
By \cite[Proposition~2.2]{yokoyama2016topological}, the union $\mathrm{LD}(v)$ consists of finitely many orbit classes such that each connected component of $\mathrm{LD}(v)$ is an orbit class.
Assertion $(4)$ is immediate from this finiteness. 

Suppose that $v$ is quasi-regular. 
The set difference $\mathrm{P}(v) - (\mathop{\mathrm{P}_{\mathrm{semi}}}(v) \sqcup \mathop{\partial_{\mathrm{P}(v)}}) = \mathrm{int}\mathrm{P}(v) - \mathop{\mathrm{P}_{\mathrm{sep}}}(v)$ is open. 
Since $\mathop{\partial_{\mathrm{P}(v)}}$ consists of finitely many orbits, the closure $\overline
{\mathop{\mathrm{P}_{\mathrm{semi}}}(v)} \subseteq \mathop{\mathrm{Sing}}(v) \sqcup \mathop{\mathrm{P}_{\mathrm{semi}}}(v)$ is a union of $\mathop{\mathrm{P}_{\mathrm{semi}}}(v)$, $\partial$-sinks, and $\partial$-sources. 
Therefore the set difference $\mathrm{P}(v) - \mathop{\mathrm{P}_{\mathrm{semi}}}(v) = \mathrm{P}(v) \setminus \overline{\mathop{\mathrm{P}_{\mathrm{semi}}}(v)}$ is open. 
\end{proof}



\begin{proposition}\label{cor:bd01}
The following statements hold for a quasi-regular flow $v$ on a compact surface $S$: 
\\
{\rm(1)} The border point set $\mathop{\mathrm{BD}}(v)$ is the union of singular points, one-sided periodic orbits, the closure of the union of limit cycles, proper semi-multi-saddle separatrices, separatrices on the boundary $\partial S$ between a $\partial$-source and a $\partial$-sink, and exceptional orbits. 
\\
{\rm(2)} The weak border point set $\mathop{\mathrm{Bd}}(v)$ is the union of singular points, periodic orbits on $\partial S$, the closure of the union of limit cycles, proper semi-multi-saddle separatrices, separatrices on the boundary $\partial S$ between a $\partial$-source and a $\partial$-sink, and exceptional orbits. 
\\
{\rm(3)} 
The set differences $\mathop{\mathrm{BD}}(v) - \mathrm{E}(v)$ and $\mathop{\mathrm{Bd}}(v) - \mathrm{E}(v)$ consist of finitely many proper orbits and {\rm(}possibly infinitely many{\rm)} periodic orbits in the closure of the union of limit cycles. 
\\
{\rm(4)} The following statements are equivalent: 
\begin{quote}
{\rm(4.1)} $\mathop{\mathrm{BD}}(v) \sqcup \mathrm{LD}(v)$ consists of finitely many orbit classes. 
\\
{\rm(4.2)} $\mathop{\mathrm{Bd}}(v) \sqcup \mathrm{LD}(v)$ consists of finitely many orbit classes. 
\\
{\rm(4.3)} There are at most finitely many limit cycles. 
\end{quote}
{\rm(5)} If $v$ is of finite type, then $\mathop{\mathrm{BD}}(v)$ {\rm(resp.} $\mathop{\mathrm{Bd}}(v)${\rm)} is the finite union of singular points, one-sided periodic orbits {\rm(resp.} periodic orbits on $\partial S${\rm)}, limit cycles, proper semi-multi-saddle separatrices, and separatrices on the boundary $\partial S$ between a $\partial$-source and a $\partial$-sink. 
\end{proposition}

\begin{proof}
From Lemma~\ref{lem010}, we have $\mathop{\mathrm{BD}}(v) = \mathop{\mathrm{Bd}}(v) \sqcup \mathop{\mathrm{Per}_{1}}(v) = \mathop{\mathrm{Sing}}(v) \sqcup \partial^{-} \mathop{\mathrm{Per}}(v) \sqcup \partial_{\mathop{\mathrm{Per}}(v)} \sqcup \mathop{\mathrm{Per}_{1}}(v) \sqcup \mathop{\mathrm{P}_{\mathrm{semi}}}(v) \sqcup \mathop{\partial_{\mathrm{P}(v)}} \sqcup \mathrm{E}(v)$. 

\begin{claim}\label{claim:021}
Assertions $(1)$ and $(2)$ hold.
\end{claim}

\begin{proof}[Proof of Claim~\ref{claim:021}]
By Lemma~\ref{lem010}, then the union $\mathop{\mathrm{P}_{\mathrm{semi}}}(v) \sqcup \mathop{\partial_{\mathrm{P}(v)}}$ is the finite union of proper semi-multi-saddle separatrices and separatrices on $\partial S$ between $\partial$-sources and $\partial$-sinks. 
Lemma~\ref{lem007} implies that $\partial^{-} \mathop{\mathrm{Per}}(v)$ is the union of periodic orbits contained in the closure of the union of limit cycles. 
From Proposition~\ref{prop4-12.5}, since $\partial^{-} \mathop{\mathrm{Per}}(v)$ contains all limit cycles of $v$, the closure of the union of limit cycles is contained in $\overline{\partial^{-} \mathop{\mathrm{Per}}(v)} \subseteq \partial \mathop{\mathrm{Per}}(v) \subseteq \mathop{\mathrm{Sing}}(v) \sqcup \partial^{-} \mathop{\mathrm{Per}}(v) \sqcup \partial^{-} \mathrm{P}(v) \subseteq \mathop{\mathrm{BD}}(v)$. 
Lemma~\ref{lem010} implies that $\partial_{\mathop{\mathrm{Per}}(v)} \sqcup \mathop{\mathrm{Per}_{1}}(v)$ is the finite union of one-sided periodic orbits in $\mathop{\mathrm{int}} \mathop{\mathrm{Per}}(v)$. 
The union of of one-sided periodic orbits is contained in $\partial^{-} \mathop{\mathrm{Per}}(v) \sqcup \partial_{\mathop{\mathrm{Per}}(v)} \sqcup \mathop{\mathrm{Per}_{1}}(v) \subseteq \mathop{\mathrm{BD}}(v)$. 
Therefore assertion {\rm(1)} holds. 
Since the set difference $\mathop{\mathrm{BD}}(v) - \mathop{\mathrm{Bd}}(v) = \mathop{\mathrm{Per}_{1}}(v)$ is the finite union of one-sided periodic orbits in $\mathop{\mathrm{int}} \mathop{\mathrm{Per}}(v)$, assertion {\rm(2)} is followed by assertion {\rm(1)}. 
\end{proof}

\begin{claim}\label{claim:022-}
The union $\mathop{\mathrm{Sing}}(v) \sqcup \partial_{\mathop{\mathrm{Per}}(v)} \sqcup \mathop{\mathrm{Per}_{1}}(v) \sqcup \mathop{\mathrm{P}_{\mathrm{semi}}}(v) \sqcup \mathop{\partial_{\mathrm{P}(v)}} = \mathop{\mathrm{BD}}(v) - (\partial^{-} \mathop{\mathrm{Per}}(v) \sqcup \mathrm{E}(v))$ is a finite union of singular points, one-sided periodic orbits in $\mathop{\mathrm{int}} \mathop{\mathrm{Per}}(v)$, proper semi-multi-saddle separatrices, and separatrices on the boundary $\partial S$ between a $\partial$-source and a $\partial$-sink.
\end{claim}

\begin{proof}[Proof of Claim~\ref{claim:022-}]
The quasi-regularity of $v$ implies the finiteness of singular points. 
The compactness of $S$ implies that $\partial_{\mathop{\mathrm{Per}}(v)}$ consists of finitely many orbits contained in the boundary $\partial S$ and so that the union $\partial_{\mathop{\mathrm{Per}}(v)} \sqcup \mathop{\mathrm{Per}_{1}}(v)$ of one-sided periodic orbits in $\mathop{\mathrm{int}} \mathop{\mathrm{Per}}(v)$ consists of finitely many orbits. 
The quasi-regularity of $v$ implies that the union $\mathop{\mathrm{P}_{\mathrm{semi}}}(v)$ of proper semi-multi-saddle separatrices consists of finitely many orbits. 
By the finiteness of $\mathop{\mathrm{Sing}}(v)$ and boundary components, the union $\mathop{\partial_{\mathrm{P}(v)}}$ of separatrices between a $\partial$-source and a $\partial$-sink on the boundary $\partial S$ consists of finitely many orbits. 
\end{proof}

\begin{claim}\label{claim:023+}
Assertion $(3)$ holds.
\end{claim}

\begin{proof}[Proof of Claim~\ref{claim:023+}]
By Claim~\ref{claim:022-}, the set diffrence $\mathop{\mathrm{BD}}(v) - (\partial^{-} \mathop{\mathrm{Per}}(v) \sqcup \mathrm{E}(v)) = \mathop{\mathrm{Sing}}(v) \sqcup \partial_{\mathop{\mathrm{Per}}(v)} \sqcup \mathop{\mathrm{Per}_{1}}(v) \sqcup \mathop{\mathrm{P}_{\mathrm{semi}}}(v) \sqcup \mathop{\partial_{\mathrm{P}(v)}}$ consists of finitely many proper orbits. 
Since $\partial^{-} \mathop{\mathrm{Per}}(v)$ is the union of periodic orbits contained in the closure of the union of limit cycles, assertion {\rm(3)} holds. 
\end{proof}

\begin{claim}\label{claim:022}
Assertion $(4)$ holds.
\end{claim}

\begin{proof}[Proof of Claim~\ref{claim:022}]
If either {\rm(4.1)} or {\rm(4.2)} holds, then {\rm(4.3)} holds. 
Conversely, suppose that there are at most finitely many limit cycles. 
By the finiteness of limit cycles, Lemma~\ref{lem007} implies that $\partial^{-} \mathop{\mathrm{Per}}(v)$ is the finite union of limit cycles. 
The Ma{\v \i}er and Markley works \cite{markley1969poincare,markley1970number} imply the closure $\overline{\mathrm{R}(v)}$ is a finite union of Q-sets. 
By \cite[Proposition~2.2]{yokoyama2016topological}, the union $\mathrm{R}(v)$ of non-closed recurrent orbits consists of finitely many orbit classes.
From Claim~\ref{claim:022-}, the disjoint union $\mathop{\mathrm{BD}}(v) \sqcup \mathrm{LD}(v)$ consists of finitely many orbit classes and so does $\mathop{\mathrm{Bd}}(v) \sqcup \mathrm{LD}(v)$, because of $\mathop{\mathrm{BD}}(v) \subseteq \mathop{\mathrm{BD}}(v)$. 
\end{proof}

\begin{claim}\label{claim:023}
Assertion $(5)$ holds.
\end{claim}

\begin{proof}[Proof of Claim~\ref{claim:023}]
Suppose that $v$ is of finite type. 
By Claim~\ref{claim:022-}, the set difference $\mathop{\mathrm{BD}}(v) - \partial^{-} \mathop{\mathrm{Per}}(v) = \mathop{\mathrm{Sing}}(v) \sqcup \partial_{\mathop{\mathrm{Per}}(v)} \sqcup \mathop{\mathrm{Per}_{1}}(v) \sqcup \mathop{\mathrm{P}_{\mathrm{semi}}}(v) \sqcup \mathop{\partial_{\mathrm{P}(v)}}$ is a finite union of singular points, one-sided periodic orbits in $\mathop{\mathrm{int}} \mathop{\mathrm{Per}}(v)$, proper semi-multi-saddle separatrices, and separatrices on the boundary $\partial S$ between a $\partial$-source and a $\partial$-sink.
The finiteness of limit cycles implies that the union $\partial^{-} \mathop{\mathrm{Per}}(v) \sqcup \partial_{\mathop{\mathrm{Per}}(v)} \sqcup \mathop{\mathrm{Per}_{1}}(v)$ is the finite union of limit cycles and one-sided periodic orbits. 
Therefore, the border point set $\mathop{\mathrm{BD}}(v)$ is the finite union of singular points, one-sided periodic orbits, limit cycles, proper semi-multi-saddle separatrices, and separatrices on the boundary $\partial S$ between a $\partial$-source and a $\partial$-sink. 
Since the set difference $\mathop{\mathrm{BD}}(v) - \mathop{\mathrm{Bd}}(v) = \mathop{\mathrm{Per}_{1}}(v)$ is the union of one-sided periodic orbits outside of the boundary $\partial S$, the weak border point set $\mathop{\mathrm{Bd}}(v)$ is the finite union of singular points, periodic orbits on $\partial S$, limit cycles, proper semi-multi-saddle separatrices, and separatrices on the boundary $\partial S$ between a $\partial$-source and a $\partial$-sink. 
\end{proof}

This completes the proof of the proposition. 
\end{proof}

By the previous proof, notice that $\mathop{\mathrm{BD}}(v)$ (resp. $\mathop{\mathrm{BD}}(v)$) for a quasi-regular flow $v$ on a compact surface $S$ is also the union of singular points, one-sided periodic orbits (resp. periodic orbits on $\partial S$), periodic orbits contained in the closure of the union of limit cycles, proper semi-multi-saddle separatrices, separatrices on the boundary $\partial S$ between a $\partial$-source and a $\partial$-sink, and exceptional orbits. 

\section{Decompositions of quasi-regular flows}

In this section, we demonstrate Theorem~\ref{lem0c}. 
We introduce some concepts to describe the proof and show statements to achieve this.

Let  $v$ be a flow on a connected compact surface $S$. 
Recall that a non-trivial circuit is a multi-saddle circuit if it is contained in a multi-saddle connection. 
We demonstrate the existence of a closed transversal with infinitely many intersections for non-closed recurrent orbits. 

\begin{lemma}\label{lem3-02a}
For any point $x \in \mathrm{R}(v)$, there is a closed transversal $\gamma$ through $O(x)$ such that the intersection $\gamma \cap O(x)$ is infinite.
\end{lemma}

\begin{proof}
Fix a point $x \in \mathrm{R}(v)$ and a transverse arc $I \subset U$ such that $x$ is the interior point of $I$. 
Then $\vert I \cap O(x) \vert = \infty$. 
By \cite[Lemma~6]{yokoyama2021poincare}, there are an orbit arc $C$ in $O(x)$ and a transverse closed arc $J \subseteq I$ such that the union $\mu := J \cup C$ is a loop with $C \cap J = \partial C = \partial J$ and that the return map along $C$ is orientation-preserving between neighborhoods of $\partial C$ in $I$. 
By the waterfall construction to the loop $\mu$, there is a closed transversal $\gamma$ containing $x$ arbitrarily near $\mu$.
Since $x$ is recurrent, we have $\vert \gamma \cap O(x) \vert = \infty$. 
\end{proof}

\begin{lemma}\label{lem:3-02a}
Each closed transversal through a point $y \in \mathrm{R}(v)$ is essential and intersects $O(y)$ infinitely many times.
\end{lemma}

\begin{proof}
Let $\nu$ be a closed transversal through a point $y$ in $\mathrm{R}(v)$. 
Since $y$ is non-closed recurrent, the intersection $\nu \cap O(y)$ is infinite.
Suppose that $\nu$ is not essential.
Let $S^*$ be the resulting closed surface from $S$ by collapsing all boundary components into singletons.
Then $\nu$ is null homotopic in $S^*$ and so the point $y \in \mathrm{P}(v)$, which contradicts the recurrence of $y$.
\end{proof}

\subsection{Surgeries for flows on surfaces}
Recall that an invariant closed disk is a center disk if and only if it is a union of one center and periodic orbits and is a neighborhood of the center.
\begin{definition}
A closed disk is a {\bf sink (resp. source) disk} if its boundary is a closed transversal and its interior consists of one sink (resp. source) and of orbit arcs of non-recurrent orbits (see Figure~\ref{disks}).
\end{definition}
\begin{figure}
\begin{center}
\includegraphics[scale=0.4]{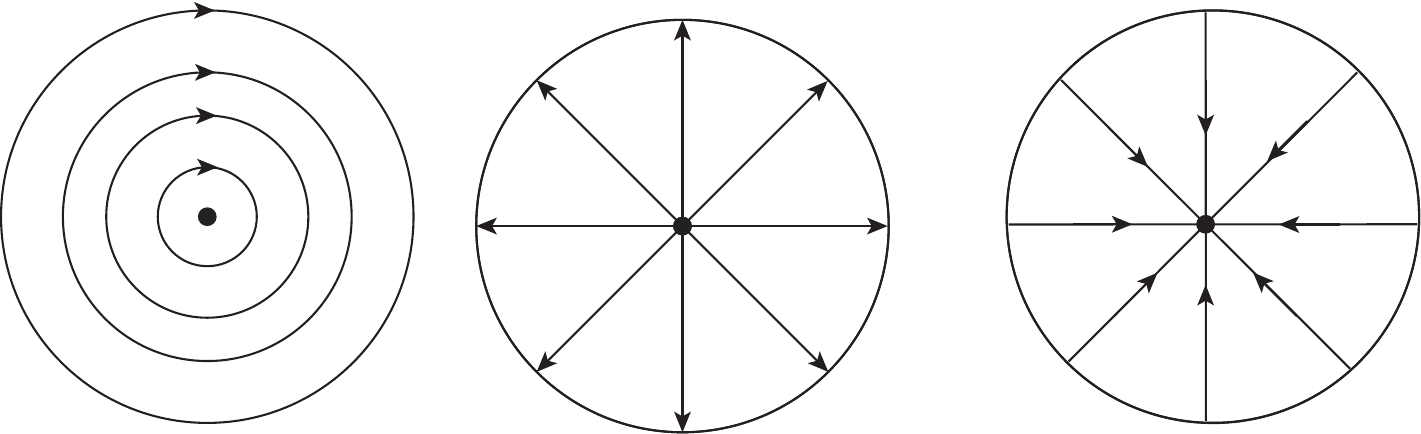}
\end{center}
\caption{A center disk, a sink disk, and a source disk}
\label{disks}
\end{figure}

We define a Cherry flow box as follows (cf. \cite{cherry1937topological} and Figure~1 in \cite{gardiner1985structure}). 

\begin{definition}
A {\bf Cherry flow box} is a disk with orbit arcs on one of the figures to the middle and right in Figure~\ref{cherrybox}. 
\end{definition}

\begin{figure}
\begin{center}
\includegraphics[scale=0.3]{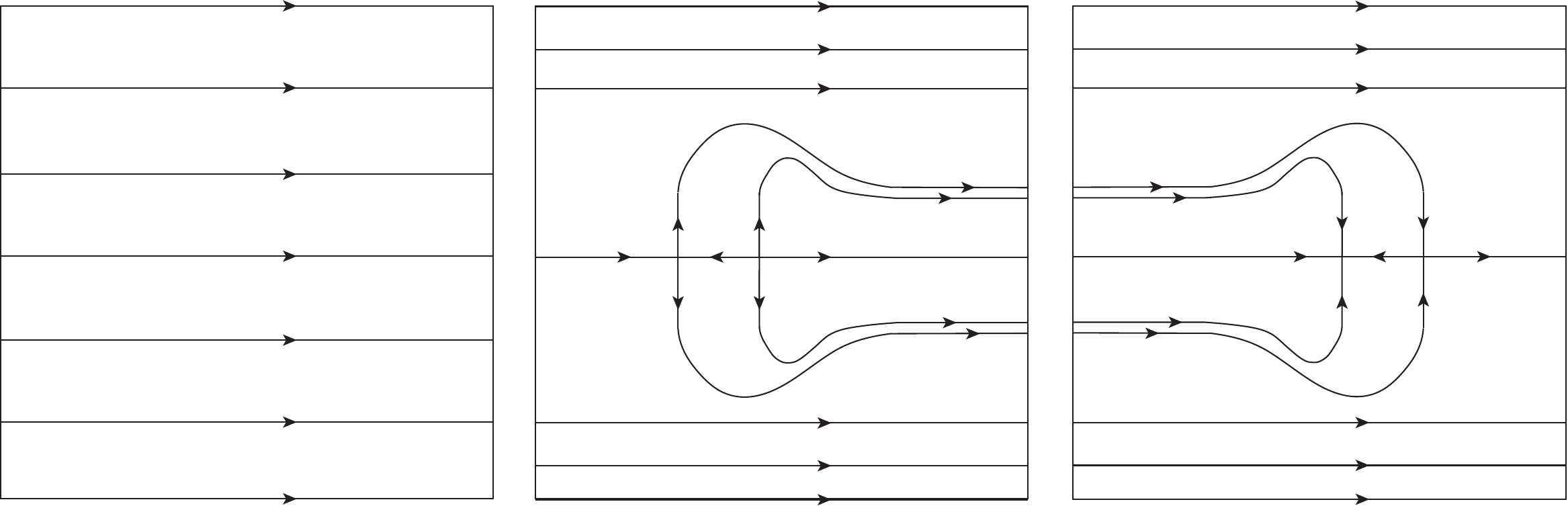}
\end{center}
\caption{A trivial flow box and Cherry flow boxes}
\label{cherrybox}
\end{figure}

\subsubsection{Operations}

A {\bf Cherry blow-up operation} is replacing a trivial flow box with a Cherry flow box.
Define operations as in  Figure~\ref{cutting} as follows: 
\begin{quote}
$\bm{\mathop{\mathrm{C}}_{o}}$ : an operation as cutting a non-null-homotopic periodic orbit and pasting one or two center disks to the new boundary. 
\\
$\bm{\mathop{\mathrm{C}}_{d}}$ : an operation as removing an essential one-sided (resp. two-sided) multi-saddle circuit which is a loop, 
 and pasting a double covering (resp. two copies) of the loop to the new boundary (resp. new two boundary components). 
\\
 $\bm{\mathop{\mathrm{C}}_{t}}$ : an operation as cutting an essential closed transversal and pasting one sink disk and one source disk.
\begin{figure}
\begin{center}
\includegraphics[scale=0.18]{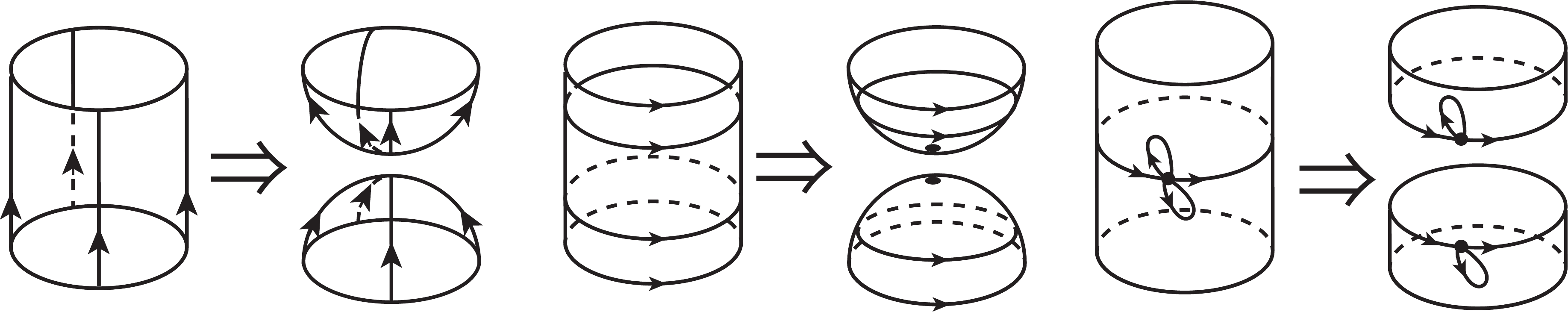}
\end{center}
\caption{Left: An example of the operation $\mathop{\mathrm{C}}_{t}$. Middle: An example of the operation $\mathop{\mathrm{C}}_{o}$. Right: An example of the operation $\mathop{\mathrm{C}}_{d}$}
\label{cutting}
\end{figure}
\\
$\bm{\mathop{\mathrm{Ch}}_{k}}$ : an inverse operation of a Cherry blow-up operation to a multi-saddle separatrix as shown in Figure~\ref{cherrybox04}. 
\begin{figure}
\begin{center}
\includegraphics[scale=0.45]{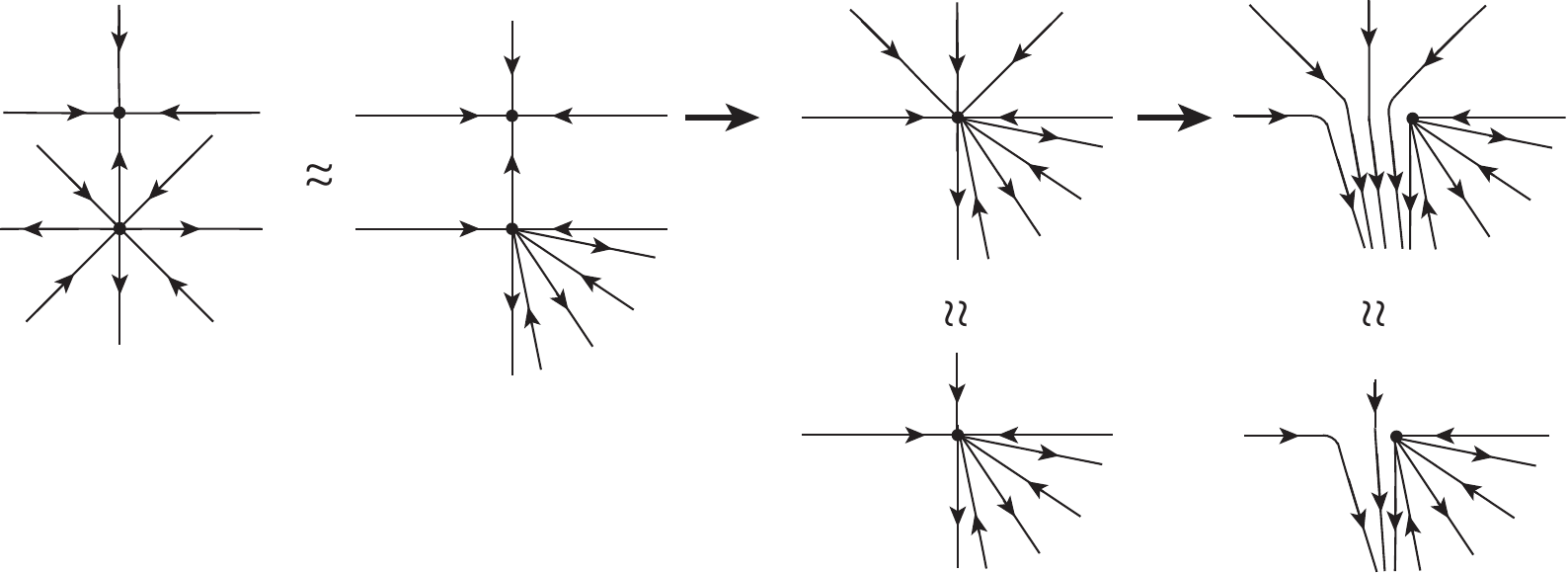}
\end{center}
\caption{Inverse operation of a Cherry blow-up operation $\mathop{\mathrm{Ch}}_{k}$ to a separatrix from a $k$-saddle to a sink}
\label{cherrybox04}
\end{figure}
\end{quote}
By cutting essential parts, we can reduce surface flows into spherical flows.

\begin{lemma}\label{cut-ess-orbit}
Let  $v$ be a quasi-regular flow on a connected compact surface $S$,
$S_{o}$ the resulting surface from $S$ by applying an operation $\mathop{\mathrm{C}}_{o}$ as possible,
$S_{d}$ the resulting surface from $S_{o}$ by applying an operation $ \mathop{\mathrm{C}}_{d}$ as possible, and
$S_{t}$ the resulting surface from $S_{d}$ by applying an operation $\mathop{\mathrm{C}}_{t}$ as possible.
Then each connected component of the surface $S_{t}$ is a subset of a sphere up to homeomorphism.
\end{lemma}

\begin{proof}
By the quasi-regularity, there are at most finitely many non-periodic limit circuits. 
Removing ($\partial$-)$0$-saddles, we may assume that there are no ($\partial$-)$0$-saddles.
Let $v$ be the resulting flow from $v$ on $S$ without ($\partial$-)$0$-saddles, $v_o$ the resulting flow on the resulting compact surface $S_o$ from $v$ by applying an operation $\mathop{\mathrm{C}}_{o}$ as possible, $v_d$ the resulting flow on the resulting compact surface $S_d$ from $v_o$ by applying an operation $ \mathop{\mathrm{C}}_{d}$ as possible, and $v_t$ the resulting flow on the resulting compact surface $S_t$ from $v_d$ by applying an operation $\mathop{\mathrm{C}}_{t}$ as possible.
Since $S$ is compact, the surface $S_{o}$ can be obtained by taking operations $C_o$ to $S$ finitely many times.
By construction, the flow $v_{o}$ has no essential periodic orbits.
Since $S$ is compact, the surface $S_{d}$ can be obtained by taking operations $C_d$ to $S_o$ finitely many times.
By construction, the flow $v_{d}$ has neither essential periodic orbits nor essential loops in $D(v)$.
Since $S_{d}$ is compact, the surface $S_{t}$ can be obtained by taking operations $C_t$ to $S_d$ finitely many times.
Then the flow $v_{t}$ has neither essential closed transversals, essential loops in $D(v)$, nor essential periodic orbits.
Since each Q-set intersects some essential closed transversal (because of Lemma~\ref{lem3-02a} and Lemma~\ref{lem:3-02a}),  the flow $v_{t}$ has no Q-sets but consists of proper orbits. 
%
Replacing $\partial$-sinks (resp. $\partial$-sources) with pairs of a $\partial$-saddle and a sink (resp. a source) (see Figure~\ref{Bsinks}), we may assume that there are neither $\partial$-sinks nor $\partial$-sources on $S_{t}$. 
Then each singular point on the boundary $\partial S_t$ is a multi-saddle.

Let $S'$ be the resulting closed surface from a connected component of $S_{t}$ by collapsing all boundary components into singletons and $v'$ the resulting flow on $S'$. 
Since any orbits of $v_t$ are proper, so are those of $v'$.
Moreover, the flow $v'$ is quasi-regular because of the construction and Proposition~\ref{prop:char_center}. 
By construction, it suffices to show that each connected component of $S'$ is a subset of a sphere up to homeomorphism. 

\begin{claim}\label{claim:024}
Each connected component of the surface $S'$ is a sphere or a projective plane. 
\end{claim}

\begin{proof}[Proof of Claim~\ref{claim:024}]
By the operation $C_o$, each periodic orbit of $v'$ is null homotopic. 
From the operation $C_t$ and the construction of $S'$ form $S_t$, each multi-saddle connection of $v'$ contains no null homotopic loops in it and so is null homotopic. 
By construction, each new singular point of $v'$ is a multi-saddle.
The quasi-regularity of $v$ implies that each singular point of the resulting flow $v'$ is either a multi-saddle, a sink, a source, or a center.
Moreover, each limit circuit is null homotopic in the closed surface $S'$. 

For an ss-multi-saddle connection of $v'$ with a $k$-saddle ($k > 0$) such that there is a separatrix from a source to it (resp. from it to a sink), applying the inverse operation of $\mathop{\mathrm{Ch}}_{k}$ to the separatrix with a source (resp. sink), we obtain an ss-multi-saddle connection with a $(k-1)$-saddle (see Figure ~\ref{cherrybox04}).
Applying the operations finitely many times and removing $0$-saddles, we may assume that there are no ss-separatrices connecting a multi-saddle and either a sink or a source.
Therefore each semi-multi-saddle separatrix is a multi-saddle separatrix.
Then, each separatrix which is not contained in the multi-saddle connection diagram $D(v')$ of $v'$ connects two of a sink, a source, and limit circuits.

Since each non-periodic limit circuit of $v'$ is null homotopic, by the finiteness of non-periodic limit circuits, collapsing closed disks each of whose boundary is a non-periodic limit circuit into singletons, the new singletons are either sinks or sources.
Thus, we may assume that $D(v') = \emptyset$ and that each periodic orbit is null homotopic.

Therefore, $v'$ consists of sinks, sources, centers, non-recurrent orbits, and null homotopic periodic orbits. 
In particular, the indices of singular points are positive if singular points exist. 
If there are no singular points, then there are non-contractible periodic orbits or non-closed recurrent orbits, which contradicts the non-existence of non-contractible periodic orbits and non-closed recurrent orbits.  
Thus there are singular points, whose indices are positive.
Then the Euler characteristic is positive. 
Hence, the Poincar\'e-Hopf theorem implies that each connected component of the surface $S'$ is a sphere or a projective plane. 
\end{proof}

\begin{claim}\label{claim:025}
Each connected component of $S'$ is a sphere. 
\end{claim}
\begin{proof}[Proof of Claim~\ref{claim:025}]
Assume that there is a connected component $\mathbb{P}$ of $S'$ is a projective plane. 
By the non-existence of multi-saddles, the Poincar\'e-Hopf theorem implies that there is just one singular point on $\mathbb{P}$. 

Suppose that there is a center. 
Removing a small open center disk $B$, the complement $\mathbb{P} - B$ is a M{\"o}bius band without singular points and so their first return map to a transverse closed arc is an orientation-reversing homeomorphism on a closed interval.   
Then, it has a unique fixed point of the first return map that corresponds to a one-sided periodic orbit, which contradicts the non-existence of essential periodic orbits. 

Suppose that there is a sink (resp. source). 
The Poincar\'e-Hopf theorem implies the $\alpha$-limit (resp. $\omega$-limit) set of each non-singular point is a periodic orbit. 
Since there is only one singular point, which is a sink (resp. source), 
there is a closed disk $B$ whose boundary is a periodic orbit and whose interior is the unstable manifold of the sink (resp. the stable manifold of the source) and so the complement $\mathbb{P} - \mathop{\mathrm{int}} B$ is a M{\"o}bius band without singular points, because the periodic orbit $\partial B$ is not essential. 
The same argument as above implies the contradiction. 
\end{proof}

Thus, each connected component of $S_{t}$ is a subset of a sphere.
\end{proof}

\begin{figure}
\begin{center}
\includegraphics[scale=0.5]{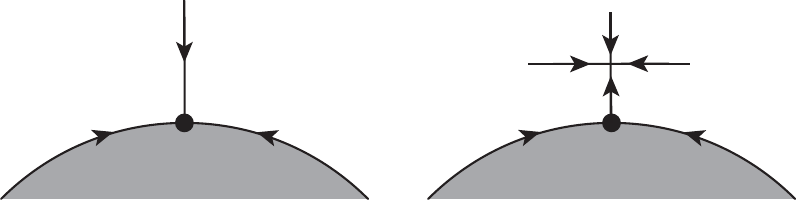}
\end{center}
\caption{Replacing a $\partial$-sink with a pair of a sink and a $\partial$-saddle}
\label{Bsinks}
\end{figure}

\begin{figure}
\begin{center}
\includegraphics[scale=0.3]{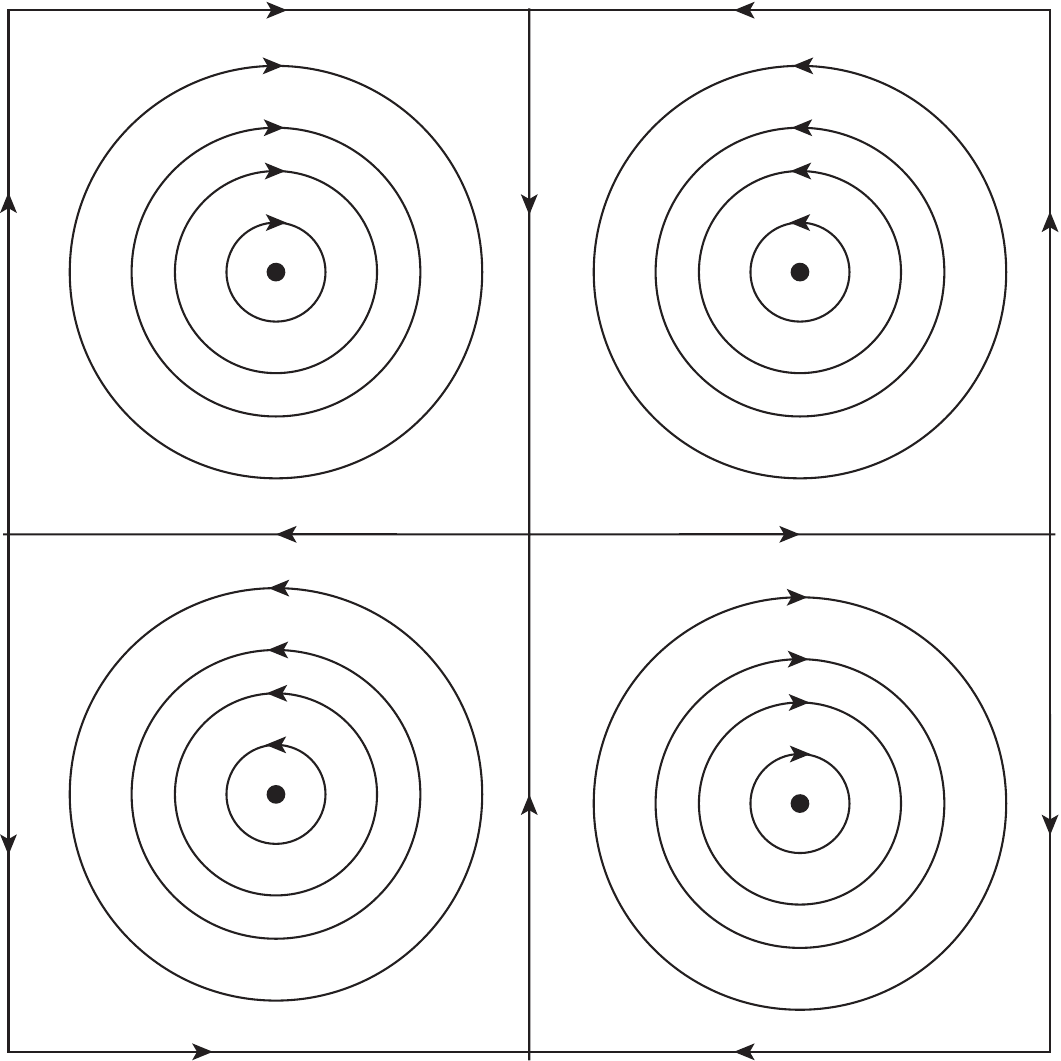}
\end{center}
\caption{A flow on a torus with heteroclinic multi-saddle separatrices but without essential closed transversals or essential periodic orbits}
\label{fig:Tayorflow}
\end{figure}

Note that the operation $\mathop{\mathrm{C}}_{d}$ is necessary in the previous lemma because of the existence of flows generated by Taylor fields as in Figure~\ref{fig:Tayorflow} and \cite[Figure~3.3.2 p.99]{ma2005geometric}.

\subsection{Decomposition of quasi-regular flow using the strict border set}

To demonstrate Theorem~\ref{lem0c}, we define the transverse and vertical boundary of an invariant open subset.

\subsubsection{$\alpha'(x)$-limit sets and $\omega'$-limit sets of points and orbits}
Define $\bm{\alpha'(x)}$ and $\bm{\omega'(x)}$ for a point $x \in X$ as follows \cite{buendia2018markus,markus1954global}: 
$$\alpha'(x) := \alpha(x) \setminus O(x)$$
$$\omega'(x) := \omega(x) \setminus O(x)$$
We call $\alpha'(x)$ the {\bf $\bm{\alpha'}$-limit set} of $x$ and $\omega'(x)$ the {\bf $\bm{\omega'}$-limit set} of $x$. 
For a subset $A \subseteq S$, define $\bm{\alpha'(A)}$ and $\bm{\omega'(A)}$ as follows:  
\[
\alpha'(A) := \bigcup_{x \in A} \alpha'(x)
\]
\[
\omega'(A) := \bigcup_{x \in A} \omega'(x)
\]
Then $\alpha'(O) = \alpha(O) \setminus O$ and $\omega'(O) = \omega(O) \setminus O$ for an orbit $O$. 
As mentioned, the definition of $\alpha'(p)$ (resp. $\omega'(p)$) in the sense of Buend{\'\i}a and L\'opes is slightly different from the definition in the sense of Markus and Neumann (cf.~\cite{markus1954global,neumann1975classification,neumann1976global}) but these definitions correspond for proper orbits. 
Note that an orbit $O$ is closed if and only if $\alpha'(O) = \emptyset$ (resp. $\omega'(O) = \emptyset$). 
Since flows of finite type have no non-proper orbits, two definitions of $\alpha'(p)$ (resp. $\omega'(p)$) in the sense of Buend{\'\i}a and L\'opes and in the sense of Markus and Neumann 
 correspond to each other for flows of finite type.

\subsubsection{Transverse boundary of an invariant open subset by flows}

Recall that a subset is locally connected if each point in it has a small connected neighborhood. 

For an invariant open subset $A$ with locally connected boundary, define the transverse boundary $\partial_{\pitchfork} A$ of $A$ as follows: 
Put $\Gamma_A := \bigcup_{x \in A} \alpha'(x) \cup \omega'(x)$. 
Let $S_1$ be the metric completion of $S - \Gamma_A$, $\partial_1 := S_1 - (S - \Gamma_A)$ the new boundary of $S_1$, and $\pi \colon S_1 \to S$ the natural continuous surjection associated to the metric completion such that a restriction $\pi\vert_{\partial_1}: \partial_1 \to \Gamma_A$ is an injection except singular points and a restriction $\pi_{S_1 - \partial_1}: S_1 - \partial_1 \to S - \Gamma_A$ is an identity mapping. 
Here, we call that a mapping $f$ is an {\bf injection except for singular points} if the restriction to the complement of the inverse image of the singular point set is injective. 
Then we define the {\bf transverse boundary}
$$\bm{\partial_{\pitchfork} A} := \pi \left(\partial_{S_1} A - (\alpha'_{v_1} (A) \cup \omega'_{v_1} (A)) \right) = \pi \left(\partial_{S_1} A - \bigcup_{x \in A} (\alpha'_{v_1} (x) \cup \omega'_{v_1} (x)) \right)$$
where $v_1$ is the resulting flow on $S_1$ induced by $v$,  $\partial_{S_1} A$ is the boundary of $A$ in $S_1$, $\alpha'_{v_1} (x)$ is the $\alpha'$-limit set of $x$ with respect to $v_1$, and $\omega'_{v_1} (x)$ is the $\omega'$-limit set of $x$ with respect to $v_1$.

\subsubsection{Left and right transverse boundary components of canonical regions}

Let $v$ be a quasi-regular flow with $\mathrm{E}(v) = \emptyset$ on a compact surface. 

For an oriented oriented invariant trivial flow box $B \cong \mathbb{D} = I \times J \subset \R^2$, denote by $\partial_{\pitchfork}^L B$ (resp. $\partial_{\pitchfork}^R B)$ the image by the mapping $\pi \colon S_1 \to S$ in the previous subsection of the boundary component of $\partial_{S_1} B - (\alpha'_{v_1} (B) \cup \omega'_{v_1} (B))$ corresponding to $I \times \{0 \}$ (resp. $I \times \{ 1 \}$) (see the figure to the left in Figure~\ref{flow-boxes}), where $\{0,1\} = \partial J$. 
By definition, we have $\partial_{\pitchfork} B = \partial_{\pitchfork}^L B \cup \partial_{\pitchfork}^R B$. 

For an oriented open periodic annulus $B \cong \mathbb{A} = \{ (x,y) \mid 1< \sqrt{x^2+y^2} < 2 \}$, denote by $\bm{\partial_{\pitchfork}^L B}$ (resp. $\bm{\partial_{\pitchfork}^R B})$ the boundary component of $\partial_{\pitchfork} B$ corresponding to $\{ ( \cos \theta, \sin \theta) \mid \theta \in [0, 2\pi ] \}$ (resp. $\{ ( 2 \cos \theta, 2 \sin \theta) \mid \theta \in [0, 2\pi ] \}$) (see the figure to the middle in Figure~\ref{flow-boxes}). 

For an open transverse annulus $B \cong \A_{\widetilde{+}}$, define $\partial_{\pitchfork}^L B = \partial_{\pitchfork}^R B := \emptyset$. 
We call that $\partial_{\pitchfork}^R B$ (resp. $\partial_{\pitchfork}^L B$) is the {\bf right (resp. left) transverse boundary}.

\subsubsection{Vertical boundary of an invariant open subset by flows}

For an invariant open subset $A$, define $\bm{\partial_{\perp}^{\alpha} A}$, $\bm{\partial_{\perp}^{\omega}}$, and $\bm{\partial_{\perp} A}$ of $A$ as follows: 
$$\partial_{\perp}^{\alpha} A := (\bigcup_{x \in A} \alpha(x)) \setminus A$$
$$\partial_{\perp}^{\omega} A := (\bigcup_{x \in A} \omega(x)) \setminus A$$
$$\partial_{\perp} A := \partial_{\perp}^{\alpha} A \cup \partial_{\perp}^{\omega} A$$
We call that $\partial_{\perp}^{\alpha} A$ (resp. $\partial_{\perp}^{\omega} A$, $\partial_{\perp} A$) is the {\bf $\bm{\alpha}$-vertical} (resp. {\bf $\bm{\omega}$-vertical}, {\bf vertical}) {\bf boundary}.

\subsubsection{Examples of vertical boundaries}
For an invariant trivial flow box $B \cong \mathbb{D} = I \times J \subset \R^2$, the $\alpha$-vertical (resp. $\omega$-vertical) boundary $\partial_{\perp}^{\alpha} B$ (resp. $\partial_{\perp}^{\omega} B)$ is the boundary component of $\partial_{\perp} B$ corresponding to $\{0 \} \times J$ (resp. $\{ 1 \} \times J$) (see the figure to the left in Figure~\ref{flow-boxes}), where $\{0,1\} = \partial I$. 
\begin{figure}
\begin{center}
\includegraphics[scale=0.4]{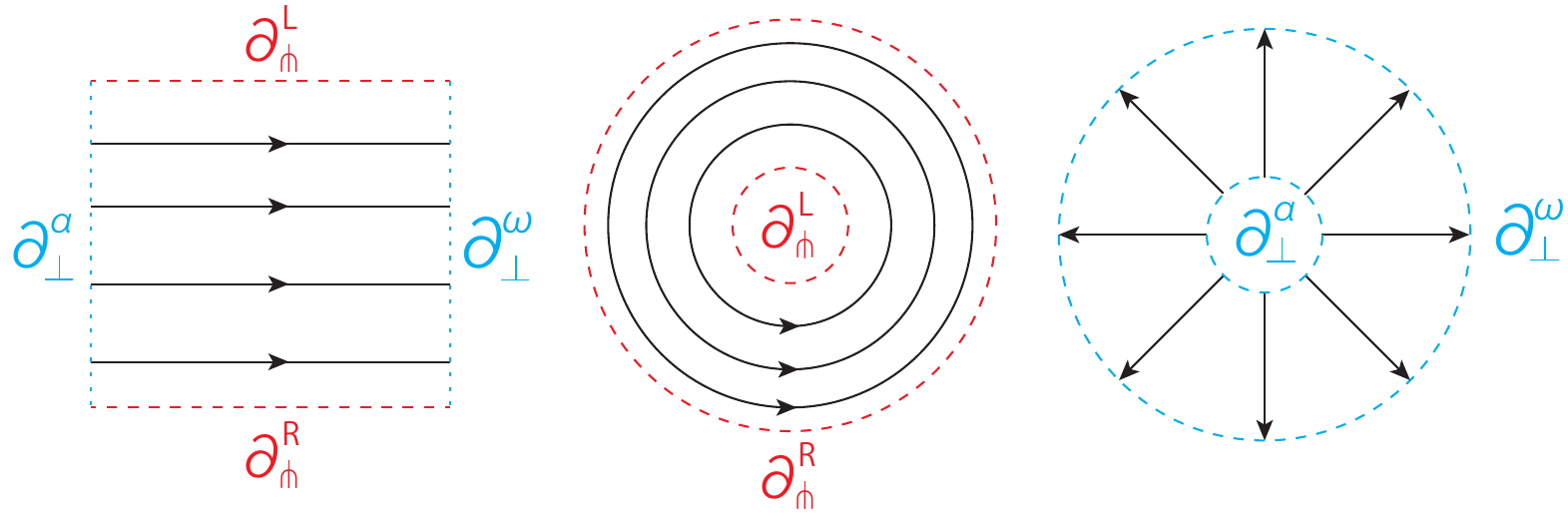}
\end{center}
\caption{A trivial flow box, a periodic annulus, and an open transverse annulus}
\label{flow-boxes}
\end{figure}

For an open periodic annulus $B \cong \A_{\widetilde{+}}$, the vertical boundary components are empty (i.e. $\partial_{\perp}^{\alpha} B = \partial_{\perp}^{\omega} B := \emptyset$). 

For an open transverse annulus $B \cong \mathbb{A} = \{ (x,y) \mid 1< \sqrt{x^2+y^2} < 2 \}$, the $\alpha$-vertical (resp. $\omega$-vertical) boundary $\partial_{\perp}^{\alpha} B$ (resp. $\partial_{\perp}^{\omega}  B)$ is the boundary component of $\partial_{\perp} B$ corresponding to $\{ ( \cos \theta, \sin \theta) \mid \theta \in [0, 2\pi ] \}$ (resp. $\{ ( 2 \cos \theta, 2 \sin \theta) \mid \theta \in [0, 2\pi ] \}$) (see the figure to the right in Figure~\ref{flow-boxes}). 

\subsubsection{Remarks of transverse boundaries and vertical boundarires}

Note that $\partial_{\pitchfork} A \cap \partial_{\perp} A \neq \emptyset$ in general and that metric completions are necessary to define the transverse boundary. 
In fact, there is a flow on an orientable closed surface $\Sigma_2$ with genus two which consists of two singular points $x_1$ and $x_2$ and non-recurrent orbits with the weak border point set $\BD = \bigsqcup_{i=1}^2 x_i \sqcup \bigsqcup_{i=1}^5 O_i$ such that the complement $\Sigma_2 - \BD$ is an invariant trivial flow box $U$ as in Figure~\ref{trans-bdry} with $\partial_{\pitchfork} U \cap \partial_{\perp} U = \bigsqcup_{i=1}^2 x_i \sqcup \bigsqcup_{i=1}^2 O_i \neq \emptyset$ and $\partial_{\pitchfork} U = \bigsqcup_{i=1}^2 x_i \sqcup \bigsqcup_{i=1}^5 O_i \supsetneq O_3 \sqcup O_4 \sqcup O_5 =  \partial U - \bigcup_{x \in U} (\alpha' (x) \cup \omega' (x)) = \partial U - \partial_{\perp} U$.  
\begin{figure}
\begin{center}
\includegraphics[scale=0.25]{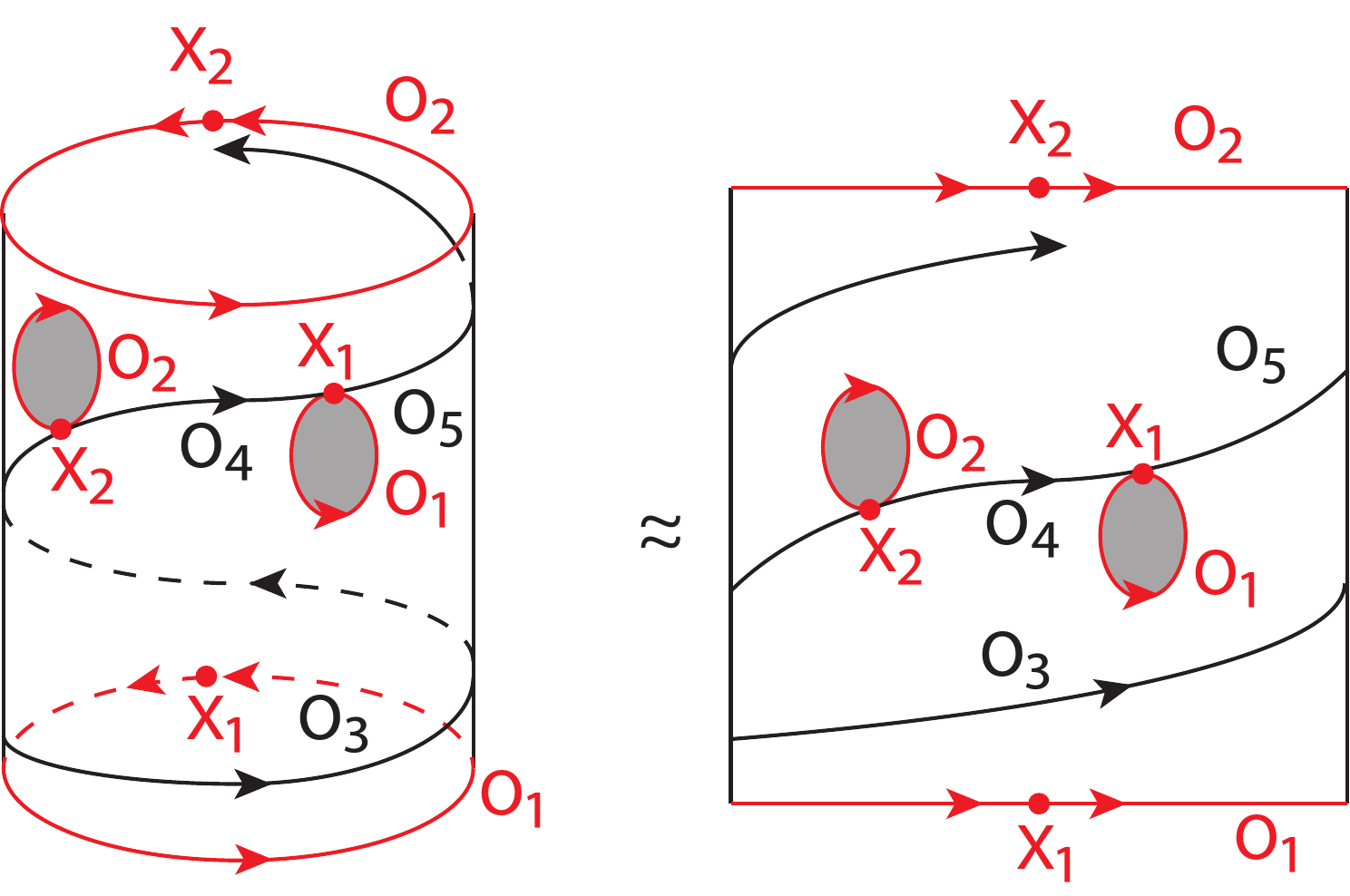}
\end{center}
\caption{A flow on an orientable closed surface of genus two is obtained by gluing a pair of two  circuits labeled by $x_1 \sqcup O_1$ and a pair of two circuits labeled by $x_2 \sqcup O_2$.}
\label{trans-bdry}
\end{figure}


\subsubsection{Self-connectedness}
Self-connectedness is defined as follows. 

\begin{definition}
A multi-saddle separatrix is {\bf self-connected} (cf. \cite{ma2005geometric}) if either its $\omega$-limit set and $\alpha$-limit set correspond or it connects multi-saddles on the same boundary component.
\end{definition}

\begin{definition}\label{def:sc_mc}
A (multi-)saddle connection is {\bf self-connected} if all the separatrices are self-connected.
\end{definition}

\begin{definition}
The (multi-)saddle connection diagram is {\bf self-connected} if so is each connected component.
\end{definition}

Notice that a homoclinic multi-saddle separatrix is self-connected and a heteroclinic multi-saddle separatrix on the boundary is self-connected.

\subsubsection{Proof of the existence of decompositions for quasi-regular flows}

For a quasi-regular flow $v$ on a compact surface, recall that $\mathop{\mathrm{Bd}}(v) = \mathop{\mathrm{Sing}}(v) \sqcup \partial^{-} \mathop{\mathrm{Per}}(v) \sqcup \partial_{\mathop{\mathrm{Per}}(v)} \sqcup \mathop{\mathrm{P}_{\mathrm{semi}}}(v) \sqcup \mathop{\partial_{\mathrm{P}(v)}} \sqcup \mathrm{E}(v)$ is the union of singular points, the closure of the union of limit cycles, periodic orbits on $\partial S$, proper semi-multi-saddle separatrices, separatrices between a $\partial$-source and a $\partial$-sink on the boundary $\partial S$, and exceptional orbits, because of Proposition~\ref{cor:bd01}. 
Since $\partial S \subseteq \mathop{\mathrm{Bd}}(v)$ for any quasi-regular flow $v$ on a compact surface $S$, we describe the complement of the weak border point set $\mathop{\mathrm{Bd}}(v)$ as follows.

\begin{proof}[Proof of Theorem~\ref{lem0c}]
Lemma~\ref{lem010} implies that the complement $S - \mathop{\mathrm{Bd}}(v)$ is open. 

\begin{claim}\label{claim:027}
We may assume that $\mathop{\partial_{\mathrm{P}(v)}} \sqcup \partial_{\mathop{\mathrm{Per}}(v)}  = \emptyset$ and that there are neither $\partial$-sinks nor $\partial$-sources. 
\end{claim}
\begin{proof}[Proof of Claim~\ref{claim:027}]
We can deform the flow $v$ on $S$ by following operations which preserve the complement $S - \mathop{\mathrm{Bd}}(v)$: 
Replacing $\partial_{\mathop{\mathrm{Per}}(v)}$ with centers, we may assume that $\partial_{\mathop{\mathrm{Per}}(v)} = \emptyset$.
Replacing $\partial$-sinks (resp. $\partial$-sources) with pairs of a $\partial$-saddle and a sink (resp. a source) as in Figure~\ref{Bsinks}, we may assume that there are neither $\partial$-sinks nor $\partial$-sources, and so that $\mathop{\partial_{\mathrm{P}(v)}} = \emptyset$. 
\end{proof}

Fix a connected component $U$ of $S - \mathop{\mathrm{Bd}}(v)$.
Proposition~\ref{lem011} implies that the open subset $U$ is contained in either $\mathop{\mathrm{int}} \mathop{\mathrm{Per}}(v)$,  $\mathop{\mathrm{int}} (\mathrm{P}(v) - \mathop{\mathrm{P}_{\mathrm{sep}}}(v))$, or $\mathrm{LD}(v)$.

\begin{claim}\label{claim:028}
If $U \subseteq \mathrm{LD}(v)$, then $U$ is in the case $(7)$.
\end{claim}
\begin{proof}[Proof of Claim~\ref{claim:028}]
Suppose that $U \subseteq \mathrm{LD}(v)$. 
Lemma~\ref{lem002} implies that $\mathrm{LD}(v)$ is open. 
By Proposition~\ref{lem011}, the connected component $U$ of $\mathrm{LD}(v)$ is an orbit class. 
Fix a point $x \in U$ and a transverse arc $I \subset U$ where $x$ is the interior point in $I$.
Since $x$ is non-closed recurrent, 
by \cite[Lemma~6]{yokoyama2021poincare}, taking a small transverse arc if necessary, by the waterfall construction, there is a closed transversal in $U$. 
By Lemma~\ref{lem:3-02a}, the closed transversal is essential and so is $U$. 
Then $U$ is desired in the case $(7)$.
\end{proof}

\begin{claim}\label{claim:029}
If $U \subseteq \mathop{\mathrm{int}} \mathop{\mathrm{Per}}(v)$, then $U$ is in the cases $(3)$, $(4)$, $(5)$ or $(6)$.
\end{claim}
\begin{proof}[Proof of Claim~\ref{claim:029}]
Suppose that $U \subseteq \mathop{\mathrm{int}} \mathop{\mathrm{Per}}(v)$.
If $\partial U = \emptyset$, then $U$ is either a torus or a Klein bottle and so is desired in the case $(3)$ or $(4)$.
Thus, we may assume that $\partial U \neq \emptyset$.
Then $U$ is an open annulus or an open M{\"o}bius band and so is desired in the case $(5)$ or $(6)$.
\end{proof}

Thus we may assume that $U \subseteq \mathop{\mathrm{int}} (\mathrm{P}(v) - \mathop{\mathrm{P}_{\mathrm{sep}}}(v))$.
The $\omega$-limit (resp. $\alpha$-limit) set of each point in $U$
is either an exceptional Q-set, a limit circuit, or a sink (resp. a source).
We can deform the flow $v$ on $S$ by the following operation which preserve $U$: 
Cutting a non-null-homotopic periodic orbit and pasting one or two center disks
(i.e. taking operation $\mathop{\mathrm{C}}_{o}$),
by induction, we may assume that each periodic orbit is null homotopic.


\begin{claim}\label{claim:030}
We may assume that there are no multi-saddle circuits in $\mathop{\mathrm{Bd}}(v)$.
\end{claim}
\begin{proof}[Proof of Claim~\ref{claim:030}]
We can deform the flow $v$ on $S$ by the following operation which preserve $U$: 
Applying the operation $\mathop{\mathrm{C}}_{d}$ (i.e. removing the union $\Gamma$ of multi-saddle circuits in $\mathop{\mathrm{Bd}}(v)$ and taking the metric completion $S_1$ of $S - \Gamma$) and collapsing the new boundary components into singletons, we may assume that there are no multi-saddle circuits in $\mathop{\mathrm{Bd}}(v)$, by replacing $S$ with the resulting surface $S_{\mathrm{col}}$. 
Here a new boundary component is a connected component of $S_1 - (S - \Gamma)$. 
%
\end{proof}

%
Then the $\omega$-limit (resp. $\alpha$-limit) set of each point in $U$ is either an exceptional Q-set or a sink (resp. a source).
Moreover, the multi-saddle connection diagram $D(v)$ is a finite forest (i.e. a finite disjiont union of trees). 
We can deform the flow $v$ on $S$ by the following operation which preserve $U$: 
Collapsing multi-saddle separatrices into singletons, we may assume that there are no multi-saddle connections of $v$. 

\begin{claim}\label{claim:031}
If $\overline{U} \cap \mathrm{E}(v) = \emptyset$, then $U$ is in the cases $(1)$ or $(2)$.
\end{claim}

\begin{proof}[Proof of Claim~\ref{claim:031}]
Suppose that $\overline{U} \cap \mathrm{E}(v) = \emptyset$.
The $\omega$-limit (resp. $\alpha$-limit) set of each point in $U$ is a sink (resp. a source).
Fix a sink $s$ which is the $\omega$-limit set of a point in $U$
and a small closed transversal $\gamma$ which is a boundary of a sink disk of the sink $s$.
If $\gamma$ is contained in $U$, then $U$ is an open annulus in $\mathrm{P}(v)$ and so is desired in the case $(2)$.
Thus, we may assume that $\gamma$ is not contained in $U$.
Then there are separatrices $O_i$ between the sink $s$ and multi-saddles such that $O_i \subset \partial U$.
Since each multi-saddle in $\partial U$ connects with a sink or a source, the boundary $\partial U$ consists of sinks, sources, multi-saddles, and semi-multi-saddle separatrices from sinks or to sources.
Taking a double $U_1$ of $\overline{U}$ if necessary, the resulting flow on the resulting closed surface $U_1$ consists of sinks, sources, $0$-saddles, and non-recurrent orbits.
Since the index of each singular point is non-negative, the existence of both a sink and a source and the Poincar\'e-Hopf theorem imply that the double is a sphere and there are exactly one sink and one source except $0$-saddles on the double $U_1$.
This means $U$ is an invariant open trivial flow box in $\mathrm{P}(v)$ and so is desired in the case $(1)$.
\end{proof}

Thus we may assume that $\overline{U} \cap \mathrm{E}(v) \neq \emptyset$.
This implies that $S$ is not spherical.

\begin{claim}\label{claim:032}
We may assume that $\mathrm{LD}(v) = \emptyset$. 
\end{claim}

\begin{proof}[Proof of Claim~\ref{claim:032}]
We can deform the flow $v$ on $S$ by the following operation which preserve $U$: 
%
Applying an operation $\mathop{\mathrm{C}}_{t}$ to essential loops in $\mathrm{LD}(v) \subset S - \overline{U}$ as possible, 
we may assume that $\mathrm{LD}(v) = \emptyset$. 
\end{proof}

\begin{claim}\label{claim:033}
The invariant subset $U$ is in the case $(1)$.
\end{claim}

\begin{proof}[Proof of Claim~\ref{claim:033}]
Cutting an essential closed transversal $\gamma_i$ intersecting $\mathrm{E}(v)$ and pasting one sink disk and one source disk (i.e. applying an operation $\mathop{\mathrm{C}}_{t}$) as possible, let $T$ be the resulting surface from $S$, $w$ the resulting flow on $T$, and $U^- \subseteq T$ the resulting subset of $U \setminus \bigcup_{i} \gamma_i$.
By Lemma~\ref{cut-ess-orbit}, the resulting surface $T$ is spherical. 
Let $U_k$ be the connected components of $U^-$ with $U^- = \bigsqcup_{k} U_k$.
Denote by $\widetilde{U}$ (resp. $\widetilde{U}_k$)
the saturation of $U^-$ (resp. $U_k$) by $w$.
By construction, each connected component $\widetilde{U}_k$ of $\widetilde{U}$ is an open trivial flow box such that $\omega(x) = a$ and $\alpha(x) = b$ for the new sink $a$ and the new source $b$. 
Then each connected component $U_k$ of $U \setminus \bigcup_{i} \gamma_i$ is an invariant open trivial flow box in $T$ with respect to $w$ whose vertical boundary components correspond to closed intervals $I_{k,\alpha}, I_{k,\omega} \subset \bigcup_{i} \gamma_i \cap U \subset S$.
In other words, each connected component $U_k \subset T$ has one $\alpha$-vertical boundary corrsponding to $I_{k,\alpha} = I_{k-1,\omega}$ and one $\omega$-vertical boundary corrsponding to $I_{k,\omega} = I_{k+1,\alpha}$.
Since $U = \bigsqcup_{k} U_k \sqcup I_{k,\omega}$ is constructed by pasting $\alpha$-vertical boundaries and $\omega$-vertical boundaries, we have $U$ is an invariant open trivial flow box in $\mathrm{P}(v)$ and so is desired in the case $(1)$.
\end{proof}

This completes the classification of connected components of the complement $S - \mathop{\mathrm{Bd}}(v)$. 

By Proposition~\ref{cor:bd01}, the weak border point set $\mathop{\mathrm{BD}}(v)$ consists of singular points, one-sided periodic orbits, proper semi-multi-saddle separatrices, separatrices on $\partial S$ between $\partial$-sources and $\partial$-sinks, the closure of the union of limit cycles, and finitely many exceptional Q-sets.
Moreover, the union $\mathop{\mathrm{Bd}}(v) \sqcup \mathrm{LD}(v)$ consists of finitely many orbit classes. 
Then the set difference $\mathop{\mathrm{Bd}}(v)  - \mathrm{E}(v)$ consists of finitely many proper orbits. 

\begin{claim}\label{claim:034}
Assertions {\rm(b)} and {\rm(c)} hold.
\end{claim}

\begin{proof}[Proof of Claim~\ref{claim:034}]
Fix a connected component $U$ of $\mathrm{LD}(v) \setminus \mathop{\mathrm{BD}}(v)$ and a point $x \in U$.  
\cite[Proposition~2.2]{yokoyama2016topological} implies that  $U \subseteq \mathrm{LD}(v)$ and $\overline{O(x)} = \overline{O(y)}$ for any $y \in U$. 
This implies assertion {\rm(b)}. 

Lemma~\ref{lem3-04} and Proposition~\ref{prop:cs} imply that the intersection $\mathop{\mathrm{BD}}(v) \cap \partial^+ \overline{O(x)} \subseteq  \partial^+ \mathrm{LD}(v) \subseteq  \partial^{-} \mathop{\mathrm{Sing}}(v) \sqcup \partial^{-} \mathrm{P}(v)$ consists of finitely many multi-saddle separatrices and singular points. 
This implies assertion {\rm(c)}. 
\end{proof}

\begin{claim}\label{claim:035}
Assertion {\rm(a)} holds.
\end{claim}

\begin{proof}[Proof of Claim~\ref{claim:035}]
Fix a connected component $U$ of $\mathrm{P}(v) \setminus \mathop{\mathrm{BD}}(v)$ and a point $x \in U$.  
Then $U$ is an invariant trivial flow box. 

Suppose that $\omega(x) \cap \mathrm{E}(v) \neq \emptyset$. 
By the construction in Claim~\ref{claim:033}, the $\omega$-limit sets of any point in $U$ is the exceptional Q-set $\omega(x)$.  
By symmetry, the $\alpha$-limit set $\alpha(y)$ of any point $y \in U$  corresponds to $\alpha(x)$. 
This implies assertion {\rm(a)}. 

Thus we may assume that $\bigcup_{y \in U} \omega(y) \cap \mathrm{E}(v) = \emptyset$. 
By symmetry,  we may assume that $\bigcup_{y \in U} \alpha(y) \cap \mathrm{E}(v) = \emptyset$. 
From Poincar{\'e}-Bendixson theorem for a quasi-regular flow, the $\omega$-limit set $\omega(y)$ of any point $y \in U$ is either a sink, a $\partial$-sink, or a semi-attracting limit circuit. 
By the flow fox theorem, the $\omega$-limit set $\omega(z)$ of each point $z$ in some small \nbd of any point $y \in U$ corresponds to $\omega(y)$, which is either a sink, a $\partial$-sink, or a semi-attracting limit circuit. 
The open condition implies that the $\omega$-limit set $\omega(y)$ of any point $y \in U$  corresponds to $\omega(x)$. 
By symmetry, the $\alpha$-limit set $\alpha(y)$ of any point $y \in U$  corresponds to $\alpha(x)$. 
This implies assertion {\rm(a)}. 
\end{proof}

\begin{claim}\label{claim:036}
Assertion {\rm(d)} holds.
\end{claim}

\begin{proof}[Proof of Claim~\ref{claim:036}]
By Lemma~\ref{lem010}, we obtain $\mathop{\mathrm{BD}}(v) - \mathrm{E}(v) = \mathop{\mathrm{Sing}}(v) \sqcup \partial^{-} \mathop{\mathrm{Per}}(v) \sqcup \partial_{\mathop{\mathrm{Per}}(v)} \sqcup \mathop{\mathrm{Per}_{1}}(v) \sqcup \partial^{-} \mathrm{P}(v)\sqcup \mathop{\mathrm{P}_{\mathrm{sep}}}(v)$. 
Let $U$ be a connected component of $S - \mathop{\mathrm{Bd}}(v)$ and $\partial$ a connected component of $\partial U$ with $\mathrm{E}(v) \cap \partial = \emptyset$. 
Then $\partial \subseteq \mathop{\mathrm{BD}}(v) - \mathrm{E}(v) = \mathop{\mathrm{Sing}}(v) \sqcup \partial^{-} \mathop{\mathrm{Per}}(v) \sqcup \partial_{\mathop{\mathrm{Per}}(v)} \sqcup \mathop{\mathrm{Per}_{1}}(v) \sqcup \partial^{-} \mathrm{P}(v)\sqcup \mathop{\mathrm{P}_{\mathrm{sep}}}(v)$. 
If $\mathop{\mathrm{Per}}(v) \cap \partial \neq \emptyset$, then the connectivity of $\partial$ implies that $\partial$ is a periodic orbit and so that assertion {\rm(d)} holds. 
Thus we may assume that $\partial \subseteq \mathop{\mathrm{Sing}}(v) \sqcup \partial^{-} \mathrm{P}(v)\sqcup \mathop{\mathrm{P}_{\mathrm{sep}}}(v)$. 
From Lemma~\ref{lem010}, the union $\mathop{\mathrm{Sing}}(v) \sqcup \partial^{-} \mathrm{P}(v)\sqcup \mathop{\mathrm{P}_{\mathrm{sep}}}(v)$ is a finite union of singular points, proper semi-multi-saddle separatrices, and separatrices on $\partial S$ between $\partial$-sources and $\partial$-sinks. 
Therefore, assertion {\rm(d)} holds. 
\end{proof}

This completes the proof. 
\end{proof}

The previous result can be simplified by modifying $\mathop{\mathrm{Bd}}(v)$. 
To achieve this, we show the following statement. 

\begin{lemma}\label{lem:reduction}
Let $v$ be a quasi-regular flow on a compact surface $S$. 
For any connected component $U$ of $S - \mathop{\mathrm{Bd}}(v)$ and $S - \mathop{\mathrm{BD}}(v)$ which is a trivial flow box or a periodic annulus, each of the left and right transverse boundaries of $U$ is an immersed line which is a finite union of closed orbits, proper semi-multi-saddle separatrices, and separatrices on $\partial S$ between $\partial$-sources and $\partial$-sinks. 
\end{lemma}

\begin{proof}
By definition of $\mathop{\mathrm{BD}}(v)$, we have $\mathop{\mathrm{BD}}(v) = \mathop{\mathrm{Bd}}(v) \sqcup \mathop{\mathrm{Per}_{1}}(v)$. 
Therefore, it suffices to show the assertion for the case $\mathop{\mathrm{Bd}}(v)$. 
Let $U$ be a connected component of $S - \mathop{\mathrm{Bd}}(v)$ which is either a trivial flow box or periodic annulus. 

\begin{claim}\label{claim:037}
If $U$ is an invariant periodic annulus, then assertion holds.
\end{claim}

\begin{proof}[Proof of Claim~\ref{claim:037}]
Suppose that $U$ is an invariant periodic annulus. 
Then $\partial U = \partial_{\pitchfork} U$. 
The quasi-regularity implies that the intersection $\mathop{\mathrm{Sing}}(v) \cap \partial \mathop{\mathrm{Per}}(v)$ consists of centers and multi-saddles. 
By Propositions~\ref{prop4-12.5} and Propositions~\ref{thm018}, we have $\partial U \subseteq \mathop{\mathrm{Cl}}(v) \sqcup \partial^- \mathrm{P}(v)$ is a finite union of centers, multi-saddles, periodic orbits, and multi-saddle separatrices.
\end{proof}

Thus, we may assume that $U$ is an invariant trivial flow box. 

\begin{claim}\label{claim:038}
If $\mathrm{E}(v) \cap \partial U = \emptyset$, then assertion holds.
\end{claim}

\begin{proof}[Proof of Claim~\ref{claim:038}]
Suppose that $\mathrm{E}(v) \cap \partial U = \emptyset$. 
By Theorem~\ref{lem0c}(d), the boundary $\partial U$ is a finite union of closed orbits, proper semi-multi-saddle separatrices, and  separatrices on $\partial S$ between $\partial$-sources and $\partial$-sinks. 
\end{proof}

Thus, we may assume that $\mathrm{E}(v) \cap \partial U \neq \emptyset$. 

\begin{claim}\label{claim:039}
$\mathrm{E}(v) \cap \partial_{\pitchfork} U = \emptyset$. 
\end{claim}

\begin{proof}[Proof of Claim~\ref{claim:039}]
Assume there is a point $x \in \mathrm{E}(v) \cap \partial_{\pitchfork} U$. 
Then $\omega(x)$ or $\alpha(x)$ is an exceptional Q-set. 
By time reversion if necessary, we may assume that $\omega(x)$ is an exceptional Q-set. 
Then $x$ is positively recurrent. 
By the orientation reversion if necessary, we may assume that $O^+(x) \subset \partial_{\pitchfork}^L U$. 
Since $x$ is positively recurrent, there is an open transverse arc $I \subset U$ whose saturation $v(I) = U$ such that $\partial (\bigcup_{y \in I} O^+(y)) \cap \partial_{\pitchfork}^L U = \{x \} \sqcup O^+(x)$. 
From Theorem~\ref{lem0c}(a), we have that $\omega(y) = \omega(z)$ for any points $y,z \in U$. 
By the generalization of the Poincar{\'e}-Bendixson theorem, the $\omega$-limit set of a point $y$ in $U$ is either a semi-attracting circuit (i.e. $\partial$-sink, a sink, a semi-attracting limit circuit) or an exceptional Q-set.
Fix a point $y \in U$. 

Assume that the $\omega$-limit set $\omega(y)$ is a semi-attracting circuit. 
Then there is a semi-attracting collar basin $\A$ of $\omega(y)$ whose boundary $\partial \A$ is a closed transversal such that $O^+(y) \cap \A \neq \emptyset$ and $x \notin \A$. 
For any point $z \in U$, since $\omega(y) = \omega(z)$, we have $O^+(z) \cap \A \neq \emptyset$. 
Since the intersection $U \cap \partial \A$ is an open connected subset of the circle $\partial \A$, the intersection $U \cap \partial \A$ is an open interval. 
Therefore, the intersection $ \partial_{\pitchfork} U \cap \partial \A$ consists of two points, and so the intersection $O^+(x) \cap \partial \A$ is a singleton. 
Since $\A$ is a semi-attracting collar basin, the $\omega$-limit set $\omega(x)$ is the semi-attracting circuit $\omega(y)$, which contradicts that $\omega(x)$ is an exceptional Q-set. 
Thus, the $\omega$-limit set of a point $y$ in $U$ is an exceptional Q-set. 

If $\omega(x) = \omega(y)$, then $x \notin \partial_{\pitchfork} U$. 
Thus $\omega(x) \neq \omega(y)$. 
Then there is an open transverse arc $I$ intersecting $(\omega(y) \cap E(v))$ such that $I \cap O(x) = \emptyset$, and there are a subarc $I' \subset I$ and an orbit arc $J \subset \omega(y)$ whose union $I' \cup J$ is a loop. 
By the waterfall construction to the loop $I' \cup J$, there is a closed transversal $\gamma$ with $\gamma \cap \omega(y) \cap E(v) \neq \emptyset$ and $\gamma \cap O(x) = \emptyset$. 
As in the proof of Theorem~\ref{lem0c}, cutting $\gamma$ and pasting one sink disk and one source disk (i.e. applying an operation $\mathop{\mathrm{C}}_{t}$), let $S'$ be the resulting surface from $S$, $v'$ the resulting flow, and $U^- \subseteq S'$ the resulting subset of $U \setminus \gamma$.
Then $v'$ is a quasi-regular flow on the compact surface $S'$ and the subset $U^-$ is a disjoint union of invariant trivial flow boxes $U^-_i$ in $\mathrm{P}(v')$ each of which is contained in the complement $S' - \mathop{\mathrm{Bd}}(v')$. 
Since the $\omega$-limit set of a point in $U^-$ with respect to $v'$ is a sink, the $\omega$-limit set of any point of $\partial_{\pitchfork} U^-_i$ is either the sink or a multi-saddle on $\partial U^-_i$. 
This means that $\omega_{v'}(x)$ is either the sink or a multi-saddle, which contradicts that $\omega_{v'}(x) = \omega_v(x)$ is an exceptional Q-set. 
This implies that $\mathrm{E}(v) \cap \partial_{\pitchfork} U = \emptyset$. 
\end{proof}

Therefore $\partial_{\pitchfork} U \subseteq \mathop{\mathrm{BD}}(v) \setminus \mathrm{R}(v) = \mathop{\mathrm{BD}}(v) \cap (\mathop{\mathrm{Cl}}(v) \sqcup \mathrm{P}(v))$. 

\begin{claim}\label{claim:040}
$\partial_{\pitchfork} U \cap \mathop{\mathrm{Per}}(v) = \emptyset$. 
\end{claim}

\begin{proof}[Proof of Claim~\ref{claim:040}]
Assume there is a point $x \in \mathop{\mathrm{Per}}(v) \cap \partial_{\pitchfork} U$. 
Then there is a closed transverse arc $\mu$ such that $x$ is a boundary component of $\mu$ and $\mu - \{ x \} \subset U$. 
Applying the flow box theorem to $\mu$, the invariance of $U$ implies that $U$ is a collar of $\mu$. 
This implies that $U$ is not simply connected and so is not a trivial flow box, which contradicts that $U$ is an open disk. 
\end{proof}

Since $\mathop{\mathrm{Sing}}(v)$ is finite, by $\partial_{\pitchfork} U \subseteq \mathop{\mathrm{Bd}}(v) \cap (\mathop{\mathrm{Sing}}(v) \sqcup \mathrm{P}(v))$, Lemma~\ref{lem010} implies that $\partial_{\pitchfork} U \subseteq \mathop{\mathrm{Sing}}(v) \sqcup \mathop{\mathrm{P}_{\mathrm{semi}}}(v) \sqcup \mathop{\partial_{\mathrm{P}(v)}}$ is a finite union of singular points, proper semi-multi-saddle separatrices, and separatrices on $\partial S$ between $\partial$-sources and $\partial$-sinks. 
\end{proof}

Recall that a periodic orbit is one-sided if and only if it is either a boundary component of a surface or has a small neighborhood which is a M{\"o}bius band. 
Theorem~\ref{lem0c} and Lemma~\ref{lem:reduction} imply the following statement.

\begin{theorem}\label{homogeneity}
Each connected component of $S - \mathop{\mathrm{BD}}(v)$ for a quasi-regular flow $v$ on a compact surface $S$ is one of the following invariant open subsets exclusively:
\\
$(1)$
A trivial flow box in $\mathrm{P}(v)$, whose orbit space is an open interval,
\\
$(2)$
An annulus in $\mathrm{P}(v)$, whose orbit space is a circle,
\\
$(3)$ 
A torus in $\mathop{\mathrm{Per}}(v)$, whose orbit space is a circle,
\\
$(4)$ 
An annulus in $\mathop{\mathrm{Per}}(v)$, whose orbit space is an open interval, or
\\
$(5)$ 
An essential subset in $\mathrm{LD}(v)$, whose orbit class space is a singleton.

Moreover, the following statements hold: 
\\
{\rm(a)} The $\omega$-limit $(\mathrm{resp.}$ $\alpha$-limit$)$ set of a point in a connected component in $\mathrm{P}(v) \setminus \mathop{\mathrm{BD}}(v)$ is the $\omega$-limit $(\mathrm{resp.}$ $\alpha$-limit$)$ set of any point in the connected component. 
\\
{\rm(b)} The closure of a point in a connected component in $\mathrm{LD}(v)$ is the closure of any point in the connected component. 
\\
{\rm(c)} Any boundary component of a connected component of $S - \mathop{\mathrm{BD}}(v)$ which does not intersect $\mathrm{E}(v)$ is a finite union of closed orbits, proper semi-multi-saddle separatrices, and separatrices on $\partial S$ between $\partial$-sources. 
\\
{\rm(d)} Each of the left and right transverse boundaries of any connected components of $S - \mathop{\mathrm{BD}}(v)$ in the cases $(1)$ and $(4)$ is an immersed line which is a finite union of closed orbits, proper semi-multi-saddle separatrices, and separatrices on $\partial S$ between $\partial$-sources and $\partial$-sinks. 
\end{theorem}

%
For a non-quasi-regular flow $v$ on a compact surface, note that the transverse boundaries of connected components of the complement of the border point set $\mathop{\mathrm{BD}}(v)$ in the cases $(1)$ and $(4)$ in the previous corollary are not connected in general and that the $\omega$-limit sets of points in the domain are different from each other in general (see Figure~\ref{ex-disk}).
\begin{figure}
\begin{center}
\includegraphics[scale=0.3]{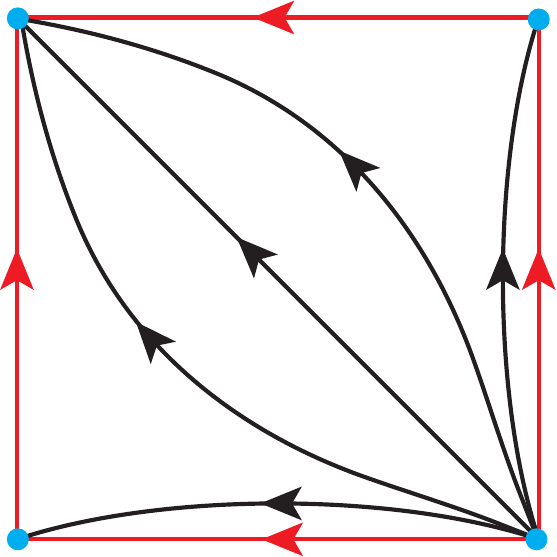}
\end{center}
\caption{The complement of the border point set $\mathop{\mathrm{BD}}(v)$ in a closed disk is an open disk whose left (resp. right) transverse boundary is a disjoint union of two open intervals and whose $\omega$-vertical boundary consists of three points such that each $\omega$-limit set is a singular point.}
\label{ex-disk}
\end{figure}

\section{Enumeration of flows of finite type}

In this section, we demonstrate the enumerability of flows of finite type. 
To show this, we construct the finiteness of $\mathop{\mathrm{Bd}}(v)$.

\subsection{Description of components of the decompositions and their boundaries}

Denote by  $\mathop{\mathrm{Sing}_c}(v)$ the set of centers, and recall that $D_{\mathrm{ss}}(v)$ is the ss-multi-saddle connection diagram. 
Propositions~\ref{prop4-12.5} and Proposition~\ref{thm018} imply the following description. 

\begin{lemma}\label{lem:dss}
The following statements hold for a quasi-regular flow $v$ on a compact surface $S$:
\\
$(1)$ The disjoint union $\partial^{-} \mathrm{P}(v) \sqcup  \mathop{\mathrm{P}_{\mathrm{sep}}}(v) = \mathop{\mathrm{P}_{\mathrm{semi}}}(v) \sqcup \mathop{\partial_{\mathrm{P}(v)}}$ is the finite union of multi-saddle separatrices, ss-separatrices, and non-closed orbits connecting a $\partial$-sink and a $\partial$-source on the boundary $\partial S$.
\\
$(2)$ The disjoint union $(\mathop{\mathrm{Sing}}(v) - \mathop{\mathrm{Sing}_c}(v)) \sqcup \partial^{-} \mathop{\mathrm{Per}}(v) \sqcup \mathop{\mathrm{P}_{\mathrm{semi}}}(v) \sqcup \mathrm{E}(v)$ is closed.
\\
$(3)$ 
$D_{\mathrm{ss}}(v) - ((\mathop{\mathrm{Sing}}(v) - \mathop{\mathrm{Sing}_c}(v)) \sqcup  \mathop{\mathrm{P}_{\mathrm{semi}}}(v) \sqcup \mathrm{E}(v)) \subseteq \partial^{-} \mathop{\mathrm{Per}}(v)$. 
\end{lemma}

\begin{proof}
Lemma~\ref{lem010} implies assertion {\rm(1)}. 

\begin{claim}\label{claim:041}
Assertion {\rm(2)} holds. 
\end{claim}

\begin{proof}[Proof of Claim~\ref{claim:041}]
Quasi-regularity implies that $(\mathop{\mathrm{Sing}}(v) - \mathop{\mathrm{Sing}_c}(v))$ is  finite and so closed. 
By Proposition~\ref{prop4-12}, we have $\overline{\partial^{-} \mathop{\mathrm{Per}}(v)} \subseteq \partial^{-} \mathrm{P}(v) \sqcup \partial^{+} \mathrm{P}(v) \subseteq (\mathop{\mathrm{Sing}}(v) - \mathop{\mathrm{Sing}_c}(v)) \sqcup \partial^{-} \mathop{\mathrm{Per}}(v) \sqcup \partial^{-} \mathrm{P}(v) \subseteq (\mathop{\mathrm{Sing}}(v) - \mathop{\mathrm{Sing}_c}(v)) \sqcup \partial^{-} \mathop{\mathrm{Per}}(v) \sqcup \mathop{\mathrm{P}_{\mathrm{semi}}}(v)$ and $\overline{\mathrm{E}(v)} = \partial^+ \mathrm{E}(v) \sqcup \mathrm{E}(v) \subseteq (\mathop{\mathrm{Sing}}(v) - \mathop{\mathrm{Sing}_c}(v)) \sqcup \mathop{\mathrm{P}_{\mathrm{semi}}}(v) \sqcup \mathrm{E}(v)$. 

From Lemma~\ref{lem002}, the union $\mathrm{LD}(v)$ is open.
Quasi-regularity implies that any limit circuits are contained in a union $(\mathop{\mathrm{Sing}}(v) - \mathop{\mathrm{Sing}_c}(v)) \sqcup \mathop{\mathrm{P}_{\mathrm{semi}}}(v)$. 
By the generalization of the Poincar{\'e}-Bendixson theorem, quasi-regularity implies that each of $\omega$-limit set and $\alpha$-limit set of a point in $\mathop{\mathrm{P}_{\mathrm{semi}}}(v)$ is either a singular point in $\mathop{\mathrm{Sing}}(v) - \mathop{\mathrm{Sing}_c}(v)$, a limit cycle in $\overline{\partial^{-} \mathop{\mathrm{Per}}(v)}$, a limit circuit in $(\mathop{\mathrm{Sing}}(v) - \mathop{\mathrm{Sing}_c}(v)) \sqcup \mathop{\mathrm{P}_{\mathrm{semi}}}(v)$, or an exceptional Q-set in $\overline{\mathrm{E}(v)}$. 
Therefore, we have $\overline{\mathop{\mathrm{P}_{\mathrm{semi}}}(v)} \subseteq (\mathop{\mathrm{Sing}}(v) - \mathop{\mathrm{Sing}_c}(v)) \cup \mathop{\mathrm{P}_{\mathrm{semi}}}(v) \cup \overline{\partial^{-} \mathop{\mathrm{Per}}(v)} \cup \overline{\mathrm{E}(v)} \subseteq (\mathop{\mathrm{Sing}}(v) - \mathop{\mathrm{Sing}_c}(v)) \sqcup \partial^{-} \mathop{\mathrm{Per}}(v) \sqcup \mathop{\mathrm{P}_{\mathrm{semi}}}(v) \sqcup \mathrm{E}(v)$. 
This completes the proof of the claim. 
\end{proof}

Recall that the ss-multi-saddle connection diagram $D_{\mathrm{ss}}(v)$ is the union of multi-saddles, multi-saddle separatrices, ss-separatrices, and ss-components.

\begin{claim}\label{claim:042}
$D_{\mathrm{ss}}(v) \subseteq (\mathop{\mathrm{Sing}}(v) - \mathop{\mathrm{Sing}_c}(v)) \sqcup \partial^{-} \mathop{\mathrm{Per}}(v) \sqcup \mathop{\mathrm{P}_{\mathrm{semi}}}(v) \sqcup \mathrm{E}(v)$. 
\end{claim}

\begin{proof}[Proof of Claim~\ref{claim:042}]
By Lemma~\ref{lem010} and Proposition~\ref{thm018}, an ss-component is either a singular point in $\mathop{\mathrm{Sing}}(v) - \mathop{\mathrm{Sing}_c}(v)$, a limit circuit in $(\mathop{\mathrm{Sing}}(v) - \mathop{\mathrm{Sing}_c}(v)) \sqcup \mathop{\mathrm{P}_{\mathrm{semi}}}(v)$, or an exceptional Q-set in $\overline{\mathrm{E}(v)} \subseteq (\mathop{\mathrm{Sing}}(v) - \mathop{\mathrm{Sing}_c}(v)) \sqcup \mathop{\mathrm{P}_{\mathrm{semi}}}(v) \sqcup \mathrm{E}(v)$. 
Since all multi-saddle separatrices and all ss-separatrices are contained in $\mathop{\mathrm{P}_{\mathrm{semi}}}(v)$,  
the claim holds. 
\end{proof}

Conversely, by definition, we obtain $(\mathop{\mathrm{Sing}}(v) - \mathop{\mathrm{Sing}_c}(v))   \sqcup \mathrm{E}(v) \subseteq D_{\mathrm{ss}}(v)$.  

\begin{claim}\label{claim:043}
$\mathop{\mathrm{P}_{\mathrm{semi}}}(v) \subseteq D_{\mathrm{ss}}(v)$
\end{claim}

\begin{proof}[Proof of Claim~\ref{claim:043}]
By the openness of $\mathrm{LD}(v)$, since $v$ is quasi-regular, the generalization of the Poincar{\'e}-Bendixson theorem for a flow with finitely many singular points implies that the $\omega$-limit set and $\alpha$-limit set of any proper semi-multi-saddle separatrix are either singular points in $(\mathop{\mathrm{Sing}}(v) - \mathop{\mathrm{Sing}_c}(v))$, limit circuits, or exceptional Q-sets. 
Then implies that the $\omega$-limit set and $\alpha$-limit set of any proper semi-multi-saddle separatrix are ss-components. 
Therefore, any proper semi-multi-saddle separatrix are ss-separatrices and so are contained in $D_{\mathrm{ss}}(v)$. 
Thus $ \mathop{\mathrm{P}_{\mathrm{semi}}}(v) \subseteq D_{\mathrm{ss}}(v)$. 
\end{proof}

Therefore assertion {\rm(3)} holds. 
\end{proof}

We describe the relations among the border point set and relative concepts as follows. 

\begin{proposition}\label{cor019}
The following statements hold for a quasi-regular flow $v$ with at most finitely many limit cycles on a compact surface $S$:
\\
$(1)$
$D_{\mathrm{ss}}(v) = (\mathop{\mathrm{Sing}}(v) - \mathop{\mathrm{Sing}_c}(v)) \sqcup \partial^{-} \mathop{\mathrm{Per}}(v) \sqcup \mathop{\mathrm{P}_{\mathrm{semi}}}(v) \sqcup \mathrm{E}(v)$ is closed.
\\
$(2)$ 
$\mathop{\mathrm{Bd}}(v) = D_{\mathrm{ss}}(v) \sqcup \mathop{\mathrm{Sing}_c}(v) \sqcup \partial_{\mathop{\mathrm{Per}}(v)} \sqcup \mathop{\partial_{\mathrm{P}(v)}}$.
\end{proposition}

\begin{proof}
By Lemma~\ref{lem010}, the union $\partial^{-} \mathop{\mathrm{Per}}(v)$ is the finite union of limit cycles.
Lemma~\ref{lem:dss} implies that assertion {\rm(1)} holds. 
From Lemma~\ref{lem010}, by the quasi-regularity, we have the following equality: 
\[
\begin{split}
\mathop{\mathrm{Bd}}(v) &= \mathop{\mathrm{Sing}}(v) \sqcup \partial^{-} \mathop{\mathrm{Per}}(v) \sqcup \partial_{\mathop{\mathrm{Per}}(v)} \sqcup \partial^{-} \mathrm{P}(v) \sqcup \mathop{\mathrm{P}_{\mathrm{sep}}}(v) \sqcup \mathrm{E}(v) 
\\
&= \mathop{\mathrm{Sing}}(v) \sqcup \partial^{-} \mathop{\mathrm{Per}}(v) \sqcup \partial_{\mathop{\mathrm{Per}}(v)} \sqcup \mathop{\mathrm{P}_{\mathrm{semi}}}(v) \sqcup \mathop{\partial_{\mathrm{P}(v)}} \sqcup \mathrm{E}(v) 
\\
&= D_{\mathrm{ss}}(v) \sqcup \mathop{\mathrm{Sing}_c}(v)  \sqcup \partial_{\mathop{\mathrm{Per}}(v)} \sqcup  \mathop{\partial_{\mathrm{P}(v)}}
\end{split}
\] 
\end{proof}

\subsubsection{Description of components of the decompositions}

Recall that a quasi-regular flow without recurrent points ($\mathrm{i.e.}$ $\mathrm{R}(v) = \emptyset$) is of finite type if there are at most finitely many limit cycles.
In the finite type case, we obtain the following statement from Lemma~\ref{lem010}, Proposition~\ref{cor019}, Theorem~\ref{lem0c}, and Theorem~\ref{homogeneity}.

\begin{proposition}\label{lem11-01}
The following statements hold for a flow $v$ of finite type on a compact surface $S$:
\\
$(1)$
The orbit spaces $(S - \mathop{\mathrm{Bd}}(v))/v$ and $(S - D_{\mathrm{ss}}(v))/v$ are finite disjoint unions of intervals and circles.
\\
$(2)$
The orbit space $(S - \mathop{\mathrm{BD}}(v))/v$ is a finite disjoint union of open intervals and circles. 
\\
$(3)$ For any connected component $C$ of $S - \mathop{\mathrm{BD}}(v)$, the boundary $\partial C$ is a finite union of closed orbits, proper semi-multi-saddle separatrices, and separatrices on $\partial S$ between $\partial$-sources and $\partial$-sinks. 
\end{proposition}

\begin{proof}
By Proposition~\ref{cor:bd01}, the border (resp. weak border) point set $\mathop{\mathrm{BD}}(v)$ {\rm(resp.} $\mathop{\mathrm{Bd}}(v)${\rm)} is closed and is the finite union of singular points, one-sided periodic orbits {\rm(resp.} periodic orbits on $\partial S${\rm)}, limit cycles, proper semi-multi-saddle separatrices, and separatrices between a $\partial$-source and a $\partial$-sink on the boundary $\partial S$. 
Therefore, assertion (3) holds. 
By Proposition~\ref{cor019}, the union $D_{\mathrm{ss}}(v) = \mathop{\mathrm{Bd}}(v) -(\mathop{\mathrm{Sing}_c}(v) \sqcup \partial_{\mathop{\mathrm{Per}}(v)} \sqcup \mathop{\partial_{\mathrm{P}(v)}})$ is the finite union of non-central singular points, limit cycles, and proper semi-multi-saddle separatrices. 

%
By Theorem~\ref{homogeneity}, each connected component of $S - \mathop{\mathrm{BD}}(v)$  is either a trivial flow box in $\mathrm{P}(v)$ whose orbit space is an interval, an annulus in $\mathrm{P}(v)$ whose orbit space is a circle, a torus in $\mathop{\mathrm{Per}}(v)$ whose orbit space is a circle, or an annulus in $\mathop{\mathrm{Per}}(v)$ whose orbit space is an open interval. 
This implies assertion (2). 

By Theorem~\ref{lem0c}, each connected component of $S - \mathop{\mathrm{Bd}}(v)$  is either a trivial flow box in $\mathrm{P}(v)$ whose orbit space is an interval, an annulus in $\mathrm{P}(v)$ whose orbit space is a circle, a torus in $\mathop{\mathrm{Per}}(v)$ whose orbit space is a circle, or an annulus/M{\"o}bius band/Klein bottle in $\mathop{\mathrm{Per}}(v)$ whose orbit space is an interval. 
The orbit space $(S - \mathop{\mathrm{Bd}}(v))/v$ is a finite disjoint union of intervals and circles.

Since the set difference $\mathop{\mathrm{Bd}}(v) - D_{\mathrm{ss}}(v) = \mathop{\mathrm{Sing}_c}(v) \sqcup \partial_{\mathop{\mathrm{Per}}(v)} \sqcup \mathop{\partial_{\mathrm{P}(v)}}$ is the union of centers, boundary components in $\mathop{\mathrm{int}}\Pv$, and  separatrices between a $\partial$-source and a $\partial$-sink on the boundary $\partial S$. 
Therefore, each connected component of $S - D_{\mathrm{ss}}(v)$ is either a trivial flow box in $\mathrm{P}(v)$ whose orbit space is an interval, an annulus in $\mathrm{P}(v)$ whose orbit space is a circle, a torus in $\mathop{\mathrm{Per}}(v)$ whose orbit space is a circle, a center disk whose orbit space is an interval, or an annulus/M{\"o}bius band/Klein bottle in $\mathop{\mathrm{Per}}(v)$ whose orbit space is an interval. 
The orbit space $S - D_{\mathrm{ss}}(v)$ is a finite disjoint union of intervals and circles.
\end{proof}


\subsection{On the (weak) border point set}

Recall several concepts to describe the combinatorial structures of the (weak) border point sets and their complements of flows of finite type.  

\subsubsection{Graphs}
Recall an ordered triple $G := (V, E, r)$ is an abstract multi-graph if $V$ and $E$ are sets and $r : E \to \{ \{ x,y \} \mid x, y \in V \}$. 

\subsubsection{Orders}
A binary relation $\leq$ on a set $P$ is a {\bf pre-order} if it is reflexive (i.e. $a \leq a$ for any $a \in P$) and transitive (i.e. $a \leq c$ for any $a, b, c \in P$ with $a \leq b$ and $b \leq c$).
For a pre-order $\leq$, the inequality $a<b$ means both $a \leq b$ and $a \neq b$. 
A pre-order $\leq$ on $X$ is a partial order if it is antisymmetric (i.e. $a = b$ for any $a,b \in P$ with $a \leq b$ and $b \leq a$).
A {\bf poset} is a set with a partial order. 
For a point $x$ of a pre-ordered set $P$, define the {\bf class} $\hat{x} := \{ y \in P \mid x \leq y, y \leq x \}$. 
The set of classes of a pre-ordered set $P$ is a decomposition and so its quotient space is a poset with respect to the induced order, which is called the {\bf class space} of $P$ and denoted by $\hat{P}$.
A pre-order order $\leq$ is a total order (or linear order) if either $a < b$ or $b < a$ for any points $a \neq b$.

\subsubsection{Specialization pre-orders of topologies}
The {\bf specialization pre-order} $\leq_{\tau}$ on a topological space $(X, \tau)$ is defined as follows:
$ x \leq_{\tau} y $ if  $ x \in \overline{\{ y \}}$.
Then the class $\hat{x} = \{ y \in X \mid \overline{\{ x \}} = \overline{\{ y \}} \}$ of a point $x \in X$.
The class space with the quotient topology is a $T_0$ space, which is called the $T_0$-tification of $X$ and denoted by $\hat{X}$.

\subsubsection{Multi-graphs}
A {\bf graph} is a cell complex whose dimension is at most one and which is a geometric realization of an abstract multi-graph.
In other words, it can be drawn such that no edges cross each other.
Such a drawing is called a {\bf surface (directed) graph}.
Note that a finite (directed) multi-graph can be embedded in a surface.
Moreover, we call that a disjoint union of a surface directed graph and finite simple closed directed curves embedded in a surface is a {\bf generalized surface directed graph}. 
\begin{definition}
An ordered triple $G = (V, E, r)$ is an {\bf abstract multi-graph with $0$-hyper-edges} if $V$ and $E$ are sets and $r : E \to \{ \{ x,y \} \mid x, y \in V \} \sqcup \{ \emptyset \}$.
\end{definition}
Note that a generalized surface directed graph is a geometric realization of an abstract multi-graph with $0$-hyper-edges on a surface such that $0$-hyper-edges are realized by disjoint simple closed directed curves.

\subsubsection{Multi-graphs as posets}
A poset $P$ is said to be {\bf multi-graph-like} if the height of $P$ is at most one and  $\left| \overline{\{ x \}} \right| \leq 3$ for any element $x \in P$.
For a multi-graph-like poset $P$, each element of $P_0$ is called a {\bf vertex}, and each element of $P_1$ is called an {\bf edge}.
Then an abstract multi-graph $G$ can be considered as a multi-graph-like poset $(P, \leq_G)$ with $V = P_{0}$ and $E = P_1$ as follows: $P = V \sqcup \mathrm{E}(v)$ and $x <_G e$ if $x \in r(e)$.
Conversely, a multi-graph-like poset $P$ can be considered as an abstract multi-graph with $V = P_{0}$, $E = P_1$, and $r: P_1 \to \{ \{ x,y \} \mid x, y \in V \}$ defined by $r(e) := \partial^+ \{ e \} = \overline{\{ e \}} - \{ e \}$.

\subsubsection{Extended orbits}
Let $v$ be a flow on a compact connected surface $S$.
We define the {\bf extended orbit space} $\bm{S/v_{\mathrm{ex}}}$ as follows. 

\begin{definition}
The extended orbit space $S/v_{\mathrm{ex}}$ is the quotient space $S/\sim_{\mathrm{ex}}$ of the equivalence relation $\sim_{\mathrm{ex}}$ defined as follows: $x \sim_{\mathrm{ex}} y$ if either $O(x) = O(y)$ or both $x$ and $y$ are contained in a multi-saddle connection.
\end{definition}

In other words, $S/v_{\mathrm{ex}}$ can be obtained from $S/v$ by collapsing the multi-saddle connections into singletons.
Note that $(S - \mathop{\mathrm{Bd}}(v))/v = (S - \mathop{\mathrm{Bd}}(v))/v_{\mathrm{ex}}$ and that each class of $\sim_{\mathrm{ex}}$ is either an orbit or a multi-saddle connection.
We define a variation of orbit space as follows. 

\begin{definition}
The {\bf extended orbit class space} $\bm{S/\hat{v}_{\mathrm{ex}}}$ is the $T_0$-tification of $S/v_{\mathrm{ex}}$. 
\end{definition}

In other words, define a relation $\approx_{\mathrm{ex}}$ on $S$ as follows: $x \approx_{\mathrm{ex}} y$ if either $\overline{O(x)} = \overline{O(y)}$ or both $x$ and $y$ are contained in a multi-saddle connection.
Then $S/\hat{v}_{\mathrm{ex}} := S/\approx_{\mathrm{ex}}$.
By \cite[Proposition~2.2]{yokoyama2016topological} and \cite[[Corollary 3.4]{yokoyama2019properness},
the definitions of the equivalences $\sim_{\mathrm{ex}}$ and $\approx_{\mathrm{ex}}$ imply the following statement.

\begin{lemma}\label{prop:bd03}
Let $v$ be a flow on a compact connected surface $S$.
The following conditions are equivalent for any invariant subset $T \subseteq S$:
\\
$(1)$ $T/\hat{v}_{\mathrm{ex}} = T/v_{\mathrm{ex}}$.
\\
$(2)$ $T/\hat{v} = T/v$.
\\
$(3)$ $\mathrm{R}(v) \cap T = \emptyset$ $(\mathrm{i.e.}$ each recurrent orbit in $T$ is closed$ )$.
\\
$(4)$ $T/v$ is $T_0$. 
\\
$(5)$ Each orbit in $T$ is proper. 
\end{lemma}

\subsection{Finiteness and Enumerability of flows of finite type}
Recall the following notations for a quasi-regular flow $v$: 
\\
$D(v)$ : The multi-saddle connection diagram $D(v)$ (i.e. the union of multi-saddles and multi-saddle separatrices)
\\
$\mathop{\mathrm{P}_{\mathrm{semi}}}(v) = \partial^{-} \mathrm{P}(v) \sqcup \mathop{\mathrm{P}_{\mathrm{ms}}}(v) \sqcup \mathop{\mathrm{P}_{\mathrm{ss}}}(v)$ : the union of proper semi-multi-saddle separatrices. 
We obtain the finiteness of the weak border point sets.

\begin{proposition}\label{thm:bd02}
The following statements hold for any flow $v$ of finite type on a compact connected surface $S$:
\\
$(1)$ The quotient space $\mathop{\mathrm{Bd}}(v)/\hat{v}_{\mathrm{ex}} = \mathop{\mathrm{Bd}}(v)/v_{\mathrm{ex}}$ is a finite abstract multi-graph $(V_{\mathop{\mathrm{Bd}}(v)}, E_{\mathop{\mathrm{Bd}}(v)})$. 
\\
$(2)$ The vertex set $V_{\mathop{\mathrm{Bd}}(v)}$ consists of centers, sinks, $\partial$-sinks, sources, $\partial$-sources, periodic orbits on $\partial S$, multi-saddle connections, and limit cycles. 
\\
$(3)$ The edge set $E_{\mathop{\mathrm{Bd}}(v)}$ consists of ss-separatrices and non-closed orbits on $\partial S$ connecting $\partial$-sinks and $\partial$-sources.
\end{proposition}

\begin{proof}
By Lemma~\ref{prop:bd03}, we obtain $\mathop{\mathrm{Bd}}(v)/\hat{v}_{\mathrm{ex}} = \mathop{\mathrm{Bd}}(v)/v_{\mathrm{ex}}$. 
Proposition~\ref{cor:bd01} implies that the weak border point set $\mathop{\mathrm{Bd}}(v)$ 
is the finite union of singular points,  periodic orbits on the boundary $\partial S$, proper semi-multi-saddle separatrices, separatrices between a $\partial$-source and a $\partial$-sink on $\partial S$, and limit cycles. 
This implies assertion (1). 

Let $E_{\mathop{\mathrm{Bd}}(v)}$ be the union of ss-separatrices and separatrices between a $\partial$-source and a $\partial$-sink on the boundary $\partial S$. 
Therefore, assertion {\rm(3)} holds.

Put $V_{\mathop{\mathrm{Bd}}(v)} := \mathop{\mathrm{Bd}}(v) - E_{\mathop{\mathrm{Bd}}(v)}$. 

\begin{claim}\label{claim:047}
Assertion {\rm(2)} holds.
\end{claim}

\begin{proof}[Proof of Claim~\ref{claim:047}]
By Lemma~\ref{lem001}, the union $\partial^{-} \mathrm{P}(v)\sqcup \mathop{\mathrm{P}_{\mathrm{sep}}}(v) = \mathop{\mathrm{P}_{\mathrm{semi}}}(v) \sqcup \mathop{\partial_{\mathrm{P}(v)}}$ is the finite union of proper semi-multi-saddle separatrices and separatrices on $\partial S$ between $\partial$-sources and $\partial$-sinks. 
Therefore, we have the following equality: 
\[
\begin{split}
\mathop{\mathrm{Bd}}(v) &= \mathop{\mathrm{Sing}}(v) \sqcup \partial_{\mathop{\mathrm{Per}}(v)} \sqcup \partial^{-} \mathop{\mathrm{Per}}(v) \sqcup \partial^{-} \mathrm{P}(v) \sqcup \mathop{\mathrm{P}_{\mathrm{sep}}}(v)
\\
& = \mathop{\mathrm{Sing}}(v) \sqcup \partial_{\mathop{\mathrm{Per}}(v)} \sqcup \partial^{-} \mathop{\mathrm{Per}}(v) \sqcup \mathop{\mathrm{P}_{\mathrm{semi}}}(v) \sqcup \mathop{\partial_{\mathrm{P}(v)}}
\end{split}
\]
On the other hand, since the union of ss-separatrices is the set difference $\mathop{\mathrm{P}_{\mathrm{semi}}}(v) \setminus D(v)$ and the union of ss-separatrices and separatrices between a $\partial$-source is $\mathop{\partial_{\mathrm{P}(v)}}$, we have $E_{\mathop{\mathrm{Bd}}(v)} = (\mathop{\mathrm{P}_{\mathrm{semi}}}(v) \setminus D(v)) \sqcup \mathop{\partial_{\mathrm{P}(v)}}$. 
This implies the following equality: 
\[
\begin{split}
V_{\mathop{\mathrm{Bd}}(v)} = \mathop{\mathrm{Bd}}(v) - E_{\mathop{\mathrm{Bd}}(v)} &= \mathop{\mathrm{Sing}}(v) \sqcup \partial_{\mathop{\mathrm{Per}}(v)} \sqcup \partial^{-} \mathop{\mathrm{Per}}(v) \sqcup (D(v) \setminus \Sv)
\end{split}
\]
By Lemma~\ref{lem010}, this means that the invariant subset $V_{\mathop{\mathrm{Bd}}(v)}$ is the finite union of singular points, periodic orbits on the boundary $\partial S$, limit cycles, and multi-saddle connections. 
The quasi-regularity implies that each singular point is either a center, a sink, a $\partial$-sink, a source, a $\partial$-source, or a multi-saddle. 
Since any multi-saddles are contained in multi-saddle connections, assertion {\rm(2)} holds.  
\end{proof}
This completes the proof.
%
\end{proof}

Proposition~\ref{cor019} and Proposition~\ref{thm:bd02} imply the following statement.

\begin{corollary}\label{lem8-04}
The following statements hold for a flow $v$ of finite type on a compact surface $S$: 
\\
$(1)$ The quotient spaces $D_{\mathrm{ss}}(v)/{v}_{\mathrm{ex}}$,  $\mathop{\mathrm{Bd}}(v)/{v}_{\mathrm{ex}}$, and $\mathop{\mathrm{BD}}(v)/{v}_{\mathrm{ex}}$ are finite abstract multi-graphs. 
\\
$(2)$ The complement $(S - D_{\mathrm{ss}}(v))/v$ $(\mathrm{resp.}$ $(S - \mathop{\mathrm{Bd}}(v))/v$, $S - \mathop{\mathrm{BD}}(v)/v )$ is a finite disjoint union of intervals and circles such that the boundary of a connected component of $S - D_{\mathrm{ss}}(v)$ $(\mathrm{resp.}$ $S - \mathop{\mathrm{Bd}}(v)$, $S - \mathop{\mathrm{BD}}(v))$ consists of finitely many closed orbits and finitely many separatrices. 
\end{corollary}

Moreover, combining Theorem~\ref{lem0c}, Proposition~\ref{lem11-01}, and Proposition~\ref{thm:bd02}, we have the following statement. 

\begin{lemma}\label{thm:enumerate}
The set of flows of finite type on compact surfaces is enumerable by using finite labeled graphs.
\end{lemma}

\begin{proof}
Since each of $\omega$-limit set and $\alpha$-limit set of a point is a closed orbit or a limit circuit, 
if there are no limit circuits then each ss-separatrix connects to a ($\partial$-)sink or a ($\partial$-)source and so the ss-multi-saddle connection diagram is a generalized surface directed graph on a compact surface. 

\begin{claim}\label{claim:048}
The set of flows of finite type without limit circuits on compact surfaces is enumerable by using finite labeled graphs.  
\end{claim}

\begin{proof}
Let $v$ be a flow of finite type without limit circuits on a compact surface. 
Then the weak border point set $\mathop{\mathrm{Bd}}(v)$ is a generalized surface directed graph on a compact surface. 
Since each multi-saddle connection diagram consists of finitely many orbits, we can enumerate the weak border point sets of flows of finite type without limit circuits on compact surfaces as a generalized surface directed graph on compact surfaces. 

Theorem~\ref{lem0c} implies that each connected component of the weak border point set $\mathop{\mathrm{Bd}}(v)$ is one of six kinds of domains. 
By Proposition~\ref{lem11-01}, the complement of the weak border point set consists of finitely many domains. 
If a domain is either a torus or a Klein bottle in $\Pv$ (i.e. in the case (3) or (4)) in Theorem~\ref{lem0c}, then we have all the topological information of the domain. 
For the remaining four domains (i.e. a trivial flow box in $\mathrm{P}(v)$, a periodic annulus, a transverse annulus, and M{\"o}bius band in $\Pv$), since the boundaries of such domains consist of finitely many orbits, adding these finite gluing information, we can construct complete finite invariants of flows of finite type without limit circuits on compact surfaces. 
In particular, we can enumerate such flows. 
\end{proof}

In the general case, by Proposition~\ref{thm:bd02}, the extended orbit space of the weak border point set is a finite abstract multi-graph with $0$-hyper-edges. 

\begin{claim}\label{claim:049}
Any flows of finite type on compact surfaces can be obtained by replacing sinks and sources into limit circuits and by pasting their boundary components if necessary. 
\end{claim}

\begin{proof}
Let $v$ be a flow of finite type on a compact surface $S$, $\Gamma$ the union of limit circuits of $v$, $S_1$ the metric completion of $S - \Gamma$, and $\partial := S_1 - (S - \Gamma)$ the new boundary of $S_1$. 
Since $S_1 - \partial = S - \Gamma$ and the new boundary $\partial$ can be immersed into $\Gamma$, there is the natural continuous surjection $\pi : S_1 \to S$ associated to the metric completion such that a restriction $\pi\vert_{\partial}: \partial \to \Gamma$ is an immersion and a restriction $\pi\vert_{S_1 - \partial}: S_1 - \partial \to S - \Gamma$ is an identity mapping as on the left of Figure~\ref{Fig:quotients}. 
Collapsing all new boundary components of $S_1$ (i.e. connected components of $\partial$) each of which is a limit circuit into singletons, we obtain a continuous surjection $q: S_1 \to S_2$ as on the right of Figure~\ref{Fig:quotients}. 
\begin{figure}
\[
\xymatrix@=18pt{
S \ar@{}[d]|{\bigcup} & & S_1\ar[ll]_{\pi} \ar[rr]^q \ar@{}[d]|{\bigcup} & &  S_2 \ar@{}[d]|{\bigcup} \\
S - \Gamma & &  S_1 - \partial \ar@{=}[ll]_{\pi\vert_{S_1 - \partial}} \ar@{=}[rr]^{q\vert_{S_1 - \partial}} & &  S_2 - q(\partial)}
\]
\caption{Canonical quotient mappings induced by the metric completion and the collapse.}
\label{Fig:quotients}
\end{figure}
Then the resulting flow $v_2$ on $S_2$ is a flow of finite type without limit circuits on a compact surface such that $v$ can be obtained from $v_2$ by replacing sinks and sources into limit circuits and by pasting boundary components if necessary. 
\end{proof}

Therefore, we can enumerate the weak border point set $\mathop{\mathrm{Bd}}(v)$ and the complement $S - \mathop{\mathrm{Bd}}(v)$ by using finite labeled graphs, and so the flow. 
\end{proof}

\section{Graphs of surface flows}\label{sec:graph}

We explicitly construct a finite recursive complete invariant of a flow of finite type as a finite labeled graph, whose existence is demonstrated by Lemma~\ref{thm:enumerate}, which 
Using the invariant, we revive the ``Markus-Neumann theorem'' (i.e. affirmatively answer Problem~\ref{prob:MN}) in the next section.

\subsection{Finite complete invariants}

Denote by $\chi(S)$ the space of flows on a compact surface $S$.
Put $\chi := \bigcup \chi(T)$, where the union runs over compact surfaces.
Denote by $\bm{\chi_F} \subset \chi$ the subspace of flows of finite type (i.e. quasi-regular flows without non-closed recurrent points such that there are at most finitely many limit cycles).
Then we have the following observation. 

\begin{lemma}
The class $\chi_F$ contains flows generated by $\Sigma(S) \sqcup \mathcal{H}^r_*(S)$ for any $r \in \Z_{\geq 0 }$, where $\Sigma^r(S)$ is the set of the Morse-Smale vector fields in the sense of Labarca and Pacifico~\cite{labarca1990stability} {\rm(}see also Definition~\ref{def:MS_with} below{\rm)} and $\mathcal{H}^r_*(S)$ is the set of regular Hamiltonian $C^r$ vector fields each of whose saddle connection is self-connected.
\end{lemma}

Recall that $\mathcal{H}^r_*(S)$ is open dense in the set $\mathcal{H}^r(S)$ of Hamiltonian $C^r$ vector fields on an orientable closed connected surface $S$ for any $r \in \Z_{\geq 1}$. 
Moreover, we obtain the following observation. 

\begin{lemma}
Suppose that $S$ is closed. 
Then the class $\chi_F$ contains flows generated by $\Sigma^r(S)$ for any $r \in \Z_{\geq 0 }$, where $\Sigma^r(S)$ is the set of the Morse-Smale $C^r$ vector fields {\rm(}see Definition~\ref{def:MS_without} below{\rm)}.
\end{lemma}

\subsubsection{Reconstruction of flows of finite type}
Roughly speaking, each flow $v$ in $\chi_F$ can be reconstructed from the following finite data (see definitions of $G^{\BD}$, $G_{\BD}$, $G_{D(v)}$, $\mathcal{L}^{\mathrm{BD}(v)}$, $l_{\BD}$, $l_{D(v)}$ below in this section):  
\[
(G^{\BD}, G_{\BD}, G_{D(v)}, \mathcal{L}^{\mathrm{BD}(v)}, l_{\BD}, l_{D(v)})
\] 
To state more precisely, we define some notations as follows:
Denote by $\mathcal{G}_l$ the set of finite abstract multi-graphs with labels.
Define a mapping $p: \chi_F \to \mathcal{G}_l \times \mathcal{G}_l \times \mathcal{G}_l$ by
$$p(v) := ((G^{\BD}, \mathcal{L}^{\BD}), (G_{\BD}, l_{\BD}), (G_{D(v)}, l_{D(v)})).$$
Denote by $\sim$ the topological equivalence on $\chi_F$ (i.e. an equivalence relation defined by $v \sim w$ if there is a homeomorphism $h: S \to T$ which is orbit-preserving (i.e. the image of an orbit of $v$ is an orbit of $w$)  and preserves the orientation of the orbits).
Write $\bm{\widetilde{\chi_F}}$ by the quotient space of $\chi_F$ by the equivalence $\sim$. 
Denote by $\bm{\widetilde{\mathcal{G}_l}}$ the quotient space of $\mathcal{G}_l$ by the isomorphisms. 
Now we can state precisely that a flow $v$ in $\chi_F$ can be reconstructed from the finite data as follows.

\begin{theorem}\label{injection}
The induced map $\widetilde{p}: \widetilde{\chi_F} \to \widetilde{\mathcal{G}_l} \times \widetilde{\mathcal{G}_l} \times \widetilde{\mathcal{G}_l}$ is well-defined and injective.
\end{theorem}

Theorem~\ref{main:04} is folowed from Theorem~\ref{injection}. 
To show this theorem, we construct $G^{\BD}$ (in \S~\ref{sec:O}), $G_{\BD}$ (in \S~\ref{sec:G_BD}), $G_{D(v)}$ (in \S~\ref{sec:G_Dv}), $\mathcal{L}^{\mathrm{BD}(v)}$ (in \S~\ref{sec:L^BD} and \S~\ref{sec:leq_perp}), $l_{\BD}$ (in \S~\ref{sec:l_BD}), $l_{D(v)}$ (in \S~\ref{sec:G_Dv}) as follows.

%

\subsection{Canonical domains and canonical regions}

We define canonical domains and canonical regions as follows. 

\begin{definition}
Each connected component of the complement $S - D_{\mathrm{ss}}(v)$ (resp. $S - \mathop{\mathrm{Bd}}(v)$, $S - \mathop{\mathrm{BD}}(v)$) of a flow $v$ of finite type on a compact surface $S$ is called a {\bf canonical domain} of $D_{\mathrm{ss}}(v)$ (resp. $\mathop{\mathrm{Bd}}(v)$, $\mathop{\mathrm{BD}}(v)$). 
\end{definition}

\begin{definition}
Let $U$ be an invariant trivial flow box (resp. invariant periodic annulus, invariant transverse annulus). 
We call that $U$ is a {\bf canonical region} if it is maximal with respect to the inclusion relation among invariant trivial flow boxes (resp. invariant periodic annuli, invariant transverse annuli). 
\end{definition}

Note that the definition of a canonical region of $v$ for the orbit complex \cite{markus1954global,neumann1975classification,neumann1976global} is slightly different from our definitions, and that a canonical domain which is not a periodic torus for a flow of finite type on a compact surface corresponds to a canonical region. 

\subsection{Graph structure $G_{\mathrm{BD}}$ with label $l_{\BD}$ of the border point set $\BD$}\label{sec:G_BD}
Let $v$ be a flow of finite type on a compact surface. 
From Proposition~\ref{thm:bd02}, denote by $\bm{G_{\mathrm{BD}(v)}} = (V_{\mathrm{BD}(v)}, E_{\mathrm{BD}(v)})$ the abstract multi-graph $\mathop{\mathrm{BD}}(v)/v_{\mathrm{ex}} = \mathop{\mathrm{BD}}(v)/\hat{v}_{\mathrm{ex}}$.
Note that each of $V_{\mathrm{BD}(v)}$ and $E_{\mathrm{BD}(v)}$ is a subset of $G_{\mathrm{BD}(v)} = \mathop{\mathrm{BD}}(v)/v_{\mathrm{ex}}$. 
In other words, the vertex set $V_{\mathrm{BD}(v)}$ consists of singular points that are not multi-saddles (i.e. sinks, $\partial$-sinks, sources, $\partial$-sources, and centers), multi-saddle connections, limit cycles, and non-limits one-sided cycles.
In other words, we have the following statement:
\[
\begin{split}
V_{\mathrm{BD}(v)} = (\Sv \setminus D(v))/v_{\mathrm{ex}} & \sqcup \{ \text{multi-saddle connection} \}
\\
& \sqcup \{ \text{limit cycle} \}
\\
& \sqcup \{ \text{non-limits one-sided cycle} \}
\end{split}
\]
The edge set $E_{\mathrm{BD}(v)}$ consists of orbits in $\mathop{\partial_{\mathrm{P}(v)}}$, ss-separatrices which are arranged in cyclic ways around limit circuits, sinks, and sources, and ss-separatrices which are arranged in linear ways around $\partial$-sinks and $\partial$-sources.
In other words, we have the following statement:
\[
\begin{split}
E_{\mathrm{BD}(v)} = \{ \text{orbit in } \mathop{\partial_{\mathrm{P}(v)}} \} \sqcup \{ \text{ss-separatrix} \}
\end{split}
\]

Notice that ss-separatrices are arranged in cyclic ways around limit circuits, sinks, and sources,  and that ss-separatrices are arranged in linear ways around $\partial$-sinks and $\partial$-sources.
Then the arrangements can be recovered by the dual graph $G^{\BD}$ and the labels $\partial^{\mathrm{BD}(v)}$ (see definition in \S~\ref{sec:boundary_label}) and $\leq^{\mathrm{BD}(v)}$ (see definition in \S~\ref{sec:leq_perp_trans}), which have the transverse information. 

\subsubsection{Label $l_{\BD}$ for $G_{\mathrm{BD}(v)}$}\label{sec:l_BD}
Define a label $l_{\BD} : V_{\BD} \to \{ 0, 1_{\partial}, 1_{\mathrm{non}}, 1_{\mathrm{two}}, 1_{\mathrm{ms}} \} =: \Sigma_{\mathrm{Top}}^0$ as follows: 
\\
$l_{\BD}(\sigma) := 0$ if $\sigma \in \mathop{\mathrm{Sing}}(v)$ is either a sink, a $\partial$-sink, a source, a $\partial$-source, or a center, 
\\
$l_{\BD}(\gamma) := 1_{\partial}$ if $\gamma$ is a periodic orbit in $\partial S$,   
\\
$l_{\BD}(\gamma) := 1_{\mathrm{non}}$ if $\gamma$ is a one-sided periodic orbit outside of $\partial S$, 
\\
$l_{\BD}(\gamma) := 1_{\mathrm{two}}$ if $\gamma$ is a two-sided periodic orbit.  
\\
$l_{\BD}(\gamma) := 1_{\mathrm{ms}}$ if $\gamma$ is a multi-saddle connection diagram.  
\\
%
Note that each multi-saddle circuit with an associated collar has either semi-repelling, semi-attracting, or identical one-sided holonomy, because of the finiteness of limit cycles. 
Then a one-sided circuit has one of these three kinds of holonomies and a two-sided cycle has one of six kinds of holonomies. 

\subsection{Abstract multi-graph $G_{D(v)}$ with label $l_{D}$ of the orbit space $D(v)/v$ of the multi-saddle connection diagram}\label{sec:G_Dv}

For a quasi-regular flow $v$ on a compact surface, denote by $\bm{G_{D(v)}}$ the abstract multi-graph $D(v)/v$ and a label $\bm{l_{D}} : G_{D(v)} \to V_{\BD}$ is defined as the natural projection $D(v)/v \to D(v)/v_{\mathrm{ex}}$.

\subsection{Dual graphs $G^{\mathrm{BD}}$ of the border point set $\BD$}\label{sec:O}
For a flow $v$ of finite type on a compact surface $S$, define the dual graph $\bm{G^{\mathrm{BD}(v)}} = (V^{\BD}, E^{\BD})$ of the border point set $\BD$ as follows:
The vertex set $V^{\BD}$ consists of canonical domains of $\BD$ (i.e. the connected components of $S - \BD$), each of which is either a periodic torus, an invariant trivial flow box, an invariant transverse annulus, or an invariant periodic annulus.
In other words, we have the following equality: 
\[
\begin{split}
V^{\BD} = \{ \text{periodic torus} \} &\sqcup \{ \text{maximal invariant trivial flow box in } S - \mathop{\mathrm{BD}}(v) \}
\\
&\sqcup \{ \text{maximal invariant transverse annulus in } S - \mathop{\mathrm{BD}}(v) \}
\\
&\sqcup \{ \text{maximal invariant periodic annulus in } S - \mathop{\mathrm{BD}}(v) \}
\end{split}
\]

An edge $\{U_i, U_j\}$ in $E^{\BD}$ exists if one of the folowing statements holds: 
\\
(1) $\dim (\overline{U_i} \cap \overline{U_j}) = 1$ and at least one of $U_i$ and $U_j$ is a periodic/ransverse annulus. 
\\
(2) $U_i \neq U_j$, $\dim (\partial_{\pitchfork} U_i \cap \partial_{\pitchfork} U_j) = 1$, and both $U_i$ and $U_j$ are trivial flow boxes.
We call that such canonical domains $U_i$ and $U_j$ are adjacent. 

Define a label $\bm{l^{\BD}} : V^{\BD} \to \{ \mathbb{D}, \widetilde{\mathbb{A}}, \mathbb{A},  \mathbb{T} \}$ as follows:
\\
$l^{\BD}(U) := \mathbb{D}$ if $U$ is an invariant open trivial flow box in $\mathrm{P}(v)$, 
\\
$l^{\BD}(U) := \widetilde{\mathbb{A}}$ if $U$ is an open transverse annulus in $\mathrm{P}(v)$, 
\\
$l^{\BD}(U) := \mathbb{A}$ if $U$ is an open annulus in $\mathop{\mathrm{Per}}(v)$, 
\\
$l^{\BD}(U) := \mathbb{T}$ if $U$ is a torus in $\mathop{\mathrm{Per}}(v)$, 
\\
By Theorem~\ref{homogeneity}, the label $l^{\BD}$ is well-defined. 
Set $\Sigma_{\mathrm{Top}}^1 := \{ \mathbb{D}, \widetilde{\mathbb{A}}, \mathbb{A},  \mathbb{T} \}$. 
%

\subsection{A spiral direction $\mathcal{D}_{\partial_{\perp}}$ for a region whose $\omega$-limit {\rm(resp.} $\alpha$-limit{\rm)} set is a circuit}\label{subsec:001}

For a flow $v$ of finite type on a compact surface, for any $U \in V^{\BD}$ with some orientation, define a label $\bm{\mathcal{D}_{\partial_{\perp}}(U) = (\mathcal{D}_{\partial_{\perp}^{\alpha}}(U), \mathcal{D}_{\partial_{\perp}^{\omega}}(U))} \in \{ R, L, 0 \} \times \{ R, L, 0 \}$ as follows: 

\begin{figure}
\begin{center}
\includegraphics[scale=0.35]{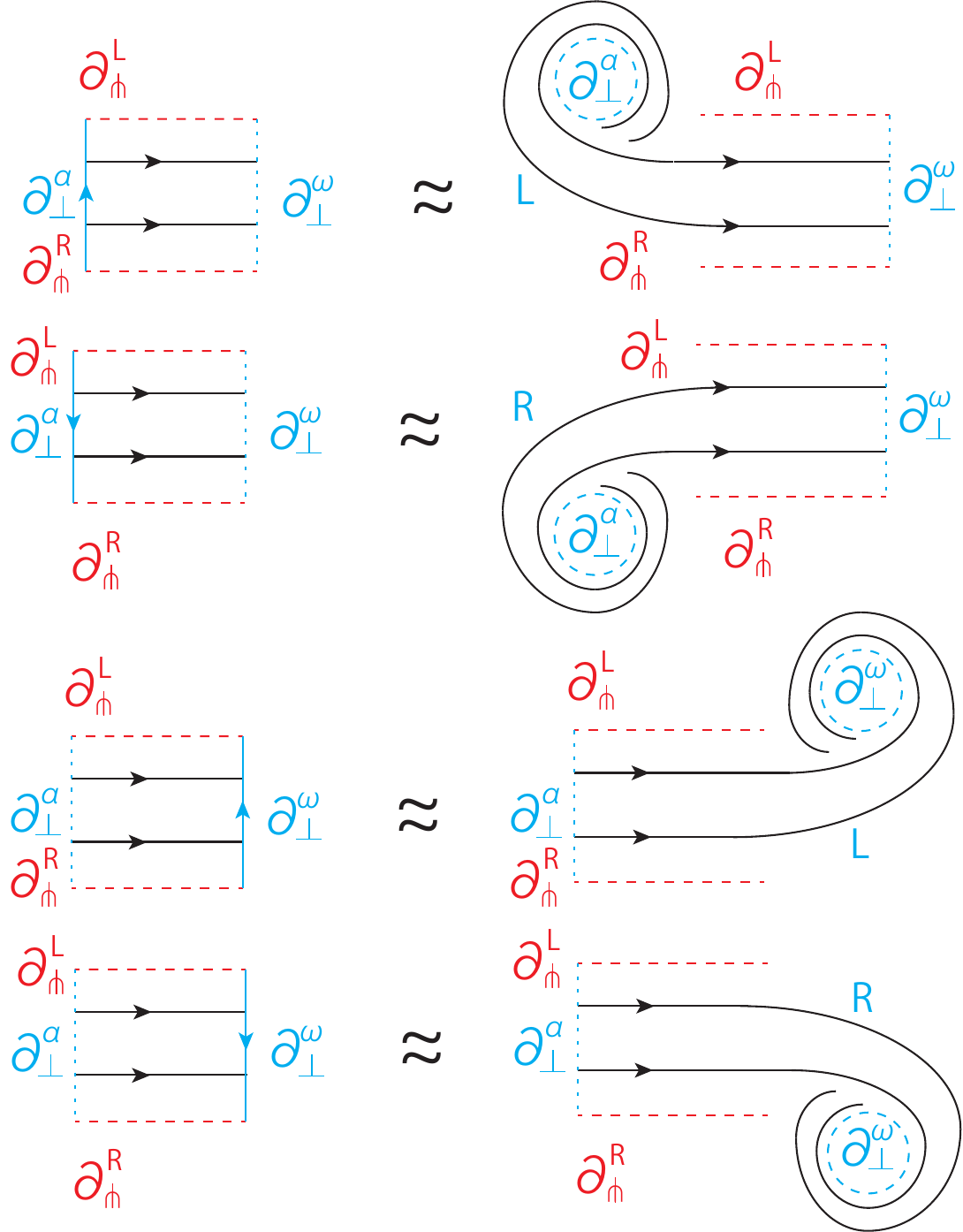}
\,\,\,
\includegraphics[scale=0.45]{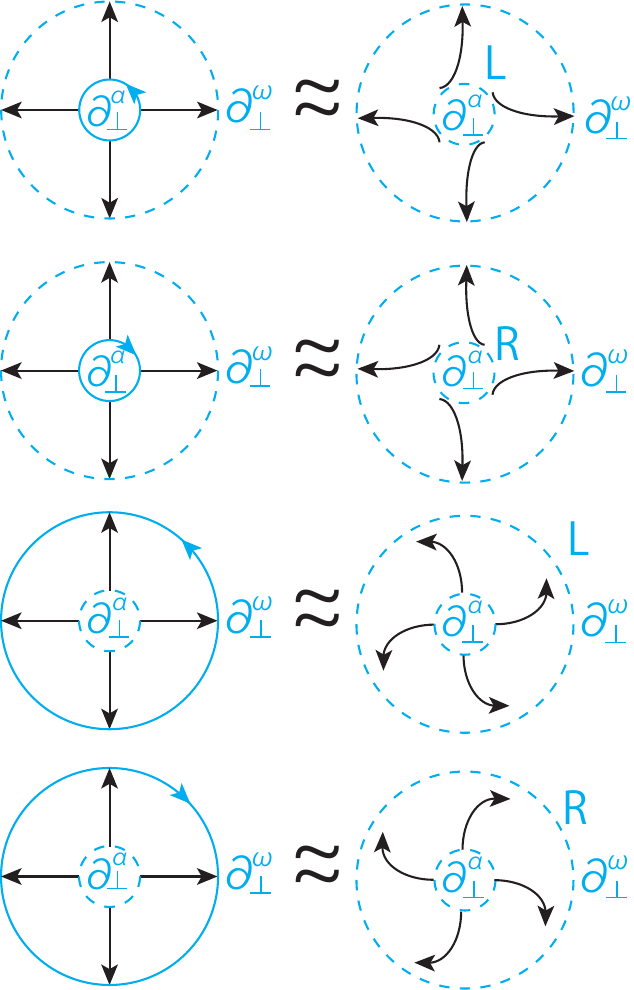}
\end{center}
\caption{Schematic pictures and their spiral structures of invariant trivial flow boxes and transverse annuli}
\label{spiral}
\end{figure}

Let $U$ be either an oriented invariant trivial flow box or an oriented invariant transverse annulus. 

\begin{definition}
The {\bf spiral direction} $\mathcal{D}_{\partial_{\perp}^{\alpha}} (U)$ is $L$ (resp. $R$) if there is a non-trivial circuit $\gamma$ such that $\alpha(x) = \gamma$ for any $x \in U$, and the invariant subset $U$ rotates anti-clockwise (resp. clockwise) around $\gamma$ as the upper part of Figure~\ref{spiral}. 
\end{definition}

\begin{definition}
The {\bf spiral direction} $\mathcal{D}_{\partial_{\perp}^{\omega}} (U)$ is $L$ (resp. $R$) if there is a non-trivial circuit $\gamma$ such that $\omega(x) = \gamma$ for any $x \in U$, and the invariant subset $U$ rotates anti-clockwise (resp. clockwise) around $\gamma$ as the lower part of Figure~\ref{spiral}. 
\end{definition}

Otherwise the {\bf spiral direction} $\mathcal{D}_{\partial_{\perp}^{\sigma}} (U)$ for any $\sigma \in \{ \alpha, \omega \}$ is defined as $0$. 

Note that the direction with respect to the local orientation on the box or annulus is well-defined even if the limit circuit $\gamma$ is contained in a non-orientable compact surface, by using the orientation reversion as in \S~\ref{sec:lor}.

\subsection{Labels for vertical boundaries and transverse boundaries}

We define labels for vertical boundaries $\partial_{\perp}^{\bullet}$ and transverse boundaries $\partial_{\pitchfork}^{\bullet}$ as follows. 

\subsubsection{Labels for vertical boundaries}

Define a label $\partial_{\perp}^{\bullet} = (\partial_{\perp}^{\alpha}/\hat{v}, \partial_{\perp}^{\omega}/\hat{v}) : V^{\BD} \to 2^{\mathrm{BD}(v)/\hat{v}} \times 2^{\mathrm{BD}(v)/\hat{v}}$ by 
$$\partial_{\perp}^{\bullet}(U) = (\partial_{\perp}^{\alpha}(U)/\hat{v}, \partial_{\perp}^{\omega}(U)/\hat{v}) := (\alpha(x)/\hat{v}, \omega(x)/\hat{v})$$ 
for any $U \in V^{\BD}$ and for some $x \in U$.
By Theorem~\ref{homogeneity}, the $\omega$-limit and $\alpha$-limit sets of any point in $U$ are so of $x$ and thus the label $\partial_{\perp}^{\bullet}$ is well-defined. 

\subsection{Label $\mathcal{L}^{\mathrm{BD}(v)}$ of $V^{\mathrm{BD}(v)}$}

We define labels for transverse descriptions as follows. 

\subsubsection{Labels for transverse boundaries}

Define a label $\partial_{\pitchfork}^{\bullet} := (\partial_{\pitchfork}^R/\hat{v}, \partial_{\pitchfork}^L/\hat{v}) : V^{\BD} \to 2^{\mathrm{BD}(v)/\hat{v}} \times 2^{\mathrm{BD}(v)/\hat{v}}$ by $$\partial_{\pitchfork}^{\bullet}(U) := (\partial_{\pitchfork}^R(U)/\hat{v}, \partial_{\pitchfork}^L(U)/\hat{v}))$$ for any $U \in V^{\BD}$. 

\subsubsection{Label $\partial^{\mathrm{BD}(v)}$ for boundaries}\label{sec:boundary_label}

Define a label $\partial^{\mathrm{BD}(v)}$ as follows: 
\[
\partial^{\mathrm{BD}(v)} := (\partial_{\pitchfork}^{\bullet}, \partial_{\perp}^{\bullet}) = (\partial_{\pitchfork}^R/\hat{v}, \partial_{\pitchfork}^L/\hat{v}, \partial_{\perp}^{\alpha}/\hat{v}, \partial_{\perp}^{\omega}/\hat{v}) : V^{\mathrm{BD}(v)} \to (2^{\mathrm{BD}(v)/\hat{v}})^4
\]

\subsubsection{Label $\leq^{\mathrm{BD}(v)}$ for boundaries of connected components of the complement $S - \mathrm{BD}(v)$}

Roughly speaking, the label $\leq^{\mathrm{BD}(v)}$ is information about how the orbits in $\mathrm{BD}(v)$ are arranged when one observes $\mathrm{BD}(v)$ from each connected component of the complement $S - \mathrm{BD}(v)$.
For any connected component $U$ of the complement $S - \mathrm{BD}(v)$, we can define a label $\leq^{\mathrm{BD}(v)} := (\leq_{\partial_{\pitchfork}^R}, \leq_{\partial_{\pitchfork}^L}, \leq_{\partial_{\perp}^{\alpha}}, \leq_{\partial_{\perp}^{\omega}})$ (see \S~\ref{sec:leq_perp_trans} below for the definition) such that the pair of $\leq_{\partial_{\perp}^{\alpha}}$ and $\leq_{\partial_{\perp}^{\omega}}$ (resp. $\leq_{\partial_{\pitchfork}^R}$ and $\leq_{\partial_{\pitchfork}^L}$) corresponds to a pair of ``orders'' of orbits on vertical (resp. transverse) boundaries of $U$, which are the restriction of the order structure (see the second order structure in \S~\ref{sec:ord_cpx} for details) introduced by a brief description as a term in Markus-Neumann theorem (see \cite[\S~5]{neumann1976global} for details). 
Though the detailed construction of $\leq^{\mathrm{BD}(v)}$ is straightforward but not concise, we arrange the precise construction in Definition~\ref{def:leq_bd} in \S~\ref{sec:leq_perp_trans} in Appendix.

\subsubsection{Definition of $\mathcal{L}^{\mathrm{BD}(v)}$}\label{sec:L^BD}

We define the label $\mathcal{L}^{\mathrm{BD}(v)}$ as follows: 
\[
\mathcal{L}^{\mathrm{BD}(v)} := [L^{\mathrm{BD}(v)}] = [\mathcal{O}, l^{\BD}, \partial^{\mathrm{BD}(v)}, \leq^{\mathrm{BD}(v)}, \mathcal{D}_{\partial_{\perp}}]
\]
Recall that the labels $\mathcal{O}$ and $l^{\BD}$ (resp. $\mathcal{D}_{\partial_{\perp}}$, $\partial^{\mathrm{BD}(v)}$, $\leq^{\mathrm{BD}(v)}$) are defined in \S~\ref{sec:O} (resp. \S~\ref{subsec:001}, \S~\ref{sec:boundary_label}, Definition~\ref{def:leq_bd} in \S~\ref{sec:leq_perp_trans}). 

\subsection{Definition of label $\mathcal{L}^{\mathrm{BD}(v)}$}
Let $v$ be a flow of finite type on a compact surface $S$. 

\subsubsection{Lebels $\partial^{\mathrm{BD}(v)}$ and $\leq^{\mathrm{BD}(v)}$ of $\mathrm{BD}(v)$}\label{subsec:p^BD}

Let $\mathcal{O}_{\mathrm{circuit,path}}$ be the set of finite circuit orders and path orders. 
Define a label 
$$\bm{L^{\mathrm{BD}(v)}} : V^{\mathrm{BD}(v)} \to \{ \pm \} \times \Sigma_{\mathrm{Top}}^1 \times (2^{\mathrm{BD}(v)/\hat{v}})^4 \times  \mathcal{O}_{\mathrm{circuit,path}}^4 \times \{ R, L, 0 \}^2$$ 
as follows: 
\[
L^{\mathrm{BD}(v)} := ( \mathcal{O}, l^{\BD}, \partial^{\mathrm{BD}(v)}, \leq^{\mathrm{BD}(v)}, \mathcal{D}_{\partial_{\perp}})
\] 
By definition, we have 
\[
L^{\mathrm{BD}(v)} = ( \mathcal{O}, l^{\BD}, \partial_{\pitchfork}^R/\hat{v}, \partial_{\pitchfork}^L/\hat{v}, \partial_{\perp}^{\alpha}/\hat{v}, \partial_{\perp}^{\omega}/\hat{v}, \leq_{\partial_{\pitchfork}^R}, \leq_{\partial_{\pitchfork}^L}, \leq_{\partial_{\perp}^{\alpha}}, \leq_{\partial_{\perp}^{\omega}}, \mathcal{D}_{\partial_{\perp}^{\alpha}}, \mathcal{D}_{\partial_{\perp}^{\omega}})
\]

\subsubsection{Local orientation reversion}\label{sec:lor}
A local orientation reversion $\bm{{V}^{\mathrm{Rev}}}: V_{\mathrm{BD}(v)} \to V_{\mathrm{BD}(v)}$ of canonical domains is an identity as a set with $\mathcal{O}^{{\mathrm{Rev}}} (U) = - \mathcal{O}(U)$, where $\mathcal{O}^{{\mathrm{Rev}}} := \mathcal{O}\circ V^{{\mathrm{Rev}}}$. 
Define $\mathcal{D}_{\partial_{\perp}^{\alpha}}^{\mathrm{Rev}}(U) = L$ (resp. $R$) if $\mathcal{D}_{\partial_{\perp}^{\alpha}} (U) = R$ (resp. L). 
Similarly, define $\mathcal{D}_{\partial_{\perp}^{\omega}}^{\mathrm{Rev}}(U) = L$ (resp. $R$) if $\mathcal{D}_{\partial_{\perp}^{\omega}} (U) = R$ (resp. L) and 
\[
\bm{\mathcal{D}_{\partial_{\perp}}^{\mathrm{Rev}}} := (\mathcal{D}_{\partial_{\perp}^{\alpha}}^{\mathrm{Rev}}, \mathcal{D}_{\partial_{\perp}^{\omega}}^{\mathrm{Rev}})
\]
Define $\mathcal{D}_{\partial_{\perp}^{\sigma}}^{\mathrm{Rev}}(U) = 0$ if $\mathcal{D}_{\partial_{\perp}^{\sigma}} (U) = 0$. 
Other revisions of labels are defined as follows: 
\begin{equation*}
\begin{split}
\bm{L^{\mathrm{BD}(v)\mathrm{Rev}}} & := (\mathcal{O}^{{\mathrm{Rev}}},  l^{\BD {\mathrm{Rev}}}, \partial^{\mathrm{BD}(v){\mathrm{Rev}}}, \leq^{\mathrm{BD}(v){\mathrm{Rev}}}, \mathcal{D}_{\partial_{\perp}}^{{\mathrm{Rev}}})\\
&\,\,= (- \mathcal{O}, l^{\BD}, \partial_{\pitchfork}^L/\hat{v},  \partial_{\pitchfork}^R/\hat{v}, \partial_{\perp}^{\bullet}, \leq_{\partial_{\pitchfork}^L}, \leq_{\partial_{\pitchfork}^R}, \leq_{\partial_{\perp}^{\alpha}}, \leq_{\partial_{\perp}^{\omega}}, \mathcal{D}_{\partial_{\perp}}^{{\mathrm{Rev}}})
\end{split}
\end{equation*}
\[
\bm{l_{\BD}^{\mathrm{Rev}}} := l_{\BD}
\]
%
%
In particular, we have $(\partial_{\pitchfork}^{R}/\hat{v})^{\mathrm{Rev}} = \partial_{\pitchfork}^L/\hat{v}$, $(\partial_{\pitchfork}^{L}/\hat{v})^{\mathrm{Rev}} = \partial_{\pitchfork}^R/\hat{v}$, $(\leq_{\partial_{\pitchfork}^R})^{\mathrm{Rev}} = \leq_{\partial_{\pitchfork}^L}$, and $(\leq_{\partial_{\pitchfork}^L})^{\mathrm{Rev}} = \leq_{\partial_{\pitchfork}^R}$.

\subsubsection{Equivalence classes of canonical domains}

Define an equivalence relation of labels as follows for a quasi-regular flow $v$ on a compact surface: 
$L^{\mathrm{BD}(v)} \sim_o L'^{\mathrm{BD}(v)}$ if either $L^{\mathrm{BD}(v)}(U) = L'^{\mathrm{BD}(v)}(U)$ or $L^{\mathrm{BD}(v) \mathrm{Rev}}(U) = L'^{\mathrm{BD}(v)}(U)$ for any $U \in V^{\mathrm{BD}(v)}$. 
\begin{definition}
Define a label $\bm{\mathcal{L}^{\mathrm{BD}(v)}} : V^{\mathrm{BD}(v)} \to ((2^{\mathrm{BD}(v)/\hat{v}})^4 \times \Sigma_{\mathrm{Top}}^1 \times \mathcal{O}_{\mathrm{circuit,path}}^4 \times \{ R, L, 0 \}^2)/\sim_o$ by 
\[
\mathcal{L}^{\mathrm{BD}(v)} := [L^{\mathrm{BD}(v)}] = [\mathcal{O}, l^{\BD}, \partial^{\mathrm{BD}(v)}, \leq^{\mathrm{BD}(v)}, \mathcal{D}_{\partial_{\perp}}]
\]
where $[L^{\mathrm{BD}(v)}(U)]$ is the equivalence class of $L^{\mathrm{BD}(v)}(U)$. 
\end{definition}

Since $2^{\mathrm{BD}(v)/\hat{v}}$ is the power set of an abstract multi-graph, any representative of the image by $\mathcal{L}^{\mathrm{BD}(v)}$ of any element consists of four signs, four abstract multi-graphs, two finite path orders, and two finite circuits orders.

\subsection{Reconstruction of the border point set $\BD$}

We observe the following statement. 

\begin{lemma}\label{reconstruction-order}
The border point set $\mathrm{BD}(v)$ of a flow $v$ of finite type on a compact surface can be reconstructed by abstract finite multi-graphs $G^{\BD}$, $G_{\BD}$, and $G_{D(v)}$ with the labels $\mathcal{L}^{\mathrm{BD}(v)}$, $l_{\BD}$, and $l_{D(v)}$ as an embedded subset. 
\end{lemma}

\begin{proof}
We show that the orbit structure and the holonomy of any element of $V_{\BD}$, the cyclic orders or the total orders of ss-separatrices around it, and the spiral structure of any element of $E_{\BD}$ can be reconstructed by the labels $l_{\BD}$ and $L^{\mathrm{BD}(v)}$. 
Precisely, the label $l_{\BD}$ reconstructs the topological type of vertices in $V_{\BD}$ (i.e. a singular point, a periodic orbit in $\partial S$, a one-sided periodic orbit outside of $\partial S$, a two-sided periodic orbit, a non-periodic limit circuit). 

Let $\sigma \in V_{\BD}$ be an element with $l_{\BD}(\sigma) = 0$. 
If there is no canonical domain whose vertical boundary contains $\sigma$, 
 then $\sigma$ is a center. 
In other words, each center can be reconstructed by labels $l_{\BD}$ and $\partial_{\perp}^{\bullet} = (\partial_{\perp}^{\alpha}/\hat{v}, \partial_{\perp}^{\omega}/\hat{v})$. 
If there is a canonical domain whose $\alpha$-vertical boundary contains $\sigma$, then $\sigma$ is a source or a $\partial$-source. 
Since the labels $\partial_{\perp}/\hat{v}, \leq_{\partial_{\perp}^{\alpha}}$, and $\leq_{\partial_{\perp}^{\omega}}$ distinguish cyclic (resp. total) orders of ss-separatrices around any singular point, each source (resp. $\partial$-source) with the cyclic (resp. total) order of ss-separatrices can be reconstructed by labels $l_{\BD}, \partial_{\pitchfork}^{\alpha}/\hat{v}$, and $\partial_{\perp}^{\alpha}/\hat{v}$. 
Similarly, each sink (resp. $\partial$-sink) with the cyclic (resp. total) order of ss-separatrices can be reconstructed by labels $l_{\BD}, \partial_{\pitchfork}^{\omega}/\hat{v}$ and $\partial_{\perp}^{\omega}/\hat{v}$. 

Let $\gamma \in V_{\BD}$ be an element with $l_{\BD}(\gamma) := 1_{\partial}$. 
Note that the local orientation $\mathcal{O}$ is necessary to determine left and right directions and that the spiral direction can be reconstructed by the labels $\partial_{\perp}^{\bullet}$ and $\mathcal{D}_{\partial_{\perp}}$. 
If there is a canonical domain whose $\alpha$-vertical (resp. $\omega$-vertical) boundary contains $\gamma$, then $\gamma$ is a semi-repelling (resp. semi-attracting) limit cycle. 
If there is a canonical domain whose transverse boundary contains $\gamma$, then $\gamma \subseteq \partial_{\mathop{\mathrm{Per}}(v)}$. 
Moreover, each periodic orbit contained in $\partial S$ with the cyclic order of ss-separatrices can be reconstructed by labels $l_{\BD}, \mathcal{O}, \mathcal{D}_{\partial_{\perp}}, \partial^{\BD}$, and $\leq^{\BD}$ (i.e. all labels except $l^{\BD}$). 
Similarly, each one-sided (resp. two-sided) periodic orbit outside of $\partial S$ with one (resp. two) cyclic order of ss-separatrices can be reconstructed by all labels except $l^{\BD}$. 

Let $\gamma \in V_{\BD}$ be an element with $l_{\BD}(\gamma) := 1_{\mathrm{ms}}$. 
The abstract multi-graph structure of $\gamma$ is reconstructed by the abstract multi-graph $G_{D(v)}$ with the natural projection $l_{D(v)}$ and the orders $\leq^{\BD}$ on $\gamma$. 
Since each multi-saddle connection is a finite disjoint union of multi-saddles and multi-saddle separatrices, the multi-saddle connection $\gamma$ can be reconstructed by all labels except $l^{\BD}$ as an embedded subset. 
\end{proof}

\subsection{Proof of Theorem~\ref{injection}}

We demonstrate Theorem~\ref{injection} as follows. 

\begin{proof}[Proof of Theorem~\ref{injection}]
Lemma~\ref{reconstruction-order} implies that $\BD$ can be reconstructed. 
Since $S - \BD$ consists of finitely many invariant trivial flow boxes, periodic tori, periodic annuli, and transverse annuli, the label $l^{\BD}$ determines the type of four kinds of canonical domains. 
Also the local orientation $\mathcal{O}$ with the orbit direction $\partial^{\mathrm{BD}(v)}$ and the spiral direction $\mathcal{D}_{\partial_{\perp}}$ determines flow directions of regions. 
This means that we can reconstruct a flow uniquely. 
\end{proof}

Note that the label $\mathcal{L}^{\mathrm{BD}(v)}$ can be replaced by $L^{\BD}$ in the previous theorem if the compact surface $S$ is orientable. 
Notice that each of the conditions of ``finite type'' (i.e. quasi-regularity, finite conditions of limit cycles, non-existence of Q-sets) is necessary to construct a complete invariant using finite data.
Indeed, the finite condition of limit cycles is necessary for finite descriptions.
As mentioned, the set of topological equivalence classes of minimal flows (resp. Denjoy flows) on a torus is uncountable.
Similarly, the singular point set is uncountable up to locally topological equivalence.
In fact, singular points with infinitely parabolic sectors as shown in Figure~\ref{sectors} form an uncountable subset because we can choose the directions of infinitely many sectors as shown in Figure~\ref{sing}.

We would like to know whether the finiteness of limit cycles can be removed by using countable sequences of labels. 

\begin{figure}
\begin{center}
\includegraphics[scale=0.25]{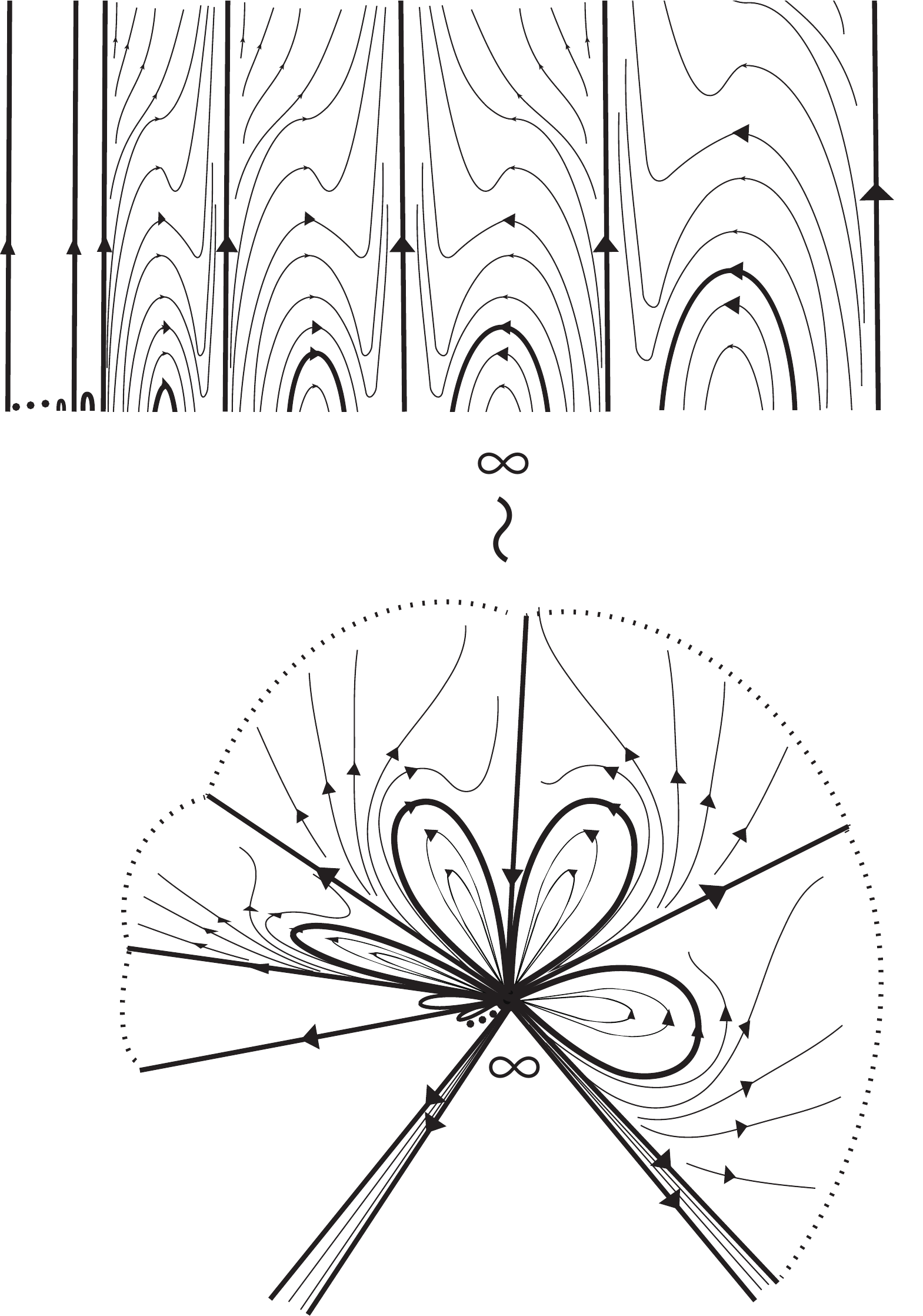}
\end{center}
\caption{A degenerate singular point with infinitely many parabolic sectors}
\label{sing}
\end{figure}

\section{On the ``Markus-Neumann theorem''}

As mentioned above, Buend{\'\i}a and L\'opez have constructed counterexamples to the ``Markus-Neumann theorem''. 
In fact, one of counterexamples (cf. the figure to the right in \cite[Figure 2]{buendia2018markus}) is a toral flow with one singular point and non-recurrent orbits such that a small neighborhood of the singular point consists of two hyperbolic sectors and two parabolic sectors (see Figure~\ref{example-sectors}), which are not quasi-regular but has only finitely sectored singular points (see Definition~\ref{def:fin_sec} for details). 
\begin{figure}
\begin{center}
\includegraphics[scale=0.2]{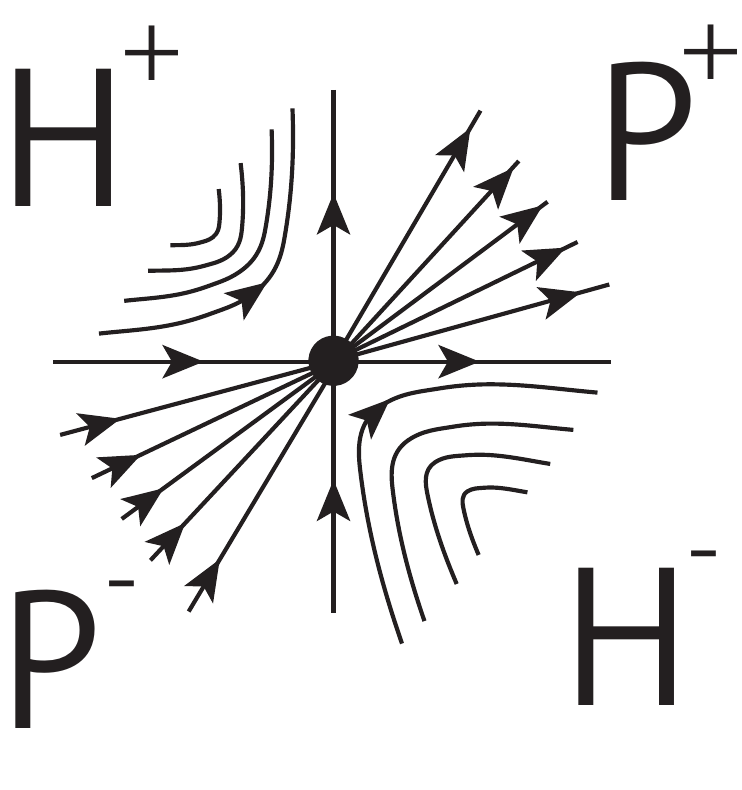}
\end{center}
\caption{A singular point whose neighborhood consists of two parabolic sectors and two hyperbolic sectors}
\label{example-sectors}
\end{figure}
Moreover, an example of a pair of flows consisting of non-recurrent orbits in a plane in \cite[Figure 3]{neumann1975classification} by Neumann is also a counterexample (see Lemma~\ref{lem:non_ex_noo_01}). 
In addition, ``Markus-Neumann theorem'' does not hold for the non-orientable original setting (see Lemma~\ref{lem:non_ex_noo_02}). 
However, we demonstrate that orbit complex is a complete invariant for flows of finite type on compact surfaces. 

\subsection{On the ``Markus-Neumann theorem'' for flows of finite type}

Let $v$ be a flow on a compact surface $S$. 
We call that an invariant subset $U \subseteq S$ is {\bf parallel} if $v\vert_U$ is topologically equivalent to either a canonical domain (i.e. an invariant trivial flow box $\mathbb{D}$, an invariant periodic annulus $\A_+$, an invariant transverse annulus $\A_{\widetilde{+}}$) as Figure~\ref{flow-boxes}, or a periodic torus $\mathbb{T}$. 
Following \cite{buendia2018markus}, we call that a point $x$ is {\bf ordinary} if it has a parallel neighborhood $U$ such that $\alpha'(x) = \alpha'(y)$ and $\omega'(x) = \omega'(y)$ for any point $y \in U$ and that there are exactly two orbits $O_- \neq O_+$ with $\alpha'(x) = \alpha'(O_-) = \alpha'(O_+)$ and $\omega'(x) = \omega'(O_-) = \omega'(O_+)$ such that $\partial U = \alpha'(x) \cup \omega'(x) \cup O_- \cup O_+$. 
A point is {\bf non-ordinary} if it is not ordinary. 
An orbit is {\bf ordinary} (resp. {\bf non-ordinary}) if so are its points. 
Note that a non-ordinary orbit is called a separatrix in the sense of Markus in papers \cite{markus1954global,neumann1975classification,neumann1976global,buendia2018markus}. 
Denote by $\bm{\Sigma_{\mathrm{Sep}}}$ the set of non-ordinary points. 

\subsubsection{Coincidence between $\mathop{\mathrm{BD}}(v)$ and  $\Sigma_{\mathrm{Sep}}$ for a flow $v$ of finite type on a compact surface}

We show that the set of non-ordinary points corresponds to the border point set $\mathop{\mathrm{BD}}(v)$ for a flow of finite type on a compact surface. 

\begin{lemma} 
Let $v$ be a flow of finite type on a compact surface $S$. 
Then $\mathop{\mathrm{BD}}(v) = \Sigma_{\mathrm{Sep}}$. 
\end{lemma}

\begin{proof}
Theorem~\ref{homogeneity} implies that each orbit in $(S - \mathop{\mathrm{BD}}(v))$ has a parallel invariant neighborhood satisfying the conditions in the definition of ``ordinary'' and so $\mathop{\mathrm{BD}}(v) \supseteq \Sigma_{\mathrm{Sep}}$. 
On the other hand, we show that $\mathop{\mathrm{BD}}(v) \subseteq \Sigma_{\mathrm{Sep}}$. 
Indeed, recall that $\mathop{\mathrm{BD}}(v) = \mathop{\mathrm{Bd}}(v) \sqcup \mathop{\mathrm{Per}_{1}}(v)$.
Lemma~\ref{lem010} implies that $\mathop{\mathrm{BD}}(v) = \mathop{\mathrm{Bd}}(v) \sqcup \mathop{\mathrm{Per}_{1}}(v) = \mathop{\mathrm{Sing}}(v) \sqcup \partial^{-} \mathop{\mathrm{Per}}(v) \sqcup \partial_{\mathop{\mathrm{Per}}(v)} \sqcup \mathop{\mathrm{Per}_{1}}(v) \sqcup \partial^{-} \mathrm{P}(v) \sqcup  \mathop{\mathrm{P}_{\mathrm{sep}}}(v)$ consists of finitely many orbits. 
Since each singular point has no parallel neighborhoods, the singular point set $\mathop{\mathrm{Sing}}(v)$ consists of non-ordinary points. 
Because each orbit in $\partial^{-} \mathop{\mathrm{Per}}(v)$ is a limit cycle $\gamma$, any its neighborhood $U$ intersects an non-recurrent orbit $O$ with either $\alpha(O) = \alpha'(O) = \gamma \neq \emptyset = \alpha'(\gamma)$ or $\omega(O) = \omega'(O) = \gamma \neq \emptyset = \omega'(\gamma)$, and so the union $\partial^{-} \mathop{\mathrm{Per}}(v)$ consists of non-ordinary points.  
By Proposition~\ref{prop:cs}, the union $\partial^{-} \mathrm{P}(v)$ consists of connecting separatrices.
From Lemma~\ref{lem:semi-multi-saddles}, we have that $\mathop{\mathrm{P}_{\mathrm{sep}}}(v) = \mathop{\mathrm{P}_{\mathrm{ms}}(v)} \sqcup \mathop{\mathrm{P}_{\mathrm{ss}}(v)} \sqcup \mathop{\partial_{\mathrm{P}(v)}}$ is the union of semi-multi-saddle separatrices in the interiror $\mathop{\mathrm{int}}\mathrm{P}(v)$. 
Fix a non-recurrent orbit $O$ whose $\omega$-limit set is a multi-saddle $x$. 
Since each multi-saddle has at most finitely many separatrices, we have $\omega'(O) = \{ x \} \neq \omega'(O')$ for all but finitely many orbits $O'$ in a neighborhood of $O$. 
By symmetry, we have $\mathop{\mathrm{P}_{\mathrm{ms}}(v)} \sqcup \mathop{\mathrm{P}_{\mathrm{ss}}(v)} \sqcup \partial^{-} \mathrm{P}(v) \subset \Sigma_{\mathrm{Sep}}$. 
Since any orbit in $\mathop{\partial_{\mathrm{P}(v)}} \sqcup \partial_{\mathop{\mathrm{Per}}(v)} \sqcup \mathop{\mathrm{Per}_{1}}(v)$ has no parallel neighborhoods, we obtain $\mathop{\partial_{\mathrm{P}(v)}} \sqcup \partial_{\mathop{\mathrm{Per}}(v)} \sqcup \mathop{\mathrm{Per}_{1}}(v) \subseteq \Sigma_{\mathrm{Sep}}$ and so $\mathop{\mathrm{BD}}(v) \subseteq \Sigma_{\mathrm{Sep}}$. 
\end{proof}

\subsubsection{Orders induced by flows}

Define 
a pre-order $\leq_{\alpha}$ (resp. $\leq_{\omega}$) on $S$ as follows: 
$$ x \leq_{\alpha} y \text{ if } \text{either } x \in \alpha(y) \text{ or } \hat{x} = \hat{y}$$
$$ x \leq_{\omega} y  \text{ if } \text{either } x \in \omega(y) \text{ or } \hat{x} = \hat{y} $$
Moreover, define a pre-order $\bm{\leq_{\alpha}}$ (resp. $\bm{\leq_{\omega}}$) on the power set $2^S$ as follows: 
$$ A \leq_{\alpha} B \text{ if } \text{either } A = B \text{ or } a \leq_{\alpha} b \text{ for any } a \in A, b \in B$$
$$ A \leq_{\omega} B  \text{ if } \text{either } A = B \text{ or } a \leq_{\omega} b \text{ for any } a \in A, b \in B$$
Define a pre-order $\bm{\leq_{\pitchfork}}$ for invariant subsets as follows: 
$$ A \leq_{\pitchfork} B \text{ if } \text{either } A \subseteq \partial_{\pitchfork} B \text{ or } A = B$$

\subsubsection{Orbit complex}\label{sec:ord_cpx}

Recall that the {\bf orbit complex} is an orbit space with cell structure, fiber type, order structure, and orientability of spirals (see \cite[\S 5]{neumann1976global} for details). 
Let $\pi_v : S \to S/v$ be the natural projection. 
A {\bf $\bm{1}$-cell} is the image $\pi_v(U)$ of a canonical domain $U$ and a {\bf $\bm{0}$-cell} is the image $\pi_v(O)$ of an orbit $O$ in $\mathrm{BD}(v)$. 
In other words, a $1$-cell is a connected component of $\pi_v(S - \mathrm{BD}(v)) = (S - \mathrm{BD}(v))/v$ and the set of $0$-cells is $\pi_v(\mathrm{BD}(v)) = \mathrm{BD}(v)/v$. 
A $1$-cell which is an open interval (resp. a circle) is called an {\bf open (resp. closed) 1-cell}. 
Notice that an open interval (resp. a circle) structure of a $1$-cell is embedded in the orbit space $S/v$. 
The {\bf fiber type} of a $i$-cell is either an open interval, a circle, or a point according to the topological type of the fiber of $\pi_v$ over a point in the orbit space. 
In particular, the fiber type of a $1$-cell is either an open interval or a circle. 
Notice that $\mathrm{BD}(v)$ is the union of the vertical boundaries and transverse boundaries of canonical domains.
Therefore, for a canonical domain $U$, the {\bf first order structure} is the union of the restrictions $\leq_{\partial_{\perp}^{\alpha}}\vert_{\partial_{\perp}^{\alpha}(U) \setminus \mathop{\mathrm{Sing}}(v)}$, $\geq_{\partial_{\perp}^{\omega}}\vert_{\partial_{\perp}^{\omega}(U) \setminus \mathop{\mathrm{Sing}}(v)}$, $\leq_{\partial_{\pitchfork}^R}\vert_{\partial_{\pitchfork}^{R} U \setminus \mathop{\mathrm{Sing}}(v)}$, and $\leq_{\partial_{\pitchfork}^L}\vert_{\partial_{\pitchfork}^{L} U \setminus \mathop{\mathrm{Sing}}(v)}$, which is either a total order or a cyclic order, where $\geq_{\partial_{\perp}^{\omega}}$ is the opposite order of $\leq_{\partial_{\perp}^{\omega}}$ and the restriction $\leq\vert_{A} := \leq\vert_{A \times A}$ of a subset $A$ on a poset $X$ is the restriction as a binary relation. 
Here the pre-orders are identified as subsets of $S \times S$ because a binary relation on a set $X$ is a subset of $X \times X$. 
The {\bf second order structure}, which is an extension of the first order structure, can be identified with the label $\leq^{\mathrm{BD}(v)}$. 
For a canonical domain $U$, the {\bf third order structure}
 is the union of the restrictions $\leq_{\alpha}\vert_{(\partial_{\perp} U) \times U}$ and $\geq_{\omega}\vert_{U \times \partial_{\perp} U}$. 
The {\bf orientability of spirals} is the label $\mathcal{D}_{\partial_{\perp}}$. 

\subsubsection{Non-existence of non-ordinary orbits for some flows}

As mentioned above, ``Markus-Neumann theorem'' does not hold even for a flow whose orbits are homeomorphic to open intervals. 
In other words, \cite[Theorem~5.2]{markus1954global} and \cite[Lemma]{neumann1975classification} (refered as \cite[Lemma~1]{neumann1976global}) don't hold as follows. 

\begin{lemma}\label{lem:non_ex_noo_01}
Let $v$ {\rm(resp.} $v'${\rm)} be the flow on the plane $\R^2$ on the left {\rm(resp.} right{\rm)} in \cite[Figure 3]{neumann1975classification}. 
Then any connected component of the complement of $\Sigma_{\mathrm{Sep}}$ for $v$ {\rm(resp.} $v'${\rm)} is not parallel.  
\end{lemma}

\begin{proof}
For any orbit of $v$ and $v'$, the $\alpha'$-limit set and the $\omega'$-limit set are the empty set. 
On the other hand, the saturation of any small neighborhood of $b_1$ (resp. $b_2$, $c_1$, $c_2$, $d_1$, $d_2$) is topologically equivalent to an open trivial flow box and so a parallel neighborhood. 
Moreover, the boundary of such a saturation consists of two orbits, each of which is either $b_1$, $b_2$, $c_1$, $c_2$, $d_1$, $d_2$, or a horizontal line in $A \sqcup E$. 
This means that $v$ and $v'$ have no non-ordinary orbits. 
Then $\Sigma_{\mathrm{Sep}} = \emptyset$ for $v$ (resp. $v'$). 
Therefore, any connected component of the complement of $\Sigma_{\mathrm{Sep}}$ for $v$ (resp. $v'$) is $\R^2$ and so not parallel. 
\end{proof}

By the previous lemma, the cell structure, the order structure, and the orientability of spirals are not well-defined in the original way by Markus and Neumann (cf. \cite[\S~5]{neumann1975classification}). 
Though two flows $v$ and $v'$ in the previous lemma are not topologically equivalent to each other, notice that the orbit spaces of two flows in the previous lemma are homeomorphic, and so their orbit complexes can not be distinguished. 
Moreover,  ``Markus-Neumann theorem'' does not hold for the non-orientable original setting as follows. 

\begin{lemma}\label{lem:non_ex_noo_02} 
There is a flow $\widetilde{v}$ {\rm(resp.} $\widetilde{v'}${\rm)} on a M{\"o}bius band $\mathbb{M}^2$ such that any connected component of the complement of $\Sigma_{\mathrm{Sep}}$ for $\widetilde{v}$ {\rm(resp.} $\widetilde{v'}${\rm)} is not parallel.  
\end{lemma}

\begin{proof}
By twisting the planes with the flows $v$ and $v'$ in the previous lemma, and by identifying a trivial flow box on $A$ with one on $E$, we can obtain a M{\"o}bius band $\mathbb{M}^2$ with the resulting flows $\widetilde{v}$ and $\widetilde{v'}$ on $\mathbb{M}^2$. 
As the same argument in the proof of the previous lemma, the flows $\widetilde{v}$ and $\widetilde{v'}$ have no non-ordinary orbits. 
Therefore, any connected component of the complement of $\Sigma_{\mathrm{Sep}}$ is $\mathbb{M}^2$ and so not parallel. 
\end{proof}

\subsubsection{Completeness of the orbit complex for flows of finite type on orientable compact surfaces}

We demonstrate that the orbit complex is a complete invariant comparing our invariant. 

\begin{theorem}\label{th:NM_re}
The orbit complex for a flow of finite type on an orientable compact surface is a complete invariant. 
\end{theorem}

\begin{proof}
Fix a flow $v$ of finite type on an orientable compact surface. 
Recall that our topological invariant for $v$ consists of three abstract multi-graphs $(G^{\BD}, G_{\BD} = \mathrm{BD}(v)/v_{\mathrm{ex}}, G_{D(v)} = D(v)/v)$ with the labels $L^{\mathrm{BD}(v)}$, $l_{\BD}$,  and $l_{D(v)}$, where $L^{\mathrm{BD}(v)} = ( \mathcal{O}, l^{\BD}, \partial^{\mathrm{BD}(v)}, \leq^{\mathrm{BD}(v)},  \mathcal{D}_{\partial_{\perp}})$. 
The local orientation $\mathcal{O}$ can be induced by the orientation of the surface.  
The orbit space implies the abstract multi-saddle connection $D(v)/v = G_{D(v)}$ with the labels $l_{D(v)}$. 
The $0$-cell structure with fiber type and the third order structure imply the abstract border point set $\mathrm{BD}(v)/v$ and so $G_{\mathrm{BD}(v)} = \mathrm{BD}(v)/v_{\mathrm{ex}}$.  
The $1$-cell structure with fiber type implies the label $l^{\BD}$. 
Since the vertex set $V^{\BD}$ of $G^{\BD}$ consists of the canonical domains of $\BD$ (i.e. the connected components of $S - \BD$) and the edges of $G^{\BD}$ exist if canonical domains are adjacent in the sense of \S~\ref{sec:O}, the orbit space with the cell structure implies the dual graph $G^{\BD} = (V^{\BD}, E^{\BD})$.  
The fiber structure implies the label $l_{\BD}$. 
The second order structure can be identified with the order $\leq^{\mathrm{BD}(v)}$. 
Combining the specialization pre-order for the quotient topology on the orbit space, the orbit structure implies labels $\partial^{\mathrm{BD}(v)}$. 
The orientability of spirals corresponds to the label $\mathcal{D}_{\partial_{\perp}}$. 
This means that the orbit complex implies our complete invariant $(G^{\BD}, G_{\BD}, G_{D(v)}, L^{\mathrm{BD}(v)}, l_{\BD}, l_{D(v)})$. 
\end{proof}

The previous result revives the ``Markus-Neumann theorem'' (i.e. affirmatively answer Problem~\ref{prob:MN}) as follows. 

\subsection{Proof of Theorem~\ref{main:03}}
Theorem~\ref{main:03} is followed from Theorem~\ref{th:NM_re}. 

\section{Applications to Hamiltonian flows and Morse-Smale flows}

We apply our decompositions to Hamiltonian flows and Morse-Smale flows. 

\subsection{Applications to Hamiltonian flows}

Recall that denote by $\bm{\mathop{\mathrm{Sing}_c}(v)}$ the set of centers.  
We have the following properties. 

\begin{corollary}\label{cor:Ham}
The following statements hold for any Hamiltonian flow $v$ with finitely many singular points on a compact surface $S$: 
\\
{\rm(1)} The flow $v$ is of finite type. 
\\
{\rm(2)} The weak border point set $\mathop{\mathrm{Bd}}(v) = D(v) \sqcup \mathop{\mathrm{Sing}_c}(v) \sqcup \partial_{\mathop{\mathrm{Per}}(v)}$ is the finite union of centers, multi-saddles, multi-saddles separatrices, and periodic orbits on the boundary $\partial S$.  
\\
{\rm(3)} The complement $S - \mathop{\mathrm{Bd}}(v)$ is the finite union of invariant open periodic annuli.  
\end{corollary}

\begin{proof}
From the existence of Hamiltonian, the surface $S$ is orientable and contains neither periodic tori,  periodic M{\"o}bius bands nor periodic Klein bottles. 
Moreover, Lemma~\ref{lem002} implies that the open invariant union $\mathrm{LD}(v)$ of locally dense orbits is empty, because of the existence of the Hamiltonian. 
Since any Hamiltonian flow on a surface is volume-preserving and so has no wandering domains, the flow $v$ has no wandering domain and so is non-wandering. 
\cite[Theorem~3]{cobo2010flows} implies that $v$ is quasi-regular. 
The non-existence of wandering domains implies that there are no limit circuits, the interior of the union $\mathrm{P}(v)$ of non-recurrent orbits is empty.  
Lemma~\ref{lem3-06} implies that $\mathrm{E}(v) = \emptyset$. 
This means that $v$ is of finite type and $S = \mathop{\mathrm{Cl}}(v) \sqcup \mathrm{P}(v) = \overline{\mathop{\mathrm{Per}}(v)}$. 

Theorem~\ref{lem0c} implies that each connected component of $S - \mathop{\mathrm{Bd}}(v)$ is an invariant open periodic annulus and that $\mathop{\mathrm{Bd}}(v)$ is the finite union of centers, multi-saddles, multi-saddles separatrices, and periodic orbits on the boundary $\partial S$. 
This means that assertions {\rm(2)}--{\rm(3)} hold. 
\end{proof}

The previous corollary means that any Hamiltonian flow with finitely many singular points on a compact surface can be obtained by gluing the boundaries of finitely many invariant open periodic annuli.  

\subsection{Applications to Morse-Smale flows}

For a closed invariant set $\gamma$,  define the stable manifold $W^s(\gamma) := \{ y \in X \mid \omega(y) \subseteq  \gamma \}$ and the unstable manifold $W^u(\gamma) := \{ y \in X \mid \alpha(y) \subseteq  \gamma \}$.
Recall a $C^r$ Morse-Smale vector field on a closed manifold for any $r \in \Z_{\geq 0}$ as follows. 

\begin{definition}\label{def:MS_without}
a $C^r$ vector field $X$ on a closed manifold is Morse-Smale if the non-wandering set consists of finitely many hyperbolic closed orbits and any point in the union $\bigcup_{O,O': \text{closed orbits}} W^s(O) \cap W^u(O')$ of intersections of the stable and unstable manifolds of closed orbits is transversal. 
\end{definition}

Moreover, under three generic conditions for differentials to guarantee structural stability, Labarca and Pacifico defined a Morse-Smale vector field on a compact manifold as follows \cite{labarca1990stability}.

\begin{definition}\label{def:MS_with}
A $C^\infty$ vector field $X$ on a compact manifold is {\bf Morse-Smale} if $(1)$ the non-wandering set $\Omega(X)$ consists of finitely many hyperbolic closed orbits; $(2)$ the restriction $X\vert_{\partial M}$ is Morse-Smale; $(3)$ for any orbits $O, O' \subset \Omega(X)$ and for any non-transversal point $x \in W^s(O) \cap W^u(O')$, we have $x \in \partial M$ and either $O$ or $O'$ is singular with respect to $X$ (i.e. $O \subseteq \mathop{\mathrm{Sing}}(X)$ or $O' \subseteq \mathop{\mathrm{Sing}}(X)$).
\end{definition}

Therefore, we say that a flow on a compact manifold is {\bf Morse-Smale} if it is topologically equivalent to a flow generated by a Morse-Smale vector field. 

Note that the two generic conditions for differentials form an open dense subset of the set of $C^\infty$ vector fields and that they are stated as follows: any closed orbit is $C^2$ linearizable, and the weakest contraction (resp. expansion) at any closed orbit is defined.
Here the weakest contraction at a singular (resp. periodic) point $p$ is defined if the contractive eigenvalue with the biggest real part among the contractive eigenvalues of $DX(p)$ (resp. $DX_f(p)$, where $X_f$ is the Poincar\'e map) is simple. 
Dually, we can define the weakest expansion at $p$. 

Under the $C^2$ linearizable condition and the eigenvalue conditions, a Morse-Smale $C^\infty$ vector field on a compact manifold is structurally stable with respect to the set of $C^\infty$ vector fields~\cite{labarca1990stability}.
Any Morse-Smale flow on a compact surface can be obtained by gluing the boundaries of finitely many transverse annuli and invariant trivial flow boxes as follows.

\begin{corollary}\label{cor:MS}
The following statements hold for any Morse-Smale flow $v$ on a compact surface $S$: 
\\
{\rm(1)} The flow $v$ is of finite type. 
\\
{\rm(2)} The invariant subset $\mathop{\mathrm{Bd}}(v) = D_{\mathrm{ss}}(v) \sqcup \mathop{\partial_{\mathrm{P}(v)}}$ is the finite union of sources, $\partial$-sources, sinks, $\partial$-sinks, saddles, $\partial$-saddles, limit cycles, proper semi-multi-saddle separatrices, and sepatarices from $\partial$-souces to $\partial$-sinks.  
\\
{\rm(3)} The complement $S - \mathop{\mathrm{Bd}}(v)$ is the finite union of invariant transverse annuli and invariant trivial flow boxes.  
\end{corollary}

\begin{proof}
By definition, the closed point set $\mathop{\mathrm{Cl}}(v)$ consists of finitely many sources, $\partial$-sources, sinks, $\partial$-sinks, saddles, $\partial$-saddles, and limit cycles and we have $S = \mathop{\mathrm{Cl}}(v) \sqcup \mathrm{P}(v) = \overline{\mathrm{P}(v)}$. 
Therefore, the flow $v$ is of finite type. 
Theorem~\ref{lem0c} implies that each connected component of $S - \mathop{\mathrm{Bd}}(v)$ is either an invariant open transverse annulus or an invariant open trivial flow box. 
From Proposition~\ref{lem11-01}, the finiteness of $\mathop{\mathrm{Bd}}(v)$ and its complement holds. 
The non-degeneracy of singular points implies that the union $\mathop{\mathrm{P}_{\mathrm{semi}}}(v)$ consists of proper semi-multi-saddle separatrices. 
Therefore $\mathop{\mathrm{Bd}}(v) = \mathop{\mathrm{Sing}}(v) \sqcup \partial^{-} \mathop{\mathrm{Per}}(v) \sqcup  \mathop{\mathrm{P}_{\mathrm{semi}}}(v) \sqcup \mathop{\partial_{\mathrm{P}(v)}} = D_{\mathrm{ss}}(v) \sqcup \mathop{\partial_{\mathrm{P}(v)}}$ is desired. 
\end{proof}

\section{Decompositions of flows with no-slip boundary condition}

From physical and differential equation points of view, the no-slip boundary condition (as in Figure~\ref{blowup}) and non-compactness appearing from properties of fluid phenomena on punctured surfaces are fundamental. 
To analyze such fluid phenomena, it is necessary to allow the degeneracy of singularities and non-compactness of surfaces. 
\begin{figure}
\begin{center}
\includegraphics[scale=0.1]{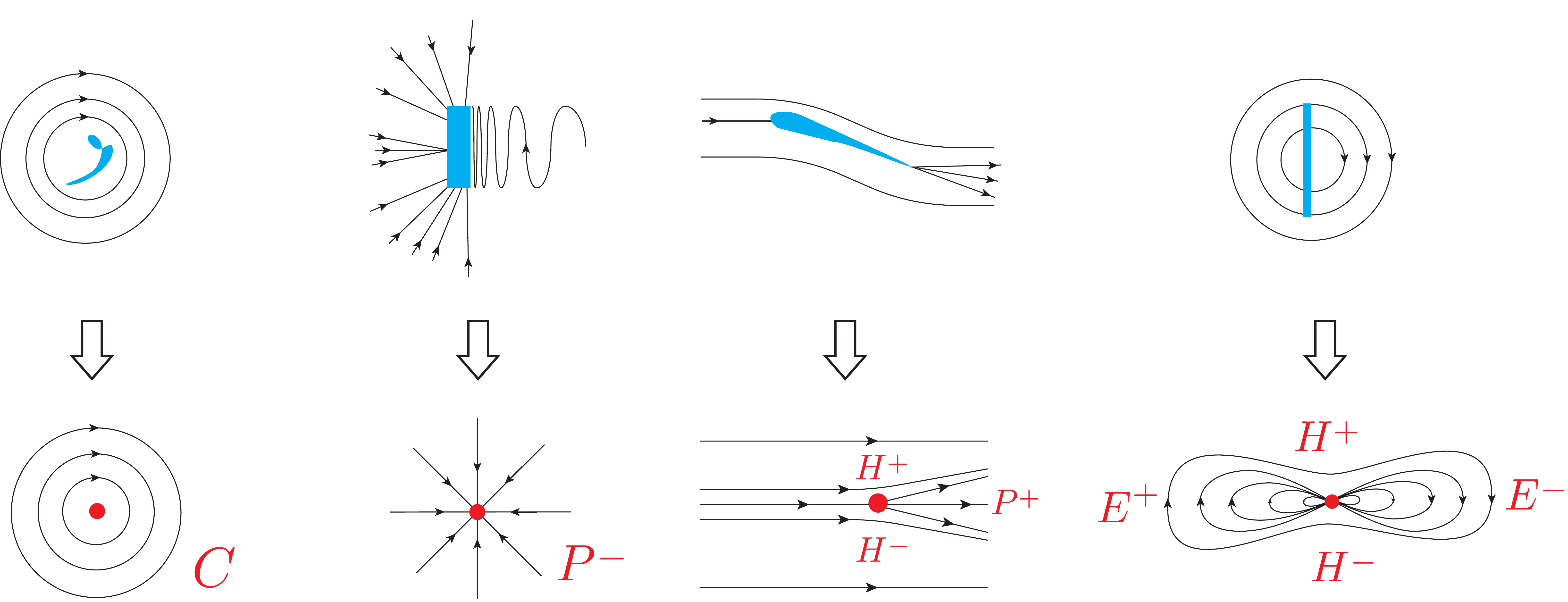}
\end{center}
\caption{Blow-downs of finitely many connected components of the singular point set. Here the symbol $C$ represents a center and the symbols $P^{\pm}$ (resp. $H^{\pm}$, $E^{\pm}$) represent parabolic (resp. hyperbolic, elliptic) sectors (see \S~\ref{sec:sector}).}
\label{blowup}
\end{figure}
To extend our descriptions of surfaces, we define blow-downs of a surface with respect to singular points and their flows as follows (see Figure~\ref{Fig:blowdown}):

\begin{figure}
\[
\xymatrix@=18pt{
S \ar@{}[d]|{\bigcup} & & S_{\mathrm{mc}}\ar[ll]_{\pi_{\mathrm{mc}}} \ar[rr]^q \ar@{}[d]|{\bigcup} & &  S_{\mathrm{col}} \ar@{}[d]|{\bigcup} \\
S - \mathop{\mathrm{Sing}}(v) & &  S_{\mathrm{mc}} - \mathop{\mathrm{Sing}}(v_{\mathrm{mc}}) \ar@{=}[ll]_{\pi_{\mathrm{mc}}|}
 \ar@{=}[rr]^{q|} & &  S_{\mathrm{col}} - \mathop{\mathrm{Sing}}(v_{\mathrm{col}})}
\]
\caption{Canonical quotient mappings induced by the metric completion and the collapse.}
\label{Fig:blowdown}
\end{figure}

\subsection{Blow-down surfaces and blow-down flows}

Let $v$ be a flow on a surface $S$ of finite genus with finitely many boundary components (and with possibly infinitely many ends (see definition of end in \S~\ref{sec:end_completion})).
Fix a Riemannian metric on $S$.
Since the singular point set $\mathop{\mathrm{Sing}}(v)$ is closed, the complement $S - \mathop{\mathrm{Sing}}(v)$ is open and so a surface. 
In particular, the set difference $S - (\mathop{\mathrm{Sing}}(v) \cup \partial S)$ is an open surface without boundary. 

Denote by $\bm{S_{\mathrm{mc}}}$ the {\bf metric completion} of the complement $S - \mathop{\mathrm{Sing}}(v)$. 
Identifying the union $\partial$ of new boundary components with the new singular points, define a flow $\bm{v_{\mathrm{mc}}}$ on $S_{\mathrm{mc}}$ such that $O_v(x) = O_{v_{\mathrm{mc}}}(x)$ for any point $x \in S - \mathop{\mathrm{Sing}}(v) = S_{\mathrm{mc}} - \mathop{\mathrm{Sing}}(v_{\mathrm{mc}})$ up to topological equivalence. 
Then $\partial = \mathop{\mathrm{Sing}}(v_{\mathrm{mc}})$ and so $S - \mathop{\mathrm{Sing}}(v) = S_{\mathrm{mc}} - \mathop{\mathrm{Sing}}(v_{\mathrm{mc}})$. 
By \cite[Theorem~3]{reeb1952certaines}, each connected component of the open surface $S - (\mathop{\mathrm{Sing}}(v) \cup \partial S)$ without boundary is homeomorphic to the resulting surface from a closed surface by removing a closed totally disconnected subset. 

Therefore, collapsing each connected component of $\mathop{\mathrm{Sing}}(v_{\mathrm{mc}})$ into a singular point, we obtain the resulting flow $\bm{v_{\mathrm{col}}}$ with totally disconnected singular points, called the {\bf blow-down flow} of $v$ with respect to singular points, on the resulting surface $S_{\mathrm{col}}$, called the {\bf blow-down surface} with respect to singular points, up to topological equivalence. 

By construction, we have the following observation. 

\begin{lemma}
Let $v$ be a flow on a surface $S$ whose singular point set has at most finitely many connected components, $v_{\mathrm{mc}}$ the resulting flow on the metric completion $S_{\mathrm{mc}}$ of $S - \mathop{\mathrm{Sing}}(v)$, and $v_{\mathrm{col}}$ the blow-down flow on the resulting surface $S_{\mathrm{col}}$. 
The following statements hold: 
\\
{\rm(1)} $S - \mathop{\mathrm{Sing}}(v) = S_{\mathrm{mc}} - \mathop{\mathrm{Sing}}(v_{\mathrm{mc}}) = S_{\mathrm{col}} -  \mathop{\mathrm{Sing}}(v_{\mathrm{col}})$. 
\\
{\rm(2)} For any point $x \in S - \mathop{\mathrm{Sing}}(v)$, we have $O_v(x) = O_{v_{\mathrm{mc}}}(x) = O_{v_{\mathrm{col}}}(x)$. 
\\
{\rm(3)} The singular point set $\mathop{\mathrm{Sing}}(v_{\mathrm{col}})$ of $v_{\mathrm{col}}$ is totally disconnected. 
\\
{\rm(4)} Each connected component of $S_{\mathrm{col}}$ is a compact surface.
\end{lemma}

Notice that $S_{\mathrm{mc}}$ and $S_{\mathrm{col}}$ may have infinitely many connected components. 

\subsubsection{Generalized concepts for constructing decompositions}


Let $\pi_{\mathrm{mc}} \colon S_{\mathrm{mc}} \to S$ be the canonical quotient map and $q \colon S_{\mathrm{mc}} \to S_{\mathrm{col}}$ the canonical quotient map. 

\begin{definition}
A flow on a surface with finitely many connected components of the singular point set is {\bf quasi-regular-like} if the blow-down flow is quasi-regular. 
\end{definition}

\begin{definition}
A boundary component $C$ of the singular point set is a {\bf weak saddle} if it the image $q(\pi_{\mathrm{mc}}^{-1}(C))$ is a multi-saddle. 
\end{definition}

An orbit is an orbit from or to a weak saddle if it is from or to a multi-saddle with respect to the blow-down flow $v_{\mathrm{col}}$. 
%
We define $\mathop{\mathrm{P}_{\mathrm{semi}}}(v)$ for a quasi-regular-like flow on surface as follows. 

\begin{definition}\label{semi_P}
For a quasi-regular-like flow $v$ on a surface, the subset $\bm{\mathop{\mathrm{P}_{\mathrm{semi}}}(v)}$ is defined as the union of non-recurrent orbits from or to weak saddles. 
\end{definition}

Notice that the union $\mathop{\mathrm{P}_{\mathrm{semi}}}(v)$ is the union of proper semi-multi-saddle separatrices if $v$ is quasi-regular, because any weak saddles for a quasi-regular flow are multi-saddles. 
Therefore, definitions of $\mathop{\mathrm{P}_{\mathrm{semi}}}(v)$ in Definition~\ref{semi_P_qr} and Definition~\ref{semi_P} coincide. 

\subsection{Border point sets}

For a quasi-regular-like flow $v$ on a compact surface, put 
\[
\begin{split}
\bm{\mathop{\mathrm{Bd}_0}(v)} &:= \mathop{\mathrm{Sing}}(v) \sqcup \partial^{-} \mathop{\mathrm{Per}}(v) \sqcup \partial_{\mathop{\mathrm{Per}}(v)} \sqcup  \mathop{\partial_{\mathrm{P}(v)}} \sqcup \mathop{\mathrm{P}_{\mathrm{semi}}}(v) \sqcup \mathrm{E}(v)
\\
\bm{\mathop{\mathrm{BD}_0}(v)} &:= \mathop{\mathrm{Bd}_0}(v)  \sqcup \mathop{\mathrm{Per}_{1}}(v)
 \end{split}
\] 

We also define the pre-border set $\bm{\mathop{\mathrm{Pd}_0}(v)}$ as follows: 
\[
\mathop{\mathrm{Pd}_0}(v) := \mathop{\mathrm{Sing}}(v) \sqcup \partial^{-} \mathop{\mathrm{Per}}(v) \sqcup \mathop{\mathrm{P}_{\mathrm{semi}}}(v) \sqcup \mathrm{E}(v)
\]

Then $\mathop{\mathrm{Bd}_0}(v) = \mathop{\mathrm{Pd}_0}(v) \sqcup \partial_{\mathop{\mathrm{Per}}(v)} \sqcup  \mathop{\partial_{\mathrm{P}(v)}}$. 
%
%
By definitions of $\mathop{\mathrm{Pd}_0}(v)$ and $\mathop{\mathrm{BD}}_0(v)$, from Proposition~\ref{cor:bd01}, we have the following observation.

\begin{lemma}\label{lem:bd_cl}
The following stratements hold for any quasi-regular-like flow $v$ on a compact surface $S$: 
\\
{\rm(1)} The restriction $v \vert_{S - \mathop{\mathrm{Pd}_0}(v)}$ of the flow $v$ is topologically equivalent to the restriction $v_{\mathrm{col}} \vert_{S_{\mathrm{col}} - \mathop{\mathrm{bd}}_0(v_{\mathrm{col}})}$ to the blow-down flow $v_{\mathrm{col}}$ on the blow-down surface $S_{\mathrm{col}}$ in the canonical way. 
\\
{\rm(2)} If $v$ is quasi-regular, then $\mathop{\mathrm{Bd}}_0(v) = \mathop{\mathrm{Bd}}(v)$ and $\mathop{\mathrm{BD}}_0(v) = \mathop{\mathrm{BD}}(v)$. 

\end{lemma}

%
%

For a point $z \in S_{\mathrm{col}}$, denote by $\alpha_{v_{\mathrm{col}}} (z)$ the $\alpha$-limit  set of $z$ in $ S_{\mathrm{col}}$ with respect to $v_{\mathrm{col}}$ and by $\omega_{v_{\mathrm{col}}} (z)$ the $\omega$-limit set of $z$ in $ S_{\mathrm{col}}$ with respect to $v_{\mathrm{col}}$. 
We describe the decomposition of a quasi-regular-like flow on a compact surface. 

\begin{theorem}\label{main:a}
Each connected component of $S - \mathop{\mathrm{Pd}_0}(v)$ for a quasi-regular-like flow $v$ on a compact surface $S$ is one of the following seven invariant open subsets exclusively:
\\
$(1)$
A trivial flow box in $\mathrm{P}(v)$, 
\\
$(2)$
An annulus in $\mathrm{P}(v)$, 
\\
$(3)$ 
An annulus in $\mathop{\mathrm{Per}}(v)$, 
\\
$(4)$ 
A torus in $\mathop{\mathrm{Per}}(v)$, 
\\
$(5)$ 
A Klein bottle in $\mathop{\mathrm{Per}}(v)$, 
\\
$(6)$ 
A M{\"o}bius band in $\mathop{\mathrm{Per}}(v)$, 
or
\\
$(7)$ 
An essential subset in $\mathrm{LD}(v)$. 

Moreover, the following statements hold: 
\\
$(a)$ For any connected component $U$ in $\mathrm{P}(v) \setminus \mathop{\mathrm{Pd}_0}(v)$ and for any points $x,y \in U$, we have $\alpha_{v_{\mathrm{col}}}(x) = \alpha_{v_{\mathrm{col}}}(y)$ and $\omega_{v_{\mathrm{col}}}(x) = \omega_{v_{\mathrm{col}}}(y)$. 
\\
$(b)$ For any connected component $U$ in $\mathrm{LD}(v)$ and for any points $x,y \in U$, we have $\overline{O(x)} = \overline{O(y)} = \overline{U} \subseteq S$. 
\\
{\rm(c)} For any connected component $U$ in $\mathrm{LD}(v)$, the boundary $\partial U$  consists of singular points and finitely many orbits from or to subsets of singular points. 
\\
$(d)$ For any connected component $U$ in $S - \mathop{\mathrm{Pd}_0}(v) = S_{\mathrm{col}} - \mathop{\mathrm{Bd}}_0(v_{\mathrm{col}}) \subseteq S_{\mathrm{col}}$, 
any connected component of $\partial_{\mathrm{col}} U$ is a finite union of closed orbits, semi-multi-saddle separatrices, and separatrices on $\partial S$ between $\partial$-sources and  $\partial$-sinks unless the boundary component $\partial_{\mathrm{col}} U$ intersects $\mathrm{E}(v_{\mathrm{col}}) = \mathrm{E}(v)$.
\end{theorem}

\begin{proof}
Let $v$ be a quasi-regular-like flow on a compact surface $S$. 
By Lemma~\ref{lem:bd_cl}, we have $\mathop{\mathrm{Bd}}(v_{\mathrm{col}}) - \mathop{\mathrm{bd}}_0(v_{\mathrm{col}}) = \partial_{\mathop{\mathrm{Per}}(v_{\mathrm{col}})} \sqcup \mathop{\partial_{\mathrm{P}(v_{\mathrm{col}})}}$. 
Therefore Theorem~\ref{main:a} is followed from Theorem~\ref{lem0c}. 
\end{proof}

\subsubsection{Decompositions of flow of finite-like type}

We call that a quasi-regular-like flow is of {\bf finite-like type} if there are at most finitely many limit cycles and each recurrent orbit is closed. 
Then, we have the following statement.

\begin{theorem}\label{main:b}
Each connected component of $S - \mathop{\mathrm{BD}}_0(v)$ for a quasi-regular-like flow $v$ on a compact surface $S$ is one of the following five invariant open subsets exclusively:
\\
$(1)$
A trivial flow box in $\mathrm{P}(v)$, 
\\
$(2)$
An annulus in $\mathrm{P}(v)$, 
\\
$(3)$ 
An annulus in $\mathop{\mathrm{Per}}(v)$, 
\\
$(4)$ 
A torus in $\mathop{\mathrm{Per}}(v)$, 
\\
$(5)$ 
An essential subset in $\mathrm{LD}(v)$. 

Moreover, if $v$ is of finite-like type, then the union $\mathop{\mathrm{BD}}_0(v_{\mathrm{col}})$ is a finite union of singular points, limit cycles, one-sided periodic orbits, and orbits from or to singular points. 
\end{theorem}

\begin{proof}
Let $v$ be a quasi-regular-like flow on a compact surface $S$. 
By Lemma~\ref{lem:bd_cl}(1) and Proposition~\ref{cor019}, the complement $S - \mathop{\mathrm{BD}_0}(v)$ can be identified with the complement $S - \mathop{\mathrm{BD}}_0(v_{\mathrm{col}})$. 
Therefore, the assertion is followed from Theorem~\ref{homogeneity}. 
\end{proof}

\subsubsection{}

We describe the relations among the border point set and relative concepts as follows. 

\begin{proposition}\label{cor019+}
The following statements hold for a quasi-regular flow $v$ with at most finitely many limit cycles on a compact surface $S$:
\\
$(1)$ $\mathop{\mathrm{Pd}_0}(v) = D_{\mathrm{ss}}(v) \sqcup \mathop{\mathrm{Sing}_c}(v)$ is closed. 
\\
$(2)$ 
$\mathop{\mathrm{Bd}}(v) = \mathop{\mathrm{Pd}_0}(v) \sqcup \partial_{\mathop{\mathrm{Per}}(v)} \sqcup \mathop{\partial_{\mathrm{P}(v)}}$ is closed.
\end{proposition}

\begin{proof}
By definition and Proposition~\ref{cor019}, assertion (1) holds. 
Lemma~\ref{lem:bd_cl} implies that assertion {\rm(2)} holds. 
\end{proof}

\section{Decompositions of flows of weakly finite type}

We weaken the regularity(-like) condition for singular points into a finitely sectored condition. 

\subsection{Flows of weakly finite type}\label{sec:sector}

Let $v$ be a flow on a surface $S$. 

\subsubsection{Sectors for flows on surfaces}

Sectors are defined as follows. 

\begin{definition}\label{def:sector}
The restriction $v\vert_A$ for a subset $A \subset S$ is a {\bf sector} for a singular point $x$ if there are a non-degenerate interval $I \subseteq [0, 2 \pi )$ and a homeomorphism $h \colon \{ x\} \sqcup A \to \{ 0 \} \sqcup \{ ( r \cos \theta, r \sin \theta ) \in \R^2 \mid r \in (0,1), \theta \in I \}$ such that $h^{-1}(0) = x$.
\end{definition}

\begin{definition}\label{def:para_sector}
A sector on an open subset $A$ for a singular point $x \in S$ is a {\bf parabolic} if the restriction of $v$ to $\{ x\} \sqcup A$ is locally topologically equivalent to the restriction on $\{ 0 \} \sqcup \{ ( r \cos \theta, r \sin \theta ) \in \R^2 \mid r \in (0,1), \theta \in [0,\pi/2] \}$ of the flow generated by $X(x,y) := (x,y)$ as on the left of Figure~\ref{sectors_all_01}. 
\end{definition}

\begin{definition}\label{def:hyp_sector}
A sector on an open subset $A$ for a singular point $x$ is a {\bf hyperbolic} if the restriction of $v$ to $\{ x\} \sqcup A$ is locally topologically equivalent to the restriction on $\{ 0 \} \sqcup \{ ( r \cos \theta, r \sin \theta ) \in \R^2 \mid r \in (0,1), \theta \in [0,\pi/2] \}$ of the flow generated by  $X(x,y) := (-x, y)$ as on the middle of Figure~\ref{sectors_all_01}. 
\end{definition}

\begin{definition}\label{def:elli_sector}
A sector on an open subset $A$ for a singular point $x$ is {\bf elliptic} if the restriction of $v$ to $\{ x\} \sqcup A$ is locally topologically equivalent to the restriction on $\{ (x,y) \in \R^2 \mid x \geq 0, x^2 + y^2 > 1 \} \sqcup \{ \infty \} \subset \R^2 \sqcup \{ \infty \}$ of the flow generated by a continuous vector field $X$ defined by $X = (0,1)$ on $\R^2$ and $X = 0$ at the point $\infty$ at infinity as on the left of Figure~\ref{sectors_all_01}. 
\end{definition}

Here the union $\R^2 \sqcup \{ \infty \}$ is the one-point compactification of $\R^2$, which is a sphere. 

\begin{figure}
\begin{center}
\includegraphics[scale=0.2]{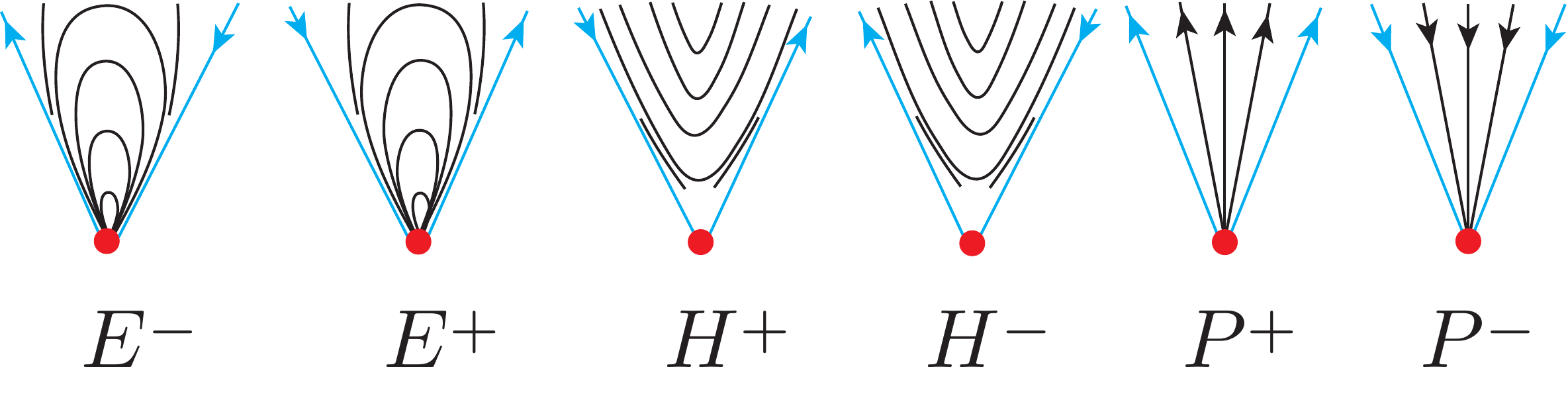}
\end{center}
\caption{Two parabolic sectors $P^-$ and $P^+$, two hyperbolic sectors $H^-$ and $H^+$ with clockwise and anti-clockwise orbit directions, and two elliptic sectors $E^+$ and $E^-$ with clockwise and anti-clockwise orbit directions, respectively.}
\label{sectors_all_01}
\end{figure}

\begin{definition}\label{def:fin_sec}
A singular point in the interior of a surface is {\bf finitely sectored} if either it is a center or there is its open neighborhood which is an open disk and is a finite union of the point, parabolic sectors, hyperbolic sectors, and elliptic sectors such that a pair of distinct sectors intersects at most two orbit arcs. 
\end{definition}

\begin{definition}
A singular point on the boundary of a surface is {\bf finitely sectored} if it is finitely sectored for the resulting flow on the double of the compact surface. 
\end{definition}


Note that a maximal open elliptic sector $U$ is invariant and has the $\omega$-limit set $\{ x \}$ and the $\alpha$-limit set $\{ x \}$ of any point $x$ in $U$ such that the orbit closure $\overline{O(y)} = O(y) \sqcup \{ x \}$ for any point $y \in U$ bounds an invariant trivial flow box in $U$. 

\begin{definition}
For a finitely sectored singular point $x$, a nonempty open parabolic sector is a {\bf maximal open parabolic sector} if it is an open parabolic sector contained in the complement of the union of maximal elliptic open sectors of $x$ and hyperbolic sectors of $x$ which is maximal with respect to the inclusion order. 
\end{definition}


Moreover, the quasi-regularity implies that each singular point is a finitely sectored singular point without elliptic sectors. 
A singular point is a multi-saddle if and only if it is a finitely sectored singular point whose sectors are hyperbolic.
Similarly, a singular point is a sink, a $\partial$-sink, a source, or a $\partial$-source if and only if it is a finitely sectored singular point whose sectors consist of exactly one parabolic sector as in Figure~\ref{Fig:quasiregular}.
In \cite[Theorem~A]{kibkalo2022topological}, any isolated singular points of a gradient vector field on a surface are characterized as non-trivial finitely sectored singular points without elliptic sectors.

\subsubsection{Flows of weakly finite type and border separatrices}

We generalize flows of finite type as follows. 

\begin{definition}
A flow is {\bf of weakly finite type} if any singular point of it is finitely sectored, there are at most finitely many limit cycles, and $\mathrm{R}(v) = \emptyset$.
\end{definition}

We generalize semi-multi-saddle separatrix and separatrices in $\mathop{\partial_{\mathrm{P}(v)}}$ as follows. 

\begin{definition}
A separatrix is a {\bf hyperbolic} {\rm(resp.} {\bf elliptic}, {\bf parabolic}{\rm)} border separatrix if it is contained in the boundary of a hyperbolic sector (resp. a maximal open elliptic sector, a maximal open parabolic sector). 
\end{definition}

Here, each of the boundaries of a hyperbolic sector, a maximal open elliptic sector, and a maximal open parabolic sector is either the topological boundary of the sector as subsets of $S$ or the intersection of the sector and the boundary $\partial S$ of the surface $S$. 

\begin{definition}
A separatrix is a {\bf border separatrix} if it is either hyperbolic, elliptic, or parabolic border separatrix. 
\end{definition}

Denote by $\bm{\mathop{\partial_{\mathrm{h}(v)}}}$ the union of hyperbolic border separatrices, and by $\bm{\mathop{\partial_{\mathrm{e/p}(v)}}}$ the union of elliptic/parabolic border separatrices which are not hyperbolic. 
By definition, note that $\mathop{\partial_{\mathrm{P}(v)}} \subseteq \mathop{\partial_{\mathrm{e/p}(v)}}$.
The union  $\bm{\mathop{\partial_{\mathrm{bs}(v)}}} := \mathop{\partial_{\mathrm{h}(v)}} \sqcup \mathop{\partial_{\mathrm{e/p}(v)}}$ is the union of border separatrices. 
We have the following statement. 

\begin{lemma}\label{lem:border_set_diff}
For any flow $v$ with finitely sectored singular points on a compact surface, we have $\partial^{-} \mathrm{P}(v) \sqcup \mathop{\mathrm{P}_{\mathrm{sep}}}(v) \subseteq \mathop{\partial_{\mathrm{bs}(v)}}$. 
\end{lemma}

\begin{proof}
Bu definitions, we obtain $\mathop{\mathrm{P}_{\mathrm{ms}}}(v) \sqcup \mathop{\mathrm{P}_{\mathrm{ss}}}(v) \subseteq \mathop{\partial_{\mathrm{h}(v)}}$ and $\mathop{\partial_{\mathrm{P}(v)}} \subseteq \mathop{\partial_{\mathrm{e/p}(v)}}$. 
Therefore, we have $\mathop{\mathrm{P}_{\mathrm{sep}}}(v) = \mathop{\mathrm{P}_{\mathrm{ms}}}(v) \sqcup \mathop{\mathrm{P}_{\mathrm{ss}}}(v) \sqcup \mathop{\partial_{\mathrm{P}(v)}} \subseteq \mathop{\partial_{\mathrm{h}(v)}} \sqcup \mathop{\partial_{\mathrm{e/p}(v)}} = \mathop{\partial_{\mathrm{bs}(v)}}$. 
By Proposition~\ref{prop:cs}, the finitely sectored property of singular points implies $\partial^{-} \mathrm{P}(v) \subseteq \mathop{\partial_{\mathrm{bs}(v)}}$. 
\end{proof}

\subsubsection{A generalization of {\rm(}weak{\rm)} border point sets for flows of weakly finite type}

Define $\bm{\mathop{\mathrm{Bd}_+}(v)}$ and $\bm{\mathop{\mathrm{BD}_+}(v)}$ for a flow $v$ with finitely sectored singular points on a compact surface $S$ as follows: 
\[
\begin{split}
\mathop{\mathrm{Bd}_+}(v) &:= \mathop{\mathrm{Sing}}(v) \sqcup \partial^{-} \mathop{\mathrm{Per}}(v) \sqcup \partial_{\mathop{\mathrm{Per}}(v)} \sqcup \mathop{\partial_{\mathrm{bs}(v)}} \sqcup \mathrm{E}(v)
\\
 \mathop{\mathrm{BD}_+}(v) &:= \mathop{\mathrm{Bd}_+}(v)  \sqcup \mathop{\mathrm{Per}_{1}}(v)
 \end{split}
 \]

By Lemma~\ref{lem010} and Lemma~\ref{lem:border_set_diff}, we have $\mathop{\mathrm{Bd}}(v) \subseteq \mathop{\mathrm{Bd}_+}(v)$ and $\mathop{\mathrm{BD}}(v) \subseteq \mathop{\mathrm{BD}_+}(v)$. 
Moreover, we have the following relations. 

\begin{lemma}
The following statements hold for any quasi-regular flow $v$ on a compact surface $S$: 
\\
{\rm(1)} The invariant subset $\mathop{\partial_{\mathrm{bs}(v)}} = \partial^{-} \mathrm{P}(v) \sqcup \mathop{\mathrm{P}_{\mathrm{sep}}}(v)$ is the finite union of proper semi-multi-saddle separatrices and separatrices on $\partial S$ between $\partial$-sources and $\partial$-sinks. 
\\
{\rm(2)} $\mathop{\mathrm{Bd}_+}(v) = \mathop{\mathrm{Bd}}(v)$. 
\\
{\rm(3)} $\mathop{\mathrm{BD}_+}(v) = \mathop{\mathrm{BD}}(v)$. 
\end{lemma}

\begin{proof}
Lemma~\ref{lem:border_set_diff} implies that $\partial^{-} \mathrm{P}(v) \sqcup \mathop{\mathrm{P}_{\mathrm{sep}}}(v) \subseteq \mathop{\partial_{\mathrm{bs}(v)}}$. 

\begin{claim}\label{claim:046}
$\mathop{\partial_{\mathrm{bs}(v)}}  \subseteq \partial^{-} \mathrm{P}(v) \sqcup \mathop{\mathrm{P}_{\mathrm{sep}}}(v)$. 
\end{claim}

\begin{proof}[Proof of Claim~\ref{claim:046}]
The quasi-regularity implies that any singular points are either multi-saddles, sinks, $\partial$-sinks, sources, $\partial$-sources, or centers. 
Therefore, any border separatrices connect among multi-saddles, $\partial$-sinks, and $\partial$-sources. 
This implies that each border separatrix is either a semi-multi-saddle separatrix or a separatrix on $\partial S$ between a$\partial$-source and a $\partial$-sink. 
Lemma~\ref{lem010}(6) completes the proof of the claim.
\end{proof}

Therefore, assertion (1) holds. 
By definitions of $\mathop{\mathrm{Bd}_+}(v)$, $\mathop{\mathrm{Bd}}(v)$, $\mathop{\mathrm{BD}_+}(v)$ and $\mathop{\mathrm{BD}}(v)$, assertions (2) and (3) hold. 
\end{proof}


\subsection{Deformations of quasi-regular flows into flows with finitely sectored singular points}

We can deform any flow with sectored singular points on a compact surface into a quasi-regular flow by finitely many operations which increase at most finitely many non-recurrent orbits but preserves both the complements of $\mathop{\mathrm{Bd}_{+}}$ and $\mathop{\mathrm{BD}}_{+}$. 
\begin{figure}[t]
\begin{center}
\includegraphics[scale=0.2]{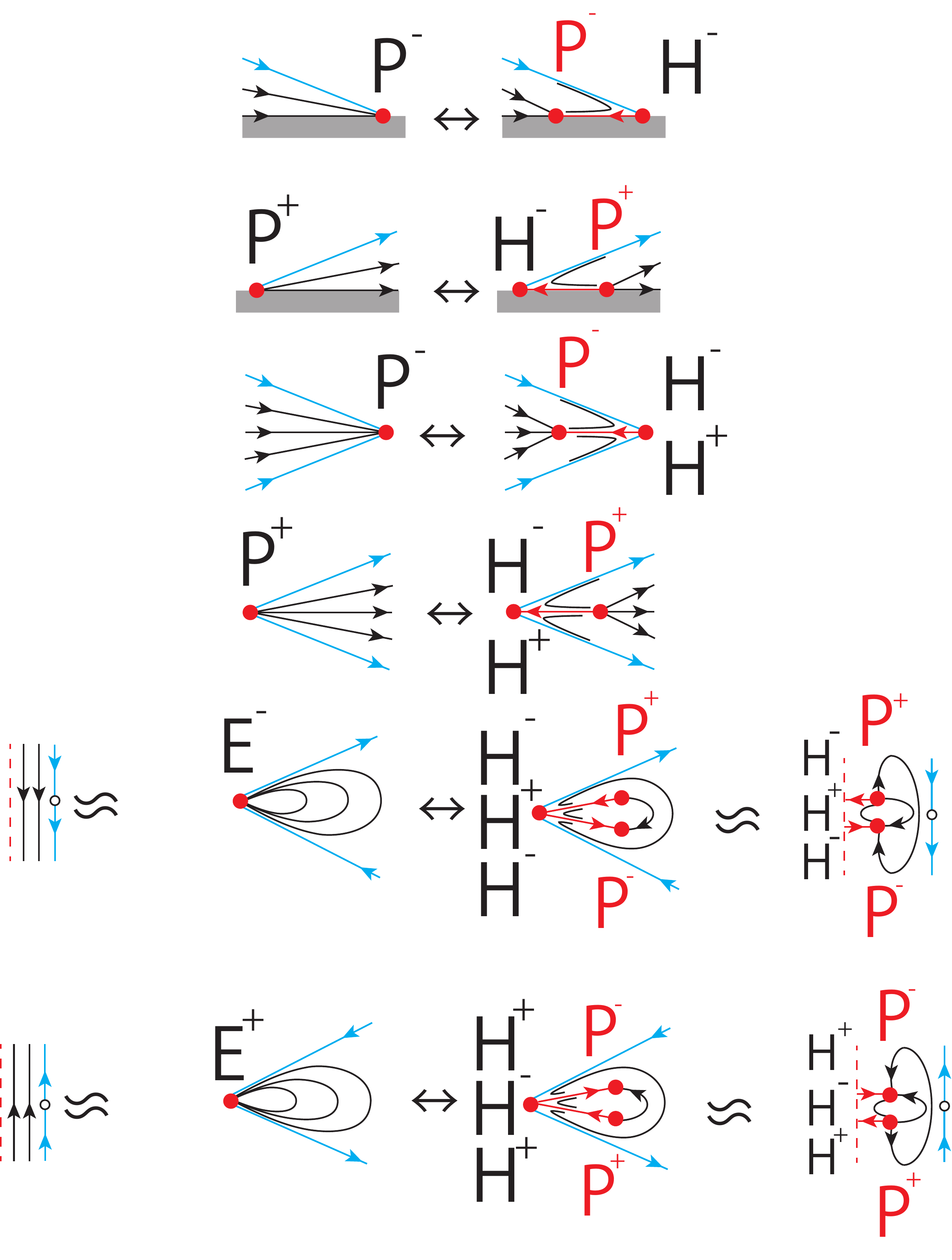}
\end{center}
\caption{Upper two columns: Two deformations of parabolic sectors on the boundary into a hyperbolic sector and a $\partial$-sink/$\partial$-source; Middle two columns: Two deformations of parabolic sectors outside of the boundary into two hyperbolic sectors, a sink/source, and a separatrix connecting singular points; Bottom two columns: Two deformations of maximal elliptic sectors into a hyperbolic sector, a source, a parabolic sector, and a separatrix connecting singular points}
\label{sectors}
\end{figure}

\subsubsection{Blow-up operations of sectors and their blow-down mapping with respect sectors}

We have the following observation about parabolic sectors. 

\begin{lemma}
Let $v_\gamma$ be a flow on a surface $S$ with a hyperbolic border separatrix $\gamma$ as in four on the top of the right in Figure~\ref{sectors}. 
Collapsing the closure $\overline{\gamma}$ into a singleton $x_\gamma$ implies a mapping $p_\gamma \colon S \to S$ which satisfies the following properties: 
\\
{\rm(1)} The singular point $x_\gamma$ has a parabolic sector. 
\\
{\rm(2)} The restriction $p_\gamma\vert_{S - \overline{\gamma}} \colon S - \overline{\gamma} \to S - \{ x_\gamma \}$ can be identified with the identical map. 
\\
{\rm(3)} The restriction $v_\gamma \vert_{S - \overline{\gamma}}$ is topologically equivalent to the restriction $v\vert_{S - \{ x_\gamma \}}$, where $v$ is the resulting flow by the collapsing operation. 
\\
{\rm(4)} $\mathop{\mathrm{BD}_+}(v) - \{ x_\gamma \} = \mathop{\mathrm{BD}_+}(v_\gamma) - \overline{\gamma}$. 
\\
{\rm(5)} The resulting flow has one parabolic or elliptic sector less than $v$, except for the number of sinks and sources.
\end{lemma}

Then the surjection $p_\gamma$ is called a {\bf blow-down mapping} to $\gamma$.  
The replacement of the singular point $x_\gamma$ into the interval $\overline{\gamma}$ by $p_\gamma^{-1}$ is called {\bf blow-up operation of a parabolic sector}. 
Moreover, the flow $v_\gamma$ is called the {\bf blow-up flow} of $v$ with respect to a parabolic sector, the flow $v$ is called the {\bf blow-down flow} of $v_\gamma$ with respect to a parabolic sector. 
In addition, the interval $\overline{\gamma}$ is called the {\bf blow-up} of the singular point $x_\gamma$ with a parabolic sector.
Notice that the closure $\overline{\gamma}$ consists of a singular point with hyperbolic sectors, a hyperbolic border separatrix, and a $\partial$-sink (resp. $\partial$-source, sink, source) as in first (resp. second, third, fourth) from the top in the right. 

Similarly, we have the following observation about elliptic sectors. 

\begin{lemma}
Let $v_{\gamma_\pm}$ be a flow on a surface $S$ with two hyperbolic border separatrices $\gamma_-$ and $\gamma_+$ as in two on the bottom of the right in Figure~\ref{sectors}. 
Collapsing the closure $\overline{\gamma_- \sqcup \gamma_+}$ into a singleton $x_{\gamma_\pm}$ implies a mapping $p_{\gamma_\pm} \colon S \to S$ which satisfies the following properties: 
\\
{\rm(1)} The singular point $x_{\gamma_\pm}$ has an elliptic sector. 
\\
{\rm(2)} The restriction $p_{\gamma_\pm}\vert_{S - \overline{\gamma_- \sqcup \gamma_+}} \colon S - \overline{\gamma_- \sqcup \gamma_+} \to S - \{ x_{\gamma_\pm} \}$ can be identified with the identical map. 
\\
{\rm(3)} The restriction $v_{\gamma_\pm} \vert_{S - \overline{\gamma_- \sqcup \gamma_+}}$ of the flow $v_{\gamma_\pm}$ is topologically equivalent to the restriction $v\vert_{S - \{ x_{\gamma_\pm} \}}$ of the resulting flow $v$ by the collapsing operation. 
\\
{\rm(4)} $\mathop{\mathrm{BD}_+}(v) - \{ x_{\gamma_\pm} \} = \mathop{\mathrm{BD}_+}(v_{\gamma_\pm}) - \overline{\gamma_- \sqcup \gamma_+}$. 
\\
{\rm(5)} The resulting flow has one parabolic or elliptic sector less than $v$, except for the number of sinks and sources.
\end{lemma}

Then the surjection $p_{\gamma_\pm}$ is called a {\bf blow-down mapping} to $\gamma_- \sqcup \gamma_+$.  
The replacement of the singular point $x_\gamma$ into the interval $\overline{\gamma}$ by $p_{\gamma_\pm}^{-1}$ is called {\bf blow-up operation of an elliptic sector}. 
Moreover, the flow $v_{\gamma_\pm}$ is called the {\bf blow-up flow} of $v$ with respect to an elliptic sector, the flow $v$ is called the {\bf blow-down flow} of $v_{\gamma_\pm}$ with respect to an elliptic sector. 
In addition, the interval $\overline{\gamma_\pm}$ is called the {\bf blow-up} of the singular point $x_{\gamma_\pm}$ with an elliptic sector.
Notice that the closure $\overline{\gamma_- \sqcup \gamma_+}$ consists of a singular point with at least three hyperbolic sectors, two hyperbolic border separatrix, a sink, and a source as in bottom two on the right. 

Taking finitely many lifts of finitely sectored singular points, we have the following blow-up operation. 

\begin{lemma}\label{lem:blow_up}
For any flow $v$ with finitely sectored singular points on a compact surface $S$, 
applying finitely many blow-up operations of maximal elliptic sectors as possible, and applying finitely many blow-up operations of parabolic sectors as possible, the following statements hold for the resulting flow $w$: 
\\
{\rm(1)} The resulting flow $w$ is a quasi-regular flow on $S$. 
\\
{\rm(2)} $\Gamma \subseteq \mathop{\mathrm{BD}_+}(w)$, where $\Gamma$ is the union of blow-ups of finitely sectored singular points. 
\\
{\rm(3)} The restriction $w\vert_{S - \Gamma}$ is topologically equivalent to the restriction $v\vert_{S - \Sv}$ in the canonical way. 
\end{lemma}

The resulting quasi-regular flow $w$ is called the {\bf blow-up flow} of $v$ with respect to sectors. 
Replacing $\mathop{\mathrm{BD}}(v)$ with $\mathop{\mathrm{BD}_+}(v)$, note that we can obtain a similar injection as in Theorem~\ref{injection} for flows of weakly finite type. 

\section{Decompositions of flows of weakly finite-like type}

We weaken the regularity-like condition and weak finite type condition for singular points into finitely sectored-like condition. 

\subsection{Flows of weakly finite-like type}


\begin{definition}
A flow $v$ is {\bf of weakly finite-like type} if the blow-down flow $v_{\mathrm{col}}$ with respect to singular points is of weakly finite type. 
\end{definition}

\begin{definition}
A flow $v$ has {\bf sectored-like singular points} if each singular point of the blow-down flow $v_{\mathrm{col}}$ with respect to singular points is finitely sectored. 
\end{definition}

Then a flow is of weakly finite-like type if and only if it has sectored-like singular points, there are at most finitely many limit cycles, and any recurrent orbits are closed.

\subsubsection{Border-like separatrices}

The border separatrices are generalized as follows. 

\begin{definition}
A separatrix is a {\bf hyperbolic} (resp. {\bf elliptic}, {\bf parabolic}) {\bf border-like separatrix} of a flow $v$ on a surface $S$ if it is contained in the boundary of a hyperbolic sector (resp. a maximal open elliptic sector, a maximal open parabolic sector) as a subset of $S_{\mathrm{col}}$. 
\end{definition}

\begin{definition}
A separatrix is a {\bf border-like separatrix} of a flow $v$ on a surface $S$ if it is either a hyperbolic, elliptic, or parabolic border-like separatrix. 
\end{definition}

Denote by $\bm{\mathop{\partial_{\mathrm{bl}(v)}}}$ the union of border-like separatrices. 
By definition, we obtain that $(\mathop{\partial_{\mathrm{P}(v)}} \sqcup \mathop{\mathrm{P}_{\mathrm{semi}}}(v)) \cup \mathop{\partial_{\mathrm{bs}(v)}} \subseteq \mathop{\partial_{\mathrm{bl}(v)}}$.

\subsubsection{A generalization of {\rm(}weak{\rm)} border point sets for flows of  sectored-like singular points}

Define $\bm{\mathop{\mathrm{Bd}_{0+}}(v)}$ and $\bm{\mathop{\mathrm{BD}_{0+}}(v)}$ for a flow $v$ with sectored-like singular points on a compact surface $S$ as follows: 
\[
\begin{split}
\bm{\mathop{\mathrm{Bd}_{0+}}(v)} &:= \mathop{\mathrm{Sing}}(v) \sqcup \partial^{-} \mathop{\mathrm{Per}}(v) \sqcup \partial_{\mathop{\mathrm{Per}}(v)} \sqcup \mathop{\partial_{\mathrm{bl}(v)}} \sqcup \mathrm{E}(v)
\\
\bm{\mathop{\mathrm{BD}_{0+}}(v)} &:= \mathop{\mathrm{Bd}_{0+}}(v)  \sqcup \mathop{\mathrm{Per}_{1}}(v)
 \end{split}
 \]

By definition, from $(\mathop{\partial_{\mathrm{P}(v)}} \sqcup \mathop{\mathrm{P}_{\mathrm{semi}}}(v)) \cup \mathop{\partial_{\mathrm{bs}(v)}} \subseteq \mathop{\partial_{\mathrm{bl}(v)}}$, we have the following observation. 

\begin{lemma}\label{lem:correspond_0}
The following statements hold for a flow $v$ with sectored-like singular points on a compact surface $S$: 
\\
{\rm(1)} If $v$ has sectored-like singular points, then $\mathop{\mathrm{Bd}_{0+}}(v) = \mathop{\mathrm{Bd}}_+(v)$. 
\\
{\rm(2)} If $v$ is quasi-regular-like, then $\mathop{\mathrm{Bd}_{0+}}(v) = \mathop{\mathrm{Bd}}_0(v)$. 
\\
{\rm(3)} If $v$ is quasi-regular, then $\mathop{\mathrm{Bd}_{0+}}(v) = \mathop{\mathrm{Bd}}_0(v)  = \mathop{\mathrm{Bd}}_+(v) = \mathop{\mathrm{Bd}}(v)$. 
\end{lemma}

\subsection{Decompositions of flows with sectored-like singular points}

For a point $z \in S_{\mathrm{col}}$, denote by $\alpha_{v_{\mathrm{col}}} (z) \subseteq S_{\mathrm{col}}$ {\rm(resp.} $\omega_{v_{\mathrm{col}}} (z) \subseteq S_{\mathrm{col}}${\rm)} the $\alpha$-limit {\rm(resp.} $\omega$-limit{\rm)} set of a point $z$ with respect to $v_{\mathrm{col}}$. 
By Theorem~\ref{lem0c}, Theorem~\ref{homogeneity}, and Lemma~\ref{lem:correspond_0}, the blow-down flows of quasi-regular flows imply the following main results. 

\begin{theorem}\label{th:bd1_1}
Each connected component of $S - \mathop{\mathrm{Bd}_{0+}}(v)$ for a flow $v$ with sectored-like singular points on a compact surface $S$ is one of the following  invariant open subsets exclusively:
\\
$(1)$
A trivial flow box in $\mathrm{P}(v)$, whose orbit space is an interval,
\\
$(2)$
An annulus in $\mathrm{P}(v)$, whose orbit space is a circle,
\\
$(3)$ 
A torus in $\mathop{\mathrm{Per}}(v)$, whose orbit space is a circle,
\\
$(4)$ 
A Klein bottle in $\mathop{\mathrm{Per}}(v)$, whose orbit space is an interval,
\\
$(5)$ 
An annulus in $\mathop{\mathrm{Per}}(v)$, whose orbit space is an interval,
\\
$(6)$ 
A M{\"o}bius band in $\mathop{\mathrm{Per}}(v)$, whose orbit space is an interval,
or
\\
$(7)$ 
An essential subset in $\mathrm{LD}(v)$, whose orbit class space is a singleton.

Moreover, the following statements hold: 
\\
$(a)$ For any connected component $U$ in $\mathrm{P}(v_{\mathrm{col}}) \setminus \mathop{\mathrm{Bd}}_{+}(v_{\mathrm{col}}) = \mathrm{P}(v) \setminus \mathop{\mathrm{Bd}_{0+}}(v)$ and for any points $x,y \in U$, we have $\alpha_{v_{\mathrm{col}}}(x) = \alpha_{v_{\mathrm{col}}}(y) $ and $\omega_{v_{\mathrm{col}}}(x) = \omega_{v_{\mathrm{col}}}(y)$. 
\\
$(b)$ For any connected component $U$ in $\mathrm{LD}(v)$ and for any points $x,y \in U$, we have $\overline{O(x)} = \overline{O(y)} = \overline{U} \subseteq S$. 
\\
$(c)$ Each of the vertical $\alpha$- and $\omega$-vertical boundaries of any connected components of $S - \mathop{\mathrm{Bd}_{0+}}(v_{\mathrm{col}})$ in the cases $(1)$ and $(2)$ is a finite union of closed orbits, proper semi-multi-saddle separatrices, and separatrices on $\partial S$ between $\partial$-sources and $\partial$-sinks as a subset of $S_{\mathrm{col}}$ unless the vertical boundary intersects $\mathrm{E}(v_{\mathrm{col}}) = \mathrm{E}(v)$. 
\\
$(d)$ Each of the left and right transverse boundaries of any connected components of $S - \mathop{\mathrm{Bd}_{0+}}(v_{\mathrm{col}})$ in the cases $(1)$ and $(4)$ is an immersed line which is a finite union of closed orbits, proper semi-multi-saddle separatrices, and separatrices on $\partial S$ between $\partial$-sources and $\partial$-sinks as a subset of $S_{\mathrm{col}}$. 
\end{theorem}

\begin{theorem}\label{cor:bd1_1}
Each connected component of $S - \mathop{\mathrm{BD}}_{+}(v)$ for a flow $v$ with sectored-like singular points on a compact surface $S$ is one of the following invariant open subsets exclusively:
\\
$(1)$
A trivial flow box in $\mathrm{P}(v)$, whose orbit space is an interval,
\\
$(2)$
An annulus in $\mathrm{P}(v)$, whose orbit space is a circle,
\\
$(3)$ 
A torus in $\mathop{\mathrm{Per}}(v)$, whose orbit space is a circle,
\\
$(4)$ 
An annulus in $\mathop{\mathrm{Per}}(v)$, whose orbit space is an interval, or
\\
$(5)$ 
An essential subset in $\mathrm{LD}(v)$, whose orbit class space is a singleton.
%
\end{theorem}

Notice the similar properties in Theorem~\ref{th:bd1_1}(a)--(d) hold in the previous theorem. 
The previous theorem implies the following description of flows of weakly finite-like type. 

\begin{corollary}\label{cor:bd2}
Each connected component of $S - \mathop{\mathrm{BD}}_{+}(v)$ for a flow $v$ of weakly finite-like type on a compact surface $S$ is one of the following invariant open subsets exclusively:
\\
$(1)$
A trivial flow box in $\mathrm{P}(v)$, whose orbit space is an interval,
\\
$(2)$
An annulus in $\mathrm{P}(v)$, whose orbit space is a circle,
\\
$(3)$ 
A torus in $\mathop{\mathrm{Per}}(v)$, whose orbit space is a circle, or
\\
$(4)$ 
An annulus in $\mathop{\mathrm{Per}}(v)$, whose orbit space is an interval. 

Moreover, the union $\mathop{\mathrm{BD}_+}(v_{\mathrm{col}})$ is the finite union of singular points, limit cycles, one-sided periodic orbits, and border separatrices of $v_{\mathrm{col}}$. 
\end{corollary}

\section{Decompositions of flows of virtual sectored(-like) type on non-compact surfaces}\label{sec:end_completion}

Since some kinds of fluid phenomena are described as flows on non-compact domains, to describe such phenomena, we generalize our results into the case for non-compact surfaces. 
Therefore, we introduce end completions of surfaces with flows. 

\subsection{Virtually border point sets}

Recall the end completion as follows, which is introduced by Freundenthal \cite{Freudenthal1931end}. 

\subsubsection{End completions of surfaces with finite genus} 
Consider the direct system $\{K_\lambda\}$ of compact subsets of a topological space $X$ and inclusion maps such that the interiors of $K_\lambda$ cover $X$.  
There is a corresponding inverse system $\{ \pi_0( X - K_\lambda ) \}$, where $\pi_0(Y)$ denotes the set of connected components of a space $Y$. 
Then the set of ends of $X$ is defined to be the inverse limit of this inverse system. 
Notice that $X$ has one end $x_{\mathcal{U}}$ for each sequence $\mathcal{U} := (U_i)_{i \in \mathbb{Z}_{>0}}$ with $U_i \supseteq U_{i+1}$ such that $U_i$ is a connected component of $X - K_{\lambda_i}$ for some $\lambda_i$. 
Considering the disjoint union $X_{\mathrm{end}}$ of $X$ and  $\{ \pi_0( X - K_\lambda ) \}$ as a set, a subset $V$ of the union $X_{\mathrm{end}}$ is an open \nbd of an end $x_{\mathcal{U}}$ if there is some $i \in \mathbb{Z}_{>0}$ such that $U_i \subseteq V$. 
Then the resulting topological space $\bm{X_{\mathrm{end}}}$ is called {\bf the end completion} (or end compactification) of $X$. 
Note that the end completion is not compact in general. 
From \cite[Theorem~3]{richards1963classification}, any connected surfaces with finite genus are  homeomorphic to the resulting surfaces from closed surfaces by removing closed totally disconnected subsets. 
Therefore, the end compactification $S_{\mathrm{end}}$ of a connected surface $S$ of finite genus is a closed surface. 

For a flow $v$ on a surface $S$ of finite genus, considering ends to be singular points, we obtain the resulting flow $\bm{v_{\mathrm{end}}}$ on the surface $S_{\mathrm{end}}$ which is a union of closed surfaces. 

\subsubsection{Virtually sectored-like singular points and virtual border point sets}
We define some concepts to state decompositions on non-compact surfaces.
 
\begin{definition}
A flow $v$ is a flow with {\bf virtual sectored-like} singular points if $\bm{v_{\mathrm{end}}}$ is a flow with sectored-like singular points. 
\end{definition}

We define $\bm{\mathop{\mathrm{Bd}}_{\widetilde{0+}}(v)}$, $\bm{\mathop{\mathrm{BD}}_{\widetilde{0+}}(v)}$ for a flow $v$ with virtual sectored-like singular points on a surface of finite genus and finite ends as follows: 
\[
\begin{split}
\mathop{\mathrm{Bd}_{\widetilde{0+}}}(v) &:= \mathop{\mathrm{Bd}_{0+}}(v_{\mathrm{end}}) \cap S
\\
\mathop{\mathrm{BD}_{\widetilde{0+}}}(v) &:= \mathop{\mathrm{BD}_{0+}}(v_{\mathrm{end}}) \cap S
\end{split}
\]
We have the following observation.

\begin{lemma}\label{lem:end}
The following statements hold for a flow $v$ with sectored-like singular points on a surface $S$ with finite genus and finite ends:
\\
{\rm(1)} The flow $v$ is a flow with virtual sectored-like singular points if and only if every end becomes a finitely sectored singular point of $v_{\mathrm{end}}$. 
\\
{\rm(2)} If $v$ is a flow with virtual sectored-like singular points, then $v_{\mathrm{end}}$ is  a flow $v$ with sectored-like singular points on the compact surface $S_{\mathrm{end}}$ such that $S - \mathop{\mathrm{Bd}}_{\widetilde{0+}}(v) = S_{\mathrm{end}} - \mathop{\mathrm{Bd}_{0+}}(v_{\mathrm{end}})$ and $S - \mathop{\mathrm{BD}}_{\widetilde{0+}}(v) = S_{\mathrm{end}} - \mathop{\mathrm{BD}_{0+}}(v_{\mathrm{end}})$. 
\end{lemma}

From Theorem~\ref{th:bd1_1}, Theorem~\ref{cor:bd1_1}, and Lemma~\ref{lem:end}, we have the following statements. 

\begin{theorem}\label{th:bd1}
Each connected component of $S - \mathop{\mathrm{Bd}}_{\widetilde{0+}}(v)$ for  a flow $v$ with virtual sectored-like singular points on a surface $S$ with finite genus and finite ends is one of the following  invariant open subsets exclusively:
\\
$(1)$
A trivial flow box in $\mathrm{P}(v)$, whose orbit space is an interval,
\\
$(2)$
An annulus in $\mathrm{P}(v)$, whose orbit space is a circle,
\\
$(3)$ 
A torus in $\mathop{\mathrm{Per}}(v)$, whose orbit space is a circle,
\\
$(4)$ 
A Klein bottle in $\mathop{\mathrm{Per}}(v)$, whose orbit space is an interval,
\\
$(5)$ 
An annulus in $\mathop{\mathrm{Per}}(v)$, whose orbit space is an interval,
\\
$(6)$ 
A M{\"o}bius band in $\mathop{\mathrm{Per}}(v)$, whose orbit space is an interval,
or
\\
$(7)$ 
An essential subset in $\mathrm{LD}(v)$, whose orbit class space is a singleton.

Moreover, the following statements hold: 
\\
$(a)$ For any connected component $U$ in $\mathrm{P}(v_{\mathrm{col}}) \setminus \mathop{\mathrm{Bd}}_{\widetilde{0+}}(v_{\mathrm{col}}) = \mathrm{P}(v) \setminus \mathop{\mathrm{Bd}}_{\widetilde{0+}}(v)$ and for any points $x,y \in U$, we have $\alpha_{v_{\mathrm{col}}}(x) = \alpha_{v_{\mathrm{col}}}(y) $ and $\omega_{v_{\mathrm{col}}}(x) = \omega_{v_{\mathrm{col}}}(y)$. 
\\
$(b)$ For any connected component $U$ in $\mathrm{LD}(v)$ and for any points $x,y \in U$, we have $\overline{O(x)} = \overline{O(y)} = \overline{U} \subseteq S$. 
\\
$(c)$ Each of the vertical $\alpha$- and $\omega$-vertical boundaries of any connected components of $S - \mathop{\mathrm{Bd}}_{\widetilde{0+}}(v_{\mathrm{col}})$ in the cases $(1)$ and $(2)$ is a finite union of closed orbits, proper semi-multi-saddle separatrices, and separatrices on $\partial S$ between $\partial$-sources and $\partial$-sinks as a subset of $S_{\mathrm{col}}$ unless the vertical boundary intersects $\mathrm{E}(v_{\mathrm{col}}) = \mathrm{E}(v)$. 
\\
$(d)$ Each of the left and right transverse boundaries of any connected components of $S - \mathop{\mathrm{Bd}}_{\widetilde{0+}}(v_{\mathrm{col}})$ in the cases $(1)$ and $(4)$ is an immersed line which is a finite union of closed orbits, proper semi-multi-saddle separatrices, and separatrices on $\partial S$ between $\partial$-sources and $\partial$-sinks as a subset of $S_{\mathrm{col}}$. 
\end{theorem}

\begin{theorem}\label{cor:bd12}
Each connected component of $S - \mathop{\mathrm{BD}}_{\widetilde{0+}}(v)$ for a flow $v$ with virtual sectored-like singular points on a surface $S$ with finite genus and finite ends is one of the following invariant open subsets exclusively:
\\
$(1)$
A trivial flow box in $\mathrm{P}(v)$, whose orbit space is an interval,
\\
$(2)$
An annulus in $\mathrm{P}(v)$, whose orbit space is a circle,
\\
$(3)$ 
A torus in $\mathop{\mathrm{Per}}(v)$, whose orbit space is a circle,
\\
$(4)$ 
An annulus in $\mathop{\mathrm{Per}}(v)$, whose orbit space is an interval, or
\\
$(5)$ 
An essential subset in $\mathrm{LD}(v)$, whose orbit class space is a singleton.
%
\end{theorem}

Notice the similar properties in Theorem~\ref{th:bd1}(a)--(d) hold in the previous theorem. 
The previous theorem implies the following description of flows of ``virtual weakly finite-like type''. 

\begin{corollary}\label{cor:vir_bd2}
Each connected component of $S - \mathop{\mathrm{BD}}_{\widetilde{0+}}(v)$ for a flow $v$ with virtual sectored-like singular points without non-closed recurrent points on a compact surface $S$ is one of the following invariant open subsets exclusively:
\\
$(1)$
A trivial flow box in $\mathrm{P}(v)$, whose orbit space is an interval,
\\
$(2)$
An annulus in $\mathrm{P}(v)$, whose orbit space is a circle,
\\
$(3)$ 
A torus in $\mathop{\mathrm{Per}}(v)$, whose orbit space is a circle, or
\\
$(4)$ 
An annulus in $\mathop{\mathrm{Per}}(v)$, whose orbit space is an interval. 

Moreover, if the flow $v$ has at most finitely many limit cycles, then the union $\mathop{\mathrm{BD}}_{\widetilde{0+}}(v_{\mathrm{col}})$ is the finite union of singular points, limit cycles, one-sided periodic orbits, and border separatrices of $v_{\mathrm{col}}$. 
\end{corollary}

\subsubsection{Flows of virtual weakly finite type}

A flow $v$ on a surface of finite genus and finite ends is of {\bf virtual weakly finite type} if $v_{\mathrm{end}}$ is of weakly finite type. 
Lemma~\ref{lem:blow_up}, Lemma~\ref{thm:enumerate}, Lemma~\ref{lem:end}, and Corollary~\ref{cor:bd2} imply the following statement. 

\begin{theorem}\label{thm:ext_enumerable}
The set of flows of virtual weakly finite type on surfaces with finite genus and finite ends is enumerable.
\end{theorem}

\appendix

\section{Characterization of centers}\label{appendix}

In this appendix, we will demonstrate Proposition~\ref{prop:char_center}, which follows from a technical lemma below. 
To state the technical lemma, we recall the concept of ``accumulated by periodic orbits'' as follows.
A singular point $x \in S$ is {\bf accumulated by periodic orbits} \cite{roussarie2021top} if for any \nbd $U$ of $x$ there is a periodic orbit in $U$. 
Notice that Roussarie \cite{roussarie2021top} introduced the concept to extend the study of germs of vector fields on the plane with an isolated singular point at the origin not accumulated by periodic orbits are studied by Bendixson \cite{bendixson1901courbes} and Dumortier \cite{dumortier1977singularities}. 
Then we state the technical lemma, which characterizes an isolated singular point accumulated by periodic orbits, which complements the studies of germs by Bendixson, Dumortier, and Roussarie. 

\begin{lemma}\label{lem:comp_center}
The following statements are equivalent for any isolated singular point $x$ of a flow $v$ on a surface $S$: 
\\
{\rm(1)} The point $x$ is accumulated by periodic orbits. 
\\
{\rm(2)} There is a sequence $(O_n)_{n \geq 0}$ of periodic orbits which bound invariant closed disks $B_n$ with $B_{n+1} \subset \mathop{\mathrm{int}} B_{n}$ such that $\{ x \} = \bigcap_{n \geq 0} B_n$, the intersection $B_n \cap \mathrm{P}(v)$ for any $n \geq 0$ is open, and the $\omega$-limit sets and $\alpha$-limit sets of any non-closed points in $B_0$ are limit cycles. 
\end{lemma}

\begin{proof}
By definition of ``accumulated by periodic orbits'', assertion (2) implies assertion (1). 
Therefore, it suffices to show that assertion (1) implies assertion (2). 

Let $v$ be a flow on a surface $S$. 
Suppose that an isolated singular point $x$ is accumulated by periodic orbits. 
Fix a small open ball $U$ of $x$ with $\Sv \cap U = \{ x \}$. 

\begin{claim}\label{claim:a02}
There is a sequence $(B_n)_{n \geq 0}$ of invariant closed disks in $U$ whose boundaries are perodic orbits with $\{ x \} = \bigcap_{n \geq 0} B_n$ such that $B_{n+1} \subset \mathop{\mathrm{int}} B_{n} \subset U$ for any $n \geq 0$. 
\end{claim}
\begin{proof}[Proof of Claim~\ref{claim:a02}]
Denote by $(O_n)_{n \geq 0}$ a sequence of periodic orbits and a sequence $(x_n)$ with $x_n \in O_n$ converging to $x$. 
Taking a subsequence of $(O_n)_{n \geq 0}$, we may assume that the sequence $(O_n)_{n \geq 0}$ consists of pairwise disjoint periodic orbits. 
Collapsing the outside of $U$ into a singular point if necessary, \cite[Gutierrez's smoothing theorem]{gutierrez1986smoothing} implies that the restriction $v\vert_U$ is topologically equivalent to a smooth flow. 
Taking a subsequence of $(O_n)_{n \geq 0}$, we may assume that $O_n \subset U$ for any $n \geq 0$. 
Since $U$ is homeomorphic to the plane, by Jordan-Schoenflies theorem, any periodic orbits $O_n \subset U$ bound closed disks in $U$, denoted by $B_n$. 
If there is an invariant closed disk $B_k$ which does not contain $x$, then the Poincar\'e-Hopf theorem implies there are singular points on $B_n \subset U$, which contradicts the isolatedness of the singular point $x$. 
Thus, any invariant closed disk $B_n$ contains $x$. 
For any $n \neq m \in \Z_{\geq 0}$, by $\partial B_n = O_n \neq O_m = \partial B_m$, since the sequence $(x_n)$ with $x_n \in O_n$ converges to $x$, either $B_n \subset \mathop{\mathrm{int}} B_m$ or $B_m \subset \mathop{\mathrm{int}} B_n$. 
Taking a subsequence of $(O_n)_{n \geq 0}$, we may assume that $B_{n+1} \subset \mathop{\mathrm{int}} B_{n} \subset U$ for any $n \geq 0$. 
Taking a subsequence of $(O_n)_{n \geq 0}$, we may assume that $\{ x \} = \bigcap_{n \geq 0} B_n$. 
\end{proof}

Replacing $U$ by an arbitrarily small invariant open disk in some $B_n$, we may assume that $U$ is an invariant open disk. 
Then the $\alpha$-limit or the $\omega$-limit sets of any points in $U - \{ x \}$ do not contain $x$. 
Since $B_n$ is an invariant closed disk, by a generalization of the Poincar\'e--Bendixson theorem and its dual statement, the non-existence of singular points in the invariant subset $U - \{ x \}$ implies that the $\omega$-limit sets and $\alpha$-limit sets of any points in $U - \{ x \}$ are periodic orbits. 
Therefore, the $\omega$-limit sets and $\alpha$-limit sets of any non-closed points in $B_n$ are limit cycles. 
Proposition~\ref{prop:cs}(3) implies that $\mathrm{P}(v) \cap B_n = B_n \setminus \Cv$ is open. 
\end{proof}

To prove Proposotion~\ref{prop:char_center}, we recall some concepts. 
A compact invariant subset $\mathcal{M}$ of a flow on a topological space $X$ is \emph{positively asymptotically stable} if, for any neighborhood $U$ of $\mathcal{M}$, there is a neighborhood $V$ of $\mathcal{M}$ with $\bigcup_{x \in V} O^+(x) \subseteq U$ such that $\{ y \in X \mid \omega(y) \subseteq \mathcal{M} \}$ is a \nbd of $\mathcal{M}$. 
Similarly, a compact invariant subset $\mathcal{M}$ of a flow on a topological space $X$ is \emph{negatively asymptotically stable} if, for any neighborhood $U$ of $\mathcal{M}$, there is a neighborhood $V$ of $\mathcal{M}$ with $\bigcup_{x \in V} O^-(x) \subseteq U$ such that $\{ y \in X \mid \alpha(y) \subseteq \mathcal{M} \}$ is a \nbd of $\mathcal{M}$, where $O^-(x) := \{ v_t(x) \mid t \leq 0 \}$.

\begin{proof}[Proof of Proposition~\ref{prop:char_center}]
If $x$ is a center, then there is an open disk $V$ with $V - \{ x \} \subseteq \Pv$. 
This means that assertion {\rm(1)} implies assertion {\rm(2)}. 
 
Conversely, suppose that there is an open disk $V$ with $V - \{ x \} \subseteq \Pv$. 
\begin{claim}\label{claim:a03}
There is no \nbd $U \subset V$ of $x$ which does not contain any closed invariant subsets except the singleton $\{ x \}$. 
\end{claim}
\begin{proof}[Proof of Claim~\ref{claim:a03}]
Assume that there is such a neighborhood. 
By results in \cite{taro1962errata,ura1964flow} (cf. \cite[Theorem 1.6]{bhatia1970attraction} or \cite[Theorem]{egawa1973remark}), since $\{ x \}$ is neither positively asymptotically stable nor negative asymptotically stable, the singular point $x$ is an $\omega$-limit or $\alpha$-limit set of a non-singular point. 
This implies that $V$ contains such non-closed points, which contradict $V \subseteq \Cv$. 
\end{proof}

Since $V - \{ x \} \subseteq \Pv$, for any \nbd $U$ of $x$, Claim~\ref{claim:a03} implies that there is a periodic orbit contained in $U$.  
This means that $x$ is accumulated by periodic orbits. 
Lemma~\ref{lem:comp_center} implies that 
there is an invariant closed disk $U \subseteq V \subseteq \{ x \} \sqcup \Pv$. 
By the Poincar\'e-Hopf theorem, any periodic orbit in $U$ bounds an invariant closed disk in $U$ which contains singular points and so the unique singular point $x$ in $U$. 
Therefore $x \in \mathop{\mathrm{int}} U$. 
Replacing $U$ by any disk bounded by a periodic orbit in $U$, we may assume that $\partial U$ is a periodic orbit. 
Since any periodic orbit bounds an invariant closed disk in $U$ containing $x$, such invariant closed disks in $U$ whose boundaries are periodic orbits are totally ordered with respect to the inclusion order. 
This means that the orbit space $U/v$ is a closed interval and so $x$ is a center. 
\end{proof}

Notice that the conditions in Lemma~\ref{lem:comp_center} are necessary. 
In fact, the isolated properties in Lemma~\ref{lem:comp_center} are necessary. 
Indeed, there is a non-isolated singular point that is not a center and has an invariant \nbd consisting of closed orbits. 
Consider the flow $v$ on a torus $\R^2/\Z^2$ generated by a vector field $(\sin (2\pi x), 0)$. 
Then $\Sv = \{ [0], [1/2] \} \times \R/\Z$ and $\R^2/\Z^2 = \Sv \sqcup \Pv$. 
Therefore, the point $[(0,0)]$ is a non-isolated singular point that is not a center and has an invariant \nbd consisting of closed orbits. 

Moreover, the condition for limit cycles in Lemma~\ref{lem:comp_center} is necessary. 
Indeed, for any flow on a sphere with exactly two singular points, which are one sink and one source, and at least infinitely many periodic orbits, there is a sequence in Lemma~\ref{lem:comp_center} except for the condition for limit cycles. 

\section{Proof of Ma{\v \i}er's result}\label{M_result}

We state a proof Ma{\v \i}er's result as follows, which is based on the sketch of a proof of  \cite[Theorem 4.2]{aranson1996maier}. 

\begin{proof}[Proof of Lemma~\ref{pos_rec}{\rm(Ma{\v \i}er)}]

Since $\omega(x) \setminus  \mathop{\mathrm{Cl}}(v) \neq \emptyset$, the point $x$ is not closed. 
If $O(x) = O(z)$, then $x \in \omega(z) = \omega(x)$ is non-closed positively recurrent. 

\begin{claim}\label{claim:01}
If $O(x) = O(z)$, then $x$ is non-closed positively recurrent. 
\end{claim}

\begin{proof}[Proof of Claim~\ref{claim:01}]
Suppose that $O(x) = O(z)$. 
By $x \in \omega(z) = \omega(x)$, the point is positively recurrent. 
Since $x$ is not closed, the assertion holds. 
\end{proof}

Thus, we may assume that $O(x) \neq O(z)$. 
Fix a non-closed point $y \in \alpha(x) \setminus \mathop{\mathrm{Cl}}(v)$. 
Then there is a transverse closed arc $I_{[-1,1]}: [-1,1] \to S$ with $y = I_{[-1,1]}(0)$ such that the negative orbit $O^+(x)$ intersects $I_{[-1,1]}([-1,0])$ infinitely many times. 
Denote by $I := I_{[-1,1]}([-1,0])$ a directed closed interval. 
Therefore there is a sequence $(x_i)_{i \in \Z_{\geq 0}}$ of points in $O^+(x) \cap I$ with $x_{i+1} \in O^+(x_i)$ which converges to $y$ monotonically from one side. 
Denote by $I_{a,b}$ the closed sub-arc in $I$ whose boundary consists of $a$ and $b$ for any points $a, b \in I$ and by $C_{a,b}$ the closed orbit arc in an orbit $O$ from $a$ to $b$ for any points $a, b \in O \cap I$. 

Assume that $x$ is not positively recurrent (i.e. $x \notin \omega(x)$). 
Then there is an open sub-arc $J$ in $I$ with $\{ x_2 \} = J \cap O^+(x)$. 
By $x_2 \in \omega(z)$, the first return map $f_{v,J}$ on $J$ induced by $v$ is well-defined and injective. 
From the finiteness of genus of $S$, by replacing $x$ with a point of $O^+(x)$, 
we may assume that the restriction of the first return map $f_{v,I}$ on the transverse closed arc $I$ induced by $v$, to a \nbd of $f_{v,I}^{-1}(O^+(x) \cap I)$ in $I$ is orientation-preserving. 
Therefore, $I$ and $O^+(x)$ intersect in a same orientation infinitely many times. 
\begin{claim}\label{claim:02}
We can define a strictly increasing subsequence $(n_i)_{i \in \Z_{\geq 0}}$ of $\Z_{\geq 0}$ with $n_i +3 \leq n_{i+1}$ and a sequence $(z_i)_{i \in \Z_{\geq 0}}$ of $J \cap O^+(z)$ with $z_{i+1} \in O^+(z_i)$ converging to $x_2$ monotonically from one side in $J$ such that $C_{z_{i-1},z_{i}} \cap \mathop{\mathrm{int}} I_{z_{i-1}, z_{i}} \neq \emptyset$, $C_{z'_{i}, z_{i}} \cap \mathop{\mathrm{int}} I_{x_{n_{i}},x_{n_{i}+1}} \neq \emptyset$, and $C_{z_0, z'_{i}} \cap I_{x_{n_{i}}, y} = \emptyset$ for any $i \in \Z_{>0}$, where $z'_{i} \in \mathop{\mathrm{int}} C_{z_{i-1},z_{i}} \cap I_{z_{i-1},z_{i}}$ is the first return image of $z_{i}$ into $I_{z_{i-1},z_i}$ induced by the time reversed flow of $v$. 
\end{claim}

\begin{proof}[Proof of Claim~\ref{claim:02}]
By induction, fix a point $z_0 \in J \cap O^+(z)$ and $n_0 = 0$ such that $O^+(z)$ intersects $I_{z_0, x_2}$ infinitely many times. 
Since the sequence $(x_k)_{k \in \Z_{\geq 0}}$ converges to $y$ monotonically from one side, by $O(x) \neq O(z)$, for any $i \in \Z_{\geq 0}$, there are an integer $k_i \geq 3$ and a point $z_{i+1} \in I_{z_i, x_2} \cap O^+(z_i)$ with $C_{z_{i},z_{i+1}} \cap \mathop{\mathrm{int}} I_{z_{i}, z_{i+1}} \neq \emptyset$, $C_{z_0,z'_{i+1}} \cap \mathop{\mathrm{int}} I_{x_{n_{i}+k_i}, y} = \emptyset$, and $C_{z'_{i+1}, z_{i+1}} \cap I_{x_{n_{i}+k_i}, y} \neq \emptyset$. 
Fix an integer $n_{i+1} \geq n_i + k_i \geq n_i +3$ such that $C_{z'_{i+1}, z_{i+1}} \cap \mathop{\mathrm{int}} I_{x_{n_{i+1}}, x_{n_{i+1}+1}} \neq \emptyset$ and $C_{z_0, z'_{i+1}} \cap I_{x_{n_{i+1}}, y} = \emptyset$. 
Then $C_{z_{i-1},z_{i}} \cap \mathop{\mathrm{int}} I_{z_{i-1}, z_{i}} \neq \emptyset$, $C_{z'_{i}, z_{i}} \cap \mathop{\mathrm{int}} I_{x_{n_{i}}, x_{n_{i}+1}} \neq \emptyset$, and $C_{z_0, z'_{i}} \cap I_{x_{n_{i}}, y} = \emptyset$ for any $i \in \Z_{>0}$. 
\end{proof}

Fix a Riemannian metric $g$ on $S$ which induces the Riemannian distance $d_g$. 
Since the sequence $(z_i)_{i \in \Z_{\geq 0}}$ of $J \cap O^+(z)$ converging to $x_2$ monotonically from one side, the sequence of the lengths of $I_{z_{i+1},z_i}$ converges to zero. 
For any $i \in \Z_{\geq 0}$, let $f_{v,I_{z_i, x_2}}$ be the first return map from $I_{z_i, x_2}$ to $I_{z_{i+1}, x_2} \subset I_{z_i, x_2}$ induced by $v$. 
Then $C_{z'_{i+1}, z_{i+1}} \cap I_{z_{i}, x_2} = \{ z'_{i+1}, z_{i+1} \}$ and $C_{z'_{i+1}, z_{i+1}} \cap \mathop{\mathrm{int}} I_{x_{n_{i+1}},x_{n_{i+1}+1}} \neq \emptyset$. 
Since $\mathop{\mathrm{int}} C_{z_i,z_{i+1}} \cap I_{z_i, z_{i+1}} \neq \emptyset$, we have that $z_i \neq z'_{i+1}$ and so that the closed intervals $I_{z'_{i+1}, z_{i+1}} \subset J$ are pairwise disjoint. 
Therefore the unions $\gamma_i := C_{z'_{i+1}, z_{i+1}} \cup I_{z'_{i+1}, z_{i+1}} \subset O^+(z) \cap J$ are pairwise disjoint loops intersecting $\mathop{\mathrm{int}} I_{x_{n_{i+1}},x_{n_{i+1}+1}}$. 
Let $\A_i$ be the connected component of $S - \bigcup_{k \in \Z_{\geq 0}} \gamma_k$ intersecting $I_{z'_{i+1},z_i}$. 

\begin{claim}\label{claim:03}
We may assume that $\A_i$ is a closed annulus whose boundary is a disjoint union $\gamma_i \sqcup \gamma_{i+1} \subset O^+(z_0) \cup J$ such that the pairwise disjoint loops $\gamma_i$ are homotopic to each other.
\end{claim}
\begin{proof}[Proof of Claim~\ref{claim:03}]
Then the boundary of any domain $\A_i$ is contained in $(O^+(z_0) \cup J) \sqcup \partial S$. 
Since there are at most finite genus and finitely many boundary components, by renumbering, we may assume that each domain $\A_i$ is annular and that the restriction of $f_{v,I_{z_i, x_2}}$ whose domain is a small neighborhood of $z'_{i+1} \in I_{z_i, z_{i+1}}$ and codomain is a small neighborhood of $z_{i+1}$ is orientation-preserving. 
Then $\A_i$ is a closed annulus whose boundary is a disjoint union $\gamma_i \sqcup \gamma_{i+1} \subset O^+(z_0) \cup J$. 
Since $S$ is compact, by renumbering, we may assume that the pairwise disjoint loops $\gamma_i$ are homotopic to each other.
\end{proof}

Then the union $\A_{i-1} \cup \A_{i}$ is also a closed annulus with $\A_{i-1} \cap \A_{i} = \gamma_i$. 
Denote by $d_0 > 0$ the distance between $\gamma_0$ and $\gamma_1$ in $\A_0$ (i.e. $d_0 := d_g(\gamma_0,\gamma_1)$, where $d_g(A,B) := \min_{a \in A, b \in B}d_g(q_0,q_1)$).

Fix a large integer $N \in \Z_{>2}$ such that the length of $I_{x_{n_{i-1}},x_{n_{i+1}}}$ is less than $d_0/2$ for any $i \geq N$. 
Then $x_2 \notin O^+(x_{n_N +1})$. 
Since $\gamma_i \cap O^+(x_2) \subset (O^+(z) \cup J) \cap O^+(x_2) = \emptyset$ for any $i \geq N$, put $D := \min \{ d_g(x_{n_N +1}, \partial \A_N), d_g(x_{n_{N+1} +1}, \partial \A_{N+1}) \} = \min \{ d_g(\{ x_{n_N +1}\} , \gamma_N \sqcup \gamma_{N+1}), d_g(\{ x_{n_{N+1} +1}\} , \gamma_{N+1} \sqcup \gamma_{N+2}) \} >0$. 

For any $i \in \Z_{\geq 0}$, applying the waterfall construction to the loop $\gamma_i$, there is a closed transversal $T_i$ isotopic to $\gamma_i$ with $x_{n_{i}},x_{n_{i}+1} \notin T_i$ such that $T_i$ intersects $\mathop{\mathrm{int}} I_{x_{n_{i}},x_{n_{i}+1}}$ transversely and $d_{H}(T_i, \gamma_i) < \min \{D, d_0, d_g(\gamma_{i-1}, \gamma_{i}), d_g(\gamma_{i}, \gamma_{i+1}) \}/4$, where $d_H$ is the Hausdorff distance. 
For any $i \in \Z_{\geq 0}$, denote by $\A'_i$ the closed annulus whose boundary is $T_i \sqcup T_{i+1}$ and which is near $\A_i$. 
Then the union $\A'_{i} \cup \A'_{i+1}$ is also a closed annulus with $\A'_{i} \cap \A'_{i+1} = T_{i+1}$ and $\partial (\A'_{i} \cup \A'_{i+1}) = T_{i} \sqcup T_{i+2}$. 

\begin{claim}\label{claim:ineq_001}
$d_g(x_{n_N +1}, \partial \A'_N) \geq 3D/4$. 
\end{claim}
\begin{proof}[Proof of Claim~\ref{claim:ineq_001}]
We have the following inequality: 
\[
\begin{split}
& \,\, d_g(x_{n_N +1}, \partial \A'_N) =  d_g(x_{n_N +1}, T_N \sqcup T_{N+1}) 
\\
= & \, \min \{ d_g(x_{n_N +1}, T_N), d_g(x_{n_N +1}, T_{N+1})\} \\
\geq & \,  \min \{ d_g(x_{n_N +1}, \gamma_N) - d_H(T_N,\gamma_{N}), d_g(x_{n_N +1}, \gamma_{N+2}) - d_H(T_{N+2},\gamma_{N+2})\} 
\\
\geq & \,  D - D/4 = 3D/4
\end{split}
\]
\end{proof}

\begin{claim}\label{claim:04}
The closed transversal $T_i$ intersects exactly once $I_{x_{n_{i}},x_{n_{i}+1}}$ at a point in $I_{x_{n_{i}},x_{n_{i}+1}}$ for any $i \in \Z_{\geq N}$. 
\end{claim}
\begin{proof}[Proof of Claim~\ref{claim:04}]
Assume that $T_i$ intersects $I_{x_{n_{i}},x_{n_{i}+1}}$ at least twice. 
Since $\A'_i$ is a closed annulus with $\partial \A'_i = T_i \sqcup T_{i+1}$ such that $T_i$ and $T_{i+1}$ are closed transversals, the transverse closed arc $I_{x_{n_{i}},x_{n_{i}+1}}$ goes outside of $\A'_i$ and goes into $\A'_i$ from $\gamma_i$ with respect to the positive or negative direction. 
The fact that the union $\bigcup_{k=0}^i \A'_k$ is a closed annulus whose boundary components are closed transversals implies that  $I_{x_{n_{i}},x_{n_{i}+1}} \cap T_k \neq \emptyset$ for any $k =0, 1, \ldots , i$. 
Since the transverse closed arc $I_{x_{n_{i}},x_{n_{i}+1}}$ goes through $\A'_0$, it contains a sub-arc in $\A'_0$ whose boundary component consists of a point in $T_0$ and a point in $T_1$. 
Then the length of $I_{x_{n_{i}},x_{n_{i}+1}}$ is more than $d_0/2$, which contradicts that the length is less than $d_0/2$. 
\end{proof}

By the previous claim, we have that $x_{n_{N}}< T_N \cap I_{x_{n_{N}},x_{n_{N}+1}} < x_{n_{N}+1} < x_{n_{N+1}} < T_{N+1} \cap I_{x_{n_{N+1}},x_{n_{N+1}+1}} < x_{n_{N+1}+1} < x_{n_{N+2}}$ in the closed interval $I$. 

\begin{claim}\label{claim:05}
$x_{n_N +1} \in \A_N$. 
\end{claim}
\begin{proof}[Proof of Claim~\ref{claim:05}]
Since $\A'_{N}$ is a closed annulus with $\partial \A'_{N} = T_{N} \sqcup T_{N+1}$, by $T_N \cap I_{x_{n_{N}},x_{n_{N}+1}} < x_{n_{N+1}} < T_{N+1} \cap I_{x_{n_{N+1}},x_{n_{N+1}+1}}$, we obtain that $x_{n_N +1} \in \A'_N$.
By Claim~\ref{claim:ineq_001}, we have the following inequalities: 
\[
\begin{split}
d_g(x_{n_N +1}, \partial \A'_N) \geq 3D/4 > D/4 &> \max \{ d_{H}(T_N, \gamma_N), d_{H}(T_{N+1}, \gamma_{N+1})\} 
\\
&\geq d_H(\partial \A'_N, \partial \A_N) \geq d_H(\A'_N, \A_N)
\end{split}
\]
Since the boundary $\partial \A'_N = T_N \sqcup T_{N+1}$ is isotopic to $\partial \A_N = \gamma_N \sqcup \gamma_{N+1}$,  the annulus $\A'_N$ is isotopic to the annulus $\A_N$ with $d_H(\partial \A'_N, \partial \A_N)< D/4 < d_g(x_{n_N +1}, \partial \A'_N)$ and so $x_{n_N +1} \in \A_N$, because the annulus $\A'_N$ contain $x_{n_N +1}$.
\end{proof}
By the same argument of the proof of the previous claim, we have $x_{n_{N+1} +1} \in \A_{N+1}$. 
From $O^+(x_{n_N +1}) \subseteq O^+(x)$ and $x_{n_{N+1} +1} \in O^+(x_{n_N +1}) \setminus \A_N$, the positive orbit $O^+(x_{n_N +1})$ intersects $\A_N$ but is not contained in $\A_N$. 
By $O^+(x_{n_N +1}) \cap (\bigcup_{k} C_{z'_{k+1}, z_{k+1}}) \subseteq O(x) \cap O(z) = \emptyset$ and $\partial \A_N = \gamma_N \sqcup \gamma_{N+1} \subset O^+(z) \cup (I_{z'_{N}, z_{N}} \sqcup I_{z'_{N+1}, z_{N+1}})$, we have $\emptyset \neq O^+(x_{n_N +1}) \cap \partial \A_N = O^+(x_{n_N +1}) \cap (\gamma_N \sqcup \gamma_{N+1})  = O^+(x_{n_N +1}) \cap (I_{z'_{N}, z_{N}} \sqcup I_{z'_{N+1}, z_{N+1}}) \subset O^+(x) \cap J = \{ x_2 \}$, which contradicts $x_2 \notin O^+(x_{n_N +1})$. 
Thus, the point $x$ is positively recurrent. 
\end{proof}

\section{Definition of $\leq^{\mathrm{BD}(v)}$}

We define the path orders and the circuit orders for any connected component of the complement $S - \mathrm{BD}(v)$. 

\subsection{Path orders $\leq_{\partial_{\pitchfork}^R}$ and $\leq_{\partial_{\pitchfork}^L}$ on a path on the transverse boundary of an invariant trivial flow box}

Let $\mu_R := (0,1) \times \{ 0 \}$ be an open interval, $\mathbb{D} := (0,1) \times (0,1)$ a canonical invariant trivial flow box whose orbits are of form $(0,1) \times \{ t \}$, $U$ an oriented open trivial flow box, and $h_R: (0,1) \times [0,1) = \mathbb{D} \sqcup \mu_R \to U \sqcup \partial_{\pitchfork}^R U$ a mapping such that the restriction $h_R\vert_{\mathbb{D}} : \mathbb{D} \to U$ is a topologically equivalent homeomorphism and the restriction $h_R\vert_{\mu_R} : \mu_R \to \partial_{\pitchfork}^R U$ is an injection except singular points.  
Since an injection except singular points lifts orbits into the domain, the interval $\mu_R$ can be considered as a union of orbits. 
Note that the restriction $h_R\vert_{\mu_R}$ need not be injective. 

Suppose that the interval $\mu_R$ consists of finitely many orbits. 
\begin{definition}
Define an equivalence relation $\sim_{\partial_{\pitchfork}^R U}$ on $\mu_R$ as follows: $x \sim y$ if there is a connected component $J$ of $h_R^{-1}(O)$ for some orbit $O$ in $\partial_{\pitchfork}^R U$ such that $J$ contains $x$ and $y$. 
\end{definition}

Then $\mu_R/\sim_{\partial_{\pitchfork}^R U}$ is a finite set. 

\begin{definition}
The total order $(J_1, J_2, \ldots , J_n)$ (i.e. $J_1 < J_2 < \cdots < J_n$) on the finite set $\mu_R/\sim_{\partial_{\pitchfork}^R U}$ is defined by the transitive closure of the following relation $J' \prec J''$ if either $\omega(J') = J''$ or $J' = \alpha(J'')$.

We call a sequence $(h_R(J_1), h_R(J_2), \ldots , h_R(J_n))$ of orbits in $\partial_{\pitchfork}^R U$ a {\bf path order} of $\partial_{\pitchfork}^R U$ and denoted by $\bm{\leq_{\partial_{\pitchfork}^R}(U)}$ (or $\bm{\leq_{\partial_{\pitchfork}^R}}$ with respect to $U$) (i.e. $h_R(J_1) \leq_{\partial_{\pitchfork}^R}  h_R(J_2) \leq_{\partial_{\pitchfork}^R}  \cdots \leq_{\partial_{\pitchfork}^R} h_R(J_n)$ with respect to $U$).
\end{definition}

Similarly, we define a path order $\leq_{\partial_{\pitchfork}^L}(U)$ of the left transverse boundary $\partial_{\pitchfork}^L U$.
Note that path orders need not be total orders because $h_R\vert_{\mu_R}$ is not injective in general, and that we can similarly define a path order unless the assumption that $\mu_R$ consists of finitely many orbits. 

\subsection{Cyclic orders $\leq_{\partial_{\perp}^{\alpha}}$ and $\leq_{\partial_{\perp}^{\omega}}$ for the vertical boundaries of canonical regions}\label{sec:leq_perp}

We define a cyclic order as follows. 

\subsubsection{Cyclic orders}
We define a cyclic relation as follows. 

\begin{definition}
Define a {\bf cyclic relation} $\sim_c$ on a finite set $F$ with $n$ points as follows:
$(O_1, O_2, \ldots , O_n) \sim_c (O_{i_1}, O_{i_2}, \ldots , O_{i_n})$ if $F = \{ O_1, O_2, \ldots , O_n\}$ and  there is an integer $k = 0, 1, \ldots , n-1$ such that $j - i_j \equiv k \mod n$ for any $j = 1, \ldots , n$.
Then an equivalence class is called a {\bf cyclic class} of $F$ and denoted by $\bm{[O_1, O_2, \ldots , O_n]}$.
\end{definition}

If a cycle class consists of one element, then it is called a trivial cycle class. 
Notice that the empty set is also a cycle class by definition. 
Note that a cyclic class of a $n$-point set corresponds to a $n$-point subset of an oriented circle.

\subsubsection{Trivial circuit orders}

We define trivial circuit orders as follows. 

\begin{definition}
For any transverse annulus $U$, define $\bm{\leq_{\partial_{\pitchfork}^{\alpha}}(U)}$ (resp. $\bm{\leq_{\partial_{\pitchfork}^{\omega}}(U)}$) as the empty set. 
\end{definition}

\begin{definition}
For any periodic annulus $U$, define $\bm{\leq_{\partial_{\perp}^{\alpha}}(U)}$ (resp. $\bm{\leq_{\partial_{\perp}^{\omega}}(U)}$) as the empty set. 
\end{definition}

\begin{definition}
For any ($\partial$-)sink (resp. ($\partial$-)source) $x$ with a basin $\A$, the trivial cycle class of the singular point is called a {\bf trivial circuit order} and denoted by $\leq_{\A}(x)$. 
\end{definition}

\subsubsection{A circuit order on semi-attracting or semi-repelling limit sets}

Let $\gamma$ be a multi-saddle circuit with a collar $\A = \mathbb{S}^1 \times (0,1)$ and an associated immersion $h\vert_{\mathbb{S}^1 \times \{ 0 \}} : \mathbb{S}^1 \times \{ 0 \} \to \gamma$ except singular points.  
Then the composition $p_{\partial} := h \circ p^{-1} : \mathbb{S}^1 \to \gamma$ is an injection except singular points, called a {\bf multi-saddle circuit} with the spiral direction for $\A$, where $p: \mathbb{S}^1 \times \{ 0 \} \to \mathbb{S}^1$ is the canonical projection. 
Since an injection except singular points lifts orbits into the domain, the circle $\mathbb{S}^1 \times \{ 0 \}$ can be considered as a union of orbits. 

Suppose that the circuit $\gamma$ consists of finitely many orbits. 

\begin{definition}
Define an equivalence relation $\sim_{\A}$ on $\mathbb{S}^1$ as follows: $x \sim_{\A} y$ if there is a connected component $I$ of $p_{\partial}^{-1}(O)$ for some orbit $O$ in $\gamma$ such that $I$ contains $x$ and $y$. 
\end{definition}

Then $\mathbb{S}^1/\sim_{\A}$ is a finite set

\begin{definition}
The induced cyclic order $[J_1, J_2, \ldots , J_n]$  (i.e. $J_1 < J_2 < \cdots < J_n < J_0$) on the finite set $\mathbb{S}^1/\sim_{\A}$ is defined by either $\omega(J_i) = J_{i+1}$ or $J_i = \alpha(J_{i+1})$ for any $i \in \{0, \ldots , n-1 \}$, where $J_0 := J_n$. 

Consider a sequence $(p_{\partial}(J_1), p_{\partial}(J_2), \ldots , p_{\partial}(J_n))$ of orbits in $\gamma$ and define an equivalence relation $\approx_{\A}$ as follows: 
\[
(p_{\partial}(J_1), p_{\partial}(J_2), \ldots , p_{\partial}(J_n)) \approx_{\A} (p_{\partial}(J_{i_1}), p_{\partial}(J_{i_2}), \ldots , p_{\partial}(J_{i_n}))
\]
 if there is an integer $k = 0, 1, \ldots , n-1$ such that $j+k = i_j$ for any $j = 1, \ldots , n$, where $J_{n + l} := J_{l}$ for $l = 0, \ldots , n-1$.
We call an equivalence class a {\bf circuit order} of $\gamma$ and denoted by $\bm{\leq_{\A}(\gamma)}$. 
 \end{definition}


Note that a circuit class is a quotient space of a cyclic order. 
We define cyclic orders for the vertical boundaries of canonical regions.

\begin{definition}\label{def:leq_perp}
For any canonical region $U$ which intersects the semi-repelling (reps. semi-attracting) basin $\A$ of a limit circuit $\gamma$, define $\bm{\leq_{\partial_{\perp}^{\alpha}}(U)}$ (resp. $\bm{\leq_{\partial_{\perp}^{\omega}}(U)}$) as $\leq_{\A}(\gamma)$. 
\end{definition}

\subsubsection{Cyclic orders for the transverse boundaries of periodic annuli}\label{sec:leq_perp_trans}

We define cyclic orders for the transverse boundaries of periodic annuli. 

\begin{definition}\label{def:leq_trans}
For any periodic annulus $U$ whose right (resp. left) transverse boundary is a circuit $\gamma$, 
define $\bm{\leq_{\partial_{\pitchfork}^{R}}(U)}$ (resp. $\bm{\leq_{\partial_{\pitchfork}^{L}}(U)}$) as the circuit order $\leq_{U}(\gamma)$. 
\end{definition}

Notice that we can define a circuit order unless the assumption that $\gamma$ consists of finitely many orbits in a similar way. 

\begin{definition}\label{def:leq_bd}
We define a lebel $\bm{\leq^{\mathrm{BD}(v)}}$ by 
\[
\leq^{\mathrm{BD}(v)} := (\leq_{\partial_{\pitchfork}^R}, \leq_{\partial_{\pitchfork}^L}, \leq_{\partial_{\perp}^{\alpha}}, \leq_{\partial_{\perp}^{\omega}}) : V^{\mathrm{BD}(v)} \to   \mathcal{O}_{\mathrm{circuit,path}}^4
\]
where $\mathcal{O}_{\mathrm{circuit,path}}$ is the set of finite circuit orders and path orders. 
\end{definition}

\bibliographystyle{abbrv}
\bibliography{../y202408}

\end{document}